\documentclass{amsart}
\usepackage{amsmath,ifthen,amsthm,amsopn,amssymb,amsfonts}
\usepackage{amscd}
\usepackage[all,cmtip]{xy}
\usepackage{amsmath,amscd,xy,graphicx,textcomp}
\usepackage{amssymb,latexsym}
\usepackage{color}
\usepackage{amsthm}
\usepackage{eucal}

\newcommand{\showcomments}{yes}
\renewcommand{\showcomments}{no}

\newsavebox{\commentbox}
\newenvironment{com}%
{\ifthenelse{\equal{\showcomments}{yes}}%
{\footnotemark
        \begin{lrbox}{\commentbox}
        \begin{minipage}[t]{1.25in}\raggedright\sffamily\tiny
        \footnotemark[\arabic{footnote}]}
{\begin{lrbox}{\commentbox}}}%
{\ifthenelse{\equal{\showcomments}{yes}}%
{\end{minipage}\end{lrbox}\marginpar{\usebox{\commentbox}}}
{\end{lrbox}}}
\newenvironment{dedication}{\vspace*{1mm}\begin{quotation} \begin{center} \begin{em}}
{\par\end{em}\end{center}\end{quotation}}

\title{The Hodge conjecture and arithmetic quotients of complex balls}

\author{Nicolas Bergeron} 
\thanks{N.B. is a member of the Institut Universitaire de France.}

\address{Institut de Math\'ematiques de Jussieu \\
Unit\'e Mixte de Recherche 7586 du CNRS \\
Universit\'e Pierre et Marie Curie \\
4, place Jussieu 75252 Paris Cedex 05, France \\}
\email{bergeron@math.jussieu.fr}
\urladdr{http://people.math.jussieu.fr/~bergeron}

\author{John Millson}
\thanks{J.M. was partially supported by NSF grant DMS-1206999}
\address{Department of Mathematics\\
University of Maryland\\
College Park, Maryland
20742, USA }
\email{jjm@math.umd.edu}
\urladdr{http://www-users.math.umd.edu/~jjm/}

\author{Colette Moeglin}
\address{Institut de Math\'ematiques de Jussieu \\
Unit\'e Mixte de Recherche 7586 du CNRS \\
4, place Jussieu 75252 Paris Cedex 05, France \\}
\email{moeglin@math.jussieu.fr}
\urladdr{http://www.math.jussieu.fr/~moeglin}

 \DeclareFontFamily{OT1}{rsfs}{}
 
\DeclareFontShape{OT1}{rsfs}{n}{it}{<-> rsfs10}{}
\DeclareMathAlphabet{\mathscr}{OT1}{rsfs}{n}{it}

\newcommand{\C}{\mathbb{C}}

\newcommand{\Z}{\mathbb{Z}}

\newcommand{\Ad}{\mathrm{Ad}}

\newcommand{\p}{\mathfrak{p}}

\newcommand{\g}{\mathfrak{g}}
\newcommand{\Mp}{\mathrm{Mp}}

\renewcommand{\H}{\mathbb{H}}

\DeclareFontFamily{OT1}{rsfs}{}

\DeclareFontShape{OT1}{rsfs}{n}{it}{<-> rsfs10}{}
\DeclareMathAlphabet{\mathscr}{OT1}{rsfs}{n}{it}
\newcommand{\Q}{\mathbb{Q}}
\newcommand{\Hom}{\mathrm{Hom}}

\renewcommand{\k}{\mathfrak{k}}

\newcommand{\R}{\mathbb{R}}

\newcommand{\Sp}{\mathrm{Sp}}

\newcommand{\Pol}{\mathrm{Pol}}
\newcommand{\Harm}{\mathrm{Harm}}

\newcommand{\Aut}{\mathrm{Aut}}

\swapnumbers
\newtheorem{thm}[subsection]{Theorem}  
\newtheorem{lem}[subsection]{Lemma}         
\newtheorem*{lem*}{Lemma}         
\newtheorem{prop}[subsection]{Proposition}
\newtheorem*{prop*}{Proposition}

\newtheorem{cor}[subsection]{Corollary}
\newtheorem*{thm*}{Theorem}

\theoremstyle{definition}
\newtheorem{defn}[subsection]{Definition}   
\newtheorem*{rem}{Remark}

\numberwithin{equation}{subsection}

\newcommand{\A}{\mathbb A}

\newcommand{\cal}{\mathcal}
\newcommand{\GL}{\mathrm{GL}}
\newcommand{\SL}{\mathrm{SL}}

\newcommand{\End}{\mathrm{End}}
\newcommand{\SU}{\mathrm{SU}}
\newcommand{\U}{\mathrm{U}}

\def\adots{\mathinner{\mkern2mu\raise1pt\hbox{.}
\mkern3mu\raise4pt\hbox{.}\mkern1mu\raise7pt\hbox{.}}}

\begin{document}

\begin{com}
{\bf \normalsize COMMENTS\\}
ARE\\
SHOWING!\\
\end{com}

\maketitle

\begin{dedication}

\vspace*{1mm}
{In memory of Raquel Maritza Gilbert beloved wife of the second author.}
\end{dedication}

\begin{abstract}  
Let $S$ be a closed Shimura variety uniformized by the complex $n$-ball. The Hodge conjecture 
predicts that every Hodge class in $H^{2k} (S , \Q)$, $k=0, \ldots , n$, is algebraic. We show 
that this holds for all degrees $k$ away from the neighborhood $]n/3 , 2n/3[$ of the middle degree. 
We also  prove the Tate conjecture for the same degrees as the Hodge conjecture and the generalized form of the Hodge conjecture in degrees away from an interval (depending on the codimension $c$ of the subvariety)  centered at the middle dimension of S.  Finally we extend most of these results to Shimura varieties associated to unitary groups of any signature.
The proofs make use of the recent endoscopic classification of automorphic representations of classical groups by \cite{ArthurBook,Mok}.  
As such our results are conditional on the
stabilization of the trace formula for the (disconnected) groups $\GL (N) \rtimes \langle \theta \rangle$ associated to base change. At present the stabilization of the trace formula has been proved only for the case of {\it connected} groups. But the extension needed is now announced, see \S \ref{org2} for more details.
\end{abstract}

\tableofcontents

\section{Introduction}

We shall consider {\it Shimura varieties} $S = S( \Gamma ) = \Gamma \backslash X$ associated to 
standard unitary groups. We first recall their construction.

\subsection{} Let $F$ be a totally real field and $E$ be an imaginary quadratic extension of $F$.
Let $V_E$ be a vector space defined over $E$ and let $(,)$ be a non-degenerate Hermitian form
on $V_E$. Suppose that $(,)$ is anisotropic, of signature $(p,q)$ at one Archimedean place and positive
definite at all other infinite places. Let $V$ be the real vector space whose points are the real points
of vector space obtained from $V_E$ by restricting scalars from $F$ to $\Q$. We then have an isomorphism of real vector spaces:
$$V \cong \oplus_{j=1}^{\mu} V^{(j)},$$
where the $V^{(j)}$'s are the completions of $V_E$ w.r.t. the different complex embeddings $\sigma_j$ of $E$ considered up to complex conjugation. We furthermore assume that the Hermitian form $(,)_1$ induced by $(,)$ on $V^{(1)}$ has signature $(p,q)$ and therefore that the Hermitian forms $(,)_j$ induced by $(,)$ on the $V^{(j)}$, $j\geq 2$, are all positive definite. By a slight abuse of notations we will denote by $(,)$ the corresponding Hermitian form on $V$. It is the orthogonal sum 
of the $(,)_j$. We finally denote by $G^{(j)}$ the subgroup of ${\rm Aut} (V^{(j)})$
which consist of isometries of $(,)_j$ and we set $G= \prod_j G^{(j)}$. The group $G$ is isomorphic
to the group of real points of the $\Q$-reductive group obtained,
by restricting scalars from $F$ to $\Q$, the group of isometries of $(,)$ on $V_E$. Note that 
$G^{(1)} \cong \U(p,q)$ and $G^{(j)} \cong \U(m)$ if $j\geq 2$. We denote by $K_{\infty}$ the maximal
compact subgroup of $G$.

A {\it congruence subgroup} of $G({\Bbb Q})$ is a subgroup $\Gamma = G({\Bbb Q}) \cap K$,
where $K$ is a compact open subgroup of $G({\Bbb A}_f)$ of the finite adelic points of $G$.
The (connected) Shimura variety $S = S(\Gamma) = \Gamma \backslash X$ is obtained as the quotient of the Hermitian symmetric space $X= G/K_{\infty} = \U (p,q) / (\U(p) \times \U (q))$ by the congruence subgroup $\Gamma$. We will be particularly interested in the $q=1$ case. Then $X$ identifies with the unit ball in $\C^p$.

\subsection{} The Shimura variety $S$ is a projective complex manifold defined over a finite abelian extension of $F$. Recall that a cohomology class on $S$ is of {\it level} $c$ if it is the pushforward of a cohomology class on a $c$-codimensional subvariety of $S$. We let $N^cH^{\bullet} (S , \Q)$ be the subspace of $H^{\bullet} (S, \Q)$ which consists of classes of level at least $c$. We first observe it is a consequence of the Hard Lefschetz Theorem \cite{GriffithsHarris}, pg. 122, that  $N^cH^{\bullet} (S , \Q)$ is stable under duality (the isomorphism given by the Hard Lefschetz Theorem).
Indeed the projection formula states that cohomological pushforward commutes with the actions of $ H^{\bullet} (S , \C)$ on the cohomologies  of a subvariety $V$ of $S$ and $S$ and hence with the operators $L_V$ and $L_S$ of the Hard Lefschetz Theorem.  Hence, if    $\beta \in H^k(S,\C)$ is the pushforward of a 
class $\alpha \in H^{k-2c}(V,\C)$ for $V$ a subvariety of codimension $c$ 
then the dual class $L^{p-k}_S (\beta)$ is the pushforward of 
$L^{p-k}_V (\alpha)$.

Hodge theory provides each $H^{n} (S)$ with a pure Hodge structure of weight $n$. We first consider the case where $q=1$:

\begin{thm} \label{thm:intro1}
Let $S$ be a connected compact Shimura variety associated to the unitary group $\U(p,1)$ and let $n$ and $c$ be non-negative integers s.t. 
$2n - c < p+1 \ \text{or} \ 2n + c > 3 p -1$ or equivalently $n \in [0,2p] \setminus ]p-\frac{p-c}{2}, p+\frac{p-c}{2}[$ . Then, we have:  
\begin{equation} \label{eq:thm1}
N^cH^n (S, \Q) = H^n (S , \Q) \cap \left( \oplus_{\begin{subarray}{c} a+b=n \\ a,b \geq c \end{subarray}}  H^{a,b} (S, \C) \right).
\end{equation}
\end{thm}

\noindent
{\it Remarks.} 1. The inclusion `$\subset$' always holds. In particular $N^cH^n (S, \Q)$ is trivial if $n <2c$. 

2. The conclusion of Theorem \ref{thm:intro1} confirms Hodge's generalized conjecture in its original formulation (with $\Q$ coefficients). Note however that --- 
as it was first observed by Grothendieck \cite{Grothendieck} --- the RHS of \eqref{eq:thm1} is not always a Hodge structure. Grothendieck corrected Hodge's original formulation but in our special case it turns out that the stronger form holds. 

3. We  prove the more precise result that  each subspace $H^{a,b} (S , \C) \oplus H^{b,a} (S, \C)$, with $3(a+b)+ |a-b| < 2(p+1)$, is defined over $\Q$ and that every rational cohomology class in that subspace is a linear combination with rational coefficients of cup products of $(1,1)$ classes and cohomology classes obtained as pushforwards of {\it holomorphic} or {\it anti-holomorphic} classes from algebraic cycles associated to smaller unitary subgroups --- the so-called `special' cycles. We want to emphasize that --- contrary to the orthogonal case studied in \cite{BMM} --- special cycles associated to smaller
unitary subgroups \emph{do not} span, even in codimension one. One has to use cycles obtained using the Lefschetz $(1,1)$ Theorem which we don't know how to realize explicitly.

\subsection{} Let $\mathcal{Z}^n (S)$ denote the $\Q$-linear span of the algebraic subvarieties of $S$ of codimension $n$. Poincar\'e duality
gives the cycle map
$$\mathcal{Z}^n (S) \to H^{2n} (S , \Q) \cap H^{n,n} (S , \C).$$
The subset $ H^{2n} (S , \Q) \cap H^{n,n} (S , \C)$ of $H^{2n} (S , \Q)$ is called the set of {\it Hodge classes}. 
Note that the image of the cycle map is precisely $N^n H^{2n} (S , \Q)$. The Hodge conjecture states that the cycle map is surjective, i.e. that every
Hodge class is algebraic. Theorem \ref{thm:intro1} and Poincar\'e duality  imply:

\begin{cor} \label{cor:intro1}
Let $S$ be a connected compact Shimura variety associated to the unitary group $\U(p,1)$ and let 
$n \in [0,p]\setminus ]\frac{p}{3} , \frac{2p}{3}[$. Then every Hodge class in $H^{2n} (S , \Q)$ is algebraic.
\end{cor}

Tate \cite{Tate} investigated the $\ell$-adic analogue of the Hodge conjecture. Recall  that $S$ is defined over a finite abelian extension $M$ of $E$. Fix a separable algebraic closure $\overline{M}$ of $M$. Seeing $S$ as a projective variety over $M$, we put $\overline{S} = S \times_M \overline{M}$. Given any prime number $\ell$ we denote by 
$$H^{\bullet}_{\ell} (\overline{S}) = H^\bullet (\overline{S} , \Q_{\ell})$$ 
the $\ell$-adic \'etale cohomology of $\overline{S}$. Recall that fixing an isomorphism of $\C$ with the completion $\C_{\ell}$ of an algebraic closure $\overline{\Q}_\ell$ of $\Q_\ell$ we have an isomorphism 
\begin{equation} \label{isomcomp}
H^{\bullet}_{\ell} (\overline{S}) \otimes \C_\ell \cong H^{\bullet} (S , \C).
\end{equation}
Given any finite separable extension $L(\subset \overline{M})$ of $M$ we let $G_L = \mathrm{Gal}(\overline{M}/L)$ be the corresponding Galois group. Tensoring with $\Q_\ell$ embeds $L$ in $\overline{\Q}_\ell$. The elements of $G_L$ then extend to continuous automorphisms of $\C_\ell$. For $j \in \Z$, let $\C_\ell (j)$ be the vector space $\C_\ell$ with the semi-linear action of $G_L$ defined by $(\sigma , z) \mapsto \chi_\ell (\sigma)^j \sigma (z)$, where $\chi_\ell$ is the $\ell$-adic cyclotomic character. We define 
$$H^{\bullet}_{\ell} (\overline{S}) (j) = \varinjlim (H^{\bullet}_{\ell} (\overline{S}) \otimes \C_\ell (j))^{G_L}.$$
Here the limit is over finite degree separable extensions $L$ of $M$. The $\ell$-adic cycle map
$$\mathcal{Z}^n (S) \to H^{2n}_{\ell} (\overline{S})$$
%\otimes \Q_{\ell} 
maps a subvariety to a class in $H^{2n}_{\ell} (\overline{S}) (n)$; the latter subspace is the space of 
{\it Tate classes}. The Tate conjecture states that the $\ell$-adic cycle map 
is surjective, i.e. that every Tate class is algebraic.

Now recall that Faltings \cite{Faltings} has proven the existence of a Hodge-Tate decomposition for the \'etale cohomology of {\it smooth} projective varieties defined
over number fields.  In particular, the isomorphism \eqref{isomcomp} maps 
$H^{m}_{\ell} (\overline{S}) (j)$ isomorphically onto $H^{j,m-j} (S , \C)$. From this, Theorem \ref{thm:intro1} and the remark following it, we get:

\begin{cor} \label{cor:intro2}
Let $S$ be a {\rm neat} connected compact Shimura variety associated to the unitary group $\U(p,1)$ and let $n\in [0,p]\setminus ]\frac{p}{3} , \frac{2p}{3}[$. Then every Tate class in $H^{2n} (S , \Q_{\ell})$ is algebraic.
\end{cor}
\begin{proof} 

It follows from the remark following Theorem \ref{thm:intro1} (and Poincar\'e duality) that the whole subspace  $H^{n,n} (S, \C)$ is spanned by algebraic cycles as long as $n \in [0,p]\setminus ]\frac{p}{3} , \frac{2p}{3}[$. 
The corollary is then a consequence of the following diagram.  
\begin{figure} [htbp]
$$
\begin{xy} 0;/r.1pc/:
(0, -10)*{Z^n(S)} = "A";
(5, -3)*{} = "A1";
(-5, -3)*{} = "A2";
(35, 30)*{H^{n, n}(\bar{S}, \C)} = "B";
(30,22)*{} = "B1";
(10, 30)*{} = "B2";
(-35, 30)*{H_{\ell}(\bar{S})(n)} = "C";
(-30,22)*{} = "C1";
(-15, 30)*{} = "C2"; 
{\ar@{>}"A1"; "B1"};
{\ar@{>}"A2"; "C1"};
{\ar@{>}"C2"; "B2"};
\end{xy}
$$\\
\end{figure}

The horizontal arrow is an isomorphism and the two diagonal arrows are the cycle maps.  We have proved the image of the right diagonal arrow spans.  Hence the image of the left diagonal arrow spans.

%It follows from the remark following Theorem \ref{thm:intro1} (and Poincar\'e duality) that the whole subspace  $H^{n,n} (S, \C)$ is spanned by algebraic cycles as long as $n \in [0,p]\setminus ]\frac{p}{3} , \frac{2p}{3}[$. 
%Since, by Faltings' theorem, the Tate space $H^{2n}_{\ell} (\overline{S}) (n)$ is mapped (by a natural map) into $H^{n,n} (S , \C)$,
%we conclude that  
%Then $H^{2n} (S , \Q) = X \oplus Y$ where $X$ and $Y$ are rational Hodge structures such that 
%$$X \otimes \C = \oplus_{\substack{a+b=2n \\ a\neq b}} H^{a,b} (S , \C) \mbox{ and } Y \otimes \C = H^{n,n} (S , \C).$$
%In particular it follows from Faltings' theorem that the only Hodge-Tate types of the quotient of $H^{2n}_{\ell} (\overline{S}) / (Y\otimes \Q_{\ell})$ are non-parallel, that is lie in $H^{2n}_{\ell}(\overline{S})(j)$ with $j \neq n$, and 
%hence that there are no Tate classes outside $Y \otimes \Q_\ell$. Now Theorem \ref{thm:intro1} implies that $Y$ 
%is spanned by algebraic classes and we conclude that 
%every Tate class is algebraic.
\end{proof}

\subsection{A general theorem} \label{010}
Theorem \ref{thm:intro1} will follow from Theorem \ref{thm:intro2} below which is the main result of our paper. 
It is related to the Hodge conjecture
through a refined decomposition of the cohomology that we now briefly describe.

\subsection{} 
As first suggested by Chern \cite{Chern} the decomposition of exterior powers of the cotangent bundle of $X=\U(p,q) / (\U(p) \times \U (q))$ under the action of the holonomy group, i.e. the maximal compact subgroup $\U(p) \times \U (q)$ of $\U(p,q)$, yields a natural notion of {\it refined Hodge decomposition} of the cohomology groups of the associated locally 
symmetric spaces.

The symmetric space $X$ being of Hermitian type, there is an element $c$ belonging to the center of
$\U(p) \times \U(q)$ such that $\Ad (c)$ induces (on the tangent space $\p_0$ of $X$ associated to
the class of the identity in $\U(p,q)$) multiplication by $i = \sqrt{-1}$. Let 
$$\g = \k \oplus \p' \oplus \p''$$
be the associated decomposition of $\g= \mathfrak{gl}_{p+q} (\C)$ -- the complexification of $\mathfrak{u}(p,q)$. Thus $\p' = \{ X \in \p \; : \; \Ad (c) X = i X \}$ (where
$\p$ is the complexification of $\p_0$) is the holomorphic tangent space. 

As a representation of $\GL(p,\C) \times \GL (q, \C)$ the space $\p'$ is isomorphic to $V_+ \otimes V_-^*$ where $V_+= \C^p$ (resp. $V_- = \C^q$) is the standard representation of $\GL(p, \C)$ (resp. $\GL(q, \C)$). 
The refined Hodge types therefore correspond to irreducible summands in the decomposition of $\wedge^{\bullet} \p^*$ as a $(\GL(p,\C) \times \GL (q, \C))$-module.
In the case $q=1$ (then $X$ is the complex hyperbolic space of complex dimension $p$) it is an exercise to check that one recovers the usual Hodge-Lefschetz decomposition. But in general the decomposition is much finer and it is hard to write down the 
full decomposition of
$$\wedge^{\bullet} \p = \wedge^{\bullet} \p' \otimes \wedge^{\bullet} \p'' $$ 
into irreducible modules. Note that the decomposition of $\wedge^{\bullet} \p'$ is already quite complicated. We have (see \cite[Equation (19), p. 121]{Fulton}):
\begin{equation} \label{GLdec}
\wedge^R (V_+ \otimes V_-^*) \cong \bigoplus_{\lambda \vdash R} S_{\lambda} (V_+) \otimes S_{\lambda^*} (V_-)^*.
\end{equation}
Here we sum over all partitions of $R$ (equivalently Young diagrams of size $|\lambda|=R$) and $\lambda^*$ is the conjugate partition (or transposed Young diagram). 

It nevertheless follows from Vogan-Zuckerman \cite{VZ} that very few of the irreducible
submodules of $\wedge^{\bullet} \p^*$ can occur as refined Hodge types of non-trivial coholomogy classes.

The ones which can occur (and do occur non-trivially for some $\Gamma$) are understood in terms of cohomological representations of $\U (p,q)$. We review these cohomological representations in section \ref{sec:CR}. We recall in particular how to associate to each cohomological representation $\pi$ of $\U(p,q)$ a {\it strongly primitive} refined Hodge type. This refined Hodge type
corresponds to an irreducible representation $V(\lambda , \mu)$ of $\U(p) \times \U(q)$ which is uniquely determined by 
some special  pair of partitions $(\lambda , \mu)$ with $\lambda$ and $\mu$ as in \eqref{GLdec}, see \cite{TG}; it is an irreducible submodule of
%kind of
$$S_{\lambda} (V_+) \otimes S_{\mu} (V_+)^* \otimes S_{\mu^*} (V_-) \otimes S_{\lambda^*} (V_-)^* .$$ The first degree where such a refined Hodge type can occur is $R=|\lambda|+ |\mu|$. 
We will use the notation $H^{ \lambda , \mu}$ for the space of the cohomology in degree $R=|\lambda|+ |\mu|$ corresponding to this special Hodge type; more precisely, it occurs in the subspace $H^{|\lambda| , |\mu|}$.

The group $\mathrm{SL}(q) = \mathrm{SL}(V_ -)$ acts on $\wedge^{\bullet} \mathfrak{p^*}$. 
In this paper we will be mainly concerned with elements of $(\wedge^{\bullet} \mathfrak{p^*} )^{\SL (q)}$ --- that is elements that are trivial 
on the $V_-$-side. Recall that there exists an invariant element 
$$c_q \in (\wedge^{2q} \mathfrak{p^*} )^{\U (p  ) \times \SL(q) },$$ 
the {\it Chern class/form} (see subsection \ref{chernform} for the definition). 
We finally note that if $\lambda$ is the partition $q+ \ldots + q$ ($a$ times) then $S_{\lambda^*} (V_-)$ is the {\it trivial} representation of $\SL (V_-)$ and $S_{\lambda} (V_+) \otimes S_{\lambda^*} (V_-)^*$ occurs in $(\wedge^{aq} \mathfrak{p^+} )^{\SL (q)}$; in that  case we use the notation $\lambda = a \times q$.

\subsection{The special refined Hodge types} 
The subalgebra $\wedge^{\bullet} (\p^*)^{\mathrm{SL}(q)}$ of $\wedge^{\bullet} (\p^*)$ is invariant under $\U ( p ) \times \U (q)$. 
Hence, we may form the  associated subbundle 
$$F= X \times_{\U ( p ) \times \U (q)} \wedge^{\bullet} ( \p^*)^{\mathrm{SL}(q)}$$ 
of the bundle 
$$X \times_{\U ( p ) \times \U (q)}  \wedge^{\bullet} ( \p^*) $$
of exterior powers of the cotangent bundle of $X$. The space of sections of $F$ --- a subalgebra of the algebra of  differential forms ---  is invariant under the Laplacian
and hence under harmonic projection, compare \cite[bottom of p. 105]{Chern}.  
We denote by $SH^{\bullet} (S , \C)$ the corresponding subalgebra of $H^{\bullet} (S , \C)$; we will refer to the refined Hodge types occurring in $SH^{\bullet} (S , \C)$
as {\it special refined Hodge types}. 
Note that 
when $q=1$ we have $SH^{\bullet} (S , \C ) = H^{\bullet} (S , \C)$. In general we shall show that\footnote{In the body of the paper we will rather write $b \times q ,a \times q$ in order to write U(a,b) instead of U(b,a).}
\begin{equation} \label{subringSC}
SH^{\bullet} (S , \C) = \oplus_{a,b=0}^{p} \oplus_{k=0}^{\min (p-a,p-b)} c_q^{k} H^{a \times q , b \times q} (S , \C ).
\end{equation}
(Compare with the usual Hodge-Lefschetz decomposition.) We set 
$$SH^{aq,bq} (S, \C ) = H^{aq,bq} (S, \C ) \cap SH^{\bullet} (S , \C)$$
so that $SH^{\bullet} (S , \C) = \oplus_{a,b=0}^{p} SH^{aq,bq} (S, \C )$.
Note that the {\it primitive} part of $SH^{aq,bq} (S, \C )$ is $H^{a \times q , b \times q} (S , \C )$. 

\subsection{} We may associate to an $n$-dimensional totally positive Hermitian subspace of $V$ a special cycle of complex codimension $nq$ in $S$ which is Shimura subvariety associated to a unitary group of type $\U(p-n,q)$ at infinity. Since these natural cycles do not behave particularly well under pullback for congruence
coverings, following Kudla \cite{Kudla} we introduce weighted sums of these natural cycles and show that their cohomology classes form a subring  
$$SC^{\bullet} (S) =\oplus_{n=0}^p SC^{2nq} (S)$$
of $H^{\bullet} (S, \Q)$. 
We shall show (see Theorem \ref{thm:MainSC}) that for each $n$ with $0 \leq n \leq p$ we have:
\begin{equation*}
SC^{2nq} (S) \subset SH^{nq,nq} (S , \C)  \cap H^{\bullet} (S, \Q).
\end{equation*}
Note that, in particular, this gives strong restrictions on the possible refined Hodge types that can occur in the classes dual to special cycles. We furthermore give an intrinsic 
characterization of the primitive part of the subring $SC^{\bullet} (S)$ in terms of automorphic representations. 
For quotients of the complex $2$-ball (i.e. $q=1$ and $p=2$) this was already obtained by Gelbart, Rogawski and Soudry \cite{GRS1,GRS2,GRS3}. 

We can now state our main theorem (see Theorem \ref{Thm:main9}):

\begin{thm} \label{thm:intro2}
Let $a$ and $b$ be integers s.t. $3(a+b)+|a-b| <2(p+q)$. First assume $a\neq b$. Then the image of the natural cup-product map
$$SC^{2\min (a,b) q} (S) \times \left( SH^{|a-b|q,0} (S , \C) \oplus SH^{0,|a-b|q} (S , \C) \right)  \rightarrow H^{(a+b)q} (S , \C )$$
spans a subspace which projects surjectively onto $H^{a\times q , b\times q} (S , \C) \oplus H^{b\times q , a \times q} (S , \C)$. 
If $a=b$ this is no longer true but the image of the natural cup-product map
$$SC^{2(a-1)q} (S) \times SH^{q,q} (S , \C) \rightarrow H^{2aq} (S, \C)$$
spans a subspace which projects surjectively onto $H^{a\times q , a\times q} (S , \C)$.
\end{thm}

\noindent
{\it Remarks.} 1. We shall see that the subrings $SH^{\bullet , 0} (S , \C)$, resp. $SH^{0, \bullet } (S , \C)$, are well understood. These are spanned 
by certain theta series associated to explicit cocyles: the {\it holomorphic and anti-holomorphic cocyles}. 

2. In case $q=1$ we will furthermore prove that if $a$ and $b$ are integers s.t. $3(a+b)+|a-b| <2(p+q)$, then the space $H^{a+b} (S , \Q)$ contains a polarized
$\Q$-sub-Hodge structure $X$ s.t.
$$X \otimes_{\Q} \C = H^{a,b} (S , \C) \oplus H^{b,a} (S , \C)$$
(see Corollary \ref{C:rational}). Theorem \ref{thm:intro1} then immediately follows from Theorem \ref{thm:intro2} and Lefschetz' $(1,1)$ Theorem.

\subsection{Organisation of the paper} \label{organization}
The proof of Theorem \ref{thm:intro2} is the combination of four main steps. 

The first step is  the work
of Kudla-Millson \cite{KudlaMillson3}. It relates the subspace
of the cohomology of locally symmetric spaces associated to unitary groups generated by 
special cycles to certain cohomology classes associated to ``special theta lifts'' from a quasi-split unitary group and using vector-valued  adelic Schwartz functions with a fixed vector-valued component at infinity.  More precisely, the special theta lift restricts the general theta lift from a quasi-split unitary group to 
Schwartz functions that have  at the distinguished  infinite place where the unitary group is noncompact  the fixed vector-valued
Schwartz function $\varphi_{nq,nq}$, see \S \ref{sec:KMlocal}, at infinity.  The Schwartz functions at the other infinite places are Gaussians (scalar-valued) and at the finite places are scalar-valued and otherwise arbitrary.  The main point is that $\varphi_{nq,nq}$ is a relative Lie algebra cocycle for the unitary group allowing one to interpret the special theta lift cohomologically.  In fact we work with the Fock model for the Weil representation and use the cocycle $\psi_{nq,nq}$ with values in the Fock model corresponding to  the cocycle  $\varphi_{nq,nq}$ with values in the Schr\"odinger model under the usual intertwiner from the Fock model to the Schr\"odinger model. 

The second step is to interpret the cocycles $\psi _{nq,nq}$ as cup-products of the holomorphic and anti-holomorphic cocycles 
$\psi_{q,0}$ and $\psi_{0,q}$. This factorization does not hold for the cocycles $\varphi_{nq,nq}$ - see Appendix C. 
This naturally leads to the general definition of  new  cocyles $\psi_{aq,bq}$. These are local (Archimedean) computations in the Fock model that culminate
in the proof that these cocycles generate the $(\g , K)$-cohomology of refined Hodge type of the space of the Weil representation, see Theorem \ref{KMgenerates}.

The third step, accomplished in  Theorem \ref{StepTwo} and depending essentially on  Theorem \ref{KMgenerates}, is to show that the intersections of the images of the general theta lift and the special theta lift just described with the subspace of  the cuspidal automorphic forms that have infinite component a cohomological 
representation corresponding to a {\it special} refined Hodge type coincide (of course the first intersection is potentially larger). In other words the special theta lift accounts for all the cohomology of special refined Hodge type that may be obtained from theta lifting.  

The last step (and it is here that we use Arthur's classification \cite{ArthurBook}) is to show  that in low degree any cohomology class in $SH^{\bullet}(S)$ can be obtained as a projection of the class of a theta series. In other words, we prove the low-degree cohomological surjectivity of the general  theta lift (for classes of special refined Hodge type). We first have to prove a criterion for an automorphic form to be in the image of the general theta lift. This criterion is analogous to a classical 
result of Kudla and Rallis \cite{KudlaRallis}, however its analogue in the unitary case doesn't seem to have been fully worked out. Building on work of Ichino, we provide this extension in Section \ref{Sec:1}. We then have to make sure that the hypotheses of this criterion are satisfied in our case.

\subsection{} \label{org2} This is where the deep theory of Arthur comes into play. Very briefly: Arthur, and Mok, classify
automorphic representations of classical groups into global packets. Two automorphic representations belong to the same packet if their partial $L$-functions are the same i.e. if the local components of the two automorphic representations are isomorphic almost everywhere. Moreover in loose terms: Arthur shows that if an automorphic form is very non tempered at one place then it is very non tempered everywhere. To conclude we therefore have to study the cohomological representations
at infinity and show that those we are interested in are very non-tempered, this is the main issue of (the automorphic) Part 3 of our paper. Arthur's work on the endoscopic classification of representations of classical groups relates the automorphic representations of unitary groups to the automorphic representations of $\GL (N)$ twisted by some automorphism $\theta$. Note however that the relation is made through the stable trace formula for unitary groups and the stable trace formula for the twisted (non connected) group $\GL (N) \rtimes \langle \theta \rangle$. 

Thus, as pointed out in the abstract,  our work is still conditional on extensions to the twisted case of results which have only been proved so far in the case of connected groups. This is now announced by Waldspurger and Moeglin-Waldspurger see \cite{WaldsSeoul}.

\medskip

{\it We  would like to thank Claire Voisin for suggesting to look at the generalized Hodge conjecture and Don Blasius and Laurent Clozel for useful references.}

\newpage

\part{Local computations}

\section{Hermitian vector spaces over $\C$} \label{linearalgebra}

We begin with some elementary linear algebra that will be important to us in what follows. The results we prove are standard, the main point is to establish notations that will be useful later.  For this section the symbol $\otimes$ will mean tensor product over $\R$, in the rest of the paper it will mean tensor product over $\C$. 

\subsection{Notations} Let $V$ be a complex vector space 
equipped with a nondegenerate Hermitian form $( , )$. 
Our Hermitian forms will be complex linear in the first argument and complex antilinear in the second.   
We will often consider $V$ as  a real vector space equipped with the almost complex structure $J$ given by $Jv = iv$. When there are several vector spaces under consideration we write $J_V$ instead of $J$. 
% We let $\overline{V}$
%denote the complex vector space conjugate to $V,J$.  Hence as real vector spaces $V$ and $\overline{V}$ coincide but the complex structures have opposite signs
%$$J_{\overline{V}} = - J_V.$$
 We will give $V^{\ast}$ the transpose almost complex structure (not the inverse transpose almost complex structure) so
\begin{equation} \label{transposecomplexstructure} 
(J \alpha)(v) = \alpha(Jv).
\end{equation}
We will  sometimes denote this complex structure by $J_{V^{\ast}}$.

Finally recall that we have a  real-valued symmetric  form $B$ and a real-valued skew symmetric form $\langle \  , \ \rangle$ or $A$ on $V$ considered as a real vector space associated to the Hermitian form $( , )$ by the formulas
\begin{align*}
B(v_1,v_2) = & \ \mathrm{Re} (v_1,v_2) \\
A( v_1,v_2)= & - \mathrm{Im} (v_1,v_2),
\end{align*}
so that we have:
\begin{equation} \label{relateAandB}
B(v_1, v_2) = A(v_1, J v_2).
\end{equation}

\subsection{The complexification of an Hermitian space and the subspaces of type $(1,0)$ and $(0,1)$ vectors} \label{complexification}
We now form the complexification $V \otimes \C = V \otimes_{\R} \C$  of $V$ considered as real vector space. The space $V \otimes \C$ has two commuting complex structures namely $J \otimes 1$ and $I_V \otimes i$.  
%The Hermitian form $( ,)$ induces a nondegenerate Hermitian form again denoted 
%$( , )$ on the complexification $V \otimes \C$  by taking the tensor product of $( ,)$ and the unary Hermitian form on $\C$.  Thus we have the formula 
 %\begin{equation}\label{formoncomplexification}
%(v_1 \otimes z_1 , v_2 \otimes z_2 ) = (v_1 ,v_2 ) z_1 \overline{z}_2.
%\end{equation}
We define the orthogonal idempotents $p'$ and $p''$ in $\mathrm{End}_{\C}(V \otimes \C)$ by
\begin{equation}\label{typeprojectors}
p' =  \frac{1}{2}(I_V  \otimes 1 - J_V \otimes i) \ \text{and} \ p'' = \frac{1}{2}(I_V \otimes 1 + J_V \otimes i)
\end{equation}
The reader will verify the equations
\begin{equation} \label{idempotents} 
p' \circ p' = p', \  p''\circ p'' = p''\ \text{and} \  p' \circ p'' = p'' \circ p =0.
\end{equation}
In what follows if $v \in V$ is given we will abbreviate $p'(v \otimes \C)$ by $v'$ and $p''(v \otimes \C)$ by $v''$. We will write
$z v'$ for $(1 \otimes z)v'$ and $z v''$ for $(1 \otimes z)v''$.  We note the formulas
\begin{equation} \label{linearandantilinear}
p'(zv) = z p'(v) \ \text{and} \ p''(zv) = \overline{z} p''(v)
\end{equation}

%$u = v \otimes 1 \in V \otimes \C$ and $z \in \C$ then the product $z u$ will be defined by
%$$ z u = ( 1 \otimes z) ( \sum_j v_j  \otimes z_j ) =  \sum_j v_j  \otimes z z_j.$$
%With this definition we have
%$$p'(zu) = zp'(u) \ \text{and}\  p''(zu) = z p''(u).$$

We define $V' = p'(V \otimes \C)$ and $V''= p''(V \otimes \C)$.  From the  Equation \eqref{idempotents} we obtain
$V \otimes \C = V' \oplus V''$.
An element of $V'$ is said to have type $(1,0)$ and an element of $V''$ is said to have type $(0,1)$. We will identify $V$ with the subspace $V \otimes 1$ in $V \otimes \C$.

\subsection{Coordinates on $V$ and the  induced coordinates on $V'$ and $V''$} \label{introducecoords}
In this subsection only we will assume that the Hermitian form $( , )$ on $V$ is positive definite. 
Let  $v_1,\cdots,v_n$ be an orthonormal  basis for $V$ over $\C$.  Then we obtain induced bases $v'_1,\cdots, v'_n$ and $v''_1,\cdots, v''_n$ for $V'$ and $V''$ respectively. We let $z_j(v),1 \leq j \leq n$ be the coordinates of $v \in V$ relative to the basis,  $z'_j(v'), 1 \leq j \leq n$ be the coordinates of $v' \in V'$ relative to the basis $v'_1,\cdots, v'_n$ and $v''_1,\cdots, v''_n$ and $z''_j(v'), 1 \leq j \leq n$ be the coordinates of $v'' \in V''$ relative to the basis $v''_1,\cdots, v''_n$.  Let $v \in V$.  Then, by applying $p'$ and $p''$ to the equation
$v = 
\sum \limits_{j=1}^n z_j(v) v_j$ and applying  Equation \eqref{linearandantilinear} we have
\begin{lem} \label{coordinateformula} 
\hfill

\begin{enumerate}
\item $z'_j(v') = z_j(v) = (v,v_j)$
\item $z''_j(v'') = \overline{z_j(v)}= (v_j,v)$
\end{enumerate}
\end{lem}

%Since $V \otimes \C$ is the complexification of $V$ there is a  real structure $\tau$, 
%an antilinear map from $V \otimes \C$ to itself, defined by
%$$\tau(v \otimes z) = v \otimes \overline{z}.$$
%We have 
%$$ (V \otimes \C)^{\tau} = V.$$ 

%Note that $\tau \circ p' \circ \tau  = p''$ and hence $\tau$ induces a complex antilinear  isometry $\tau:V' \to V''$.
%We see that $\tau(v'_j) = v''_j$ whence 
%$\begin{equation} \label{conjugation}
%\tau( \sum_{j} z_j v'_j) = \sum_j \overline{z}_j v''_j
%\end{equation}

%It is useful to define  maps
%$f':V'_+ \to V_+ $ and $f'':V''_+ \to V^*_+$  given by
%\begin{equation} \label{tautologicalmaps}
%f'(v') =v \ \text{and} \ f''(v'') = v^*
%\end{equation}
%Both maps are clearly complex linear isomorphisms.  We note that $f'$ is induced by the map $f: V \otimes \C \to V$ given by
%$ f(v \otimes z) = zv$ and $f''$ is induced by the map $\overline{f}: V \otimes \C \to V$ given by $\overline{f}(v \otimes z) = \overline{z} v$. 

\subsection{The induced Hodge decomposition of $V^{\ast}$}

There is a corresponding decomposition $V^{\ast} \otimes \C = (V^{\ast})' \oplus (V^{\ast})''$ induced by the almost complex structure
$J_{V^{\ast}}$. The complexified canonical pairing $(V^{\ast} \otimes \C) \otimes_{\C} (V \otimes \C) \to \C$ given by $(\alpha \otimes z) \otimes_{\C}  (v \otimes w) \to (\alpha(v)) (zw)$
induces isomorphism $(V^{\ast})' \to (V')^{\ast}$ and $(V^{\ast})'' \to (V'')^{\ast}$.  We will therefore make the identifications
$$(V^{\ast})'= (V')^{\ast} \ \text{and} \ (V^{\ast})'' = (V'')^{\ast}$$
without further mention. In particular if $\{v_1,\cdots,v_n\}$ is a basis for $V$  and 
$\{f_1,\cdots,f_n\}$ is the dual  basis for $V^{\ast}$, then
$\{f'_1,\cdots,f'_n\}$ is the basis for $(V^{\ast})'$ dual to the basis   $\{v'_1,\cdots,v'_n\}$ for $V'$.

%\subsection{The induced  Hermitian form on $V^{\ast}$}

%For $v \in V$ we define $v^* \in V^*$ by 
%\begin{equation}\label{Berger}
%v^*(u) = (u,v)
%\end{equation}
%We note that $v \to v^*$ is complex antilinear: $(iv)^* = -i v^*$.
%The relation between a nondegenerate Hermitian form on $V$ and the induced form on the dual
%is therefore
%$$( v^*_1, v^*_2)= (v_2,v_1)  = \overline{(v_1,v_2)}.$$

\subsection{The positive almost complex structure $J_0$ associated to a Cartan involution} \label{pacs}
We now assume that $V,( ,  )$ is an indefinite  Hermitian space of signature $(p,q)$.  
We choose once and for all an orthogonal splitting of complex vectors spaces
$V= V_+ + \ V_ -$ of $V$ such that the restriction of $( , )$ to $V_+$ is positive definite and the restriction to  $V_-$ is negative definite. 
Such a splitting is determined by the choice of $V_-$ and consequently corresponds to a point in the symmetric space of $V$.  We can obtain a
positive definite Hermitian form $( , )_0$ depending on the choice of $V_ -$ by changing the sign of $( , )$ on $V_-$.  The positive definite form $( , )_0$ is called (in classical terminology) a minimal majorant of $( , )$.   Let $\theta_{V_-}$ be the involution which is $I_{V_+}$ on $V_+$ and $-I_{V_-}$ on $V_ -$.  Since $V_+$ and $V_-$ are complex subspaces $J$ 
and $\theta_{V_-}$ commute.  Then $\theta_{V_-}$ is a Cartan involution of $V$ in the sense that it is an order two isometry of $( , )$  such that its centralizer in $\U(V)$ is a maximal compact subgroup.  All Cartan involutions are of the form $\theta_{V_-}$ for some splitting of $V= V_+ + V_-$ as above. 
We note that we have
\begin{equation} \label{newform}
(v_1,v_2)_0 = (v_1,\theta_{V_-}v_2).
\end{equation}
We have
\begin{equation}
|(v,v)| \leq (v,v)_0, v \in V.
\end{equation}
For this reason $(\ ,\ )_0$ is called a (minimal) majorant of $( \ , \ )$.

By taking real and imaginary parts of $(\  ,\  )_0$ we obtain
a {\it positive definite} symmetric form $B_0( , )$ and a symplectic form $A_0( , )$ such that 
$$(v_1,v_2)_0 = B_0(v_1,v_2) - i A_0(v_1,v_2).$$
%We note that the new forms are related to the old ones by
%$$B_0(v_1 ,v_2)= B(v_1, \theta_{V_-} v_2) \ \text{and} \ A_0(v_1 ,v_2)= A(v_1, \theta_{V_-} v_2).$$
Define a new complex structure $J_0$ by
$$J_0 = \theta_{V_-} \circ J = J \circ \theta_{V_-}.$$
We note that the new form $( , )_0$ is still Hermitian with respect to the old complex structure $J$,\footnote{However $( , )_0$ is  not Hermitian with respect to the new complex structure $J_0$.} that is
$$(Jv_1,v_2)_0 = i(v_1,v_2)_0 \ \text{and} \ (v_1,J v_2)_0 = -i (v_1,v_2)_0,$$
and that $J_0$ is an isometry of $( , )_0$, that is 
$$(J_0 v_1, J_0 v_2)_0 = (v_1,v_2)_0.$$

We claim that
\begin{equation} \label{positivedefinitealmostcomplexstructure}
B_0(v_1,v_2) = A(v_1, J_0 v_2).
\end{equation}
Indeed we have 
\begin{align*}
B_0(v_1,v_2) = & \mathrm{Re} (v_1, v_2)_0  = \mathrm{Re} (v_1, \theta_{V_-} v_2) = \mathrm{Im} \ i (v_1, \theta_{V_-} v_2) = \mathrm{Im} \big( - (v_1, J \theta_{V_-} v_2) \big)\\
             =& - \mathrm{Im} (v_1,J \theta_{V_-} v_2) = A( v_1,J \theta_{V_-} v_2) =  A( v_1,J_0 v_2) .
\end{align*}
The claim follows. 

It follows from \eqref{positivedefinitealmostcomplexstructure} that  $J_0$ is a positive definite almost complex with respect to the symplectic form $A( \ , \ )$.  For the convenience of the reader we recall this basic definition.  

\begin{defn} 
Given a symplectic form $A(\  ,\  )$ and an almost complex structure $J_0$ we say $A(\  ,\  )$ and $J_0$ are {\it compatible} if $J_0$ is an isometry of $A(\  ,\  )$ and we say 
$J_0$ is {\it positive (definite)} with respect to $A$ if $J_0$ and $A$ are compatible and moreover  the symmetric form $B_0(v_1,v_2) =  A(v_1 , J_0 v_2)$ is positive definite.
\end{defn}

It now follows from the above discussion that there is a one-to-one correspondence between minimal Hermitian majorants of $( , )$,  positive almost complex structures $J_0$ commuting with $J$ such that the  product $J_0 J$ is a Cartan involution and points of the symmetric space of $\U(V)$ (subspaces $Z$ of dimension $q$ such that the restriction of $(\ ,\ )$ to $Z$ is negative definite.  Henceforth, we will call such positive complex structures $J_0$ {\it admissible}.

We will therefore have to deal with {\em two} different almost complex structures hence two notions of type $(1,0)$ vectors.  To deal with this we use the following notations.  

\begin{defn}
\begin{enumerate}
\item If $U$ is a subspace of $V$ which is $J$-invariant then $U'$, resp. $ U''$, will denote the subspace of type $(1,0)$, resp. type $(0,1)$, vectors for the {\em indefinite} almost complex structure
$J$ acting on $U \otimes \C$, hence, for example $V'_+$ is $+i$ eigenspace of $J$ acting on $V_+ \otimes \C$.
\item If $U$ is a subspace of $V$ which is $J_0$-invariant then $U^{\prime_0}$, resp. $ U^{\prime \prime_0}$, will denote the subspace of type $(1,0)$,
resp. type $(0,1)$, vectors for the {\em definite} almost complex structure
$J_0$ acting on $U \otimes \C$.
\end{enumerate} 
\end{defn}

\section{Cohomological unitary representations} \label{sec:CR}

\subsection{Notations} Keep notations as in section \ref{linearalgebra} and let $m=p+q$. In this section
$G= \U (V) \cong \U (p,q)$ and $K \cong \U (p ) \times \U (q)$ is a maximal compact subgroup of $G$ associated to the Cartan
involution $\theta_{V_-}$.
We let $\g_0$ the real Lie algebra of $G$ and $\g_0 = \k_0 \oplus \p_0$ be the Cartan decomposition associated to $\theta_{V_-}$. 
If $\mathfrak{l}_0$ is a real Lie algebra we denote by $\mathfrak{l}$ its complexification $\mathfrak{l} = \mathfrak{l}_0 \otimes \C$. 

\subsection{Cohomological representations} A unitary representation $\pi$ of $G$ is {\it cohomological} if it has nonzero $(\g , K)$-cohomology $H^{\bullet} (\g , K ; V_{\pi})$.

Cohomological representations are classified by Vogan and Zuckerman in \cite{VZ}:
let $\mathfrak{t}_0$ be a Cartan subalgebra of $\k_0$. A {\it $\theta$-stable parabolic subalgebra} $\mathfrak{q} = \mathfrak{q} (X) \subset \g$ is associated to an element $X \in i \mathfrak{t}_0$. It is defined as the direct sum
$$\mathfrak{q} = \mathfrak{l} \oplus \mathfrak{u},$$
of the centralizer $\mathfrak{l}$ of $X$ and the sum $\mathfrak{u}$ of the positive eigenspaces of
$\mathrm{ad} (X)$. 
Since $\theta X = X$, the subspaces $\mathfrak{q}$, $\mathfrak{l}$ and $\mathfrak{u}$ are all invariant under $\theta$, so 
$$\mathfrak{q} = \mathfrak{q} \cap \k \oplus \mathfrak{q} \cap \p,$$
and so on. Let $R = \dim ( \mathfrak{u} \cap \p )$.

Associated to $\mathfrak{q}$, there is a well-defined, irreducible essentially unitary representation $A_{\mathfrak{q}}$ of $G$; it is characterized by the following
properties. Assume that a choice of a positive system $\Delta^+(\mathfrak{l})$ of the roots of $\mathfrak{t}$ in $\mathfrak{l}$ has been made compatibly with 
$\mathfrak{u}$. Let $e(\mathfrak{q})$ be a generator of the line $\wedge^R (\mathfrak{u} \cap \p)$; we shall refer to such a vector as a {\it Vogan-Zuckerman vector}. Then $e(\mathfrak{q})$
is the highest weight vector of an irreducible representation $V(\mathfrak{q})$ of $K$ contained in 
$\wedge^R \p$ (and whose highest weight is thus necessarily $2\rho (\mathfrak{u} \cap \p)$). The representation $A_{\mathfrak{q}}$ is then uniquely characterized by the following two properties:
\begin{multline}\label{P1}
A_{\mathfrak{q}} \ \mbox{\it is essentially unitary with trivial central character and} \\ \mbox{\it with the same infinitesimal character as the trivial representation},
\end{multline}
\begin{equation} \label{P2}
\mathrm{Hom}_K (V (\mathfrak{q}) , A_{\mathfrak{q}}) \neq 0.
\end{equation}
Note that (the equivalence class of) $A_{\mathfrak{q}}$ only depends on the intersection $\mathfrak{u} \cap \mathfrak{p}$ so that two parabolic subalgebras 
$\mathfrak{q} = \mathfrak{l} \oplus \mathfrak{u}$ and $\mathfrak{q}' = \mathfrak{l}' \oplus \mathfrak{u}'$ which satisfy $\mathfrak{u} \cap \mathfrak{p} = 
\mathfrak{u}' \cap \mathfrak{p} '$ yield the same (equivalent class of) cohomological representation. Moreover 
$V(\mathfrak{q})$ occurs with multiplicity one in $A_{\mathfrak{q}}$ and $\wedge^R \mathfrak{p}$, and 
\begin{equation} \label{VZKtype}
H^{\bullet} (\g , K ; A_{\mathfrak{q}} ) \cong \mathrm{Hom}_{L \cap K} (\wedge^{\bullet -R} (\mathfrak{l} \cap \mathfrak{k}) , \C ).
\end{equation}
Here $L$ is a subgroup of $K$ with Lie algebra $\mathfrak{l}$.

In the next paragraphs we give a more explicit parametrization of the cohomological modules of $G$.

\subsection{The Hodge  decomposition of the complexified tangent space $\mathfrak{p}$  of the symmetric space of $\mathrm{U}(p,q)$ at the basepoint} 
\label{complexifiedtangentspace}
We first give the standard  development of the Hodge decomposition of $\mathfrak{p}$.  In what follows we  will use a subscript zero to denote a real algebra (subspace of a real algebra) and omit the subscript zero for its complexification.  For example we have $\mathfrak{g}_0 = \mathfrak{u}(p,q)$ and $\mathfrak{g} =  \mathfrak{g}_0 \otimes \C$. 

We start by making the usual identification $\mathfrak{g} \cong \End(V)$ given by 
\begin{equation}\label{identificationwithgl}
A \otimes z \to zA, \quad A \in \mathfrak{g}_0 
\end{equation}
Rather than using pairs of numbers between $1$ and $p+q$ (so for example $e_{i,j}$) to denote the usual basis  elements of $\End(V)$ we will use the isomorphism  
$ \End(V) \cong V \otimes V^{\ast}$, see below.  We then have 
$$\mathfrak{g} = V \otimes V^*  =(V_+ \otimes V^*_+ ) \oplus (V_+ \otimes V^*_ -) \oplus  (V_- \otimes V^*_+ ) \oplus  (V_- \otimes V^*_- ).$$
In terms of the above splitting (and identification) we have
\begin{equation} \label{Cartandecompositionone}
\begin{split}  
& \mathfrak{k} = (V_+ \otimes V^*_+ ) \oplus (V_- \otimes V^*_- ) \\
& \mathfrak{p} =  (V_+ \otimes V^*_- ) \oplus (V_- \otimes V^*_+ ).
\end{split}
\end{equation}

\subsection{} Now consider a basis $\{ v_i  \; : \; 1 \leq i \leq m \}$ for $V$ adapted to the decomposition $V=V_+ + V_-$. The following index convention will be useful in what follows. 
We will use ``early'' Greek letters $\alpha$ and $\beta$ to index basis vectors and coordinates belonging to $V_+$ and ``late'' Greek letters $\mu$ and $\nu$ to index basis vectors and coordinates belonging to $V_-$. We furthermore suppose that
$$(v_{\alpha},v_{\beta}) = \delta_{\alpha,\beta} \ \text{and} \ (v_{\mu},v_{\nu})= - \delta_{\mu,\nu}.$$
Then the matrix of the Hermitian form $( , )$ on $V$ w.r.t. to the basis $\{ v_i \}$ is the diagonal matrix $\left( \begin{smallmatrix}
1_p & \\ 
& -1_q 
\end{smallmatrix} \right)$. We therefore end up with the usual matrix realization of the Lie algebra $\mathfrak{g}_0$ of $\U (p,q)$ where a $m$ by $m$ complex matrix 
$\left(\begin{smallmatrix} A & B\\
                            C & D
\end{smallmatrix}\right)$ belongs to $\mathfrak{g}_0$ if and only if  $A^{\ast} = -A$, $D^{\ast} = -D$ and $B^{\ast} = C$.
In that realization we have:
\begin{enumerate}
\item $\mathfrak{k}_0 = \left(\begin{smallmatrix} A & 0 \\
  0& D   \end{smallmatrix} \right)$  with $A$ and $D$ skew-Hermitian,
\item $\mathfrak{k} = \left(\begin{smallmatrix} A & 0 \\
  0& D   \end{smallmatrix} \right)$  with $A$, resp. $D$, an arbitrary $p$ by $p$, resp. $q$ by $q$, complex matrix, 
\item  $\mathfrak{p}_0 = \left( \begin{smallmatrix} 0 & B \\
 B^* & 0                \end{smallmatrix} \right)$ with $B$ an arbitrary $p$ by $q$ complex matrix, and
\item $\mathfrak{p} = \left( \begin{smallmatrix} 0 & B \\
 C & 0                \end{smallmatrix} \right)$ with $B$, resp. $C$, an arbitrary $p$ by $q$, resp. $q$ by $p$, complex matrix. 
\end{enumerate}

\subsection{} For $v \in V$ we define $v^* \in V^*$ by 
\begin{equation}\label{Berger}
v^*(u) = (u,v)
\end{equation}
and $v_1 \otimes v_2^* \in V \otimes V^* = \End{(V)} \cong \mathfrak{g}$ by 
\begin{equation} \label{identificationwithdual}
 (v_1 \otimes v^{\ast}_2) (v) = (v,v_2) v_1.
\end{equation}
If $f \in V^*$ and $v \in V$ we will define $v \otimes f \in \End{(V)}$ in the same way.  
We will use this map to identify $V \otimes V^*$ and $\End{(V)}$ henceforth. 

We next note that we may identify $V^* \otimes V$ with $(V \otimes V^*)^* = \End(V)^* =\End(V^*)$ by the formula 
$$\langle f_1 \otimes v_1, v_2 \otimes f_2\rangle = f_1(v_2) f_2(v_1).$$
 We will use $ ^t A$ for the element of $ \End(V^*)$ corresponding to $A \in \End(V)$ hence $^tA (f) = f \circ A$ .   Using the above identifications we have 
$$  ^t(v \otimes f) = f \otimes v.$$
The adjoint map $A \to A^*$ relative to the Hermitian form $(\ , \ )$ is the antilinear map given by
$$ (u \otimes v^*)^* = ( v \otimes u^*).$$
Note that  $A \in \End(V)$ is in $\mathfrak{g}_0$ if and only if $A^* = -A$. Hence the conjugation map $\sigma_0$ of $\End(V)$ relative to the real form  $\mathfrak{g}_0$ is given by $\sigma_0(u \otimes v^*) = - (v \otimes u^*)$.  From either of the two previous sentences we have
\begin{lem} \label{intheunitaryla}
Let $x,y \in V$.  Then $x \otimes y^* - y \otimes x^*$ and $i( x \otimes y^* + y \otimes x^*)$ are in $\mathfrak{g}_0$.
\end{lem}

We now define a basis  for  $\mathfrak{p}_0$ by defining the basis vectors $e_{\alpha,\mu}$, $f_{\alpha, \mu}$, $1 \leq \alpha \leq p$, $p+1 \leq \mu \leq p+q$ that follow. We will not need a basis for $\mathfrak{k}_0$  in what follows.  By Lemma \ref{intheunitaryla} it follows that the elements below are in fact in $\mathfrak{g}_0$.  Here the matrices show only the action on the pair of basis vectors  $v_i,v_j$ in the formula immediately to the left of the matrix in the order in which they are given.  All other basis vectors are sent to zero
$$e_{\alpha,\mu} =  - v_{\alpha} \otimes v^{\ast}_{\mu}  +   v_{\mu} \otimes v^{\ast}_{\alpha}  =   \begin{pmatrix} 0 &  1\\ 1 &  0 \end{pmatrix}$$
and
$$f_{\alpha,\mu} =   i (- v_{\alpha} \otimes v^{\ast}_{\mu} -v_{\mu} \otimes v^{\ast}_{\alpha}) = \begin{pmatrix} 0 & i\\ -i & 0 \end{pmatrix}.$$

%We first define a basis for $\mathfrak{k}_0$
%by defining  basis vectors   $e_{\alpha,\beta}$, $e_{\mu,\nu}$ for $1 \leq \alpha <\beta \leq p$, $p+1 \leq \mu<\nu \leq p+q$  and $ f_{\alpha,\beta}$, $f_{\mu,\nu}$ for $1 \leq \alpha < \beta \leq p$, $p+1 \leq \mu < \nu \leq p+q$,  by
%$$e_{\alpha,\beta} =  v_{\beta} \otimes v^{\ast}_{\alpha}  -   v_{\alpha} \otimes v^{\ast}_{\beta}  =   \begin{pmatrix} 0 &  - 1\\ 1 &  0 \end{pmatrix},$$
%$$f_{\alpha,\beta} =   i  (v_{\beta} \otimes v^{\ast}_{\alpha}  +   v_{\alpha} \otimes v^{\ast}_{\beta} ) =   \begin{pmatrix} 0 &  i\\ i &  0 \end{pmatrix},$$
%$$e_{\mu,\nu} =    v_{\mu} \otimes v^{\ast}_{\nu} - v_{\nu} \otimes v^{\ast}_{\mu}    =   \begin{pmatrix} 0 &  - 1\\ 1 &  0 \end{pmatrix}$$
%and
%$$f_{\mu,\nu} =   - i (v_{\nu} \otimes v^{\ast}_{\mu}  +   v_{\mu} \otimes v^{\ast}_{\nu})  =   \begin{pmatrix} 0 &  i\\ i &  0 \end{pmatrix}.$$

%$h_{\alpha,\alpha} \ \text{and}
% \  h_{\mu,\mu} \in \mathfrak{t}_0 $ by 
%$h_{\alpha,\alpha} = i v_{\alpha} \otimes v^*_{\alpha}$
%and $h_{\mu,\mu} = - i v_{\mu} \otimes v^*_{mu}$.

We now describe the Ad$(K)$-invariant almost  complex structure $J_{\mathfrak{p}}$ acting on $\mathfrak{p}$ that induces the structure of Hermitian symmetric space
on $\U (p,q)/ (\U(p) \times \U(q))$.  Let $\zeta = e^{i \frac{\pi}{4}}$.  Then $\zeta$  satisfies  $\zeta^2 =i$.  Let  $a(\zeta)$ be the  diagonal $m$ by $m$
block 
matrix given by
$$a(\zeta) = \begin{pmatrix}  \zeta & 0\\
                              0 & \zeta^{-1}
\end{pmatrix}.$$
Then $a(\zeta)$ is in the center of $\U(p) \times \U(q)$ and the adjoint   action Ad$(a(\zeta))$ of  on $\mathfrak{p}$ induces the required almost complex structure, that is we have
\begin{equation} \label{complexstructureonp}
J_{\mathfrak{p}} =  \mathrm{Ad} ( a(\zeta)). 
\end{equation}
We have:
\begin{equation}\label{homogeneityunderJ}
a(\zeta) v_{\alpha} = \zeta \  v_{\alpha}, \ a(\zeta) v_{\mu} = \zeta^{-1}  v_{\mu}, \ a(\zeta) v_{\alpha}^{\ast}  = \zeta^{-1}  v_{\alpha}^{\ast} \ \text{and} \ a(\zeta) v_{\mu}^{\ast}  = \zeta \ v_{\mu}^{\ast}.
\end{equation}
In particular we have:
$$J_{\mathfrak{p}} e_{\alpha,\mu} =  f_{\alpha , \mu} \ \text{and} \ J_{\mathfrak{p}} f_{\alpha,\mu} = - e_{\alpha,\mu}, \quad 1 \leq \alpha,  p \leq \mu \leq p+q.$$

\subsection{} We define  elements $x_{\alpha,\mu}$ in $\mathfrak{p}'$ and $y_{\alpha,\mu}$ in $\mathfrak{p}''$ in $\mathfrak{g}$ by

\begin{align*} 
 &x_{\alpha,\mu} =  \  - v_{\alpha} \otimes v^{\ast}_{\mu} \ = \  \begin{pmatrix} 0 & \ 1\\ 0 &\  0 \end{pmatrix} \ \text{whence} \ J_{\mathfrak{p}} x_{\alpha,\mu} = i x_{\alpha,\mu}, \\
&y_{\alpha,\mu} = \sigma_0(x_{\alpha,\mu}) =    \  v_{\mu} \otimes v^{\ast}_{\alpha} =  \begin{pmatrix} 0 & 0\\ 1 & 0 \end{pmatrix} \ \text{whence} \ J_{\mathfrak{p}} y_{\alpha,\mu} = -i y_{\alpha,\mu}.
\end{align*}
 
The set $\{ x_{\alpha,\mu}: 1 \leq \alpha \leq p,\ p+1 \leq \mu \leq p+q \}$ is a basis for $\mathfrak{p}'$. In the corresponding matrix realization we have:
$$\mathfrak{p}' =  V_+ \otimes V^{\ast}_- = \left\{ \left( 
\begin{array}{cc} 0 & B \\ 0 & 0 \end{array} \right) \; : \; B \in \mathrm{M}_{p\times q} (\C) \right\}.$$
Similarly, the set $\{ y_{\alpha,\mu}: 1 \leq \alpha \leq p, \ p+1 \leq \mu \leq p+q \}$ is a basis for $\mathfrak{p}''$ and we have:
$$\mathfrak{p}'' = V_- \otimes V^{\ast}_+ =  \left\{ \left( 
\begin{array}{cc} 0 & 0 \\ C & 0 \end{array} \right) \; : \; C \in \mathrm{M}_{q\times p} (\C) \right\}.$$
Hence we have
$$\sigma_0( \mathfrak{p}') = \mathfrak{p}''.$$
As a consequence of the above computation we  note that we have isomorphisms of $K = \U(p) \times \U(q)$ modules 
$$\p' \cong \mathrm{M}_{p\times q} (\C)  \ \text{and} \ \mathfrak{p}'' \cong \mathrm{M}_{q\times p} (\C)$$
and the above splitting into $B,C$ blocks corresponds to the splitting $\mathfrak{p} = \mathfrak{p}' + \mathfrak{p}''$.

Using the identification $(U \otimes U^*)^* = U^* \otimes U$ we have 
\begin{equation}\label{dualsofsubalgebras}
(\mathfrak{p}')^* =  V^*_+ \otimes V_- \ \text{and} \ (\mathfrak{p}'')^* = V^*_- \otimes V_+ .
\end{equation}
Hence the transpose $\phi: V \otimes V^* \to V^* \otimes V$ given by $^t(v \otimes f)= f \otimes v$ induces isomorphisms
 $\mathfrak{p}'' \to (\mathfrak{p}')^*$ and $\mathfrak{p}' \to (\mathfrak{p}'')^*$. On the above basis these maps 
are given by
\begin{equation}\label{identificationsofdual}
^t y_{\alpha,\mu} = ^t (v_{\mu} \otimes v^*_{\alpha})=  v^*_{\alpha} \otimes v_{\mu} \ \text{and} \
^t x_{\alpha,\mu} =  ^t (-v_{\alpha} \otimes v^*_{\mu})= - v^*_{\mu} \otimes v_{\alpha}  .
\end{equation}
We will set  $\xi'_{\alpha,\mu} = v^*_{\alpha} \otimes v_{\mu}$ and $\xi''_{\alpha,\mu} = - v^*_{\mu} \otimes v_{\alpha}.$

We now give four definitions that will be important in what follows.  The notation below is  chosen to  help to make clear later that
the adjoints of the cocycles $\psi_{bq,aq}$ that we  construct and study in \S \ref{sec:KMlocal} of degrees $(a+b)q$ are {\it completely decomposable} in the sense that their values at a point of $x \in V^{a+b}$ are wedges of $(a+b)q$ elements of $\mathfrak{p}^{\ast}$. 
\begin{defn}
Suppose $x \in V_+$.  Then we define $\widetilde{x} \in \wedge^q(\mathfrak{p}') = \wedge^q(V_+ \otimes (V_-)^*) $ by
\begin{equation} \label{xtilde}
\widetilde{x} = (-1)^q(x \otimes v^*_{p+1}) \wedge \cdots \wedge (x \otimes v^*_{p+q}).
\end{equation}
\begin{rem} By \cite{FultonHarris}, pg. 80, there is an equivariant embedding $f_q: Sym^q(V_+) \otimes \wedge^q((V_-)^*) \to 
\wedge^{q}(V_+ \otimes (V_-)^*)$. Then $\widetilde{x} = f(x^{\otimes q} \otimes (v^*_{p+1}\wedge  \cdots \wedge v^*_{p+q}))$
\end{rem}
Suppose now that $f \in V^*_+$. Then we define $\widetilde{f} \in \wedge^q(\mathfrak{p}'')= \wedge^q(V_- \otimes (V_+)^*) $ by
\begin{equation} \label{ftilde}
\widetilde{f} = ( v_{p+1} \otimes f) \wedge \cdots \wedge (v_{p+q} \otimes f).
\end{equation}
\end{defn}
Using the transpose maps of Equation \eqref{identificationsofdual} we obtain 
$^t\widetilde{f} \in \wedge^q((\mathfrak{p}')^*)= \wedge^q((V_+)^* \otimes V_-) $ is given by 
\begin{equation} \label{ftilde} 
^t\widetilde{f} = ( f \otimes  v_{p+1}  ) \wedge \cdots \wedge (  f \otimes v_{p+q} )
\end{equation}
and $^t\widetilde{x}  \in \wedge^q((\mathfrak{p}'')^*)= \wedge^q((V_-)^* \otimes V_+) $ is given by
\begin{equation} \label{xtilde}
 ^t\widetilde{x} = (-1)^q( v_{p+1}^* \otimes x ) \wedge \cdots \wedge ( v_{p+q}^* \otimes x).
\end{equation}

\subsection{Theta-stable parabolic subalgebras}
Fix the Borel subalgebra of $\k$ to be the algebra of matrices in $\k = \mathfrak{u} ( p) \times \mathfrak{u} (q)$ (block diagonal), which are upper-triangular on $V_+=\C^p$ and lower-triangular on $V_-=\C^q$ w.r.t. these bases. We may take $i\mathfrak{t}_0$ as the algebra of diagonal matrices $(t_1 , \ldots , t_{p+q} )$. 

The roots of $\mathfrak{t}$ occuring in $\mathfrak{p}'$ are the linear forms $t_{\alpha}-t_{\mu}$. We now classify all the $\theta$-stable parabolic subalgebras $\mathfrak{q}$ of $\mathfrak{g}$. Let $X = (t_1 , \ldots , t_{p+q})$
be such that its eigenvalues on the Borel subalgebra are non-negative. Therefore 
$$t_1 \geq \ldots \geq t_p \quad \mbox{ and } \quad t_{p+q} \geq \ldots \geq t_{p+1}.$$
In \cite{TG} we associate two Young diagrams $\lambda_+$ and $\lambda_-$ to $X$: 
\begin{itemize}
\item The diagram $\lambda_+$ is the sub-diagram of $p\times q$ which consists of the boxes of coordinates $(\alpha , \mu)$ s.t. $t_{\alpha} > t_{\mu}$.
\item The diagram $\lambda_-$ is the sub-diagram of $p\times q$ which consists of the boxes of coordinates $(\alpha , \mu)$ s.t. $t_{p-\alpha+1} < t_{q-\mu+1}$.
\end{itemize}
A description of the possible pairs $(\lambda_+ , \lambda_-)$ that can occur is given in \cite[Lem. 6]{TG}.\footnote{Beware that in this reference $\mu$ refers to a sub-diagram of $p\times q$ 
which --- in our current notations --- corresponds to the complementary diagram of $\lambda_-$ in $p\times q$.}

Recall that we have associated to the parabolic subalgebra $\mathfrak{q}=\mathfrak{q} (X)$
the representations $V(\mathfrak{q})$ and $A_{\mathfrak{q}}$. The equivalent classes of both these representations only depends on the pair $(\lambda_+ , \lambda_-)$. We will therefore denote by $V(\lambda_+ , \lambda_-)$ and $A(\lambda_+ , \lambda_-)$ these representations.

\subsection{} To any Young diagram $\lambda$, we associate the irreducible $K$-representation 
$$V( \lambda) = S_{\lambda} (V_+) \otimes S_{{}^t \- \lambda} (V_-)^*.$$
Here $S_{\lambda} (\cdot )$ denotes the Schur functor (see \cite{FultonHarris}) and ${}^t \- \lambda \subset q \times p$ is the transposed diagram.
The $K$-representation $V(\lambda )$ occurs with multiplicity one in $\wedge^{|\lambda|} (V_+ \otimes V_-^*)$ where $|\lambda|$ is the size of $\lambda$. The $K$-representation 
$V(\lambda_+ , \lambda_-)$ is the Cartan product of $V(\lambda_+)$ and $V(\lambda_-)^*$. In our special situation --- that of Vogan-Zuckerman $K$-types ---
it occurs with multiplicity one in $\wedge^R \p$ where $R=|\lambda_+ | + |\lambda_- |$. 

Note that if $\lambda \subset p \times q$ is a Young diagram, we have:
$$S_{\lambda} (V_+)^* \cong S_{\lambda^{\vee}} (V_+) \otimes (\wedge^p V_+)^{-q},$$
where $\lambda^{\vee} = (q- \lambda_p , \ldots , q - \lambda_1)$ is the complementary diagram of $\lambda$ in $p \times q$.
We conclude that we have:
\begin{equation} \label{VZtype}
V(\lambda_+ , \lambda_-) \cong \left( S_{\lambda_+ + \lambda_-^{\vee}} (V_+) \otimes (\wedge^p V_+)^{-q} \right) \otimes \left( S_{{}^t \- \lambda_- + {}^t \- \lambda_+^{\vee}} (V_-) \otimes (\wedge^q V_- )^{-p} \right).
\end{equation}

\subsection{} Recall that, as a $\GL (V_+) \times \GL (V_-)$-module, we have (see \cite[Equation (19), p. 121]{Fulton}):
\begin{equation} \label{decGL}
\wedge^R ( V_+ \otimes V_-^*) \cong \bigoplus_{\lambda \vdash R} S_{\lambda} (V_+) \otimes S_{{}^t \- \lambda} (V_-)^* = \bigoplus_{\lambda \vdash R} V(\lambda ).
\end{equation}
Here we sum over all partition of $R$ (equivalently Young diagram of size $|\lambda|=R$).
We will see that, as far as we are concerned with special cycles, we only have to consider the subalgebra $(\wedge^{\bullet} \p)^{\rm special}$ of $\wedge^{\bullet} \p$ generated by the 
submodules $\wedge^{\bullet} (V_+ \otimes V_-^*)^{\SL (V_-)}$, resp. $\wedge^{\bullet} ((V_+ \otimes V_-^*)^*)^{\SL (V_-)}$, of $\wedge^{\bullet} \p^+$, resp. $\wedge^{\bullet} \p^-$. 
It amounts to consider the submodule of \eqref{decGL} which corresponds to the $\lambda$ of type $b\times q = (q^b)$. We conclude that
\begin{equation} \label{decSC1}
\begin{split}
(\wedge^{\bullet} \p)^{\rm special} & = \bigoplus_{a,b =0}^p S_{b \times q} (V_+) \otimes S_{a \times q} (V_+)^* \otimes (\wedge^q V_-)^{a-b} \\
& = \bigoplus_{a,b =0}^p S_{b \times q} (V_+) \otimes S_{(p-a) \times q} (V_+) \otimes (\wedge^p V_+)^{-q} \otimes (\wedge^q V_-)^{a-b} .
\end{split}
\end{equation}
This singles out certain parabolic subalgebras that we describe in more detail below. Before that we recall the description of the invariant forms.

\subsection{The Chern form} \label{chernform} Let $\lambda \subset p \times q$ be a Young diagram. Given a basis $\{ z_{\ell} \}$ of $V(\lambda)$ we denote by $\{ z_{\ell}^* \}$ the dual basis of $V(\lambda)^*$ and set
$$C_{\lambda} = \sum_{\ell} z_{\ell} \otimes z_{\ell}^* \in V(\lambda) \otimes V(\lambda )^* \subset \wedge^{|\lambda|, |\lambda|} \p.$$
The element $C_{\lambda}$ is independent of the choice of basis $\{ z_{\ell} \}$; it belongs to $(\wedge^{\bullet} \p )^K$. Now $C_{\lambda}$ belongs to $(\wedge^{\bullet} \p)^{\rm special}$ if and only if 
$\lambda = n \times q$, for some $n=0, \ldots , p$, and $C_{n \times q} = C_q^n$ in $\wedge^{\bullet} \p$ where 
$C_q= C_{(q)}$ is the {\it Chern class}. We conclude:

\begin{prop}
The subspace of $K$-invariants in $(\wedge^{\bullet} \p)^{\rm special}$ is the subring generated by the Chern class $C_q$.
\end{prop}

The $(q,q)$-invariant form on the symmetric space $X$ associated to the Chern class is called the {\it Chern form} in \cite{KudlaMillson3} where it is expressed in terms of the curvature two-forms 
$\Omega_{\mu , \nu} = \sum_{\alpha=1}^p \xi^{\prime \prime}_{\alpha  \nu} \wedge \xi^{\prime}_{\alpha \nu}$  by the formula
\begin{equation} \label{euler2}
c_q = \left( \frac{-i}{2\pi} \right)^q \frac{1}{q !} \sum_{\sigma \bar{\sigma} \in \mathfrak{S}_{q}} \mathrm{sgn} (\sigma \bar{\sigma}) \Omega_{p+ \sigma (1) ,p + \bar{\sigma} (1)} \wedge \ldots \wedge \Omega_{p+\sigma (q) ,p+ \bar{\sigma} (q )} \in \wedge^{q , q} \p^*.
\end{equation}

We now give a detailed description of the modules occuring in \eqref{decSC1}.

\subsection{The theta-stable parabolic $\mathrm{Q}_{b,0}$  and the Vogan-Zuckerman vector $e(bq,0)$ } \label{holpar}
We first define the theta-stable parabolics $\mathrm{Q}_{b,0}$ which will be related to the cohomology of type $(bq,0)$.  These parabolics will be maximal parabolics.  Suppose $b$ is a positive integer such that  $b <p$.  Let $E_b \subset V_+ $ be the span of $v_1,\cdots, v_b$.
We define $\mathrm{Q}_{b,0}$ to be the stabilizer of $E_b$. Equivalently, $\mathrm{Q}_{b,0}$  is the theta-stable parabolic corresponding to
$X =(\underbrace{1,1,\cdots,1}_{b},0,\cdots,0) \in i \mathfrak{t}_0$.   We now compute the nilradical of the Lie algebra $\mathfrak{q}_{b,0}$ 
of $\mathrm{Q}_{b,0}$.

Let $F_b$ be the orthogonal complement of $E_b$ in $V_+$ whence $F_b = \mathrm{Span}\{ v_{b+1}, \cdots v_p\}$.  Hence, since $V = V_+ \oplus V_-$, we have 
\begin{equation}\label{directsumdecompofV}
 V = E_b \oplus F_b \oplus V_ -.
\end{equation} 
Put $C_b = F_b \oplus V_-$.  Thus we have decomposed $V$ into the subspace $E_b$ and its orthogonal complement $C_b$ in $V$. 
Let $\mathfrak{u}_{b,0}$  be the nilradical of the Lie algebra $\mathfrak{q}_{b,0}$ of the Lie group $\mathrm{Q}_{b,0}$.  
We now  have
\begin{lem}\label{holomorphicnoncompactnilradical}
Using the identification $\End(V) \cong V \otimes (V^{\ast})$ 
$$\mathfrak{u}_{b,0} \cap \mathfrak{p}  \cong E_b \otimes (V_-^{\ast}) \subset \mathfrak{p}'.$$
Hence
$\{- v_{\alpha} \otimes v_{\mu}^* : 1 \leq \alpha \leq b,\  p+1 \leq \mu \leq p+q \}$ 
is a basis for $\mathfrak{u}_{b,0}\cap \mathfrak{p}$.
\end{lem}

\begin{proof}
 It is standard that
the nilradical of the maximal parabolic subalgebra which is the stabilizer of a complemented  subspace $E_b$  is the space of homomorphisms from the given complement $C_b$ into $E_b$ whence, using the above identification,
$$\mathfrak{u}_{b,0} = E_b \otimes C_b^{\ast} = (E_b \otimes F_b^{\ast})  \oplus (E_b \otimes V_-^{\ast}).$$
Clearly we have 
$$ \mathfrak{u}_{b,0} \cap \mathfrak{k} = E_b \otimes F_b^{\ast}  \ \text{and} \ \mathfrak{u}_{b,0}  \cap  \mathfrak{p} =  E_b \otimes V_-^{\ast}.$$
\end{proof}

The following lemma follows immediately from Lemma \ref{holomorphicnoncompactnilradical}.
\begin{lem} \label{holomorphicVZvector}
The vector  $e(bq,0) \in \wedge^{bq,0}(\mathfrak{p}) \cong \wedge^{bq}(\mathfrak{p}') $ associated to $\mathfrak{q}_{b,0}$ by 
\begin{equation}\label{KZVvector}
e(bq,0) = (-1)^{bq} \widetilde{v_1}\wedge \cdots \wedge \widetilde{v_b}
\end{equation}
is a Vogan-Zuckerman vector for the theta stable parabolic $\mathfrak{q}_{b,0}$.
\end{lem}
Note that $S_{b \times q}(V_+)$ is the irreducible representation for $\Aut(V_+)$ which has highest weight $ q \varpi_b$ where $\varpi_b$ is the $b$-th fundamental weight (i.e. the highest weight of the $b$-th exterior power of the standard representation). 
From Lemma \ref{holomorphicVZvector} and the general theory of Vogan-Zuckerman, we  have
\begin{lem}\label{formulaforhoolomorphicVZvectors}
The Vogan-Zuckerman vector $e(bq,0)$ is a highest weight vector of the  irreducible $K_{\C} \cong \GL(V_+) \times \GL(V_-)$-submodule   
$V(b \times q) := V(b \times q, 0) = S_{b \times q}(V_+) \otimes (\wedge^q(V_-^*))^{b}$ in $\wedge^{bq,0} \mathfrak{p}$. 
\end{lem}

\noindent
{\it Remark.}
As a representation of $K_{\C} = \mathrm{GL}(p) \times \mathrm{GL}(q)$ the representation $V(b \times q) 
\cong S_{b \times q}(\C^p) \otimes (\wedge^q (\C^q))^{-b}$ has highest weight 
$$(\underbrace{q,\cdots,q}_b,\underbrace{0,\cdots,0}_{p-b} ;\underbrace{-b,\cdots,-b}_q).$$
 
\subsection{The theta-stable parabolic $\mathrm{Q}_{0,a}$ and the Vogan-Zuckerman vector $e(0,aq)$}
Suppose $a$ is a positive integer such that $a <p$. Once again we let $E_a$ be the span of $\{v_1,\cdots,v_a\}$ and $F_a$ be the span of $\{v_{a+1},\cdots,v_p\}$.  Let $F^{\ast}_a$
be the span of $\{v^*_{p-a+1},\cdots,v^*_p\}$.    
We define $\mathrm{Q}_{0,a}$ to be the stabilizer of $F^*_a \subset V^{\ast}$. We note that the stabilizer of $F^*_a$ is the same as the stabilizer  of its annihilator $(F^*_a)^{\perp} = E_a + V_- \subset V$.  Thus $\mathrm{Q}_{0,a}$ is the theta-stable parabolic corresponding to
$X =(\underbrace{1,1,\cdots,1}_{a},\underbrace{0,0,\cdots,0}_{p-a},\underbrace{1,1,\cdots,1}_{q}) \in i \mathfrak{t}_0$.

The proof of the following lemma is similar to that of Lemma \ref{holomorphicnoncompactnilradical}. 
Let $\mathfrak{u}_{0,a}$ be the nilradical of the Lie algebra $\mathfrak{q}_{0,a}$ of the parabolic $\mathrm{Q}_{0,a}$. We have
\begin{lem}\label{antiKazhdannilradical}
$$\mathfrak{u}_{0,a} \cap \mathfrak{p} \cong V_- \otimes F^{\ast}_a \subset \mathfrak{p}''.$$
Hence $\{v_{\mu} \otimes v^*_{\alpha}: p-a+1 \leq \alpha \leq p,\  p+1 \leq \mu \leq p+q \}$ is a basis for $\mathfrak{u}_{0,a} \cap \mathfrak{p}$.
\end{lem}

We obtain the Vogan-Zuckerman vector $e(0,aq) \in \wedge^{0,aq}(\mathfrak{p}) = \wedge^{aq}(\mathfrak{p}'')$ by (see \eqref{ftilde})
$$e(0,aq)=\widetilde{v^*_{p-a+1}}\wedge \cdots \wedge \widetilde{v^*_p}.$$ 
The reader will observe that that 
$$e(0,aq) = \pm  \widetilde{w_0} \sigma_0(e(0,bq)$$
where $\widetilde{w_0}$ is the element of $\U(p)$ that exchanges the basis vectors $v_\alpha \ \text{and} \ v_{p+1 - \alpha},\  1 \leq \alpha \leq p$. The reader will also observe  that $e(0,aq)$
is a  weight vector for  the diagonal Cartan in $\mathfrak{u}( p)_{\C}$ with weight $(\underbrace{0,\cdots,0}_{p-a},\underbrace{-q,\cdots,-q}_a)$, that is, a highest weight of the representation $S_{ a \times q }((\C^p)^*)$. 
From Lemma \ref{antiKazhdannilradical} we have

\begin{lem}
The Vogan-Zuckerman vector $e(0,aq)$ is a generator for $\wedge^{aq}(\mathfrak{u}_{0,aq} \cap \mathfrak{p})$.  As such it is a highest weight vector for 
$Q_{K,a}$ and (from the above weight formula) it is the highest weight vector for the irreducible $K_{\C}$-submodule
$V(0, a \times q) = S_{(a \times q )}(V_+^{\ast}) \otimes (\wedge^qV_-)^{a}$ in $\wedge^{0,aq}\mathfrak{p}$.
\end{lem}

\noindent
{\it Remark.}
As a representation of $K_{\C} = \mathrm{GL}(p) \times \mathrm{GL}(q)$ the representation $S_{(a \times q )}((\C^p)^{\ast}) \otimes (\wedge^q(\C^q))^{a}$
has highest weight 
$(\underbrace{0,\cdots,0}_{p-a},\underbrace{-q,\cdots,-q}_a ;\underbrace{a,\cdots,a}_q)$.

\subsection{The theta-stable parabolic $\mathrm{Q}_{b,a}$  and the Vogan-Zuckerman vector $e(bq,aq)$ }
We now define the theta-stable parabolics $\mathrm{Q}_{b,a}$ which shortly will be related to the cocycles of Kudla-Millson and their 
generalization. Here we assume that $a$ and $b$ are positive integers satisfying $a +b \leq p$ whence $b \leq p-a$.   
The associated theta-stable parabolics will be next-to-maximal parabolics, that is, stabilizers of two step flags.  

As before, we let  $E_b \subset V_+ $ be the span of $v_1,\cdots, v_b$ and $E_{p-a} =F^*_{p-a} \subset V_+$ be the span of $v_1,\cdots, v_{p  -a}$.  
Since $b \leq p-a$ we find $E_b \subset E_{p-a} $ and we obtain the  two  step flag 
$$\mathcal{F}_{b,a}=E_a \subset E_{p-a} + V_- \subset V.$$  
%Let $D_{b,a}$ be the orthogonal complement of $E_b$ in $E_{p-a}$ so
%$$E_{p-a} = E_b + D_{b,a}.$$ 
%Also as before, we let  $F_{p-a} \subset V_+$ be the span of $v_{p-a+1},\cdots, v_p$ (so $F_{p-a}$ is the orthgonal complement of $E_{p-a}$ 
%in $V_+$) and let  $F^{\ast}_a$ be the span of $\{v^*_{p-a+1},\cdots,v^*_p \}$.  Then $E_{p-a}$ is the annihilator of $F^*_a$ in $V_+$ and $E_{p-a}+ V_-$ is the annihilator of
%$F^*_a$ in $V$ .    We have
%$$V_+ = E_{p-a} + F_{p-a} = E_b + D_{b,a} + F_{p-a}.$$   
%Since $V = V_+ \oplus V_-$, we have 
%\begin{equation}\label{directsumdecompofV}
 %V = E_b \oplus D_{b,a} \oplus F_{p-a}  \oplus V_-. 
%\end{equation} 
%Hence, we have the two step (complemented) flag
%$$\mathcal{F}_{b,a} =  E_b \subset (F^*_a)^{\perp} =  E_b + D_{b,a} + V_-  \subset V =E_b +  D_{b,a}+ V_- + F_{p-a}.$$
Let $\mathrm{Q}_{b,a}$ be the stabilizer of the flag $\mathcal{F}_{b,a}$. Thus $\mathrm{Q}_{b,a}$ is the theta-stable parabolic corresponding to
$X =(\underbrace{1,1,\cdots,1}_{a},\underbrace{0,0,\cdots,0}_{p-a},\underbrace{-1,-1,\cdots,-1}_{q}) \in i \mathfrak{t}_0$. 
 Since $\mathrm{Q}_{b,a}$  is the intersection of stabilizers of the subspaces $E_b$ and $E_{p-a} + V_- $ comprising the flag $\mathcal{F}_{b,a}$, the $\mathrm{Q}_{b,a}$ is the intersection of the two previous ones
\begin{lem}
$$\mathrm{Q}_{b,a} = \mathrm{Q}_{b,0} \cap \mathrm{Q}_{0,a}.$$
\end{lem}

Let $\mathfrak{u}_{b,a}$  be the nilradical of the Lie algebra $\mathfrak{q}_{b,a}$ of the Lie group $\mathrm{Q}_{b,a}$.  
It is a standard result that if two parabolic subalgebras $\mathfrak{q}_1$ and $\mathfrak{q}_2$ intersect in a parabolic subalgebra
$\mathfrak{q}$ then the nilradical of $\mathfrak{q}$ is the sum of the nilradicals of $\mathfrak{q}_1$ and $\mathfrak{q}_2$. Hence, we have
\begin{equation}\label{KMnoncompactnilradical}
\mathfrak{u}_{b,a} \cap \mathfrak{p}  \cong (E_b \otimes V^{\ast}_-) \oplus (V_- \otimes F^{\ast}_a)= 
(\mathfrak{u}_{b,0} \cap \mathfrak{p}) + (\mathfrak{u}_{0,a} \cap \mathfrak{p})  .
\end{equation}
We obtain as a corollary that 
$\{ v_{\alpha} \otimes v^*_{\mu}, \ v_{\mu} \otimes  v^*_{\beta}: 1 \leq \alpha \leq b, \ p-a+1 \leq \beta \leq p,  \ p+1 \leq \mu \leq p+q \}$ 
is a basis for $\mathfrak{u}_{b,a} \cap \mathfrak{p}$.

\medskip
\noindent
{\it Remark.} We have 
$$\mathfrak{u}_{b,a}  \cap  \mathfrak{p}' = E_b \otimes V_-^{\ast} \ \text{and} \ \mathfrak{u}_{b,a}  \cap  \mathfrak{p}''= V_- \otimes F_a^{\ast}.$$

\medskip

We then define the Vogan-Zuckerman vector $e(bq,aq) \in \wedge^{bq,aq}\mathfrak{p}_{\C}  \cong (\wedge^{bq}\mathfrak{p}') \otimes (\wedge^{aq}\mathfrak{p}')  $ associated to $\mathrm{Q}_{b,a}$ by
\begin{multline}\label{KVZvector}
e(bq,aq) = e(bq,0) \wedge e(0,aq)  \\ = (-1)^{bq}(\widetilde{v_1}\wedge \cdots \wedge \widetilde{v_b}) \wedge 
(\widetilde{v^*_{p-a+1}}  \wedge \cdots \wedge  \widetilde{v^*_p}) \in \wedge^{bq}(V_+ \otimes (V_-)^{\ast})  \otimes \wedge^{aq} (V_- \otimes (V_+)^{\ast}) .
\end{multline}

Note that the Cartan product of the representations  $S_{(b \times q )}(V_+) \otimes (\wedge^q(V^{\ast}_-))^{b}$ and $S_{(a \times q )}(V^{\ast}_+)  \otimes (\wedge^q(V_-))^{a}$ is the irreducible representation for $K_{\C}$ which has highest weight 
$$(\underbrace{q,\cdots,q}_b,0,\cdots,0,\underbrace{-q,\cdots,-q}_a;\underbrace{a-b,\cdots,a-b}_q).$$ 

\begin{lem}\label{formulaforKVZvectors}
The Vogan-Zuckerman vector $e(bq,aq)$ is a generator for $\wedge^{(a+b)q}(\mathfrak{u}_{b,a}\cap \mathfrak{p} )$.  As such it is the highest weight vector of the irreducible $K_{\C}$-submodule $V(b\times q,a\times q ) \subset \wedge^{bq,aq}\mathfrak{p}$ isomorphic to the Cartan product of the representations  $S_{(b \times q )}(V_+) \otimes (\wedge^q(V^{\ast}_-))^{b}$  and $S_{(a \times q )}(V^{\ast}_+)  \otimes (\wedge^q(V_-))^{a}$. 
\end{lem}

\subsection{} Wedging with the Chern class defines a linear map in
$$\mathrm{Hom}_{K}( (\wedge^{\bullet} \p)^{\rm special} , (\wedge^{\bullet} \p)^{\rm special}).$$
We still denote by $C_q$ the linear map. It follows from \cite[Prop. 10]{TG} that
$C_q^k (V(b\times q,a\times q))$ is a {\it non-trivial} $K$-type in $(\wedge^{(b+k)q,(a+k)q} \p)^{\rm special}$ if and only if $k \leq p-(a+b)$. This leads to the following:

\begin{prop} \label{PropLit}
The irreducible $K$-types $V(b\times q,a\times q)$, with $a+b \leq p$, are the only Vogan-Zuckerman $K$-types that occur in the subring $(\wedge^{\bullet} \p )^{\rm special}$.
Moreover:
$$\mathrm{Hom}_{K} \left(V (b\times q ,a \times q) , (\wedge^{nq} \p )^{\rm special }\right) = \left\{ \begin{array}{ll}
\C \cdot C_{q}^{k} & \mbox{ if } n-(a+b)= 2k ,  \\
0 & \mbox{ otherwise}. 
\end{array} \right.$$
Here $k$ is any integer in $\{0 , \ldots , p-(a+b)\}$. 
\end{prop}
\begin{proof} The Vogan-Zuckerman $K$-types are the representations $V( \lambda_+ , \lambda_- )$. Now it follows from \eqref{VZtype} that if such a $K$-type
occurs in $(\wedge^{\bullet} \p )^{\rm special}$ then $S_{{}^t \- \lambda_- + {}^t \- \lambda_+^{\vee}} (V_-)$ is isomorphic to a power of $\wedge^q V_-$ but this can only happen if the diagram ${}^t \- \lambda_- + {}^t \- \lambda_+^{\vee}$ has shape $q \times c$ for some $c$ this forces both ${}^t \- \lambda_-$ and ${}^t \- \lambda_+^{\vee}$ to be of this shape. We conclude that there exists integers $a,b$ such that $\lambda_+ = b \times q$ and $\lambda_- = a \times q$. This proves the first assertion of the proposition.

We now consider the decomposition \eqref{decSC1} into irreducibles. It follows from the Littlewood-Richardson rule (see e.g. \cite[Corollary 2, p. 121]{Fulton}) that 
$$\dim \mathrm{Hom}_{\GL (V_+)} (S_{((2q)^b , q^{p-a-b})} (V_+ ) , S_{(q^B)} (V_+) \otimes S_{(q^{p-A})} (V_+))$$
equals the number of Littlewood-Richardson tableaux of shape  $((2q)^b , q^{p-a-b}) / (q^B)$ and of weight $(q^{p-A})$. 
The latter is $0$ if $A \leq a$ or $B \leq b$. Now if $A \geq a$ the shape $((2q)^b , q^{p-a-b}) / (q^B)$ is the disjoint union of two rectangles and it is immediate that there is at most one semistandard filling of $((2q)^b , q^{p-a-b}) / (q^B)$ of content $(q^{p-A})$ that satisfies the reverse lattice word condition, see \cite{Fulton}, \S 5.2, page 63, (the first row must be filled with ones, the second with twos etc.). 
We conclude that the multiplicity of $S_{((2q)^b , q^{p-a-b})} (V_+ )$ in $S_{(q^B)} (V_+) \otimes S_{(q^{p-A})} (V_+)$ equals $1$ if $A=a+k$ and $B=b+k$ for
some $k=0 , \ldots , p-(a+b)$, and $0$ otherwise.
Since in the former case
$$C_q^{k} (V(b \times q,a \times q)) \cong  S_{((2q)^b , q^{p-a-b})} (V_+ ) \otimes (\wedge^p V_+ )^{-q} \otimes (\wedge^q V_-)^{a-b}$$
in $(\wedge^{(a+b+2k)q} \p )^{\rm special }$, this concludes the proof.
\end{proof}

\medskip
\noindent
{\it Remark.} Proposition \ref{PropLit} implies the decomposition \eqref{subringSC} of the Introduction. 

\medskip

The following proposition shows that the Vogan-Zuckerman types $V(b \times q,a \times q)$ are essentially the only $K$-types to give
small degree cohomology.

\begin{prop} \label{P:cohomrep}
Consider a cohomological module $A_{\mathfrak{q}}$ and let $V(\mathfrak{q}) = V(\lambda_+ , \lambda_-)$ the corresponding Vogan-Zuckerman $K$-type. Suppose that $R= \dim (\mathfrak{u} \cap \mathfrak{p})$ is strictly less than $p+q-2$. Then: 
either $(\lambda_+ , \lambda_-) = (b \times q , a \times q)$, for some non-negative integers $a,b$ such that $a+b \leq p$, or $(\lambda_+ , \lambda_-) = (p \times b , p \times a)$
for some non-negative integers $a,b$ such that $a+b \leq q$.
\end{prop}
\begin{proof} See \cite[Fait 30]{TG}.
\end{proof}

\medskip
\noindent
{\it Remark.} The cohomological modules corresponding to the second case of Proposition \ref{P:cohomrep} correspond to the very same module
where we just exchange the roles of $p$ and $q$.

\medskip

\section{The action of $\U(p) \times \U(q) \times \U(a) \times \U(b)$ in the twisted Fock model } 

In this section we review the construction of the Fock model for the the Weil  representation of the dual pair $\U(p,q) \times \U(a,b)$.  We will thereby  explain  the  reversal of $a$ and $b$ in the notations of the preceding sections: relative Lie algebra cohomology for $\mathfrak{u}(p,q)$ of Hodge bidegree  $bq, aq$ comes from the dual pair $\U(p,q) \times \U(a , b )$. 

It is an important  point that we must twist the restriction of the Weil representation to $\widetilde{\U}(p,q)$ (see below for the notation)  by the correct half-integer power of the determinant in order to that  the special  cocycles $\psi_{bq,aq}$ --- to be defined in the next section --- be $\U(p) \times \U(q)$-equivariant maps 
from $\wedge^{\bullet}(\mathfrak{p})$ to the Weil representation.  There is a unique such power (namely $\det^{\frac{a-b}{2}}$). { \it Otherwise $\psi_{bq,aq}$ would not  be a relative Lie algebra cochain}.  

It is an important  point in what follows that to detect the twist part of the action
of $\U(p) \times \U(q)$ on the Fock model  it is enough to determine the action of $\U(p) \times \U(q)$  on the vacuum vector $\psi_0$ (the constant polynomial $1$ in the Fock model).  Thus  Lemmas 
\ref{restrictionofvacuumcharacteraequalsone} and \ref{restrictionofvacuumcharacteraequalszero} and their accompanying corollaries and remarks
 which give formulas for the action of the restriction of the Weil representation to $\U(p) \times \U(q)$ (actually their covers) on $\psi_0$ will play an important role in what follows.

\subsection{The square root of the determinant }
In what follows we will need the square root of the determinant and its properties for various unitary groups.  Let $\U (V)$ be the isometry group of a Hermitian space $V$  and let $\mathfrak{u}(V)$ be its Lie algebra.  Hence we have a Lie algebra homomorphism $\mathrm{Tr}: \mathfrak{u}(V) \to \C$ (the trace).   We define the covering group  $\widetilde{\U}(V)$ of $\U(V)$ to be the pull-back by $\det$  of the covering $\pi:S^1 \to S^1$ given by  $\pi(z)= z^2$.  Hence
$$\widetilde{\U}(V) = \{(g,z): g \in \U(V), z \in S^1 \ \text{with} \ \det(g) = z^2\}.$$
We then define $\det^{1/2}:\widetilde{\U}(V) \to S^1$ by 
$${\det}^{1/2}((g,z)) = z.$$
For $k \in \Z$ we define the character $\det_{\U(V)}^{k/2}$ by $\det_{\U(V)}^{k/2} = \big( \det_{\U(V)}^{1/2})^k$. 
We leave the proofs of the two following lemmas to the reader.

\begin{lem} \label{pullbackofdet}
Suppose  $f:H \to \U(V)$ is a homomorphism. Then the pull-back by $f$ to $H$ of   the cover $\widetilde{\U}(V) \to \U(V)$  is equal to 
the pull-back of the  cover $\pi:S^1 \to S^1$ above by $\det\circ f: H \to S^1$.  In particular the pull-back of the cover by $f$ is trivial  
if and only if $H$ has a character $\chi:H \to S^1$ such that 
$\chi(h)^2 = \det \circ f(h), h \in \ H$. There is a lift $\widetilde{f}:\widetilde{H} \to \widetilde{\U}(V)$ of $f$ such that
$${\det}^{1/2}_{\widetilde{\U}(V)} \circ \widetilde{f} = ({\det}_{\U(V)}\circ f)^{1/2}.$$
\end{lem}

We will also need. 

\begin{lem} \label{squarerootofdeterminant}
Define $\chi$ to be the unique character of the universal cover of $\U(V)$ with derivative $\frac{k}{2} \cdot 
\mathrm{Tr}$.  Then $\chi$ descends to the  cover $\widetilde{\U}(V)$ of $\U(V)$ and the descended character is $\det_{\U(V)}^{k/2}$.  
\end{lem}

\medskip
\noindent
{\it Remark.} We see from the lemma that ``$\det_{\U(V)}^{1/2}$ has the same functorial properties as the half-trace''. 

\medskip

There are three  results  concerning the behaviour of $\det_{\U(V)}^{1/2}$ under homomorphisms which will  need below. The first two are special cases of  Lemma \ref{pullbackofdet}. The first result  is

\begin{lem}\label{restrictiontoblocks}
Suppose $V_1 \subset V_2$ is a subspace of a Hermitian space such that the restriction of the form on $V_2$ to $V_1$ is nondegenerate. Then  we have an inclusion $\widetilde{\U}(V_1) \to \widetilde{\U}(V_2)$ and 
$$ \mathrm{det}_{\U(V_2)}^{1/2}|_{\widetilde{\U}(V_1)} = \mathrm{det}_{\U(V_1)}^{1/2}.$$
\end{lem}

The second is the behaviour under the diagonal action of $\U(V)$ on the direct sum $V^a$.  

\begin{lem}\label{diagonalactiononasum}
Let $V$ be a Hermitian space and $a$ be a positive  integer.  Let $f: \U(V) \subset \U(V^a)$ be the diagonal inclusion. Then we have
$\det_{\U(V^a)} \circ f = \det_{\U(V)}^a$.  Hence the cover 
$\widetilde{\U}(V^a) \to \widetilde{\U}(V^a )$ pulls back to the nontrivial covering group $\widetilde{\U}(V)$ if and only if $a$ is odd.  
Furthermore we have 
$$ \mathrm{det}_{\U(V^a)}^{1/2} |_{\widetilde{\U}(V)}=  \mathrm{det}_{\U(V)}^{a/2}.$$
\end{lem}

The third result we need is the behaviour of $\det_{\U(V)}^{1/2}$ under complex conjugation.  Suppose  we have chosen a basis $\{v_1,v_2,\cdots,v_n\}$ for $V$ such that for all $i,j$ we have $(v_i,v_j) \in \R$.   Let $V_0 \subset V$ be the real form of $V$ given by $V_0 = \mathrm{span}_{\R}(\{v_1,\cdots,v_n\})$. Let $\tau_{V_0}$ be complex conjugation of $V$ relative to the real form $V_0$. The above assumption on the inner products of basis vectors is equivalent to 
\begin{equation} \label{innerproductsreal} 
(\tau_{V_0}(x), \tau_{V_0}(y)) = \overline{(x,y)}, 1 \leq i,j \leq n \ \text{and} \ x,y \in V.
\end{equation} 
Equation \eqref{innerproductsreal} implies that if $g \in \U(V)$ then  $\tau_{V_0} \circ g \circ \tau_{V_0} \in \U(V)$ and hence $\tau_{V_0}$ induces a conjugation map  
$\tau_0: \U(V) \to \U(V)$ given by 
$\tau_0(g) =  \tau_{V_0} \circ g \circ \tau_{V_0}$. We note that the matrix of $\tau_0(g)$ relative to the basis $v_1,\cdots, v_n$ is the conjugate of the matrix of $g$ and hence
\begin{equation} \label{conjugationanddet}
\det(\tau_0(g)) = \overline{\det(g)} = {\det} ^{-1}(g).
\end{equation}
It then follows that  $\tau_0$ induces a map of coverings
$\widetilde{\tau}_0: \widetilde{\U}(V) \to \widetilde{\U}(V)$ 
given by
$$\widetilde{\tau}_0(g,z) = (\tau_0(g), \overline{z}).$$

The next elementary lemma will be  important in what follows.

\begin{lem}\label{conjugationchangessign}
Let $V$ and $V_0$ be as above.   Then   we have
$$ \mathrm{det}_{\U(V)}^{1/2}\circ \widetilde{\tau}_0  =  \mathrm{det}_{\U(V)}^{-1/2}.$$
\end{lem}

\begin{proof}
We have 
$${\det}^{1/2}(\widetilde{\tau}_0(g,z)) = {\det}^{1/2}((\tau_0(g), \overline{z}))= \overline{z} = z^{-1} = ({\det}^{1/2}(g,z))^{-1}.$$
\end{proof}

\subsection{The Fock model of the oscillator representation associated to a point in the symmetric space of $\U(p,q)$} \label{par:PFS}
We now recall the description of the Fock model of the Weil representation and the associated {\it polynomial Fock space}. We keep the notation of  Section \ref{linearalgebra}. There is one Fock model for each positive definite complex structure.  However since we will be concerned only with the restriction to the unitary group $\U(V)$ we will limit ourselves to the positive definite structures coming from the symmetric space of $\U(V)$. We will see below  that there is one such model (structure)  for each point (splitting $V = V_+ + V_-$ ) in the symmetric space of the unitary group.  In what follows we will assume $\dim V_+  = p$ and $\dim V_- = q$  and set
$ m = p+q$.

Recall that there is a canonical  positive almost complex structure $J_0$  associated to the Hermitian majorant $(,)_0$ of $(,)$ corresponding to the decomposition $V=V_+ + V_-$ given by the formula
$$ J_0 = J \circ \theta_{V_-}.$$
The $+ i$ eigenspace $V^{\prime_0}$ of $J_0$ acting on $V \otimes \C$ is $V'_+ + V''_-$.

We define the  Gaussian $\varphi_0$ on $V^{ \prime_0}$ associated to the majorant $( , )_0$ by
$$\varphi_0(v^{\prime_0}) =  \exp( - \pi (v^{\prime_0},v^{\prime_0})_0).$$
Here we have transferred the majorant $( , )_0$ from $V$ to $V'$ using the canonical isomorphism 
$v \to v^{\prime_0} = \frac{1}{2}( v \otimes 1 - J_0 v \otimes i )$. 
We finally define the Gaussian measure $\mu$ on $V^{ \prime_0}$ by
$$\mu = C \varphi_0 \mu_0$$
where $\mu_0$ is Lebesgue measure on $V^{ \prime_0}$  and $C$ is chosen so that the measure of $V^{ \prime_0}$ is one. 

We  define the {\it Fock space} $\mathcal{F}(V)$ to be the space of $J_0$-holomorphic functions (technically the `half-forms') on $V^{\prime_0}$ which are square integrable for the Gaussian measure. 

We define the {\it polynomial Fock space} $\mathcal{P}(V)$ to be the subspace of $\mathcal{F}(V)$ consisting of $J_0$-holomorphic polynomials.
Identifying as usual polynomial functions on a space with the symmetric algebra on its dual, we conclude that
$$ \Pol(V^{\prime_0}) \cong \mathrm{Sym}((V^{\prime_0})^{\ast})= \mathrm{Sym}((V'_+ + V''_-)^{\ast}) \cong \mathrm{Sym}((V'_+)^{\ast}) \otimes \mathrm{Sym}((V''_-)^{\ast}).$$

It will be important in what follows to note that since $V'$ and $V''$ are dually paired by (the complex bilinear extension of) the 
symplectic form $A$ we have
$$\mathrm{Sym}((V'_+)^{\ast})\cong \mathrm{Sym}(V''_+) \ \text{and} \ \mathrm{Sym}((V''_-)^{\ast})\cong \mathrm{Sym}(V'_-).$$

Hence
\begin{equation} \label{definitionofpolynomialFockspace}
\mathcal{P}(V) = \mathrm{Sym}((V'_+)^{\ast}) \otimes \mathrm{Sym}((V''_-)^{\ast}) \cong \mathrm{Sym}(V''_+) \otimes \mathrm{Sym}(V'_-) .
\end{equation}
Note that the spaces $\mathcal{F}(V)$ and $\mathcal{P}(V)$ depend on the choice of positive almost complex structure 
$J_0$.  

The  point of the previous construction is that there is a {\it unitary}  representation  of the metaplectic group $\omega: \Mp(V, \langle ,  \rangle) \to \U(\mathcal{F}(V))$. This action provides a model of the Weil representation called the Fock model. 
In what follows we will abbreviate $\Mp(V, \langle ,  \rangle)$ to  $\Mp$ and its maximal compact subgroup given by the two-fold cover of the unitary group of the {\it positive} Hermitian space $(V  ,(,)_0)$ to   $\widetilde{\U}_0$. In case we have chosen a basis for $V$ then we will write $\Mp(2m,\R)$ in place of $\Mp$.

The polynomial Fock space $\mathcal{P} (V)$ is precisely the space of $\widetilde{\U}_0$-finite vectors of the Weil representation $\omega$ (see e.g. \cite{HoweT,BorelWallach}) and
the action of $\widetilde{\U}_0$ on the polynomial Fock space is given by the following formula, see \cite{Folland}, Proposition 4.39, pg. 184. If $k \in \widetilde{\U}_0$ and $P \in \mathcal{P} (V)$ then 
\begin{equation} \label{Weiltwist}
(\omega (k) \cdot P) (v') = \mathrm{det}_{\U_0} (k)^{-\frac12} P(k^{-1} v').
\end{equation}

\medskip
\noindent
{\it Remark.}
The determinant factor comes from the fact we should have multiplied $P$ in the above by the ``half-form''(square root of the complex volume form)  
$$\sqrt{dz_1\wedge dz_2\wedge \cdots dz_m}.$$

We note that the constant polynomial $1$ satisfies 
$$\omega (k)  \cdot 1 = \mathrm{det}_{\U_0} (k)^{-\frac12} \cdot 1.$$
In general for each model of the Weil representation there is a unique vector $\psi_0$ which corresponds to $1$ traditionally called the vacuum vector.  The vacuum vector then transforms according to
$$\omega (k)  \cdot \psi_0 = \mathrm{det}_{\U_0} (k)^{-\frac12} \cdot \psi_0.$$

\medskip

\subsection{The restriction of the Fock model to $\widetilde{\U}(p,q)$ and its half-determinant twists}
In this section we will study the ``restriction of the Weil representation of $\Mp(2m,\R)$ to $\widetilde{\U}(p,q)$''. We use quotations because the map $ \U(p,q) \to \Sp(2m,\R)$ involves a conjugation.

\subsubsection{The inclusion $\widetilde{\j}_{\U(p,q)}$ of $\widetilde{\U}(p,q)$ in $\Mp(2m,\R)$ }
The conjugation  alluded to above comes about because  the matrix $M'$ of the symplectic form $A$ relative to the natural basis $\cal{B} = \{v_1,\cdots, v_m, i v_1,\cdots, iv_m \}$  for the real vector space $V_{\R}$ underlying $V$  is given by  
\begin{equation}\label{wrongmatrix}
M' = \begin{pmatrix}  0 & I_{p,q}\\ -I_{p,q} &0 \end{pmatrix} \ \text{instead of} \ M=  \begin{pmatrix}  0 & I_m\\ -I_m &0 \end{pmatrix}.
\end{equation}
Here $I_{p,q}$ is the $m$ by $m$ matrix given by 
$$I_{p,q} = \begin{pmatrix} I_p & 0\\ 0 & - I_q \end{pmatrix}.$$
Hence, the basis {\it $\cal{B}'$ is not a symplectic basis}. 
Accordingly we let $\cal{B}'$ be the new basis given by 
$$\cal{B}' = \{v_1,\cdots, v_p, - v_{p+1},\cdots,-v_{p+q},  i v_1,\cdots, iv_m \}.$$ 
Then $\cal{B}'$ is a symplectic basis and the change of basis matrix $Z_{p,q}$ is given by
\begin{equation} \label{changeofbasismatrix}
Z_{p,q} = \begin{pmatrix} I_{p,q} & 0 \\ 0 & I_m \end{pmatrix}.
\end{equation}

In what follows let $\Sp'(2m,\R)$, resp. $\mathfrak{sp}'(2m,\R)$, denote the group of isometries of the form $M'$ 
(so $g M' g^* = M'$), resp. the Lie algebra of this group.  
We find then that under the canonical map $i_{\mathcal{B}}: \GL(n,\C) \to \GL(2n,\R)$ (associated to the map $\GL(V) \to \GL(V_{\R})$ by the basis $\cal{B}$) the image of $\U(p,q)$ lies in $\Sp'(2m,\R)$.  Note
$$i_{\mathcal{B}}(a+ ib ) = \begin{pmatrix} a & - b\\ b & a \end{pmatrix}.$$
We let $\j_{\mathcal{B}}$ be the restriction of   $i_{\mathcal{B}}$ to $\U(p,q)$. 
We define $F_{p,q}:\Sp'(2m,\R)) \to \Sp(2m,\R)$ by
$F_{p,q} = \Ad (Z_{p,q})$ and 
we define  $\j_{\U(p,q)}$ by 
$$\j_{\U(p,q)} = F_{p,q} \circ \j_{\mathcal{B}}.$$ 
 We find that $\j_{\U(p,q)}$ maps $\U(p,q)$ into $\Sp(2m,\R)$. 

We let $\widetilde{\U}(p,q)$ be the pull-back of the metaplectic cover of $\Sp(2m,\R)$ by the embedding $\j_{\U(p,q)}$.  It is well known  that   $\widetilde{\U}(p,q)$ is the cover  obtained by taking the square root of $\det: \U(p,q) \to \mathbb{S}^1$, see for example \cite{Paul} \S 1.2. 

We let $\Mp'(2m,\R)$ be the pull-back by $F_{p,q}$  of the metaplectic covering of $\Sp(2m,\R)$.  Hence we have a lift 
$\widetilde{F}_{p,q}$ of $F_{p,q}$  such that  $\widetilde{F}_{p,q}:  \Mp'(2m,\R) \to \Mp(2m,\R)$.  Since the covering $\widetilde{\U}(p,q)$ of $\U(p,q)$ is pulled back from the metaplectic covering
of $\Sp(2m,\R)$ by the composition $F_{p,q} \circ \j_{\mathcal{B}}$ we also have  lifts $\widetilde{\j}_{\mathcal{B}}$ of 
$j_{\mathcal{B}}$ and $\widetilde{\j}_{\U(p,q)}$ of $\j_{\U(p,q)}$. Since two coverings of a  map that agree at a point agree everywhere (here the point is the identity) we have 

\begin{equation} \label{inclusionofU}
\widetilde{\j}_{\U(p,q)} = \widetilde{F}_{p,q}  \circ \widetilde{\j}_{\mathcal{B}}.
\end{equation}

\subsubsection{The action of $\widetilde{\U}(p) \times \widetilde{\U}(q)$ on the vacuum vector $\psi_0$}
Let $\tau_{V_-}$ be the real linear transformation that is the negative of complex conjugation on $V_-$.  Then in terms of the basis $\cal{B}$ the matrix of $I_{V_+} \oplus \tau_{V_-}$ is $Z_{p,q}$. Note that $\tau_q= \Ad \tau_{V_-}$ acting on $\U(q)$ is complex conjugation and lifts to
the operator $\widetilde{\tau}_q$ on $\widetilde{\U}(q)$ given by 
$$\widetilde{\tau}_q (g,z) = (\tau_q(g),\overline{z}).$$  
In what follows we will need the multiplication map $\widetilde{\mu}_{p,q}: \widetilde{\U}(p) \times \widetilde{\U}(q) \to \widetilde{\U}(p,q)$ given by
$$\widetilde{\mu}_{p,q}(((k_1,1),z_1)((1,k_2),z_2) = ((k_1,k_2),z_1z_2).$$
Note that $\widetilde{\mu}_{p,q}$ factors through the inclusion $\widetilde{\U(p) \times \U(q)} \to \widetilde{\U}(p,q)$ and that it has kernel $\Z/2$.  
We have an analogous map $\widetilde{\mu}_{p+q}: \widetilde{\U}(p) \times \widetilde{\U}(q) \to \widetilde{\U}(p+q)$.

We claim that the following diagram commutes. First the induced diagram of maps on the base space commutes, see Equation \eqref{complexconjugation}.  Hence the diagram of homomorphisms  of covering groups commutes --- see the sentence preceding  Equation \eqref{inclusionofU}).
\[
\begin{CD} 
\widetilde{\U}(p) \times \widetilde{\U}(q)  @>\widetilde{\mu}_{p,q}>> \widetilde{\U}(p,q) @>\widetilde{\j}_{\U(p,q)} >> \Mp'(2m,\R)\\ 
@V 1 \times \widetilde{\tau_q} VV                @.                                                   @VV \widetilde{F}_{p,q}V \\
 \widetilde{\U}(p) \times \widetilde{\U}(q) @>\widetilde{\mu}_{p+q}>> \widetilde{\U}(p+q) @ >  \widetilde{\j}_{\U(p+q)}>> \Mp(2m,\R)
\end{CD} 
\]

\medskip
\noindent
{\it Remark.}  In the diagram we are comparing two different mappings  of $\widetilde{\U}(p) \times \widetilde{\U}(q)$ into $\Mp(2m,\R)$. The top  mapping factors through $\widetilde{\U}(p,q)$  and the bottom mapping factors through $\widetilde{\U}(p+q)$.
\medskip

By Lemma \ref{restrictiontoblocks} we have
\begin{lem} \label{resU(p,q)}
$$\mathrm{det}_{\U(p,q)}^{1/2}|_{\widetilde{\U}(p) \times \widetilde{\U}(q)} = \mathrm{det}_{\U(p)}^{1/2}\otimes \mathrm{det}_{\U(q)}^{1/2}.$$
\end{lem}

We  now  use the above diagram,  Lemma \ref{pullbackofdet} and Lemma \ref{conjugationchangessign} to prove
\begin{lem} \label{restrictionofvacuumcharacteraequalsone}
$$\big(\omega\circ \widetilde{\mu}_{p,q}\big)(\psi_0) = (\mathrm{det}_{\U(p)}^{- 1/2}\otimes \mathrm{det}_{\U(q)}^{ 1/2}) \psi_0.$$
\end{lem} 

\begin{proof}  
Let $k= (k_1, k_2) \in K$.  In what follows we will abbreviate $\widetilde{\tau_q}(k_2)$ to $\overline{k_2}$.  

Recall that  $\psi_0$ is the vacuum vector (so $1$ in the Fock model).  Going around the top of the diagram  and noting that the composition of the top right horizontal arrow with the right vertical arrow is $\widetilde{\j}_{\U(p,q)}(k_1,k_2)$,  we obtain by definition 
$ \omega\big(\widetilde{\j}_{\U(p,q)}(k_1,k_2)\big) \psi_0 = \omega(k_1,k_2) \psi_0$.
Going around the diagram the other way we obtain 
\begin{equation*}
\begin{split}
\omega ((k_1,k_2) )\psi_0 & = \big( \omega|_{\widetilde{\U}(p+q)}  (  (1 \otimes \widetilde{\tau}_q)(k_1,k_2)) \big)\psi_0 \\ 
& =  \big(\omega|_{\widetilde{\U}(p+q)}    (k_1,\overline{k}_2) \big)\psi_0 \\
& = \mathrm{det}_{\U(p+q)}^{ - 1/2} (k_1,\overline{k_2}) \psi_0 \\
& = \mathrm{det}_{\U(p)}^{ - 1/2}( k_1)\cdot  \mathrm{det}_{\U(q)}^{ - 1/2}(\overline{k_2})  \psi_0 \\
& = \mathrm{det}_{\U(p)}^{ - 1/2}( k_1)\cdot  \mathrm{det}_{\U(q)}^{ + 1/2}(k_2) \psi_0  .
\end{split}
\end{equation*}
\end{proof} 

Here the next-to-last equality is Lemma \ref{restrictiontoblocks} and the last equality is Lemma \ref{conjugationchangessign}. 
We now have 

\begin{cor}
\begin{equation}\label{correcttwistforaequalsone}
(\omega \otimes \mathrm{det}_{\U(p,q)}^{1/2})|_K (\psi_0) = \mathrm{det}_{\U(q)}(k_2) \psi_0 .
\end{equation}
\end{cor}

\medskip
\noindent
{\it Remark.} \label{aequalsonebequalszero}
 Lemma \ref{restrictionofvacuumcharacteraequalsone} and its corollary  give the formula for the twist of  the standard action of $\U(p) \times \U(q)$ in the polynomial Fock model we will need in \S  
\ref{TheWeilrepresentation} for the case $a=1$, $b=0$ (so $W= W_+$ is a complex line equipped with a positive unary Hermitian form). Note that
in this case $a - b =1$. 
\medskip

We conclude this subsection with the formula for the action of $\widetilde{\U} ( p)\times \widetilde{\U}(q)$ on $\psi_0$  we will need in \S  \ref{TheWeilrepresentation} for the case $a=0,b=1$ (so $W = W_-$ is a complex line equipped with a negative unary Hermitian form). 
In this case tensoring with $W$ has the effect of changing the Hermitian form $( \  ,\   )$ on $V$ to its negative $(\ ,  \ )' = - (\  , \   )$. 
The matrix $M''$ of the symplectic form associated to $( ,  )'$  relative to the standard basis $\cal{B} = \{ v_1,\cdots,v_m, iv_1,\cdots,iv_m \}$ is the negative of the matrix $M'$ above. We identify the isometry groups of the two Hermitian forms  $(\ ,\ )$ and $(\ ,\ )$ with $\U(p,q)$ using the standard basis.   We obtain a new embedding of $\U(p,q)$ into $\Sp(2m,\R)$ which we will denote $j_{\U(q,p)}$ given by
$$j_{\U(q,p)} = \Ad W_{p,q}\circ j_{\mathcal{B}}$$
where
\begin{equation} \label{changeofbasismatrix}
W_{p,q} = \begin{pmatrix}- I_{p,q} & 0 \\ 0 & I_m \end{pmatrix}.
\end{equation}
We let $\omega'$ be the pull-back of the Weil representation to $\U(p,q)$ using the embedding $j_{\U(q,p)}$. 

Recall that $\tau_0:\U(p,q) \to \U(p,q)$ is complex conjugation and $\widetilde{\tau}_0$ is its lift to $\widetilde{\U}(p,q)$.  We have
\begin{lem} \label{relatingtheembeddings}
$$\omega' = \omega \circ \widetilde{\tau}_0.$$
\end{lem}
\begin{proof}
Note first that $Z_{p,q} \circ W_{p,q} = I_{m,m}$ where 
\begin{equation} \label{changeofbasismatrix}
I_{m,m} = \begin{pmatrix}- I_m & 0 \\ 0 & I_m \end{pmatrix}.
\end{equation}
Hence the two embeddings are related by
$$j_{\U(q,p))} = \Ad I_{m,m} \circ j_{\U(p,q))}.$$
But the embedding $j_{\cal{B}}: \GL(m,\C) \to \GL(2m,\R)$ satisfies
\begin{equation} \label{complexconjugation}
j_{\cal{B}} \circ \tau_0 = \Ad I_{m,m} \circ j_{\cal{B}}
\end{equation}
and hence the two embeddings are related by 
$$j_{\U(q,p))} = j_{\U(p,q))}\circ \tau_0.$$
The lemma follows by lifting the previous identity to the two-fold covers.
\end{proof}

It follows from Lemma \ref{conjugationchangessign} that  the action of $\omega'|K$ on the vaccum vector is by the conjugate of the character for the action of $\omega|K$ and hence from Lemma \ref{restrictionofvacuumcharacteraequalsone} we obtain

\begin{lem} \label{restrictionofvacuumcharacteraequalszero}
$$\omega' |_K (\psi_0) = (\mathrm{det}_{\U( p)}^{1/2}\otimes \mathrm{det}_{\U(q)}^{- 1/2}) \psi_0.$$
\end{lem} 

\begin{cor}
\begin{equation}\label{correcttwistforaequalszero}
(\omega' \otimes \mathrm{det}_{\U(p,q)}^{-1/2})|_K (\psi_0) = \mathrm{det}_{\U(q)}^{-1}(k_2) \psi_0.
\end{equation}
\end{cor}

\medskip
\noindent
{\it Remark.} \label{aequalszerobequalsone} Lemma \ref{restrictionofvacuumcharacteraequalszero} and its corollary  give the formula for the twist of  the standard action of $\U(p) \times \U(q)$ in the polynomial Fock model we will need in \S  
\ref{TheWeilrepresentation} for the case $a=0$, $b=1$. Note that
in this case $a - b =-1$. 
\medskip

\subsection{The tensor product of Hermitian vector spaces} \label{tensorproduct}
Let $V$ be a complex vector space of dimension $m = p+q$ equipped with a Hermitian form $( , )_V$ of signature $(p,q)$ and $W$ be a complex vector
space of dimension $a+b$ equipped with a Hermitian form $( , )_W$ of signature $(a,b)$.  We will regard $V$ as a real vector space equipped with the almost complex structure $J_V$ and $W$ as a real vector space equipped with the almost complex structure $J_W$.  We may regard the tensor product $V \otimes_{\C} W$ as the quotient of the tensor product $V \otimes_{\R} W$ by the relations
$$(J_V (v)) \otimes w = v \otimes (J_W (w))$$
for all pairs of vectors $v \in V$, $w \in W$. Thus we have an almost complex structure $J_{V \otimes W}$ on $V \otimes_{\C} W$ given  by
\begin{equation}\label{almostcomplexontensorproduct}
J_{V \otimes W} = J_V \otimes I_W = I_V  \otimes J_W.
\end{equation}
From now on all tensor products will be over $\C$ unless the contrary is indicated.  We will acccordingly abbreviate $V \otimes_{\C}W$ to
$V \otimes W$. We have
\begin{lem} \label{primetimesdoubleprime}
$$V' \otimes W'' = 0 \ \text{and} \ V'' \otimes W'=0.$$
Accordingly we have
$$ (V \otimes W ) \otimes_{\R} \C = (V' \otimes W') + (V '' \otimes W'').$$
\end{lem}
\begin{proof}
Suppose $v' \in V'$ and $w'' \in W''$.  Then
$$J_{V \otimes W}(v' \otimes w'') = J_V(v') \otimes w'' = i (v' \otimes w'').$$
But also 
$$J_{V \otimes W}(v' \otimes w'') = v' \otimes J_W(w'') = -i (v' \otimes w'').$$
\end{proof}

The tensor product  $((  , )) := (  , )_V \otimes (  , )_W$ is a Hermitian form on $V \otimes W$.  We let $\langle \langle  ,  \rangle \rangle $ or $A_{V\otimes W}$ denote  the symplectic form $A_{V\otimes W}$ on the real vector space underlying $V \otimes W$ (which we will again denote $V \otimes W$); it is given by the negative of the imaginary part of $((  , ))$.  Hence we  have
\begin{equation}\label{symplecticformontensorproduct}
A_{V\otimes W} = A_V \otimes B_W + B_V \otimes A_W.
\end{equation}
Clearly we have an embedding $\U(V) \times \U(W) \rightarrow \mathrm{Aut} (V \otimes W,
 \langle \langle,   \rangle \rangle)$.  It is standard that
this product is a dual reductive pair, that is each factor is the full centralizer of the other in the symplectic group $\mathrm{Aut} (\langle \langle,  \rangle \rangle)$.

We can now  use the direct sum decompositions  $V = V_+ \oplus V_-$ and $W= W_+ \oplus W_-$ and the considerations of \S \ref{pacs}  to change the indefinite almost complex structure $J_{V \otimes W}$ to an admissible positive almost complex structure $(J_{V \otimes W})_0$ on  $V \otimes W$. Indeed we can split $V \otimes W$ into a sum of the positive definite space $V_+ \otimes W_+ + V_- \otimes W_-$ and the negative definite space $V_+ \otimes W_- + V_- \otimes W_+$.  The corresponding ``Cartan involution'' is $\theta_V \otimes \theta_W$.
The Cartan involution $\theta_V \otimes \theta_W$ allows us to define a positive definite Hermitian form $((\ ,\ ))_0$ (the corresponding majorant
of $((\ ,\ ))$) by
$$(( v_1 \otimes w_1, v_2 \otimes w_2))_0  = ((  v_1 \otimes w_1,  \theta_V(v_2) \otimes \theta_W(w_2) )) = 
(v_1, \theta_V (v_2))_V  \ (w_1,\theta_W (w))_W .$$

We define the positive definite complex structure $(J_{V \otimes W})_0$ corresponding to the previous splitting by
\begin{multline}\label{positivetensorcomplexstructure}
(J_{V \otimes W})_0 = (J_{V \otimes W})\circ (\theta_V \otimes \theta_W) \\ = (J_V \otimes I_W)\circ (\theta_V \otimes \theta_W) = (I_V \otimes J_W) \circ (\theta_V \otimes \theta_W).
\end{multline}
We emphasize that  the definite complex structure $(J_{V \otimes W})_0$  depends on the choice of splittings of $V$ and $W$.

We can now  compute the spaces
$(V \otimes W)^{\prime_0}$ of type $(1,0)$ vectors and $(V \otimes W)^{\prime \prime_0}$ of type $(0,1)$ vectors  for $(J_{V \otimes W})_0$ acting on $(V \otimes W) \otimes_{\R} \C$. 

Note that we have
$$ \dim_{\C}( (V \otimes W)^{\prime_0}) = \dim_{\C}(V \otimes W) = (a+b)(p+q).$$

\begin{lem} \label{typeonezerooftensor}
We have $(\mathrm{U}(p) \times \mathrm{U}(q)) \times (\mathrm{U}(a) \times \mathrm{U}(b))$-equivariant isomorphisms  of complex vector spaces 
\begin{align*} 
(V \otimes W)^{\prime_0} \cong & (V'_+ \otimes W'_+ ) \oplus (V''_+ \otimes W''_-) \oplus (V''_- \otimes W''_+) \oplus (V'_- \otimes W'_-)\\
(V \otimes W)^{\prime \prime_0} \cong & (V''_+ \otimes W''_+)  \oplus (V'_+ \otimes W'_-) \oplus (V'_- \otimes W'_+) \oplus (V''_- \otimes W''_-).
\end{align*}
Under the pairing of $(V \otimes W)^{\prime_0}$ with $(V \otimes W)^{\prime \prime_0}$ induced by the symplectic form each of the four subspaces on the right is dually paired with the
space immediately below it.
\end{lem}
\begin{proof} 
The space  $(V \otimes W)\otimes_{\R} \C$ is the quotient of the space  $ (V \otimes_{\R}\C) \otimes_{\C} (W \otimes_{\R}\C)$ by the relation
that makes the action of  $J_V \otimes 1 \otimes 1 \otimes 1$ equal to that of $1 \otimes 1 \otimes J_W \otimes 1$.  The operation of passing to the quotient corresponds to setting all tensor  products of spaces with superscript  prime factors with spaces with superscript double prime factors equal to zero according to Lemma \ref{primetimesdoubleprime}.  Before passing to the quotient we have a direct sum decomposition with $4 \times 4 = 16$  summands, after passing to the quotient we have a direct sum decomposition with eight summands.  With the above identification we have
\begin{align*}
(V \otimes W) \otimes_{\R} \C = &\underbrace{([(V'_+\otimes W'_+) \oplus (V'_+ \otimes W'_-) \oplus V'_- \otimes W'_+) \otimes (V'_-\otimes W'_-)]}_{+i}\\
& \oplus [\underbrace{[(V''_+\otimes W''_+) \oplus (V''_+ \otimes W''_-) \oplus V''_- \otimes W''_+) \otimes (V''_-\otimes W''_-)]}_{-i}.
\end{align*}
Here the subscript $\pm i$ indicates the eigenvalue of $J_{V\otimes W}$ on the summand.  Thus the first four summands comprised the subspace of type $(1,0)$ vectors for $J_{V\otimes W}$ and the last four summands comprise the subspace of type $(0,1)$ vectors for $J_{V\otimes W}$.
Now the involution $\theta_V \otimes \theta_W$ is also diagonal relative to the above eight summand decomposition with corresponding eigenvalues $(+1,-1,-1,+1,+1,-1,-1,+1)$.
Since $(J_{V\otimes W})_0 = J_{V\otimes W}\circ (\theta_V \otimes \theta_W)$ the second and third summands above move  into the space of type $(0,1)$
vectors for $(J_{V\otimes W})_0 $ and the sixth and seventh summands move  into the space of type $(1,0)$ vectors for $(J_{V\otimes W})_0$.
\end{proof}

\subsection{The  polynomial Fock model for a unitary dual pair}\label{polynomialFockmodel} \label{unitarydualpair}
Our goal in this subsection is to  describe  the polynomial Fock space $\mathcal{P}(V \otimes W)$.  Our main interest
will be the subspace $\mathcal{P}(V_+ \otimes W) \subset \mathcal{P}(V \otimes W)$ and its description as the algebra of polynomials on the 
space of $p$ by $a+b$ complex matrices $\mathrm{M}_{p \times (a+b)} (\C) = \mathrm{M}_{p\times a} (\C) \oplus \mathrm{M}_{p \times b} (\C)$. 

By Lemma \ref{typeonezerooftensor} we have 
\begin{equation} \label{polyFock}
 \mathcal{P}(V\otimes W)= \Pol((V'_+ \otimes W'_+) \oplus (V''_+ \otimes W''_-)) \otimes \Pol((V'_- \otimes W'_-) \oplus (V''_- \otimes W''_+)). 
\end{equation}
We will abbreviate the first factor in the tensor product on the right to $\mathcal{P}_+$ and the second factor to $\mathcal{P}_-$. 
We now choose an orthonormal basis $\{w_1,w_2,\cdots,w_a\}$ for $W_+$ and a basis $\{w_{a+1},w_{a+2}, \cdots ,w_{a+b}\}$ for
$W_-$ which is orthonormal for the restriction of $-( , )_W$ to $W_-$.

In this paper we will be primarily  concerned with the space $\mathcal{P}_+$.  Accordingly we will  suppose that $u   \in (V_+ \otimes W)^{\prime_0}$.
Then there exist unique $x'_1,\cdots,x'_a \in V'_+$ and $y''_1,\cdots,y''_b \in V''_+$ such that
\begin{equation} \label{tensorproductdecomposition}
u = \sum_{i =1}^a x'_i \otimes w'_i + \sum_{j= 1}^{b} y''_j \otimes w''_{a+j}.
\end{equation}
 
We may accordingly  represent the  element  
$u$
of $(V_+ \otimes W)'$
by the element 
$$ (x'_1,\cdots,x'_a;y''_1,\cdots,y''_b) = (\mathbf{x}'; \mathbf{y}'') \in
(V'_+)^a \oplus (V''_+)^b.$$  
Then by using the basis $v'_1,\cdots,v'_p, v''_1,\cdots,v''_p$ we may finally represent an element
of $(V_+ \otimes W)^{\prime_0}$ as a $p$ by $a+b$ matrix with complex entries.  Thus we have 
$$(V_+ \otimes W)^{\prime_0} \cong \mathrm{M}_{p\times a} (\C)\oplus \mathrm{M}_{p\times b}(\C).$$
We  will think of a point on the right as a $p$ by $a+b$ matrix $Z(\mathbf{x}'; \mathbf{y}'')$ divided into a left $p$ by $a$ block 
$Z'(\mathbf{x}') = (z'_{\alpha, i}(\mathbf{x}'))$ and a right $p$ by $b$ block $Z''(\mathbf{y}'') =(z''_{\alpha,j}(\mathbf{y}''))$.  
We will use these matrix
coordinates henceforth (at times we will drop the arguments $\mathbf{x}'$ and $\mathbf{y}''$). By Lemma \ref{coordinateformula} we have
$$ z'_{\alpha, i}(\mathbf{x}')  = (x_i, v_{\alpha}), 1 \leq i \leq a, 1 \leq \alpha \leq p \   \text{and}
 \  z''_{\alpha, j}(\mathbf{y}'') = ( v_{\alpha}, y_j), 1 \leq j \leq b, 1 \leq \alpha \leq p.$$
Here we  have used Greek letter(s) $\alpha$ for the indices belonging to $V$ and Roman letters $i,j$ for the indices belonging to $W$.

%The following figure illustrates the case, $p=3, a=2, b=2$. 

%$$
%\left(
%\begin{array}{cc|cc}
%z'_{11} & z'_{12} & z''_{11}  & z''_{12} \\
%z'_{21} & z'_{22} & z''_{21} & z''_{22} \\
%z'_{31} & z'_{32} & z''_{31} & z''_{32}
%\end{array}
%\right)
%$$

The polynomial Fock model is then the space of polynomials in  $z'_{\alpha,i}$ and $z''_{\alpha,j} $ as above.   From now on, we will usually  work with this matrix description of the polynomial Fock space hence we will identify
$$\mathcal{P}_+= \Pol(((V'_+ \otimes W'_+) \oplus (V''_+ \otimes W''_-)) \cong \Pol( \mathrm{M}_{p\times a} (\C)\oplus \mathrm{M}_{p\times b}(\C) ).$$

\subsection{The twisted action of $(\U(p) \times \U(q)) \times (\U(a) \times \U(b))$} \label{TheWeilrepresentation} Recall that we have a unitary representation of the metaplectic group $\omega:\Mp(V\otimes W, \langle \langle , \rangle \rangle) \to \U(\mathcal{F}(V\otimes W))$. As above we denote by $\widetilde{\U}_0$ the maximal compact subgroup 
$\widetilde{\U} (V \otimes W, ((,))_0)$ of $\Mp(V\otimes W, \langle \langle , \rangle \rangle)$. The space of $\widetilde{\U}_0$-finite vectors of the Weil representation $\omega$ is precisely the polynomial Fock space $\mathcal{P} (V\otimes W) = \mathcal{P}_+ 
\otimes \mathcal{P}_-$, we refer to \cite{HoweT,Paul} for more details. 
We review how certain subgroups (subalgebras) of $\widetilde{\U}(V) \times \widetilde{\U}(W) $ act in this model. We have natural inclusion maps
$$\U (V) \times \U (W) \to \U (V \otimes W) \mbox{ and } \U(V \otimes W ) \to \Sp (V \otimes W , \langle \langle , \rangle \rangle).$$
We have previously described two-fold covers $\widetilde{\U}(V)$, resp. $\widetilde{\U}(W)$ and $\widetilde{\U} (V \otimes W)$, of $\U(V)$, $\U (W)$ and $\U (V \otimes W)$ with characters $\det_{\U(V)}^{1/2}$, resp. $\det_{\U(W)}^{1/2}$ and $\det_{\U(V \otimes W)}^{1/2}$. 
Lemma \ref{restrictiontoblocks} and Lemma \ref{diagonalactiononasum} then imply that we have
\begin{equation} \label{resdet}
\mathrm{det}_{\U(V \otimes W)}^{1/2} |_{\widetilde{\U}(V)} = \mathrm{det}_{\U( V)}^{(a+b)/2} \mbox{ and }  \mathrm{det}_{\U(V \otimes W)}^{1/2} |_{\widetilde{\U} (W)} = \mathrm{det}_{\U( W)}^{(p+q)/2} .
\end{equation}
If  $(k,\ell) \in \Z^2$ the restriction of the Weil representation $\omega$ to $\widetilde{\U} (V) \times \widetilde{\U} (W)$ 
twisted by the characters $\det_{\U (V)}^{k/2} \otimes \det_{\U (W)}^{\ell/2}$ will be denoted $\omega_{k,\ell}$.  Since the Weil representation of
$\widetilde{\U}(V\otimes W)$ twisted by ${\det}^{1/2}$ descends to $\U(V\otimes W)$, it follows from Equation \eqref{resdet} that $\omega_{k,\ell}$ descends to $\U(V) \times \U(W)$ if and only if $k \equiv a+b \ ({\rm mod} \ 2)$ and $\ell \equiv p+q \ ({\rm mod} \ 2)$.

We note  that it follows from \eqref{Weiltwist} that the compact subgroup
$\widetilde{\U} ( p) \times \widetilde{\U} (q) \times \widetilde{\U} (a ) \times \widetilde{\U} (b) $ acts on $\mathcal{P}$ by the usual action up to a central character. The explicit computation of this central
character is given by the following proposition.

\begin{prop} \label{P:weiltwisted}
The group $\widetilde{\U} ( p) \times \widetilde{\U} (q) \times \widetilde{\U} (a ) \times \widetilde{\U} (b) $ acts on the line $\C \psi_0$ (so the constant polynomials in the Fock model) under  the  twisted Weil representation $\omega_{k, \ell}$  by the character 
$\mathrm{det}_{\U( p)}^{\frac{k+b-a}{2}} \otimes \mathrm{det}_{\U(q)}^{\frac{k + a-b}{2}} \otimes \mathrm{det}_{\U (a)}^{\frac{\ell+q-p}{2}} \otimes \mathrm{det}_{\U (b)}^{\frac{\ell +p-q}{2}}$. 
\end{prop}
 The Proposition will follow from Lemma \ref{restrictiontoblocks} and the next Lemma (which will be seen to follow from Lemmas \ref{restrictionofvacuumcharacteraequalsone}, \ref{diagonalactiononasum} and \ref{restrictionofvacuumcharacteraequalszero}).

First by applying Lemma \ref{restrictiontoblocks} to the ``block inclusions'' $\U (W_+ ) \times \U (W_-) \subset \U (W)$ and $\U (V_+ ) \times \U (V_- ) \subset \U (V)$
we first get
\begin{equation} \label{res1}
\begin{split}
& \mathrm{det}_{\U (V)}^{k/2} |_{\widetilde{\U}(V_+) \times \widetilde{\U} (V_-)} = \mathrm{det}_{\U (V_+ )}^{k/2} \otimes \mathrm{det}_{\U (V_- )}^{k/2} \mbox{ and } \\
& \mathrm{det}_{\U (W)}^{\ell/2} |_{\widetilde{\U}(W_+) \times \widetilde{\U} (W_-)} = \mathrm{det}_{\U (W_+ )}^{\ell/2} \otimes \mathrm{det}_{\U (W_- )}^{\ell/2}  .
\end{split}
\end{equation}

Now the Proposition follows from the next lemma.

\begin{lem} \label{res2}
The group $\widetilde{\U} ( p) \times \widetilde{\U} (q) \times \widetilde{\U} (a ) \times \widetilde{\U} (b) $  acts on the line $\C \psi_0$   under the (untwisted) Weil representation $\omega$  by the character 
$\mathrm{det}_{\U( p)}^{\frac{b-a}{2}} \otimes \mathrm{det}_{\U(q)}^{\frac{ a-b}{2}} \otimes \mathrm{det}_{\U (a)}^{\frac{q-p}{2}} \otimes \mathrm{det}_{\U (b)}^{\frac{p-q}{2}}$.
\end{lem}

\begin{proof}
By symmetry between $V$ and $W$ it is enough to compute the action of  $\widetilde{\U} (p) \times \widetilde{\U} (q )$ under the untwisted Weil representation  on $\C \psi_0 $. Considering the 
tensor product of $V$ with a Hermitian space of signature $(a,b)$ amounts to looking at the diagonal
action of $\U(V)$ on the direct sum of $a$ copies of $V$ and $b$ copies of $V$ with the sign of the Hermitian form changed. Hence by Lemma \ref{diagonalactiononasum} we are reduced to the special cases $a=1$, $b=0$ and $a=0$, $b=1$. The first case is 
Lemma \ref{restrictionofvacuumcharacteraequalsone} and the second case is Lemma \ref{restrictionofvacuumcharacteraequalszero}.
\end{proof}

\subsection{} From now on we will always assume that $k=a-b$ and $\ell =p+q$ and will now use the symbol   $\omega$ to denote the (twisted) representation
$\omega_{a-b,p+q}$. The choice of $k$ will turn out to be very important: 
 indeed it follows from Proposition \ref{P:weiltwisted} that the restriction of  $\omega$ to the  group $\U(p) \times \U(q)$ then acts on the line $\C \psi_0$ (so the constant polynomials in the Fock model)  by the character $ 1\otimes \det_{\U(q)}^{a-b}$. As a consequence the group $\U(p) \times \U(q)$ acts on 
$\mathcal{P}_+$ by the  tensor product of the standard action of $\U(p)$ and the character $\det_{\U(q)}^{a-b}$. 
We will see later that the above twist is the correct one to ensure  our cocycle $\psi_{bq,aq}$ is $(\U(p) \times \U(q))$-equivariant, see Lemma \ref{correcttwist}.

To summarize, if we represent the action of the Weil representation $\omega$ in terms of the $p$ by $a+b$ matrix representation of the Fock model, see Subsection \ref{unitarydualpair},  
$$\mathcal{P}_+ = \Pol( \mathrm{M}_{p\times a} (\C)\oplus \mathrm{M}_{p\times b}(\C) )$$ 
we have:

\begin{thm} \label{bottomlineforthetwist} 
\hfill
\begin{enumerate}
\item  The action of the group $\U(a) \times \U(b)$  induced by the twisted Weil representation $\omega_{a-b,p+q}$  
on polynomials in the matrix variables is the tensor product of the character  $\det_{\U (a)}^{q} \otimes \det_{\U (b)}^p$ with  the action induced by the natural action on the rows (i.e. from the right) of the matrices. Note that each row has $a+b$ entries. The group   $\U(a)$ acts on on the first $a$ entries of each row  and $\U(b)$ acts on the last $b$ entries of each row.   
\item The action of the group $\U(p)$ is induced by the natural action on the columns (i.e. from the left) of the matrices, acting on the left half of the matrix by the standard action  and on the right half by the dual of the standard action so {\it there is no determinant twist}. 
\item The group $\U (q)$ simply scales all polynomials  by the central character $\det_{\U(q)}^{a-b}$.
\end{enumerate} 
\end{thm} 

The representation $\omega$  yields a correspondence between certain equivalence classes of irreducible admissible representations of $\U(a,b)$ and $\U(p,q)$. The correspondence between $K$-types is explicitly described in \cite{Paul} using the Fock model (and following Howe \cite{HoweT}), we will also refer to \cite{KashiwaraVergne}

\section{The  special $(\mathfrak{u}(p,q),K)$-cocycles $\psi_{bq,aq}$} \label{sec:KMlocal}
In this section we introduce special cocycles 
$$\psi_{bq,aq} \in \mathrm{Hom}_K \left( \wedge^{bq,aq} \mathfrak{p} , \mathcal{P} (V \otimes W) \right)$$ 
with values in the polynomial Fock space. 

We will first define the cocycles 
$\psi_{bq,0}$ of Hodge bidegree $(bq,0)$ (and similarly $\psi_{0,aq}$
of Hodge bidegree $(0,aq)$). We will  give formulas for  
their  dual maps $\psi^*_{bq,0}$,  resp. $\psi^*_{0,aq}$, as the values of these dual maps at  $x \in(V \otimes W)^{\prime_0}$ are decomposable as  exterior products of $bq$, resp. $aq$, elements of $\mathfrak{p}^*$ depending on $x$.

\subsection{Harmonic and special harmonic polynomials}
In this subsection we review the ``lowering operators''  coming from the action of the  space $\mathfrak{p}''_{\U(a,b)}$.
We leave to the reader the task of writing out the formulas for $\mathfrak{u}(W)$ analogous to those of Subsection \ref{complexifiedtangentspace}
for $\mathfrak{u}(V)$, in particular of proving $\mathfrak{p}''_{\U(W)} \cong W_- \otimes W^*_+$.  Hence, in the notation of Subsection \ref{complexifiedtangentspace} we have a basis $\{w_i \otimes w^*_{a+j}: 1 \leq i \leq a, 1 \leq j \leq b \}$ for $\mathfrak{p}''_{\U(a,b)}$. 
We define 
$$\Delta_{i,j} = \Delta^{+}_{i,j} = \sum_{\alpha =1}^p \frac{\partial^2} {\partial z'_{\alpha,i} \partial z''_{\alpha, j}}  \ \text{for} \  1 \leq i \leq a, 1 \leq j \leq b.$$
Thus $\Delta_{i,j}$ is a second order differential operator.    Then we have (up to a scalar multiple)
\begin{prop}\label{loweringoperators}
$$\omega(w_i \otimes w^*_{a+j} ) = \Delta^{+}_{i, j}.$$
\end{prop}
Here $\omega$ is the (infinitesimal) oscillator representation for the dual pair $\mathfrak{u}(p) \times \mathfrak{u}(a,b)$.
The  proposition is a straightforward computation and is implicit in \cite{KashiwaraVergne}, Equation 5.1. 

\medskip
\noindent  
{\it Remark.}
We have analogous Laplace operators $\Delta^{-}_{i,j}$ on $\Pol((V'_- \otimes W'_- )\oplus (V''_- \otimes W''_+))$.

\medskip

We define the subspace of {\it harmonic polynomials}
$$\Harm(( V'_+ \otimes W'_-) \oplus (V''_+ \otimes W''_-) ) \subset \Pol( (V_+ \otimes W)^{\prime_0}) =  \Pol( ( V'_+ \otimes W'_+) \oplus (V''_+ \otimes W''_-))$$
to be the subspace of polynomials annihilated by the Laplace operators $\Delta_{i,j}, 1 \leq i,j \leq n $.  

We will abbreviate $\Harm(( V'_+ \otimes W'_-) \oplus (V''_+ \otimes W''_-) )$ to $\mathcal{H}_+$ henceforth.  The subspace 
 $\mathcal{H}_- = \Harm(( V'_- \otimes W'_-) \oplus (V''_- \otimes W''_+ ) ) \subset \Pol(( V'_- \otimes W'_-) \oplus (V''_- \otimes W''_+ ) )$ is defined analogously to $\mathcal{H}_+$ as the simultaneous kernels of the operators $\Delta^{-}_{i,j}$. 
We emphasize that $\mathcal{H}_+$  and $\mathcal{H}_-$ are not closed under multiplication.

Note however that the subalgebra $\Pol( V'_+ \otimes W'_+)$ of $\mathcal{P}(V'_+ \otimes W')$  is contained in the subspace of  harmonic polynomials,
$$ \Pol( V'_+ \otimes W'_+) \subset \Harm(( V'_+ \otimes W'_-) \oplus (V''_+ \otimes W''_-) ).$$

We will call an element of $\Pol( V'_+ \otimes W'_+)$ a {\it special} harmonic polynomial.

Following Kashiwara-Vergne \cite{KashiwaraVergne} we  define elements $\Delta_k \in \mathcal{P}(V_+ \otimes W)$ and $\widetilde{\Delta}_{\ell} \in \mathcal{P}(V_+ \otimes W)$ for $1 \leq k \leq p, 1 \leq \ell \leq p$ and $a,b \leq p$,
 by
\begin{align*}
\Delta_k(\mathbf{x}',\mathbf{y}'') = & \Delta_k(\mathbf{y}'') = \det ( z''_{\alpha,j}))=\det ( (v_\alpha, y_j)), \ 1 \leq \alpha \leq k , \ 1 \leq j \leq k  \\
\widetilde{\Delta}_{\ell}(\mathbf{x}',\mathbf{y}'') = & \widetilde{\Delta}_{\ell}(\mathbf{x}') = \det ( z'_{\alpha,j})) = \det ( (x_j,
 v_{\alpha})), \ p-\ell+1 \leq \alpha \leq p , \ 1 \leq j \leq \ell .
\end{align*}
We note that $\Delta_k$ and  $\widetilde{\Delta}_{\ell}$ are special harmonic, hence any power of $\Delta_k$ and any power of  $\widetilde{\Delta}_{\ell}$ is also special harmonic.  Also the reader will verify
\begin{lem}
Suppose $k+\ell \leq p$.  Then for any natural numbers $\ell_1$ and $\ell_2$ the product 
$\Delta_k^{\ell_1} \cdot  (\widetilde{\Delta}_{\ell})^{\ell_2}$ is harmonic.
\end{lem}

\subsection{Some special cocyles}\label{Specialcocycles}
We now give formulas for cocycles that will turn out to be generalizations of the special cocyles constructed by Kudla-Millson. 

The domain of the relative Lie algebra cochains is  the exterior algebra $\wedge^*\mathfrak{p}$ which factors according to 
\begin{equation} \label{exteriorfactors}
\wedge^* \mathfrak{p} = (\wedge^* \mathfrak{p}') \otimes (\wedge^* \mathfrak{p}'') = \wedge^* (V_+ \otimes V^*_-) \otimes \wedge^* (V_- \otimes V^*_+).
\end{equation}
We will consider only  very special cochains whose range is the positive definite Fock model $\mathrm{Pol}((V_+ \otimes W)^{\prime_0})= \mathcal{P}_+$. Recall that the space $\mathrm{Pol}((V_+ \otimes W)^{\prime_0})$ factors according to
\begin{equation}\label{Fockfactors}
 \mathrm{Pol}((V_+ \otimes W)^{\prime_0}) =  \mathrm{Pol}(V'_+ \otimes W'_+) \otimes \mathrm{Pol}(V''_+ \otimes W''_-) .
\end{equation}
The key point is that each of the two factorizations has a strong  ``disjointness property'' namely  in  its decomposition into irreducible $\U(p)$-representations, the only  irreducible common to both factors is the trivial representation.  

\subsubsection{A restriction on Hodge types}
Note that 
$$\mathrm{Pol}(V'_+ \otimes W'_+) \cong \mathrm{Sym}(V''_+ \otimes W''_+) \ \text{and} \ \mathrm{Pol}(V''_+ \otimes W''_-) \cong \mathrm{Sym}(V'_+ \otimes W'_-).$$
Then, since  any cochain  $\psi_{k,\ell}$ is $\U(p)$-equivariant we obtain
\begin{lem} \label{forcingvalues}
Suppose $\psi_{k,0}$ is a cochain of bidegree $(k,0)$ taking values in $\mathcal{P}_+$.  Then it must take values in the second factor
$\mathrm{Pol}(V''_+ \otimes W''_-) \cong \mathrm{Sym}(V'_+ \otimes W'_-)\cong \mathrm{Sym}((V'_+)^b)$ of the previous tensor product.

Equivalently,  suppose $\psi_{0,\ell}$ is a cochain of bidegree $(0,\ell) $ taking values in $\mathcal{P}_+$.  Then it must take values in the first  factor of \eqref{Fockfactors} namely, 
$\mathrm{Pol}(V'_+ \otimes W'_+) \cong \mathrm{Sym}(V''_+ \otimes W''_+)\cong \mathrm{Sym}((V''_+)^a)$.
\end{lem} 

\medskip
\noindent
{\it Remark.} \label{reversedhodgedegree}
If we insist on the standard convention $\dim{W}_+ = a$ and $\dim{W}_- =b$ (as we are going to do) then the Hodge degrees of the special cocycles we will construct will be of the form $(bq,aq)$. Thus in our previous notation the cocycle $\psi_{bq,0}$ gives rise to a polynomial function of  $y''_1,\cdots, y''_b$ (the right half of the matrix) and  the cocycle $\psi_{0,aq}$ give rise to a polynomial function of $x'_1,\cdots, x'_a$ (the left half of the matrix). 

\medskip

\subsection{} Our immediate goal now is to give the definitions of and the properties of  the cocycles 
$\psi_{bq,0}$ of Hodge bidegree $(bq,0)$ and $\psi_{0,aq}$
of Hodge bidegree $(0,aq)$.  As pointed out above these cocycles have the special property  
\begin{enumerate} 
\item $\psi_{bq,0}: \wedge^{bq}(V_+ \otimes V^*_-)= \wedge^{bq}(\mathfrak{p}') \to \mathrm{Sym}^{(bq)}((V''_+ \otimes W''_-)^*)$
\item $\psi_{0,aq}: \wedge^{aq}( V_- \otimes   V^*_+)= \wedge^{bq}(\mathfrak{p}'') \to \mathrm{Sym}^{(aq)}( (V'_+ \otimes W'_+)^*)$. 
\end{enumerate}

As stated above if we first evaluate the cocycles $\psi_{bq,0}$ resp.  $\psi_{0,aq}$ at points in $V''_+ \otimes W''_ -$ resp. $V'_+ \otimes W'_+$ the resulting elements of $\bigwedge^{\bullet}(\mathfrak{p}^{\ast})$ are completely decomposable.   To formalize this decomposability property (which will be very useful for computations) and also to understand how the cocycles transform under $\widetilde{\U}(p) \times  \widetilde{\U}(q) \times \widetilde{\U}(a) \times \widetilde{\U}(b)$ it is better to give formulas for the 
the  dual maps $\psi^*_{bq,0}$ and $\psi^*_{0,aq}$.  These maps will satisfy 
\begin{enumerate}
\item $\psi^*_{bq,0}: \mathrm{Sym}^{bq}(V''_+ \otimes W''_-)  \to \wedge^{bq}(V^*_+ \otimes V_-)= \wedge^{bq}((\mathfrak{p}')^*) $
\item $\psi^*_{0,aq}: \mathrm{Sym}^{aq}(V'_+ \otimes W'_+)  \to \wedge^{aq}(V^*_- \otimes V_+  )= \wedge^{bq}((\mathfrak{p}'')^*)$.
\end{enumerate} 

\medskip
\noindent
{\it Remark.}
Note that  $\psi_{bq,0} \in  \mathrm{Sym}^{(bq)}((V''_+ \otimes W''_-)^*) \otimes \wedge^{bq}((\mathfrak{p}')^*) $ and 
$\psi^*_{bq,0} \in \wedge^{bq}((\mathfrak{p}')^*) \otimes \mathrm{Sym}^{(bq)}((V''_+ \otimes W''_-)^*)  )$ are interchanged by the map that switches the polynomial and exterior  factors.

\medskip

The  defining formula for $\psi^*_{bq,0}$ is then the following. Let $\mathbf{y}'' = (y''_1,\cdots,y''_b)$. 
Then we have

\begin{equation} \label{formulaforKazhdan}
\psi^*_{bq,0}(\mathbf{y}'') = ^t\widetilde{y^*_1} \wedge ^t  \widetilde{y^*_2}\wedge \cdots \wedge ^t \widetilde{y^*_b} \in \wedge^{bq}(V^*_+ \otimes V_-).
\end{equation}

\medskip
\noindent
{\it Remark.} It is important to observe that the above tensor is a wedge product of $bq$ vectors in $(\mathfrak{p}')^* $
depending on $\mathbf{y}''$, that is
$$\psi^*_{bq,0}(\mathbf{y}'') = [(y^*_1\otimes v_{p+1})\wedge \cdots \wedge (y^*_1\otimes v_{p+q})]\wedge \cdots \wedge [(y^*_b\otimes v_{p+1})\wedge \cdots \wedge (y^*_b\otimes v_{p+q})].$$
\medskip

We have a similar formulas for $\psi^*_{0,aq}$. Let $\mathbf{x}' = (x'_1,\cdots,x'_a)$ 
Then we have
\begin{equation} \label{formulaforantiKazhdan}
\psi^{\ast}_{0,aq}(\mathbf{x}') = ^t\widetilde{x_1} \wedge  \cdots \wedge ^t\widetilde{x_a} \in \wedge^{aq}(V^*_- \otimes V_+) = \wedge^{aq}((\mathfrak{p}')^*).
\end{equation}

It is immediate from the above  defining formulas that the holomorphic and anti-holomorphic cocycles  factor exterior  according to

\begin{lem} \label{completefactorization}
Suppose $a = u + v$ and $b = r+s$.  Then we have the factorization
\begin{enumerate}
\item $\psi_{bq,0} = \psi_{rq,0} \wedge \psi_{sq,0}$
\item $\psi_{0,aq} = \psi_{0,uq} \wedge \psi_{0,vq}$.
\end{enumerate}
\end{lem}
 The exterior product $\wedge$  in Lemma \ref{completefactorization}  is the ``outer exterior product'' associated to the product in the coefficient ring 
(in which the  cocycles take values) 
$$ \Pol( (V'_+ \otimes W'_+) \otimes \Pol( (V''_+ \otimes W''_-)  \to \Pol\big( (V'_+ \otimes W'_+) \oplus  (V''_+ \otimes W''_-)\big).$$
Here we note that if $V = A \oplus B$ then we have a multiplication map (isomorphism)
$\Pol(A) \otimes \Pol(B) \to \Pol(V).$

%\begin{proof}
%From Proposition \ref{formulaforcocycles} we have 
%$$\psi^*_{0,aq}(x'_1,x'_2,\cdots, x'_u, x'_{u+1},\cdots, x'_a) = \psi^*_{0,uq}(x'_1,x'_2,\cdots, x'_u)\ \wedge \  \psi^*_{0,v q}( x'_{u+1},\cdots, x'_a)$$
%and
%$\psi^*_{bq,0}(y''_1,y''_2,\cdots, y''_r, y''_{r+1},\cdots, y''_b)
 %= \psi^*_{rq,0}(y''_1,y''_2,\cdots, y''_r)\  \wedge \ \psi^*_{sq,0}( y''_{r+1},\cdots, y''_ s).$$
%\end{proof}
\subsubsection{The cocycles of Hodge type (bq,0)}
We now give a coordinate formula for $\psi_{q,0}$. Recall that $\xi'_{\alpha,\mu} \in (\mathfrak{p}')^* = (V^*_+ \otimes V_-)$ is given by
$\xi'_{\alpha,\mu} = v^*_{\alpha} \otimes v_{\mu}, 1 \leq \alpha \leq p, p \leq \mu \leq p+q$.

\begin{lem} \label{formulaforvarphi}
We have:
$$\psi_{q,0} = \sum \limits _{1 \leq \alpha_1,\cdots,\alpha_q \leq p} (z''_{\alpha_1} z''_{\alpha_2} \cdots z''_{\alpha_q}) \otimes (\xi'_{\alpha_1, p+1} \wedge \cdots \wedge \xi'_{\alpha_q,p+1}) .$$
\end{lem}
\begin{proof}
Let $y = \sum \limits_{1 \leq \alpha \leq p} z_{\alpha} v_{\alpha}$ and hence $y^* = \sum \limits_{1 \leq \alpha \leq p}\overline{z_{\alpha}} v^*_{\alpha}$. Hence 
\begin{align*}
^t \widetilde{y^*} & = \sum \limits _{1 \leq \alpha_1,\cdots,\alpha_q \leq p}\big((v^*_{\alpha_1}\otimes v_{p+1}) \wedge   \cdots \wedge (v^*_{\alpha_q}\otimes v_{p+1})\big) \otimes (\overline{z}_{\alpha_1} \overline{z}_{\alpha_2} \cdots \overline{z}_{\alpha_q})\\
&= \sum \limits _{1 \leq \alpha_1,\cdots,\alpha_q \leq p}(\xi'_{\alpha_1, p+1} \wedge \cdots \wedge \xi'_{\alpha_q,p+1}) \otimes ((z''_{\alpha_1} z''_{\alpha_2} \cdots z''_{\alpha_q}).
\end{align*}
Here the last equation is justified  by Lemma \ref{coordinateformula} which states that we have $\overline{z(y)} = z''(y'')$.  The lemma then follows because $\psi_{q,0}$ and $\psi^*_{q,0}$ are related by switching the polynomial and exterior tensor factors.  
\end{proof}
The reader will derive a coordinate formula for $\psi_{bq,0}$ by taking the $b$-fold outer exterior power of the formula above.
We will see later, Proposition \ref{valuesofcocycles}, that $\psi_{bq,0}$ is nonzero.

We will now   prove that $\psi_{q,0}$ is closed and hence, by Lemma \ref{completefactorization}, $\psi_{bq,0}$ is closed since the differential $d$ is a graded derivation of the outer exterior product. In this case we have $a=0$ and $b=1$. Hence we let $W= W_-$ be a one dimensional complex vector space with basis $w_1$ equipped with a Hermitian form $(\ ,\ )$ such that $(w_1,w_1) = -1$. We apply Equation \eqref{polyFock} with $W_+ =0$ to conclude that the Fock model $\mathcal{P}(V\otimes W)$ for the dual pair $\U(V) \times \U(W)$ is given by 
$$\mathcal{P}(V\otimes W)= \Pol(V''_+ \otimes W''_-) \otimes \Pol((V'_- \otimes W'_-).$$
We will use $z''_{\alpha}$ for the  coordinates on $V''_+ \otimes W''_-$ relative to the basis $\{v''_{\alpha} \otimes w''_1,,1 \leq \alpha \leq p\}$
and $z'_{\mu}$ for the  coordinates on  $V'_- \otimes W'_-$ relative to $\{v'_{\mu} \otimes w'_1,,p+1 \leq \alpha \leq p+q \}$.
As usual we let $\omega$ denote the (infinitesimal) Weil representation.  We then have
\begin{lem} \label{actionofxandy}
\begin{equation}
\omega(x_{\alpha,\mu}) = z''_{\alpha}z'_{\mu}  \ \text{and} \   \omega(y_{\alpha,\mu})  = \frac{\partial^2}{\partial z''_\alpha \partial z'_\mu}.
\end{equation}
\end{lem}

  As usual we define $\partial$, resp. $\overline{\partial}$, to be the bidegree $(1,0)$, resp. $(0,1)$, parts of the  differential $d$.  
It is then an immediate consequence of Lemma \ref{actionofxandy} that we have 
\begin{equation} \label{danddbar}
\partial = \displaystyle \sum_{\alpha =1}^{p} \displaystyle \sum_{\mu = p+1}^{p+q} z_\alpha ^{\prime \prime} z_\mu^{\prime }  \otimes A(\xi'_{\alpha \mu}) \ \text{and} \ 
\overline{\partial} =  \displaystyle \sum_{\alpha =1}^{p} \displaystyle \sum_{\mu = p+1}^{p+q}  \frac{\partial^2}{\partial z''_\alpha \partial z'_\mu} \otimes A(\xi''_{\alpha \mu}).
\end{equation}
It is then clear that $\overline{\partial}  \psi_{q,0} = 0$.  

\begin{lem}\label{qoclosed}
We have $d\psi_{0,q} = \partial\psi_{q,0} = 0$
\end{lem}

\begin{proof}
We have $\partial \psi_{q,0} = \displaystyle \sum_{\beta=1}^p \sum_{\mu=p+1}^{p+q} \displaystyle \sum_{1 \leq \alpha_1,\cdots,  \alpha_q \leq p} 
(z''_{\beta} z'_{\mu}) (z''_{\alpha_1} ... z''_{\alpha_q})  \otimes  \xi'_{\beta, \mu} \wedge \xi'_{\alpha_1, p+1} \wedge \cdots \wedge \xi'_{\alpha_q,p+q} $
Fix a value   $\mu = p+k$ in the second sum.  We then have  the subsum

$$S_\mu =  \displaystyle \sum_{\beta=1}^p \displaystyle \sum_{1 \leq \alpha_1,\cdots,  \alpha_q \leq p} 
z''_{\beta} z''_{\alpha_1}  ... z''_{\alpha_q} \otimes \otimes  \xi'_{\beta, \mu} \wedge \xi'_{\alpha_1, p+1} \wedge \cdots \wedge \xi'_{\alpha_q,p+q}.  $$
Clearly $S_{\mu}$ may be factored according to
$$S_\mu =   (\displaystyle \sum_{1 \leq \beta, \alpha_k \leq p} z'_\beta z'_{\alpha_k} \otimes \xi_{\beta, \mu} \wedge \xi_{\alpha_k, \mu}) \wedge \omega$$
for a certain  $q-2$-form $\omega$. Clearly the first factor is zero.
\end{proof}
 
We now  study the transformation properties of  $\psi^{\ast}_{bq,0}$ and hence of $\psi_{bq,0}$.  
From formula \eqref{formulaforKazhdan} we see that 
$\psi^*_{bq,0}$ is a homogeneous (of degree $q$ in each $y''_j, 1 \leq j \leq b$ and hence of total degree $bq$) assignment of  an element $\psi^*_{bq,0}(\mathbf{y}'')$  in $\wedge^{bq}(V^*_+ \otimes V_- )$ 
to a $b$-tuple 
$\mathbf{y}'' = (y''_1,\cdots,y''_b) \in (V''_+)^b \cong V''_+ \otimes W''_-$.  
From the above  formula it is clear that  $\psi^*_{bq,0}$ is a $\U(V_+)$-equivariant map.  

Recall, see \cite[p. 80]{FultonHarris}, that there  are quotient maps of $\U(V_+)$-modules 
$$\mathrm{Sym}^{bq}(V_+ \otimes W''_-) \to  S_{b \times q}(V''_+) \otimes S_{b \times q}(W''_-) \cong S_{b \times q}(V''_+) 
\otimes (\wedge^b W''_-)^{q} $$
and 
$$\wedge^{bq}(V^*_+ \otimes V_-) \to  S_{b \times q}(V^*_+) \otimes S_{q \times b}(V_-) \cong S_{b \times q}(V^*_+) 
\otimes \wedge^q( V_-)^{b}   .$$

Now note that $(\wedge^b W''_-)^{q} $ and $(\wedge^q V_-)^{b}$ are one dimensional and hence as $\U(V_+)$-modules  we have 
\begin{equation} \label{indentifications}
S_{b \times q}(V''_+) \otimes (\wedge^b W''_-)^{q} \cong  S_{b \times q}(V^*_+)  \otimes (\wedge^q V_-)^{b}.
\end{equation}

We now prove that  $\psi^*_{bq,0}$ induces the above isomorphism (so the lower horizontal arrow in the next diagram). 
In what follows we will use the  symbols $\U(V_+)$ and $\U(p)$ and $\U(V_-)$ and $\U(q)$ interchangeably. 

\begin{lem} \label{invariantversionforb}
The map $\psi^*_{bq,0}$ induces a commutative diagram of $\U(V_+)$-modules
\[ \begin{CD}
%S_{b \times q}(V''_+) \otimes (\wedge^b W''_- )^{q} @>>> S_{b \times q}(V^*_+) \otimes (\wedge^q V_- )^{b} \\
%@VVV              @VVV \\
\mathrm{Sym}^{bq}(V''_+ \otimes W''_-) @>\psi^*_{bq,0}>> \wedge^{bq}(V^*_+ \otimes V_-)\\
@VVV                @ VVV\\
S_{b \times q}(V''_+) \otimes (\wedge^b W''_- )^{\otimes q} @>\psi^*_{bq,0}>> S_{b \times q}(V^*_+) 
\otimes (\wedge^q V_- )^{\otimes b}
\end{CD}
\]
\end{lem}
\begin{proof}
We first note that the inclusion map (corresponding to the special case $b=1$ of the  above quotient map)
$ \iota: S^q(V_+^*) \otimes \wedge^q(V_- ) \to \wedge^q(V_+^* \otimes V_-)$
is given by 
$$\iota(f^{\otimes q} \otimes v_{p+1} \wedge \cdots \wedge v_{p+q})= (f \otimes v_{p+1}) \wedge \cdots (f \otimes v_{p+q}) =\widetilde{f}.$$
Hence 
$$\iota((y^*)^{ \otimes q} \otimes v_{p+1} \wedge \cdots \wedge v_{p+q}) = ^t\widetilde{ y^*}.$$
Hence $^t\widetilde{ y^*}$ transforms under $\U(V_-)$ according to $\det_{\U(V_-)}$
and hence $^t\widetilde{ y_1 ^*} \wedge \cdots \wedge ^t\widetilde{ y_b^*}$ transforms under $\U(V_-)$  according to $\det_{\U(V_-)}^b$.

Now recall from \eqref{GLdec} that we have 
$$\wedge^{bq}(V^*_+ \otimes V_-) = \bigoplus S_{\lambda} (V_+^*) \otimes S_{\lambda'} (V_-).$$
But in order that 
$\widetilde{ y_1 ^*} \wedge \cdots \wedge \widetilde{ y_b^*}$ transform under $\U(V_-)$  according to $\det_{\U(V_-)}^b$ the 
Young diagram $\lambda'$ must
be a $q \times b$ rectangle and hence $\lambda$ must be a $b \times q$ rectangle. 
But again by \cite[p. 80]{FultonHarris} we have 
$$S^{bq}(V^*_+ \otimes W_-) = \bigoplus S_{\lambda} (V_+^*) \otimes S_{\lambda} (W_-)$$
where the sum is over all  Young diagrams $\lambda$ with $bq$ boxes and at most $\min(p,b)$ rows.
Since the map $\psi_{bq,0}$ is $\U((V_+)^*)$-equivariant it must factor through the summand where $\lambda$ is a $b$ by $q$ rectangle. 
\end{proof}

We obtain 
\begin{lem} \label{formulaforphib} 
$$ \psi_{bq,0} \in   S_{b \times q}((V''_+)^*) \otimes (\wedge^b ((W''_- )^*)^{\otimes q}) \otimes S_{b \times q}(V^*_+ ) \otimes (\wedge^q (V_- ))^{\otimes b}$$
\end{lem} 

\begin{proof}
By dualizing the result of the previous lemma we obtain
$$\psi_{bq,0} \in \Hom(S_{b \times q}(V_+)
\otimes (\wedge^q (V^*_- ))^{\otimes b}, 
S_{b \times q}((V''_+)^*) \otimes (\wedge^b ((W''_- )^*)^{\otimes q}) .$$
We then use the isomorphism $\Hom(U_1,U_2) \cong U_2 \otimes U^*_1$. 
\end{proof}

\begin{cor}\label{transformlawforb}
The cochain $\psi_{bq,0}$ is invariant under $\U(p)$ acting by the standard action and transforms under $\U(q) \times \U(b)$ according 
to $\det_{\U(q)}^{b} \otimes  \det_{\U(b)}^{q}$. 
\end{cor}

 \subsubsection{The cocycles of Hodge type (0,aq) } 
We first  give a coordinate formula for $\psi_{q,0}$. Recall that $\xi''_{\alpha,\mu} \in (\mathfrak{p}'')^* = (V^*_- \otimes V_+)$ is given by
$\xi''_{\alpha,\mu} = v^*_{\mu} \otimes v_{\alpha}, 1 \leq \alpha \leq p, p \leq \mu \leq p+q$.  Let $x \in V_+$ be given by
$x = \sum_{\alpha}z_{\alpha} v_{\alpha}$. 
\medskip

The next lemma is proved in the same way as Lemma \ref{formulaforvarphi}.
\begin{lem} \label{formulaforvarphitwo}
We have 
$$\psi_{0,q}(x') = \sum \limits _{1 \leq \alpha_1,\cdots,\alpha_q \leq p} (z'_{\alpha_1} z'_{\alpha_2} \cdots z'_{\alpha_q}) \otimes (\xi''_{\alpha_1, p+1} \wedge \cdots \wedge \xi''_{\alpha_q,p+1}) $$
\end{lem}

We will now  prove that $\psi_{0,q}$ is closed and hence, as before, by Lemma \ref{completefactorization}, $\psi_{0,aq}$ is closed. In this case we have $a=1$ and $b=0$. Hence we let $W= W_+$ be a one dimensional complex vector space with basis $w_1$ equipped with a Hermitian form $(\ ,\ )$ such that $(w_1,w_1) = 1$. We apply Equation \eqref{polyFock} with $W_- =0$ to conclude that the Fock model $\mathcal{P}(V\otimes W)$ for the dual pair $\U(V) \times \U(W)$ is given by 
$$\mathcal{P}(V\otimes W)= \Pol(V'_+ \otimes W'_+) \otimes \Pol((V''_- \otimes W''_+).$$
We will use $z'_{\alpha} $ for the  coordinates on $V'_+ \otimes W'_+$ relative to the basis $\{v'_{\alpha} \otimes w'_1,,1 \leq \alpha \leq p\}$
and $z''_{\mu}$ for the  coordinates on  $V''_- \otimes W''_+$ relative to $\{v''_{\mu} \otimes w''_1,,p+1 \leq \alpha \leq p+q \}$.
As usual we let $\omega$ by the action of the (infinitesimal) Weil representation.  We then have
\begin{lem} \label{actionofxandyoq}
\begin{equation}
\omega(x_{\alpha,\mu}) =  \frac{\partial^2}{\partial z'_\alpha \partial z''_\mu} \ \text{and} \   \omega(y_{\alpha,\mu})  = z'_{\alpha}
z''_{\mu}.
\end{equation}
\end{lem}

It is then an immediate consequence of Lemma \ref{actionofxandyoq} that we have 
\begin{equation} \label{danddbar}
\partial = \displaystyle \sum_{\alpha =1}^{p} \displaystyle \sum_{\mu = p+1}^{p+q}  \frac{\partial^2}{\partial z'_\alpha \partial z''_\mu}  \otimes A(\xi'_{\alpha \mu}) \ \text{and} \ 
\overline{\partial} =  \displaystyle \sum_{\alpha =1}^{p} \displaystyle \sum_{\mu = p+1}^{p+q} z'_{\alpha}
z''_{\mu}  \otimes A(\xi''_{\alpha \mu}) .
\end{equation}
It is then clear that $\partial  \psi_{q,0} = 0$.

The next lemma is proved in the same way as Lemma \ref{qoclosed}. 

\begin{lem}\label{oqclosed}
We have $d \psi_{0,q} = \overline{\partial} \psi_{0,q}  = 0.$ 
\end{lem}

The following lemma is proved in the same way as Lemma \ref{invariantversionforb}. 
\begin{lem} \label{invariantversionfora}
The map $\psi^*_{0,aq}$ induces a commutative diagram 
\[ \begin{CD}
\mathrm{Sym}^{aq}(V'_+ \otimes W'_+) @>\psi^*_{0,aq}>> \wedge^{aq}(V_+ \otimes V^*_-)\\
@VVV                @ VVV\\
S_{a \times q}(V'_+) \otimes (\wedge^a W'_+)^{{\otimes q}} @>\psi^*_{0,aq}>> S_{a \times q}(V_+) 
\otimes (\wedge^q V^*_-)^{{\otimes a}}.
\end{CD}
\]
\end{lem}
The following lemma is proved in the same way as Lemma \ref{formulaforphib}.
\begin{lem}\label{formulaforphia}
We have 
$$ \psi_{0,aq} \in S_{a \times q}(V_+) \otimes (\wedge^q (V^*_- ))^{\otimes a} \otimes S_{a \times q}((V'_+)^*) \otimes (\wedge^a ((W'_+ )^*)^{\otimes q}).$$
\end{lem} 
We then have as before
\begin{cor}\label{transformlawfora}
The cochain $\psi_{0,aq}$ is invariant under $\U(p)$ acting by the standard action and transforms under $\U(q) \times \U(a)$ according 
to $\det_{\U(q)}^{-a} \otimes  \det_{\U(a)}^{-q}$. 
\end{cor}

We now define the general special cocycles $\psi_{bq,aq}$ of type $(bq,aq)$ by
\begin{def}\label{generalcocycles}
$$\psi_{bq,aq}= \psi_{bq,0} \wedge \psi_{0,aq}$$
\end{def}
Since  these cocycles are wedges of cocycles they are themselves closed.

We now summarize the properties  of  the  special cocycles.

\begin{prop}\label{formulaforcocycles}
Let $x_1,x_2,\cdots,x_a;y_1,\cdots,y_b \in V_+$ be given.  Put 
$$  \mathbf{x}' = (x'_1,x'_2,\cdots,x'_a) \ \mbox{ and } \ \mathbf{y}'' = (y''_1,,y''_2,\cdots,y''_b).$$ 
Then we have
\begin{enumerate}
\item $\psi^*_{bq,0}(\mathbf{x}', \mathbf{y}'') = \psi^*_{bq,0}(\mathbf{y}'') = ^t\widetilde{y^*_1} \wedge ^t\widetilde{y^*_2} \wedge  \cdots \wedge ^t \widetilde{y^*_b} \in \wedge^{bq}(V^*_+ \otimes V_-) \cong \wedge^{bq}((\mathfrak{p}')^*)$;
\item $\psi_{0,aq}(\mathbf{x}', \mathbf{y}'') = \psi^*_{aq,0}(\mathbf{x}') = ^t\widetilde{x_1} \wedge ^t\widetilde{x_2} \cdots \wedge ^t\widetilde{x_a} \in \wedge^{aq}(V^*_- \otimes V_+  )\cong \wedge^{aq}((\mathfrak{p}'')^*)$;
\item $\psi^*_{bq,aq}(\mathbf{x}', \mathbf{y}'') = \big(\widetilde{^ty^*_1}\wedge ^t\widetilde{y^*_2} \wedge \cdots \wedge \ ^t \widetilde{y^*_b}\big)\wedge \big(^t\widetilde{x_1} \wedge ^t\widetilde{x_2}\wedge 
 \cdots \wedge \ ^t\widetilde{x_a}\big)  
\in \wedge^{aq}(V^*_+ \otimes V_-) \otimes \wedge^{bq}
(V^*_- \otimes V_+ ) \cong \wedge^{aq}((\mathfrak{p}')^*) \otimes \wedge^{bq}((\mathfrak{p}'')^*) $;
\item The cochain  $\psi_{bq,aq}$  is a cocycle.
\item  The cocycle $\psi_{aq,aq}$  is the representation in the Fock model of the cocycle $\varphi_{aq,aq}$ in the Schr\"odinger model of Kudla and Millson. 
\item The cocycle $\psi_{bq,aq}$ transforms under $\U(q)$ according to $\det_{\U(q)}^{b-a}$.
\item The cocycle $\psi_{bq,aq}$  is invariant under $\SL(q)$.
\item The cocycle $\psi_{bq,aq}$ transforms under $\U(a) \times \U(b)$ according to $\det_{\U(a)}^{-q} \otimes \det_{\U(b)}^q$. 
\end{enumerate}
\end{prop}
\begin{proof} The only item that is not yet proved is (5). Note that (6) and  (7) follow from Corollaries  \ref{transformlawforb} and 
\ref{transformlawfora}.  We will prove  (5) in Appendix C. 
%As for (4) --- which won't be needed in this paper --- we just note that by definition of $\psi_{bq,aq}$, 
%it is enough to prove that both the holomorphic and anti-holomorphic
%cochains are cocycles. This follows from Proposition \ref{valuesofcocycles} below and the fact --- proved in \cite[III Proposition 6.1]{KashiwaraVergne} --- that the polynomial $\Delta_b$ 
%is a highest weight vector for the group $\U(b) \times \U (p) \times \U (q)$;  under the action of $\U (p,q)$ it generates a cohomological module isomorphic to 
%$A(b \times q, 0)$, see \cite{Notes} for more details. 
%Finally the fact that the cocycle $\psi_{aq,aq}$ is indeed the representation in the Fock model of the cocycle of Kudla and Millson is essentially contained in \cite{KudlaMillson3}; since it is crucial to our work, we provide more details in Appendix C.
\end{proof} 

\medskip
\noindent
{\it Remark.}
By Lemma \ref{completefactorization} the general cocycle $\psi_{bq,aq}$ factors as a product of the basic holomorphic and anti-holomorphic coycles $\psi_{q,0}$ and $\psi_{0,q}$. 
\medskip

\subsection{} In  \S \ref{TheWeilrepresentation} we pointed out that if $a$ and $b$ had opposite parity then  to descend the Weil representation restricted to 
$\widetilde{\U}(p,q)$ we needed to twist by an odd power of $\det^{1/2}_{\U(p,q)}$.  The following lemma shows that this odd  power is {\it uniquely determined} (even in the case that $a$ and $b$ have the same parity) by the condition that the special cocycles $\psi_{bq,aq}$ is a relative Lie algebra cochain with values in the Weil representation.

\begin{lem} \label{correcttwist}
The cocycle  $\psi_{bq,aq}$ considered as a linear map   from $\wedge^*\mathfrak{p}$ to $\mathcal{P}_+$ is $\U(p) \times \U(q) $-equivariant  if and only if we twist the restriction of the Weil representation to $\widetilde{\U}(p,q)$ by the character  $\det^{\frac{a-b}{2}}_{\U(p,q)}$  . In this case the action of $\widetilde{\U}(p)$ on polynomials will factor through the action induced by  the standard action of $\U(p)$ on $V_+$ and the action of $\widetilde{\U}(V_-)$ will
simply scale all polynomials by $\mathrm{det}_{\U(q)}^{b-a}$. 
\end{lem}

\begin{proof}
It is clear there exists at most one twist such that $\psi_{bq,aq}$ is equivariant.  Hence, it suffices to prove  the ``if'' part of the lemma.   The ``if" part follows from Theorem \ref{bottomlineforthetwist} and (6) of Proposition \ref{formulaforcocycles}. 
\end{proof}

\subsection{The values of the special cocycles on $e(bq,0)$, $e(0,aq)$  and $e(bq,aq)$} 
We now evaluate our special cocycles in the Vogan-Zuckerman vectors.

\begin{prop}\label{valuesofcocycles}
We have:
\begin{enumerate}
\item $\psi_{bq,0}(e(bq,0))(\mathbf{x}', \mathbf{y}'') =  \Delta_b(\mathbf{y}'')^q = 
\Delta_b (z''_{\alpha,j})^q$
\item $\psi_{0,aq}(e(0,aq))(\mathbf{x}', \mathbf{y}'') = \widetilde{\Delta}_a(\mathbf{x}')^{q} =  \widetilde{\Delta}_a (z'_{\alpha,i})^q$
\item $\psi_{bq,aq}(e(bq,aq) )(\mathbf{x}', \mathbf{y}'') = \psi_{bq,0}(e(bq,0))(\mathbf{y}'') \  \psi_{0,aq}(e(0,aq))(\mathbf{x}') =
\widetilde{\Delta}_a(\mathbf{x}')^q \ \Delta_b(\mathbf{y}'')^q =  \widetilde{\Delta_a} (z'_{\alpha,j})^q \  \Delta_b (z''_{\alpha,j})^q$.
\end{enumerate}
\end{prop}
\begin{proof}
We first prove (1). By Equation \eqref{KVZvector} we have
$$e(bq,0) = (-1)^{bq}\widetilde{v_1} \wedge \widetilde{v_2} \wedge \cdots \wedge  \widetilde{v_b}.$$
Combining this formula with the first formula in Proposition \ref{formulaforcocycles} we have
\begin{equation} \label{innerproductofbqvectors}
\begin{split}
\big(\psi_{bq,0}(e(bq,0)\big)( \mathbf{x}',\mathbf{y}'') & = \big(\psi^{\ast}_{bq,0}(\mathbf{x}',\mathbf{y}'')\big)(e(bq,0) ) \\
& = 
(-1)^{bq}\big( ^t\widetilde{y^*_1} \wedge ^t\widetilde{y^*_2}  \wedge \cdots \wedge ^t \widetilde{y^*_b}\big) 
\bigg( \big(\widetilde{v_1} \wedge \widetilde{v_2} \wedge \cdots \wedge  \widetilde{v_b} \big)\bigg).
\end{split}
\end{equation}

Recall  the definitions
$$^t \widetilde{y^*_1} \wedge ^t\widetilde{y^*_2}  \wedge \cdots \wedge  ^t\widetilde{y^*_b} = [(y_1^*\otimes v_{p+1})\wedge \cdots 
\wedge (y_1^*\otimes v_{p+q})]\wedge \cdots \wedge [(y_b^*\otimes v_{p+1})\wedge \cdots \wedge (y_b^*\otimes v_{p+q})]$$
and 
$$
\widetilde{v_1} \wedge \widetilde{v_2} \wedge \cdots \wedge  \widetilde{v_b}= [(v_1 \wedge v^*_{p+1}) \wedge \cdots \wedge (v_1 \wedge v^*_{p+q})] \wedge \cdots \wedge [ (v_a \wedge v^*_{p+1}) \wedge \cdots \wedge (v_b \wedge v^*_{p+q})].$$

From Equation \eqref{innerproductofbqvectors} and the two equations immediately above, we see that $(-1)^{bq} \psi_{bq,0}(e(bq,0))(\mathbf{y}'')$
is the determinant of the  $bq$ by $bq$ matrix $A(\mathbf{y}'')$ with matrix entries $(y^*_i \wedge v_{p+j})(v_k \wedge v^*_{p+ \ell}))= 
(v_k,y_i) \delta_{j,\ell}$ arranged in some order (with more work we could show $A(\mathbf{y}'') =   -I_q \otimes Z''(\mathbf{y}'')$ but we prefer to avoid this computation and procede more invariantly).  By definition $(v_k,y_i) = z''_{i,k}(\mathbf{y}'')$.  Hence the above matrix entry is either $z''_{i,k}$ or zero and hence $\psi_{bq,0}(e(bq,0))(\mathbf{y}'')$ is a polynomial of degree at most $bq$ in the entries $z''_{i,k}$
of the $b$ by $b$ matrix $Z''(\mathbf{y}'')$.  Hence $\psi_{bq,0}(e(bq,0))(\mathbf{y}'')$ is a polynomial $p(Z'')$ on the space of $b$ by $b$ matrices $Z''$.  But Corollary \ref{transformlawforb} and  Proposition \ref{formulaforcocycles} we have for $g \in \U(b)$
$$\psi_{bq,0}(e(bq,0))(\mathbf{y}''g) = {\det}_{\U(b)}^q (g) \psi_{bq,0}(e(bq,0)(\mathbf{y}'') \ \text{and \ hence} \  p(Z''g) = {\det}_{\U(b)}^q (g) p(Z''). $$

Hence in case $\det(Z''(\mathbf{y}'')) \neq 0$ we have 
$$p(Z''(\mathbf{y}'')) = p(I_b)\det(Z''(\mathbf{y}''))^q.$$  
By Zariski density of the invertible matrices (and the fact that every $n$ by $n$ matrix $Z''$ may be written as $Z''(\mathbf{y}'')$ for a suitable $\mathbf{y}''$)  the above equation holds for all 
$b$ by $b$ matrices $Z''$.
It remains to evaluate the value of $p$ on the identity matrix $I_b$.
This follows by setting $y_i = v_i, 1 \leq i \leq b$ and observing that each successive term in the $bq$-fold product  $^t\widetilde{y^*_1} \wedge ^t\widetilde{y^*_2}  \wedge \cdots \wedge ^t \widetilde{y^*_b}$ is the negative of the dual basis vector for the corresponding term in $ e(bq,0) = \widetilde{v_1} \wedge \widetilde{v_2} \wedge \cdots \wedge  \widetilde{v_b}$. Hence we obtain the determinant of $-I_{bq}$.

The proof of (2) is similar. 

We now observe that  formula (3) follows from (1) and (2). By Equation \eqref{KVZvector} we have
$$e(bq,aq)) =  e(bq,0) \wedge e(0,aq)$$
and by Proposition \ref{formulaforcocycles} we have 
$$\psi_{bq,aq}( \mathbf{x}', \mathbf{y}'') = \psi_{bq,0}(\mathbf{y}'') \wedge \psi_{0,aq}(\mathbf{x}').$$
Hence we have 
\begin{equation*}
\begin{split}
\psi_{bq,aq}(e(bq,aq) )( \mathbf{x}', \mathbf{y}'') & = \big(\psi_{bq,0}(\mathbf{y}'') \wedge \psi_{0,aq}(\mathbf{x}') \big)( e(bq,0) \wedge e(0,aq)) \\
& = \psi_{bq,0}(e(bq,0))(\mathbf{y}'')   \psi_{0,aq}(e(0,aq))(\mathbf{x}').
\end{split}
\end{equation*}
The Proposition follows.
\end{proof}

We conclude that the polynomials $\psi_{bq,0}(e(bq,0))$ and $\psi_{0,aq}(e(0,aq))$ are (special) harmonic for all $a,b$ and  if $a+b\leq p$ then $\psi_{bq,aq}(e(bq,aq) )$ is harmonic. The polynomial $\psi_{bq,0}(e(bq,0))$ transforms under
$\U(a) \times \U(b)$  according to the one dimensional representation $1 \otimes \det^{-q}$. The polynomial $\psi_{0,bq}(e(0,aq))$ transforms under $\U(a) \times \U(b)$ according to the one dimensional representation $\det^{q} \otimes 1$.  The polynomial  $\psi_{bq,aq}(e(bq,aq) )$  transforms under $\U(a) \times \U(b)$
according to the one dimensional representation $\det^q \otimes \det^{-q}$.

\subsection{The  cocycle $\psi_{bq,0}$  generates the $S_{b \times q}(V_+) \otimes (\wedge ^ q V^*_-)^{b}$ isotypic component for the action of 
$\U(p) \times \U(q)$ on the polynomial Fock space }
\hfill
In this section we will abbreviate the space of harmonic polynomials $\Harm(V_+ \otimes W)$ to $\mathcal{H}_+$ and the space 
$\Harm(V_- \otimes W)$ to $\mathcal{H}_-$. 
 The goal of this subsection is to prove the following theorem
\begin{thm}\label{Kazhdangenerates}
We have
$$\mathrm{Hom}_K \big(  S_{b \times q}(V_+) \otimes (\wedge ^ q V^*_-)^{b}, \mathcal{P}( V \otimes W) \big) = \mathcal{U}\big(\mathfrak{u}(a,b)_{\C}\big) \cdot \psi_{bq,0}.$$
\end{thm}

Theorem \ref{Kazhdangenerates} will be a consequence of  the next three lemmas. We will henceforth abbreviate the representation $S_{b \times q}(\C^p)\otimes (\wedge ^ q V^*_-)^{b}$ to $V(bq)$.

Recall that  $e(bq,0) = \widetilde{v_1} \wedge \cdots \wedge \widetilde{v_b} \in \wedge^{bq}(V_+ \otimes V^{\ast}_ -)$ is the Vogan-Zuckerman vector.
We have seen in Proposition \ref{valuesofcocycles} that 

$$\psi_{bq,0} ( e(bq)) = \Delta_b(\mathbf{y}'')^q $$
and consequently the 
 value of $\psi_{bq,0}$ on the Vogan Zuckerman vector $e(bq,0)$ is a (special) harmonic polynomial, that is
$$\psi_{bq,0} ( e(bq,0)) \in \Pol(V''_+ \otimes W''_-) \subset \mathcal{H}_+.$$
We now have 
\begin{lem} \label{compactharmonics}
$$\mathrm{Hom}_K \big(  V(bq) , \mathcal{H}_+ \otimes \mathcal{H}_- \big)= \C \psi_{bq,0}.$$
\end{lem}

Clearly Lemma \ref{compactharmonics} follows from the 
\begin{lem} \label{positiveharmonics} The representation 
$ V(bq)$ of $\U(p)$ occurs  once in  $ \mathcal{H}_+$ and the one dimensional representation $\det_{\U(q)}^{-b}$ of $\U(q)$ occurs  once in $ \mathcal{H}_-$. Hence
$$\Hom_{\U ( p) \times \U (q)} \big(V(bq) , \mathcal{H}_+ \otimes \mathcal{H}_- \big) = \C \psi_{bq,0}.$$ 
\end{lem}
\begin{proof} 
We first prove that  $S_{b \times q}(V_+)$ of $\U(p)$ occurs  once in  $ \mathcal{H}_+$. 
Indeed, the  actions of the groups $\U(p)$ and $\U(a) \times \U(b)$ on $ \mathcal{H}_+$ form a dual pair. Furthermore the 
correspondence of unitary representations $\tau:\U( p)^{\vee}
 \to \U(a)^{\vee} \times \U(b)^{\vee}$ is given in \cite[Theorem 6.3]{KashiwaraVergne}.
From their formula we see that $\lambda = b \times q$ corresponds to the one dimensional representation $1 \otimes \det^q$ of $\U(a) \times \U(b)$.
Since the multiplicity of the representation with highest weight $\lambda$ of $\U(p)$ corresponds to the dimension of the corresponding representation $\tau(\lambda)$ which is {\it one} in this case, we have proved that $S_{b \times q}(V_+)$ occurs once as claimed.  We note that we may realize this occurrence explicity as follows. First note that 
$$\mathrm{Sym}^{bq}(V'_+ \otimes W'_-) \cong \mathcal{P}^{(bq)}(V''_+ \otimes W''_-) \subset \mathcal{H}_+. $$
But by \cite[p. 80]{FultonHarris}, we have 

$$S_{b \times q}(V'_+) \otimes S_{b \times q}(W'_-) \subset \mathrm{Sym}^{bq}(V'_+ \otimes W'_-).$$

We note that since $\dim(W_-) = b$, the $b$-th exterior power of $W'_-$ is the top exterior power and we have
$$S_{b \times q}(W'_-) \cong (\wedge^b W'_-)^{q}.$$
Consequently $\U(W_-)$ acts on the second factor by $\det_{\U(b)}^q$.

It remains to prove that the  representation $\det_{\U(q)}^{-b}$ of $\U(q)$ occurs  once in the oscillator representation action on $\mathcal{H}_-$ the harmonic polynomials in
$$\mathcal{P}_- = \Pol( (V'_- \otimes W'_-) ) \otimes \Pol((V''_- \otimes W''_+)). $$ 
Since the oscillator representation action of $\U(q)$ is the standard action twisted by $\det_{\U(q)}^{-b}$ this is equivalent to showing that the
{\it trivial} representation of $\U(q)$ occurs once in the standard action of $\U(q)$ on $\mathcal{H}_-$.  It occurs at least once because the 
constant polynomials are harmonic.  But as above, by \cite[Theorem 6.3]{KashiwaraVergne}, the trivial representation of $\U(q)$ corresponds to the trivial representation of $\U(a) \times \U(b)$ and consequently it has multiplicity one and the lemma follows. 
\end{proof}

Theorem \ref{Kazhdangenerates} is now a consequence of the following result  of Howe, see \cite[Proposition 3.1]{HoweT}.

\begin{lem} \label{Howeresult}
We have
$$\Hom_K \big( V(bq), \mathcal{P}(V \otimes W)\big) \\ = \mathcal{U}(\mathfrak{u}(a,b)_{\C})\cdot \Hom_K(V(bq) , \mathcal{H}_+ \otimes \mathcal{H}_-).$$
\end{lem}

Hence, combining Lemma \ref{Howeresult} and Lemma \ref{compactharmonics} we obtain
$$\Hom_K \big( V(bq) , \mathcal{P}(V \otimes W)\big)= \mathcal{U}(\mathfrak{u}(a,b)_{\C})\cdot  _{bq,0}.$$
Theorem \ref{Kazhdangenerates} is now proved. 

\medskip
\noindent
{\it Remark.}
The reader will apply reasoning similar to that immediately above to deduce that the  cocycle $\psi_{0,aq}$  generates (over $\mathcal{U}(\mathfrak{u}(a,b)_{\C})$) the $S_{a \times q}(V^*_+) \otimes (\wedge ^ q V_-)^{a}$ isotypic component of the polynomial Fock space.

\medskip

\subsection{The  cocycle $\psi_{bq,aq}$ generates the $V(bq,aq)$ isotypic component of the polynomial Fock space}
We recall that $V(bq,aq)$ is the representation of $\U(p)\times \U(q) $ with highest weight the sum of the two previous highest weights:
$$(\underbrace{q,q,\cdots,q}_b,0,0,\cdots,0,\underbrace{-q,-q,\cdots,-q}_a; \underbrace{a-b,a-b,\cdots,a-b}_q).$$  
We also note that this  representation is the Cartan product of  $S_{b \times q}(\C^p) \otimes \det_{\U(q)}^{-b}$ and   $S_{a \times q}((\C^p)^{\ast})\otimes \det_{\U (q)}^{a}$.  Here we remind the reader that the Cartan product of two irreducibles is the irreducible subspace of the tensor product generated by the highest weight vectors of the factors (so the irreducible whose highest weight is the {\em sum} of the two highest weights of the factors). 
We now have

\begin{thm}\label{KMgenerates}
We have
$$\mathrm{Hom}_K \big(  V(bq,aq) , \mathcal{P}( V \otimes W ) \big)=\mathcal{U}(\mathfrak{u}(a,b)_{\C}) \cdot _{bq,aq}.$$
\end{thm}

Theorem \ref{KMgenerates} is proved the same way as  Theorem \ref{Kazhdangenerates}. Once again we have  a multiplicity one result in $\mathcal{H}_+ \otimes \mathcal{H}_-$.

\begin{lem} \label{compactharmonicsKM}
We have
$$\mathrm{Hom}_K \big(  V(bq,aq) , \mathcal{H}_+ \otimes \mathcal{H}_- \big)= \C _{bq,aq}.$$
\end{lem}
\begin{proof}
The product group  $\U( p) \times (\U(a) \times \U(b))$ acts as a dual pair  on $\mathcal{H}(V_+ \otimes (W'_+ \oplus W''_ -))$. Hence 
the dual representation of $\U(a) \times \U(b)$ has  highest weight 
$$(\underbrace{q,q,\cdots,q}_a, \underbrace{-q,-q,\cdots,-q}_b)$$ 
hence is $\det^q \otimes \det^{-q}$.
Hence, it is {\it one dimensional}  so the Cartan product of $S_{b \times q} (\C^p )$ and $S_{a \times q} (\C^p )^*$ has multiplicity one in $\mathcal{H}_+ $.

We leave the proof that the one-dimensional representation $\det_{\U(q)}^{a-b}$ of $\U(q)$ has multiplicity one in $\mathcal{H}_-$ to the reader
(once again the constant polynomials transform under $\U(q)$ by this twist). 
The lemma follows.
\end{proof}

Now Theorem \ref{KMgenerates} follows from the result of Howe, see Lemma \ref{Howeresult}.

\subsection{The polynomial Fock space}

 We refer the reader to Appendix C for the notation used in the following paragraph and further details.  

In the study of the global theta correspondence beginning in Section \ref{Sec:1} we  will need to consider the cocycles $\varphi_{bq,aq}$ with values in the Schr\"odinger model for the oscillator representation of $\U(V) \times \U(W)$  corresponding to the cocycles $\psi_{bq,aq}$ defined  above with values in the Fock model. In order to give a precise statement of the relation between them  we recall there is an intertwining operator, the Bargmann transform $B_{V\otimes W}$, see \cite{Folland}, page 40 and page 180,  from the Schr\"odinger model of the oscillator representation of $\U(V) \times \U(W)$ to the Fock model. We define the polynomial Fock space $\mathbf{S}(V \otimes E) \subset \mathcal{S} (V \otimes E)$ to be the image of the  of the holomorphic polynomials in the Fock space under the inverse Bargmann transform $B_{V\otimes W}^{-1}$.   Here $\mathcal{S} (V \otimes E)$  is the Schwartz space.  We then have

\begin{equation}\label{relationbetweencocycles}
\varphi_{bq,aq} = (B^{-1}_{V\otimes W} \otimes 1)(\psi_{bq,aq})
\end{equation}

\newpage

\part{The geometry of Shimura varieties}

\section{Shimura varieties and their cohomology}

\subsection{Notations} \label{10.1} Let $E$ be a CM-field with totally real field maximal subfield $F$ with $[F: \Q] = d$. We denote by $\A_{\Q}$ the ring of ad\`eles of $\Q$ and by $\A$ the ring of ad\`eles of $F$. We fix $d$ non-conjugate complex embeddings $\tau_1 , \ldots , \tau_d : E \to \C$ and denote by $x \mapsto \bar x$ the non trivial
automorphism of $E$ induced by the complex conjugation of $\C$ w.r.t. any of these embeddings.  We identify $F$, resp. $E$, with a subfield of $\R$, resp. $\C$, via $\tau_1$. 

Let $V$, $(,)$ be a nondegenerate anisotropic Hermitian vector space over $E$ with $\dim_E V =m$. Note that we take $(,)$ to be conjugate linear in the second 
argument. We denote by $V_{\tau_i} = V \otimes_{E , \tau_i} \C$ the complex Hermitian vector space obtained as the completion of $V$ w.r.t. to the complex embedding $\tau_i$. Choosing a suitable isomorphism $V_{\tau_i} \cong \C^m$, we may write
$$(u,v) = {}^t \- u H_{p_i , q_i} \overline{v}, \quad \forall u,v \in \C^m,$$
where 
$$H_{p_i , q_i} = \left( 
\begin{array}{cc}
1_{p_i} & \\
& -1_{q_i} 
\end{array} \right),$$
and $(p_i,q_i)$ is the signature of $V_{\tau_i}$. We will consider in this paper only those $V$, $(,)$ such that 
$q_2 = \ldots = q_d =0$ and let $(p,q)=(p_1,q_1)$. By replacing $(,)$ by $-(,)$ we can, and will, assume that $p \geq q$.

Now let us fix a non-zero element $\alpha \in E$ with $\bar \alpha = -\alpha$.
We have two associated $F$-bilinear forms:
$$B(u,v) = \frac12 \mathrm{Tr}_{E/F} (u,v) \mbox{ and } \langle u, v \rangle = \frac12 \mathrm{Tr}_{E/F} (\alpha^{-1} (u,v)).$$
The form $B(,)$ is symmetric while $\langle , \rangle$ is alternating. They satisfy the identities:
\begin{equation}
B(u , \alpha v) = - b(\alpha u , v) , \quad \langle u, \alpha v \rangle = - \langle \alpha u , v \rangle.
\end{equation}
Note that, for $\lambda \in E$, we have $\langle \lambda u , v \rangle = \langle u , \bar \lambda v \rangle$ and that the Hermitian form is given by 
\begin{equation}
(u,v) = \langle \alpha u, v \rangle + \alpha \langle u, v \rangle.
\end{equation}
The non-degenerate alternating form $\langle, \rangle$ --- or equivalently the Hermitian form $(,)$ --- determines an {\it involution} $*$ of $\mathrm{End}_E (V)$, i.e. a $F$-linear anti-automorphism $\mathrm{End}_E (V) \to \mathrm{End}_E (V)$ of order $2$ by:
$$(gu,v)=(u,g^*v) \mbox{ or equivalently } \langle gu, v \rangle = \langle u , g^* v \rangle.$$
In the following we will consider $\mathrm{End}_E (V)$ as a $\Q$-algebra; it is simple with center $F$ and the involution $*$ is of the second kind (see e.g. \cite[\S 2.3.3]{PlatonovRapinchuk}).

\subsection{Unitary group (of similitudes) of $V$} 
We view the unitary group in $m$ variables $\U (V)$ as a reductive algebraic groups over $F$. We let $G_1 =  \mathrm{Res}_{F/\Q} \U (V)$, so that for any $\Q$-algebra $A$
\begin{equation*}
\begin{split}
G_1 (A)  & = \left\{ g \in \mathrm{End}_E (V) \otimes_{\Q} A \; : \; (g u , g v ) = (u,v) \mbox{ for all } u,v \in V \otimes_{\Q} A \right\} \\
& = \left\{ g \in \mathrm{End}_E (V) \otimes_{\Q} A \; : \; g g^* = 1 \right\}.
\end{split}
\end{equation*}
The embeddings $\tau_i:E \to \C$ in particular induce an isomorphism 
$$G_1( \R) = \U (p,q) \times \U (m)^{d-1}.$$

The group of unitary similitudes $\mathrm{GU} (V)$ is the algebraic group over $F$ whose points in any $F$-algebra $A$ are given by
\begin{equation*}
\begin{split}
\mathrm{GU}(V) (A) & = \{ g \in \mathrm{End}_E (V) \otimes_{F} A \; : \; \exists \lambda (g) \in A^{\times} \mbox{ s.t. }  (g \cdot , g \cdot ) = \lambda (g) (\cdot , \cdot)  \} \\
& =  \left\{ g \in \mathrm{End}_E (V) \otimes_F A \; : \; g g^* = \lambda (g) \in A^{\times} \right\}.
\end{split}
\end{equation*}
Here $\lambda$ is the {\it similitude norm}. We let $G=  \mathrm{Res}_{F/\Q} \mathrm{GU} (V)$.\footnote{We warn the reader that in this part of the paper $G$ refers 
to the unitary {\it similitude} group and that we now refer to the usual unitary group as $G_1$.} Consider the rational torus $\mathrm{Res}_{F/ \Q} \mathbb{G}_{\mathrm{m} F}$ whose group of rational points is $F^{\times}$. We abusively denote by $\lambda : G \to \mathrm{Res}_{F/ \Q} \mathbb{G}_{\mathrm{m} F}$ the homomorphism of algebraic groups over $\Q$ induced by the similitude norm. Set 
$$\mathrm{GU} (a,b) = \{ A \in \GL_{a+b} (\C ) \; : \; {}^t A H_{a,b} \overline{A} = c(A) H_{a,b}, \ c(A) \in \R^{\times} \}.$$
The embeddings $\tau_i : E \to \C$ induce an isomorphism 
$$G (\R) \cong \mathrm{GU} (p,q) \times \mathrm{GU} (m)^{d-1}.$$
It is useful to point out that the groups of $E$-points in $\mathrm{GU} (V)$ is isomorphic to
$\GL (V) \times E^{\times}$, inside which the $F$-group $\mathrm{GU} (V)$ is defined as
$$\{ (g,t) \in \GL (V) \times E^{\times} \; : \; (t (g^*)^{-1} , \overline{t} ) = (g,t) \}.$$
In this formulation the similitude norm $\lambda$ is the projection on the second factor. The determinant on $\U (V)$ induces --- by restriction of scalars --- a character $\det : G_1 \to \mathrm{Res}_{E/\Q} \mathbb{G}_{\mathrm{m} E}$.\footnote{Here by definition the group of rational points of $\mathrm{Res}_{E/\Q} (\mathbb{G}_{\mathrm{m} E})$ is $E^{\times}$.} By the above discussion $\lambda$ and $\det$ generate the character group of $G$, where on the rational group these characters are related by $\lambda (g)^m = \mathrm{N}_{E/F} (\det (g))$. 

Consider the rational torus $\mathrm{Res}_{E/\Q} \mathbb{G}_{\mathrm{m} E} \times \mathrm{Res}_{F/ \Q} \mathbb{G}_{\mathrm{m} F}$ whose group of rational points is $E^{\times} \times F^{\times}$ 
and let $T$ be the rational subtorus defined by the equation $\mathrm{N}_{E/F} (x) = t^m$ ($x \in \mathrm{Res}_{E/\Q} \mathbb{G}_{\mathrm{m} E}$, $t \in \mathrm{Res}_{F/ \Q} \mathbb{G}_{\mathrm{m} F}$). Then $T=G/G_{\rm der}$ is the maximal torus quotient of $G$ (see Kottwitz \cite[\S 7]{Kottwitz} for a slightly different situation); we denote by 
$$\nu : G \to T = G / G_{\rm der}$$
the quotient map. 

Let $T_1$ be the rational group defined as the kernel of the norm homomorphism $\mathrm{N}_{E/F} = \mathrm{Res}_{E/\Q} \mathbb{G}_{\mathrm{m} E} \to \mathrm{Res}_{F/\Q} \mathbb{G}_{\mathrm{m} F}$. Then $T_1 = G_1 / G_{\rm der}$ is the maximal abelian quotient of $G_1$; the quotient map $G_1 \to T_1$ is induced by $\det$.

If $m$ is even, say $m=2k$, then the torus $T$ is isomorphic to $T_1  \times \mathrm{Res}_{F/ \Q} \mathbb{G}_{\mathrm{m} F}$, the isomorphism being given by 
$(x,t) \mapsto (xt^{-k} , t)$ for $(x,t) \in T$. 

If $m$ is odd, say $m=2k+1$, then $T$ is isomorphic to  $\mathrm{Res}_{E/\Q} \mathbb{G}_{\mathrm{m} E}$, the isomorphism being given by $(x,t) \mapsto xt^{-k}$ for $(x,t) \in T$. 

Note that in any case the quotient $T/T_1$ is isomorphic to the rational torus $\mathrm{Res}_{F/\Q} \mathbb{G}_{\mathrm{m} F}$.
For future reference we recall that the derived group $G_{\rm der}= \mathrm{Res}_{F/\Q} \SU (V)$ is connected and simply connected as an algebraic group.

\subsection{Shimura data} The symmetric space $X$ associated to $G_{\rm der} (\R)$ --- or equivalently the symmetric space associated to $G(\R)$ (modulo its center) --- is also the space 
$$X = \U(p,q) / (\U ( p) \times \U ( q))$$
of negative $q$-planes in $V_{\tau_1}$. It is isomorphic to a bounded symmetric domain in $\C^{pq}$. Following Kottwitz \cite{Kottwitz} and the general theory of Deligne \cite{Deligne,Milne} the pair $(G , X)$ defines a Shimura variety $\mathrm{Sh} (G, X)$ which has a canonical model over the reflex field $E(G,X)$. More precisely, let $\mathbb{S}$
be the real algebraic group $\mathrm{Res}_{\C/ \R} \mathbb{G}_{\mathrm{m} \C}$, so that $\mathbb{S} (\R) = \C^{\times}$, and define a  homomorphism 
of real algebraic group $h_0 :\mathbb{S}  \to G$ as follows. 
Since 
$$G(\R) \cong \mathrm{GU} (p,q) \times \mathrm{GU} (m)^{d-1},$$
it suffices to define the components $h_j$, $j=1, \ldots , d$, of $h_0$. 

For $j>1$, we take  $h_j$ to be the trivial homomorphism. For $j=1$, fix a base point $x_0 \in X$; it corresponds to $x_0$ a negative $q$-plane $V_-$ in $V_{\tau_1}$. Let us simply write $V$ for $V_{\tau_1}$ in the remaing part of this paragraph. Now let $V_+ \subset V$ denote the orthogonal complement of $V_-$ w.r.t. $(,)$. As in Section \ref{linearalgebra} we associate to the decomposition $V=V_+ + V_-$ a positive definite Hermitian form $(,)_{x_0}$ --- the associated minimal majorant --- by changing the sign of $(,)$ on $V_-$. By taking the real part of $(,)_{x_0}$ we obtain a positive definite symmetric form $B(,)_{x_0}$. Let $\theta_{x_0}$ be the Cartan involution
which acts by the identity on $V_+$ and by $-\mathrm{id}$ on $V_-$ and let $J_{x_0} = \theta_{x_0} \circ J$ be the corresponding positive almost complex structure on $V$.
We then have:
$$B(u,v)_{x_0} = \langle J_{x_0} u , v \rangle = - \langle u , J_{x_0} v \rangle.$$
For $a+ib \in  \C$, let
$$h (a+ib) = a + b J_{x_0} \in \mathrm{End} (V).$$
The map $h$ defines an $\R$-algebra homomorphism s.t.
\begin{itemize}
\item $h (z)^* = h (\bar z)$, where here again $*$ is the involution on $\mathrm{End} (V)$ determined by $(,)$, and 
\item the form $\langle h (i) u,v \rangle$ is symmetric and positive definite on $V$.
\end{itemize} 
Note that $h (z) h (z)^* = |z|^2$. We conclude that the restriction of $h$ to $\C^{\times}$ defines a homomorphism of real algebraic groups
$h_1 = \C^{\times} \to \mathrm{GU} (V)$. 

We let $h_0=(h_1, \ldots , h_d)$; it defines a homomorphism 
of real algebraic group $h_0 :\mathbb{S}  \to G$. 
The space $X$ may then be viewed as the space of conjugates of $h_1$ by $\mathrm{GU} (V)$ or, equivalently, of $h_0$ by $G(\R)$.

Now we have $\mathbb{S} (\C) = \C^{\times} \times \C^{\times}$, where we order the factors s.t. the first factor corresponds to the identity embedding $\C \to \C$. Recall that we have decomposed $V=V_{\tau_1}$ as $V=V_+\oplus V_-$, so that $h_1 (z)$ acts by $z$, resp. $\bar z$, on $V_+$, resp. $V_-$, w.r.t. this decomposition the Hermitian matrix of $(,)$ is diagonal equal to $H_{p,q}$. Identifying the complexification of $\mathrm{GU}(V)$ with $\GL_m (\C) \times \C^d$ the complexified homomorphism $h_1 : \mathbb{S} (\C) \to GL_m (\C) \times \C^d$ is given by
$$h_{1 \C} (z,w) = \left( 
\begin{array}{cc}
z 1_p & \\
& w1_q  
\end{array} \right) \times zw.$$
Let $\mu : \C^{\times} \to G (\C) \cong (\GL_m (\C) \times \C^{\times})^d$ be the restriction of the complexification of $h_{0 \C}$ of $h_0$ to the first factor. Up to conjugation, we may assume that
the image of $\mu$ is contained in a maximal torus of $G$ defined over $\Q$; it therefore defines a cocharacter of $G$. By definition the {\it reflex field} $E(G , X)=E(G,h_0)$ is the subfield of $\overline{\Q}$ corresponding to the subgroup of $\mathrm{Gal} (\overline{\Q} / \Q)$ of elements fixing the conjugacy class of $\mu$. It is a subfield of any extension of $\Q$ over which $G$ splits. In particular $E(G,X)$ is a subfield of $E$. We can decompose $V=V_{\tau_1}$ as $V=V_+\oplus V_-$, so that $h(z)$ acts by $z$, resp. $\bar z$, on $V_+$, resp. $V_-$, and $E(G,X)$ is precisely the field of definition of the representation $V_+$ of $E$. Since the Hermitian space $V$ is $F$-anisotropic we conclude that $E(G,X)=E$.

\subsection{The complex Shimura variety} The pair $(G,X)$, or $(G,h_0)$, gives rise to a Shimura variety $\mathrm{Sh} (G,X)$ which is defined over the reflex field $E$. In particular if $K 
= \prod_p K_p \subset G (\A_{\Q}^f)$, with $K_p \subset G (\Q_p)$, is an open compact subgroup of the finite adelic points of $G$, we can consider $\mathrm{Sh}_K (G , X)$; it is a projective variety over $E$
whose set of complex points is identified with
\begin{equation}
S(K) = \mathrm{Sh}_K (G , X) (\C) = G(\Q) \backslash (X \times G( \A_{\Q}^f)) / K.
\end{equation}
We will always choose $K$ to be {\it neat} in the following sense : For every $k \in K$, there exists some prime $p$ such that the semisimple part of the 
$p$-component of $k$ has no eigenvalues which are roots of unity other than $1$. Every compact open subgroup of $G(\A_{\Q}^f)$ contains a neat subgroup of
finite index.

In general $S(K)$ is not connected; this is a disjoint union of spaces of the type $S(\Gamma)$ discussed in the introduction, for various arithmetic subgroups
$\Gamma \subset G_1 (\Q) = \U (V) (F)$. Since we have assumed $K$ to be neat, these arithmetic subgroups are all torsion free.
 
The connected components of $S(K)$ can be described as follows. Write 
\begin{equation} \label{10.4.2}
G(\A_{\Q}^f ) = \sqcup_j G(\Q) g_j K
\end{equation}
with $g_j \in G (\A_f)$. Then 
\begin{equation}
S(K) \cong \sqcup_j S(\Gamma_j ),
\end{equation}
where $S(\Gamma_j) = \Gamma_j \backslash X$ and $\Gamma_j$ is the image in the adjoint group $G_{\rm ad} (\R)$ of the subgroup
\begin{equation}
\Gamma_j ' =g_j K g_j^{-1} \cap G (\Q)
\end{equation}
of $G(\Q)$ (see \cite[Lemma 5.13]{Milne}).

\subsection{The structure of $S_K$}
Let $Z$ be the center of $G$ and, as above, let $T$ be the largest abelian quotient of $G$. There are homomorphisms $Z \hookrightarrow G \stackrel{\nu}{\to} T$, and we define 
$$Y = T(\R) / \mathrm{Im} (Z(\R) \to T(\R)).$$
Recall that $T(\R) = (\C^{\times})^d$ if $m$ is odd, resp. $T(\R) = (\R^{\times} \times \U_1)^d$ if $m$ is even, with $U_1$ the complex unit circle. 
The image of $Z(\R)$ is $(\C^{\times})^d$, resp. $(\R_{>0} \times \U_1)^d$. Therefore, $Y=\{ 1 \}$ if $m$ is odd, and $Y= \{ \pm 1 \}^d$ if $m$ is even.

Since $G_{\rm der}$ is simply connected, the set of connected components of the complex Shimura variety $S(K)$ can be identified with the double coset
$$T(\Q) \backslash (Y \times T(\A_{\Q}^f )) / \nu (K),$$
see \cite[p. 311]{Milne}. 

The action of the Galois group $\mathrm{Gal} (\overline{E} / E)$ on $\pi_0 (S(K))$ is decribed in \cite[p. 349]{Milne} (see also \cite[p. 14]{PappasRapoport}):
consider the composition $\nu \circ \mu$; it is defined over the reflex field $E$ and therefore defines a homomorphism
$$r : \A_E^{\times} \to T (\A_{\Q}).$$
The action of $\mathrm{Gal} (\overline{E} / E)$ on $\pi_0 (S(K))$ factors through the abelian quotient $\mathrm{Gal} (E^{\rm ab} / E)$ and if, given 
$\sigma \in \mathrm{Gal} (E^{\rm ab} / E)$, we let $s \in \A_E^{\times}$ be its antecedent by the Artin reciprocity map and let $r(s) = (r(s)_{\infty} , r(s)_f) \in 
T(\R) \times T (\A_{\Q}^f)$, then
\begin{equation}
\sigma [y, a]_K = [r(s)_{\infty} y , r(s)_f a ]_K , \quad \forall \ y \in Y, \ a \in T( \A_{\Q}^f).
\end{equation}

\subsection{Cohomology of Shimura varieties}
We are interested in the cohomology groups $H^{\bullet} (S(K) , R)$ where $R$ is a $\Q$-algebra. If $K' \subset K$ is another compact subgroup of 
$G(\A_{\Q}^f)$ we denote by $\mathrm{pr}: S(K') \to S(K)$ the natural projection. It induces a map 
$$\mathrm{pr}^* : H^{\bullet} (S(K) , R) \to H^{\bullet} (S(K') , R).$$
Passing to the direct limit over $K$ via the maps $\mathrm{pr}^*$ we obtain:
$$H^{\bullet} (\mathrm{Sh} (G,X) , R) = \lim_{\substack{\to \\ K}} H^{\bullet} (S(K) , R).$$

The cohomology groups $H^{\bullet} (\mathrm{Sh} (G , X) , \C)$ are $G(\A_{\Q}^f)$-modules. For any character $\omega$ of $Z(\A_{\Q}^f)$, we denote by 
$H^{\bullet} (\mathrm{Sh} (G , X) , \C) (\omega)$ the $\omega$-eigenspace. Denote also by $\widetilde{\omega}$ the character of $Z (\A_{\Q})$ trivial on $Z(\R) Z (\Q)$ and with finite part 
$\omega$. Then one knows that 
\begin{equation}
H^{\bullet} (\mathrm{Sh} (G , X) , \C) (\omega) \cong H^{\bullet} (\mathfrak{g} , K_{\infty} ; L^2 (G , \widetilde{\omega})),
\end{equation}
where $\mathfrak{g}$ is the Lie algebra of $G(\R)$, $K_{\infty}$ is the stabilizer of a point in the symmetric space $X$ and $L^2 (G , \widetilde{\omega})$ is the Hilbert space 
of measurable functions $f$ on $G(\Q ) \backslash G (\A_{\Q})$ such that, for all $g \in G(\A_{\Q})$ and $z \in Z(\A_{\Q})$, $f(g z) = f(g) \widetilde{\omega} (z) $ and $|f|$ is square-integrable on 
$G(\Q ) Z(\A_{\Q}) \backslash G (\A_{\Q})$. 

Since $G$ is anisotropic each $L^2 (G,  \widetilde{\omega})$ decomposes as a direct sum of irreducible unitary representations of $G(\A_{\Q})$ with finite multiplicities. A representation $\pi$ which occurs in this way
is called an {\it automorphic representation} of $G$; it is factorizable as a restricted tensor product of admissible representations. We shall write $\pi = \pi_{\infty} \otimes \pi_f$ where $\pi_{\infty}$ is a unitary 
representation of $G(\R)$ and $\pi_f$ is a representation of $G(\A_{\Q}^f)$. We denote by  $\chi (\pi)$, resp. $\chi (\pi_f)$, its {\it central character} $\tilde{\omega}$, resp. $\omega$ and by $\mathrm{m} (\pi)$ its multiplicity in $L^2 (G,  \widetilde{\omega})$.

\subsection{Representations with cohomology} Let $\mathrm{Coh}_{\infty}$ be the set of unitary representations $\pi_{\infty}$ of $G(\R)$ (up to equivalence) such that 
\begin{equation}
H^{\bullet} (\mathfrak{g} , K_{\infty} ; \pi_{\infty}) \neq 0,
\end{equation}
where $\mathfrak{g}$ is the Lie algebra of $G(\R)$ and $K_{\infty}$ is the stabilizer of a point in the symmetric space $X$. Note that $K_{\infty}$ is the centralizer in $G$ of the 
maximal compact subgroup $K_1$ of $G_1 (\R)$. Given a representation $\pi$ of $G$ we denote by $\pi_1$ its restriction 
to $G_1$ and say that $\pi$ is {\it essentially unitary} if $\pi_1$ is unitary. Recall from Section \ref{sec:CR} that cohomological representations of $G_1 (\R)$ 
are classified by Vogan and Zuckerman in \cite{VZ}. 

Now the representation theory of $G$ is substantially identical to that of $G_1$. Let $Z$ and $Z_1$ denote the centers or $G$ and $G_1$, respectively. 
Then $G=ZG_1$, and every representation (local or global) of $G_1$ extends to $G$; it suffices to extend its central character.

Since we only consider cohomological representations of $G (\R)$ having trivial central character the classification of $\mathrm{Coh}_{\infty}$ amounts to the Vogan-Zuckerman classification. In particular, the set $\mathrm{Coh}_{\infty}$ is finite. For any $\pi_f$, set 
$$\mathrm{Inf} (\pi_f) =\{ \pi_{\infty} \in \mathrm{Coh}_{\infty} \; : \; \mathrm{m} (\pi_{\infty} \otimes \pi_f)  \neq 0 \}.$$
Let $\mathrm{Coh}_f$ be the set of $\pi_f$ such that $\mathrm{Inf} (\pi_f)$ is non-empty.

We will be particularly interested in the cohomological representations $A(b\times q , a \times q)$; we denote by $\mathrm{Coh}_f^{b,a}$ the set of $\pi_f$ such that 
$\mathrm{Inf} (\pi_f)$ contains $A(b\times q , a \times q)$.

\subsection{} Let $\mathcal{H}_K$ be the Hecke algebra of $\Q$-linear combinations of $K$-double cosets in $G (\A_{\Q}^f)$. If $\pi_f$ is a representation of 
$G (\A_{\Q}^f)$, we let $\pi_f^K$ denote the representation of $\mathcal{H}_K$ on the space of $K$-fixed vectors of $\pi_f$.

Over $\C$, there is an $\mathcal{H}_K$-isomorphism 
\begin{equation} \label{Mdec}
H^{\bullet} (S(K) , \C) \to \bigoplus_{\pi_f \in \mathrm{Coh}_f} H^{\bullet} (\pi_f , \C) \otimes \pi_f^K,
\end{equation}
where 
$$H^{\bullet} (\pi_f, \C) = \bigoplus_{\pi_{\infty} \in \mathrm{Inf} (\pi_f)} \mathrm{m} (\pi_{\infty} \otimes \pi_f ) H^{\bullet} (\mathfrak{g} , K_{\infty} ; \pi_{\infty}).$$

Given two integers $a$ and $b$ we denote by $H^{b \times q , a \times q} (S(K) , \C)$ the part of $H^{\bullet} (S(K) , \C)$ which corresponds to the 
cohomological representation $\pi_{\infty} = A(b\times q , a \times q)$, so that the $\mathcal{H}_K$-isomorphism induces the isomorphism:
\begin{multline} \label{Mdecab}
H^{b \times q , a \times q} (S(K) , \C)  \\ \to \bigoplus_{\pi_f \in \mathrm{Coh}_f^{b,a}} \mathrm{m} (A(b\times q , a \times q) \otimes \pi_f ) H^{(a+b)q} (\mathfrak{g} , K_{\infty} ; A(b\times q , a \times q)) \otimes \pi_f^K.
\end{multline}

\subsection{Rational subspaces of the cohomology groups}
Since the action of $\mathcal{H}_K$ is defined on $H^{\bullet} (S(K) , \Q)$, we obtain a $\overline{\Q}$-form of \eqref{Mdec}:
\begin{equation} \label{Mdec2}
H^{\bullet} (S(K) , \overline{\Q}) \to \bigoplus_{\pi_f \in \mathrm{Coh}_f} H^{\bullet} (\pi_f , \overline{\Q}) \otimes \pi_f^K (\overline{\Q}),
\end{equation}
where $H^{\bullet} (\pi_f , \overline{\Q})$ and $\pi_f^K (\overline{\Q})$ are $\overline{\Q}$-forms of $H^{\bullet} (\pi_f , \C)$ and $\pi_f^K$, respectively.
By considering arbitrary small $K$ we obtain a $\overline{\Q}$-form $\pi_f (\overline{\Q})$ of any $\pi_f \in \mathrm{Coh}_f$. 
Moreover: since $\mathrm{Gal} (\overline{\Q} / \Q)$ acts on $H^{\bullet} (S(K) , \overline{\Q})$ via its action on the coefficients $\overline{\Q}$ it permutes the summands in \eqref{Mdec2} and  therefore induces an action $(\sigma , \pi_f ) \mapsto \pi_f^{\sigma}$ of $\mathrm{Gal} (\overline{\Q} / \Q)$ on $\mathrm{Coh}_f$. We let $A(\pi_f) = \{ \sigma \in \mathrm{Gal} (\overline{\Q} / \Q) \; : \; 
\pi_f^{\sigma} \cong \pi_f \}$ be the stabilizer of $\pi_f$ and denote by $\Q (\pi_f)$ the corresponding number field. 
Given $[\pi_f ] \in \mathrm{Coh}_f / \mathrm{Gal} (\overline{\Q} / \Q)$, we define
\begin{equation*}
\begin{split}
W ([\pi_f ] ) & = \bigoplus_{\sigma \in \mathrm{Gal} (\overline{\Q} / \Q) / A(\pi_f)} H^{\bullet} (\pi_f^{\sigma} , \overline{\Q}) \otimes \pi_f^{\sigma} (\overline{\Q}) \\
& = \bigoplus_{\sigma \in \mathrm{Hom} (\Q (\pi_f ) , \overline{\Q})} \mathrm{Hom}_{G(\A_{\Q}^f)} ( \pi_f^{\sigma} , H^{\bullet} (S(K) , \overline{\Q}) ) \otimes  \pi_f^{\sigma} (\overline{\Q}) ,
\end{split}
\end{equation*}
so that 
$$H^{\bullet} (S(K) , \overline{\Q}) \cong \bigoplus_{[\pi_f ] \in \mathrm{Coh}_f / \mathrm{Gal} (\overline{\Q} / \Q)} W ([\pi_f ] )^K.$$

\begin{thm} \label{thm:rational}
Suppose that $\pi_f \in \mathrm{Coh}_f$ contributes to $H^{\bullet} (S(K) , \C)$. Then:
\begin{enumerate}
\item The subspace $W([\pi_f ])^K \subset H^{\bullet} (S(K) , \overline{\Q})$ is a polarized $\Q$-sub-Hodge structure of 
$H^{\bullet} (S(K) , \Q)$.
\item If moreover $\pi_f \in \mathrm{Coh}_f^{b,a}$ with $3(a+b)+ |a-b| <2m$, then we have:\footnote{Recall that $SH^{\bullet}$ is defined in the Introduction.}
\begin{multline*}
(W([\pi_f ] )^K \otimes_{\Q} \C) \cap SH^{(a+b)q} (S(K) , \C) \\ 
\subset  H^{b\times q , a \times q} (S(K) , \C) \oplus H^{a\times q , b \times q} (S(K) , \C).
\end{multline*}
\end{enumerate}
\end{thm}
\begin{proof} The first part is classical. It follows from the fact that the Hecke algebra $\mathcal{H}_K$ acts as algebraic correspondences on $S(K)$ which yield morphisms of the rational Hodge structure $H^{\bullet}  (S(K) , \Q)$. The polarization then comes from the cup product and Poincar\'e duality, both of which are functorial for algebraic correspondences. 

We postpone the proof of the second part until Section \ref{sec:appl}. One important ingredient is the global theta correspondence that we review in the next section.
\end{proof}

\medskip

In the special $q=1$ case we have $SH^{\bullet} (S(K) , \C) = H^{\bullet} (S(K) , \C)$ and we get the following:

\begin{cor} \label{C:rational}
Suppose $q=1$ and let $a$ and $b$ be integers s.t. $3(a+b)+ |a-b| <2m$. Then, the space $H^{a+b} (S(K) , \Q)$ contains a polarized 
$\Q$-sub-Hodge structure $X$ such that 
$$X\otimes_{\Q} \C = H^{a,b} (S(K) , \C) \oplus H^{b,a} (S(K) , \C).$$
\end{cor}

\medskip
\noindent
{\it Remark.} In particular the sub-space $H^{1,1} (S(K) , \C) \subset H^2 (S(K) , \C)$ is defined over $\Q$ as long as $p>2$. 
Note that it is not the case if $p=2$, see \cite{BR}. We will see that when $p>2$ the subspace $H^{1,1} (S(K) , \C) \subset H^2 (S(K) , \C)$ is generated by theta lifts from unitary groups of signature $(1,1)$ at infinity. This is no more true when $p=2$ but
the subspace which is generated by classes obtained by theta lifts --- or equivalently the subspace associated with endoscopic 
representations --- is indeed defined over $\Q$, see \cite{BR}. Granted this Michael Harris has proposed us another proof that 
$H^{1,1} (S(K) , \C) \subset H^2 (S(K) , \C)$ is defined over $\Q$ when $p>2$: Indeed classes obtained by theta-lift restrict to classes obtained 
by theta lifts to any sub-Shimura variety associated to a smaller unitary group $\U (2,1)$. The general result now reduces to the theorem of Blasius and 
Rogawski via Oda's trick using the restriction theorem of Harris and Li \cite{HarrisLi}.

\section{The global theta correspondence} \label{Sec:1}

We keep notations as in \S \ref{10.1}.

\subsection{The theta correspondence} \label{par:2.2}
Let $W$ be a $n$-dimensional vector space over $E$ equipped with a skew-Hermitian $\langle , \rangle$ which is conjugate linear in the first argument. 
We take $V$ to be a left $E$ vector space and $W$ to be a right $E$ vector space. These conventions comes into play when considering
the tensor product ${\Bbb W} =  W \otimes_E V$; as an $F$-vector space it is endowed with the symplectic form
$$[,]={\rm tr}_{E/F} \left( \langle , \rangle \otimes \overline{(,)}  \right)$$
where ${\rm tr}_{E/F}$ denotes the usual trace of $E$ over $F$. We let $\Sp ({\Bbb W})$ be the corresponding symplectic 
$F$-group. Then $(\U(V) , \U (W))$ forms a reductive
dual pair in $\Sp ({\Bbb W})$, in the sense of Howe \cite{Howe}. 

\medskip
\noindent
{\it Remark}. We can define a Hermitian space $W'$ by $W' = W$ (viewed as a left $E$ vector space via $aw=wa$) and 
$$(w_1 , w_2 ) = \alpha^{-1} \langle w_2 , w_1 \rangle.$$
We will sometimes abusively refer to $W$ as a Hermitian space; note however that this involves the choice of $\alpha$.

\medskip

Let ${\rm Mp}({\Bbb W})$ be the metaplectic two-fold cover of ${\rm Sp} ({\Bbb W})$
(see Weil \cite{Weil2}). Fix a choice of a non-trivial character $\psi$ of ${\Bbb A}/F$ and denote by
$\omega=\omega_{\psi}$ the corresponding (automorphic) Weil representation of ${\rm Mp}({\Bbb W})$, as in \cite{Howe}. 

A {\it complete polarization} ${\Bbb W}={\Bbb X} + {\Bbb Y}$, where ${\Bbb X}$ and ${\Bbb Y}$ 
are maximal totally isotropic subspaces of ${\Bbb W}$, leads to the realization of $\omega$ on $L^2 ({\Bbb X})$. This is known as the 
Schr\"odinger model for $\omega$, see Gelbart \cite{Gelbart}. In that way $\omega$ is realized as an automorphic representation of $\Mp ({\Bbb W})$. 
The maximal compact subgroup of $\Sp ({\Bbb W})$ is $\mathrm{U} = \mathrm{U}_{nm}$, the unitary group in $nm$ variables. 
We denote by $\widetilde{\mathrm{U}}$ its preimage in $\Mp ({\Bbb W})$. 
The associated space of smooth vectors of $\omega$ is the Bruhat-Schwartz space $\mathcal{S} ({\Bbb X}(\A))$. The 
$(\mathfrak{sp} , \widetilde{\mathrm{U}})$-module associated to $\omega$ is made explicit by the realization of $\omega$ in the Fock model. 
Using it, one sees that the $\widetilde{\mathrm{U}}$-finite vectors in $\omega$ is the subspace 
$\mathbf{S}({\Bbb X}(\A)) \subset \mathcal{S} ({\Bbb X}(\A))$ obtained by replacing, at each infinite place, the Schwartz space by the {\it polynomial Fock space} $\mathbf{S}({\Bbb X}) \subset \mathcal{S} ({\Bbb X})$, i.e. the image of holomorphic polynomials on $\C^{nm}$ under the intertwining map from the Fock model of the oscillator representation to the Schr\"odinger model. 

\subsection{} We denote by $\U_m (\A)$, $\U_n (\A)$, $\Sp_{2nm} (\A)$ and $\Mp_{2nm} (\A)$ the adelic points of respectively $\U(V)$, $\U (W)$,
$\Sp ({\Bbb W})$ and $\Mp ({\Bbb W})$.
According to Rao, Perrin and Kudla \cite{Kudla2}, for any choice of a 
pair of characters $\chi = (\chi_1 , \chi_2 )$ of ${\Bbb A}_E^{\times} / E^{\times}$ whose restrictions
to ${\Bbb A}^{\times}$ satisfy $\chi_1 |_{{\Bbb A}^{\times}} = \epsilon_{E/F}^m$ and 
$\chi_2 |_{{\Bbb A}^{\times}} = \epsilon_{E/F}^{n}$, there exists a homomorphism
\begin{eqnarray} \label{ichi}
\tilde{\i}_{\chi} : \U_m (\A ) \times \U_n (\A) \rightarrow \Mp_{2nm} (\A)
\end{eqnarray}
lifting the natural map
$$\i : \U_m (\A ) \times \U_n (\A)  \rightarrow \Sp_{2nm} (\A),$$
and so, we obtain a representation $\omega_{\chi}$ of $\U_m (\A ) \times \U_n (\A)$ on 
$\mathbf{S}({\Bbb X}(\A))$.\footnote{At infinity the choices of $\chi_1$ and $\chi_2$ correspond to the choice of a pair of integers $(k, \ell)$ with 
$k \equiv m \ ({\rm mod} \ 2)$ and $\ell \equiv n \ ({\rm mod} \ 2)$ and $\omega_{\chi}$ yields $\omega_{k,\ell}$ as in \S \ref{TheWeilrepresentation}.}

The global metaplectic group $\Mp_{2nm} (\A )$ acts in $\mathcal{S} ({\Bbb X}(\A))$ via $\omega$ and preserves the dense subspace 
$\mathbf{S}({\Bbb X}(\A))$. 
For each $\phi \in \mathbf{S}({\Bbb X}(\A))$ we form the theta function 
\begin{equation}
\theta_{\psi , \phi} (x) = \sum_{\xi \in {\Bbb X}(F)} \omega_{\psi} (x ) (\phi) (\xi) 
\end{equation}
on $\Mp_{2nm} (\A )$. Pulling the oscillator representation 
$\omega_{\psi}$ back to $\U_m (\A ) \times \U_n (\A)$ using the map \eqref{ichi} we get a
a smooth, slowly increasing function $(g,g') \mapsto \theta_{\psi, \chi, \phi} (g',g) = \theta_{\psi, \phi} (\tilde{\i}_{\chi} (g',g))$ on 
$\U (V) \backslash \U_m (\A ) \times \U (W) \backslash \U_n (\A)$; see \cite{Weil,Howe}.

\subsection{} \label{rem:2.4}
{\it Remark.} Let $\chi ' = (\chi_1 ' , \chi_2 ' )$ be another pair of characters of ${\Bbb A}_E^{\times} / E^{\times}$ whose restrictions
to ${\Bbb A}^{\times}$ satisfy $\chi_1' |_{{\Bbb A}^{\times}} = \epsilon_{E/F}^m$ and 
$\chi_2 ' |_{{\Bbb A}^{\times}} = \epsilon_{E/F}^{n}$, and put $\mu = \chi_1' \chi_1^{-1}$ and $\nu= \chi_2 ' \chi_2^{-1}$. Since $\mu_{| \A^{\times}} = \nu_{| \A^{\times}} = 1$, we can define
characters $\mu'$ and $\nu'$ of $\A^1_E$ --- the adelic points of the kernel of the norm $\mathrm{N}_{E/F}$ --- by setting $\mu' (x/ \bar x) = \mu (x) $ and $\nu ' ( x / \bar x ) = \nu (x)$. Let 
$\mu_n = \mu ' \circ \det$ and $\nu_m = \nu' \circ \det$ be the associated characters of $\U_n (\A)$ and $\U_m (\A)$, respectively. Then it follows from the explicit formulas contained in \cite{Kudla2} that
$$\omega_{\psi} (\tilde{\i}_{\chi '} (g, g')) = \omega_{\psi} (\tilde{\i}_{\chi} (g, g')) \nu_m (g) \mu_n (g').$$

\subsection{The global theta lifting} \label{par:1.4}
We denote by $\mathcal{A}^c (\U (W))$ the set of irreducible cuspidal automorphic representations 
of $\U_{n} (\A)$, which occur as irreducible subspaces in the space of cuspidal automorphic
functions in $L^2 ( \U (W) \backslash \U_{n} (\A))$. As in \cite{KR} we will denote by $[\U_n ]$ the quotient $ \U (W) \backslash \U_{n} (\A)$. 
For a $\pi ' \in \mathcal{A}^c (\U (W))$, the integral
\begin{eqnarray} \label{theta}
\theta_{\psi, \chi, \phi}^f (g) = \int_{[\U_n]} \theta_{\psi, \chi, \phi} (g,g') f(g') dg' , 
\end{eqnarray}
with $f \in H_{\pi '}$ (the space of $\pi '$), defines an automorphic function on 
$\U_m (\A)$~:  the integral \eqref{theta} 
is well defined, and determines a slowly increasing function on $\U (V) \backslash \U_m (\A)$.
We denote by $\Theta_{\psi , \chi , W}^V (\pi ')$ the space of the automorphic representation generated by all 
$\theta_{\psi, \chi, \phi}^f (g)$ as $\phi$ and $f$ vary, and call $\Theta_{\psi, \chi , W}^V (\pi ')$ the $(\psi , \chi)$-theta lifting 
of $\pi '$ to $\U_m (\A)$. Note that, since $\mathbf{S}({\Bbb X}(\A))$ is dense in $\mathcal{S} ({\Bbb X}(\A))$ we may as well let 
$\phi$ vary in the subspace $\mathbf{S}({\Bbb X}(\A))$.

We can similarly define $\mathcal{A}^c (\U (V))$ and 
$\Theta_{\psi , \chi , V}^W$ the $(\psi , \chi)$-theta correspondence from $\U (V)$ to $\U (W)$.

\begin{defn} 
We say that a representation $\pi \in \mathcal{A}^c (\U (V))$ {\it is in the image of the cuspidal $\psi$-theta correspondence from a smaller group} 
if there exists a skew-Hermitian space $W$ with $\dim W \leq m$, a representation $\pi ' \in \mathcal{A}^c (\U (W))$ and a pair of characters $\chi$ such that
$$\pi =\Theta_{\psi, \chi , W}^V (\pi ').$$
\end{defn}

\subsection{Local signs} \label{par:5.10} Given a representation $\pi \in \mathcal{A}^c (\U (V))$ in the image of the cuspidal $\psi$-theta correspondence from a smaller group $\U (W)$ we associate to $\pi$ local signs in the following way: Let $v$ be a finite place of $F$. By a theorem of Landherr \cite{Landherr}, for each $n$ there are exactly two different classes of isomorphism of $n$-dimensional Hermitian spaces over $E_v$:
\begin{enumerate}
\item For $n=2r+1$ odd, let $W_{r,r}$ denote the Hermitian space of dimension $2r$ over $E_v$ with maximal isotropic subspaces of dimension $r$, then the two classes
are represented by $W^{\pm} = W_{r,r} \oplus W_{1}^{\pm}$ where $W_1^{\pm} \cong E_v$ is the one dimensional Hermitian space over $E_v$ with Hermitian form $(x,y)= \alpha \bar{x} y$,
where $\alpha \in F^{\times}_v$ with $\epsilon_{E_v/F_v} (\alpha ) = \pm 1$.
\item For $n=2r$ even, then the two classes are represented by $W^+ = W_{r,r}$ and $W^- = W_{r-1 , r-1} \oplus W_2^-$ where $W_2^-$ is an anisotropic space of dimension $2$.
\end{enumerate}
Now we associate to $\pi$ the local sign $\varepsilon (\pi_v) = \pm 1$ depending on whether $W \cong W^+$ or $W \cong W^-$. The conservation
relation conjecture of Harris, Kudla and Sweet \cite[Speculations 7.5 and 7.6]{HKS} --- whose relevant part to us has been proved by Gong and Greni\'e \cite{GongGrenie} --- implies that this local sign is well defined and only depends 
on $\pi$. 

Note that the local sign $\varepsilon (\pi_v)$ is equal to $1$ at all but finitely many places and that we have:
$$\prod_{v < \infty} \varepsilon (\pi_v) = 1.$$

\subsection{Extension of the theta correspondence to similitude groups} The extension of the theta correspondence to unitary similitude groups has been worked out in details by Michael Harris in \cite[\S 3.8]{Harris93}. It is based on the observation that the map
$$i : \mathrm{GU} (W) \times \mathrm{GU} (V) \to \GL (W \otimes_E V), \quad i(g',g) (w \otimes v) = wg' \otimes g^{-1} v$$
takes the algebraic subgroup
$$\mathrm{G}(\U(V) \times \U (W)) := \{ (g,g') \in \mathrm{GU} (V) \times \mathrm{GU} (W) \; : \; \lambda (g) = \lambda (g') \}$$
into $\Sp (W \otimes_E V)$. The definition of \eqref{theta} also has to modified a bit, see \cite[(3.8.2)]{Harris93}. We will however
abusively refer to \eqref{theta} even while working with unitary similitude groups. 

We finally note that an automorphic representation $\pi$ of $G$ is in the image of the extension to unitary similitude groups of the
theta correspondence from a smaller group $\mathrm{GU}(W)$ if $\pi_1$ is in the image of the $\psi$-theta correspondence from $\U (W)$. 
In that case we will loosely say that $\pi$ is in the image of the $\psi$-theta correspondence from $\U (W)$. 

The main automorphic ingredient of our paper is the following theorem. 

\begin{thm} \label{thm:10.10}
Let $a$ and $b$ be integers s.t. $3(a+b)+|a-b| <2m$ and let $\pi_f \in \mathrm{Coh}_f^{b,a}$. Set $\pi = A(b \times q , a \times q)  \otimes \pi_f$, then 
$\pi$ is in the image of the $\psi$-theta correspondence from a smaller group $\U(W)$ of signature $(a,b)$ at infinity. 
\end{thm}
\begin{proof} The proof of this theorem is the goal of Part 3. There we work with unitary groups rather than similitude groups. To be more precise the theorem follows from
Proposition \ref{prop:main} which is stated and proved in Part 3. 
The only remaining thing to be proved is that the signature at infinity is $(a,b)$: this follows from the fact that $A(b \times q ,a \times q)$ is the image of the local theta correspondance from a group $\U(W,{\mathbb C}/{\mathbb R})$ of dimension $a+b$ if and only if the signature is $(a,b)$. This follows from work of Annegret Paul \cite{Paul}.
\end{proof}

Since $Z(\A_{\Q}^f)$ embeds into $T(\A_{\Q}^f)$ via the map $\nu$, it acts on the disconnected Shimura variety $S(K)$ by permuting the connected components. As a corollary we conclude:

\begin{cor} \label{cor:10.11}
Let $S$ be any connected component of $S(K)$ and let $a$ and $b$ be integers such that $3(a+b)+|a-b| <2m$. Then $H^{b \times q , a\times q} (S , \C)$ is generated by classes of theta lifts from unitary groups of signature $(a,b)$ at infinity. 
\end{cor}

\section{Special cycles}

\subsection{Notations} We keep notations as in \S \ref{10.1} and follow the adelization \cite{Kudla}
of the work of  Kudla-Millson. Let $n$ be an integer $0\leq n \leq p$. Given an $n$-tuple $\mathbf{x}=(x_1 ,  \ldots , x_n ) \in V^n$ we let 
$U=U(\mathbf{x})$ be the $F$-subspace of $V$
spanned by the components of $\mathbf{x}$. We write $(\mathbf{x},\mathbf{x})$ for the $n\times n$ Hermitian matrix  with $ij$th
entry equal to $(x_i , x_j)$.  
Assume $(\mathbf{x},\mathbf{x})$ is totally positive semidefinite of rank $t$. Equivalently:
as a sub-Hermitian space $U \subset V$ is totally positive definite of dimension $t$. 
In particular: $0 \leq t \leq p$ (and $t\leq n$). 
The constructions of the preceding section can therefore be made with the space $U^{\perp}$ in place
of $V$. Set $H=\mathrm{Res}_{F/\Q} \mathrm{GU} (U^{\perp})$. There is a natural morphism $H \rightarrow G$ and (the image of) $H$ is isomorphic to $G_U$ 
the pointwise stabilizer of $U$ in $G$; we will abusively use both notations. Recall that we can realize the symmetric space $X$ as the set of negative $q$-planes in $V_{v_0}$. We then let $X_H$ be the
subset of $X$ consisting of those $q$-planes which lie in $U^{\perp}_{v_0}$. 

\subsection{Shimura subvarieties}
There is a natural morphism $i_U : \mathrm{Sh} (H , X_H) \to \mathrm{Sh} (G , X)$ which is defined over the reflex field $E$. If $K$ is an open compact
subgroup of $G(\A_{\Q}^f)$ we set $K_H = H (\A_{\Q}^f) \cap K$. The variety $\mathrm{Sh}_{K_H} (H , X_H)$ is projective defined over $E$ and the set of
its complex points identifies with 
$$S_H (K_H ) = H (\Q) \backslash X_H \times H (\A_{\Q}^f) / K_H.$$

Now given an element $g \in G(\A_{\Q}^f )$, we may shift the natural morphism $i_U$ by $g$ to get 
\begin{equation}
\begin{split}
i_{U,g,K} : \ H (\Q) \backslash X_H \times H (\A_{\Q}^f) / K_{H,g} & \to G(\Q)   \backslash X \times G( \A_{\Q}^f) / K \\
H(\Q) (z,h) K_{H,g} & \mapsto G(\Q ) (z,hg) K,
\end{split}
\end{equation}
where $K_{H,g} = H (\A_{\Q}^f) \cap gKg^{-1}$. We denote by $Z(U,g,K)$ the algebraic cycle $i_{U,g,K} (S_H (K_{H,g} ))$. It is defined over the 
reflex field $E$. 

\subsection{Connected cycles}
Suppose $K$ is {\it neat} and let $g \in G(\A_{\Q}^f)$. Set $\Gamma_g ' = gKg^{-1} \cap G(\Q)$ and $\Gamma_{g,U} ' = g K g^{-1} \cap H( \Q) = \Gamma_g ' \cap H (\Q)$. Now let $\Gamma_g$, resp. $\Gamma_{g,U}$, denote the image of $\Gamma_g '$, resp. $\Gamma_{g,U} '$, in the adjoint group $G_{\rm ad} (\R)$,
resp. $H_{\rm ad} (\R)$. Then the natural map $\Gamma_{g,U} z \mapsto \Gamma_g z$ yield a (totally geodesic) immersion of $\Gamma_{g,U} \backslash X_H$ into $\Gamma_g \backslash X$. We will denote the corresponding (connected) cycle by $c(U,g,K)$. 

We now introduce composite cycles that may be seen as composed of the connected cycles $c(U,g,K)$ or of the algebraic cycles $Z(U,g,K)$.

\subsection{Special cycles}
Given $\beta \in \mathrm{Her}_n (E)$ a $n \times n$ totally positive Hermitian matrix --- we use the notation $\beta \gg 0$  for such an Hermitian matrix --- we define 
\begin{equation}
\Omega_{\beta} = \left\{ \mathbf{x} \in V^n \; : \; \frac12 (\mathbf{x} , \mathbf{x}) = \beta \right\}.
\end{equation}
The natural action of $G(\A_{\Q}^f)$ on $V(\A_{\Q}^f)^n$ restrict to an action on $\Omega_{\beta} (\A_{\Q}^f)$. Then: any $K$-invariant compact subset of 
$\Omega_{\beta} (\A_{\Q}^f)$ decomposes as a union of at most finitely many disjoint $K$-orbits.

Now let $ \in \mathcal{S} (V(\A_{\Q}^f)^n)$ be a $K$-invariant Schwartz function on $V(\A_{\Q}^f)^n$.
For $\beta$ as above, with $\Omega_{\beta} (F) \neq \emptyset$, fix $\mathbf{x}_0 \in \Omega_{\beta} (\A_{\Q}^f)$ and write 
\begin{equation}
(\mathrm{supp} \ \varphi) \cap \Omega_{\beta} (\A_{\Q}^f) = \bigsqcup_r K \cdot \xi_r^{-1} \cdot \mathbf{x}_0
\end{equation}
for some finite set of representatives $\xi_r G(\A_{\Q}^f)$. Following \cite{Kudla} we introduce the {\it special cycle}:
\begin{equation}
Z(\beta , \varphi , K) = \sum_r \varphi (\xi_r^{-1} \cdot \mathbf{x}_0) Z(U(\mathbf{x}_0 ) , \xi_r , K).
\end{equation}

\medskip
\noindent
{\it Remark.} The cycle $Z(\beta , \varphi , K)$ is independent of the choices of $\mathbf{x}_0$ and $\xi_r$. It is is an algebraic cycle defined over the reflex field $E$. 

\medskip

Recall the definition of the $g_j$'s in $G(\A_{\Q}^f)$ in \eqref{10.4.2}. Following the proof of \cite[Proposition 5.4]{Kudla} we get:
\begin{equation} \label{Zconnected}
Z(\beta , \varphi , K) = \sum_j \sum_{\substack{\mathbf{x} \in \Omega_{\beta} (F) \\ {\rm mod} \ \Gamma_{g_j}'}} \varphi (g_j^{-1} \mathbf{x})  c(U(\mathbf{x}) , g_j , K).
\end{equation}

Write $g \mapsto \omega(g)$ for the natural action of $g \in G(\A_{\Q}^f)$ on $\mathcal{S}  (V(\A_{\Q}^f)^n)$ given by $(\omega (g) \varphi ) (x) = \varphi (g^{-1} x)$. The analogues of \cite[Propositions 5.9 and 5.10]{Kudla} then yield:
\begin{prop} \label{prop:Kudla}
1. For any $g \in G(\A_{\Q}^f)$, we have: $Z(\beta , \omega (g) \varphi , gKg^{-1}) = Z(\beta , \varphi , K) \cdot g^{-1}$.

2. Suppose that $K' \subset K$ is another open compact subgroup of $G(\A_{\Q}^f)$, and let $\mathrm{pr} : S(K') \to S(K)$ be the natural projection. Then we have: $\mathrm{pr}^* (Z(\beta , \varphi , K)) = Z(\beta , \varphi , K')$.
\end{prop}

\subsection{The ring of special cycles}
Let $\mathcal{S} (V(\A_{\Q}^f)^n)_{\Z}$ be the space of locally constant functions on $V(\A_{\Q}^f)^n$ with compact support and values in $\Z$. For any 
commutative ring $R$, let 
$$\mathcal{S} (V(\A_{\Q}^f)^n)_R = (\mathcal{S} (V(\A_{\Q}^f)^n)_{\Z}) \otimes_{\Z} R;$$
note that the natural action $\omega$ turns it into a $G(\A_{\Q}^f)$-module. It follows from Proposition \ref{prop:Kudla} that for any Hermitian $n\times n$-matrix $\beta \gg 0$ we get a $G(\A_{\Q}^f)$-equivariant map
\begin{equation}
\begin{split}
\mathcal{S}  (V(\A_{\Q}^f)^n)_{\Q} & \to H^{2qn} (\mathrm{Sh} (G, X) , \Q) \\
\varphi & \mapsto [ \beta , \varphi ] := [Z(\beta , \varphi )].
\end{split}
\end{equation}
Following Kudla \cite{Kudla}, in order to study the $G(\A_{\Q}^f)$-submodule which is the image of this map, we first extend this construction to the case where 
$\beta$ is only (totally) positive {\it semidefinite}. The expression \eqref{Zconnected} is still well defined. Denoting by $t$ the rank of $\beta$, one obtains
a class 
$$[\beta , \varphi ]^0 = [Z(\beta , \varphi)] \in H^{2qt} (\mathrm{Sh} (G, X) , \Q).$$ 
Now recall that the symmetric domain $X$ has a natural K\"ahler form $\Omega$ and that, for any compact open subgroup $K \subset G(\A_{\Q}^f)$, $\frac{1}{2\pi i} \Omega$
induces a $(1,1)$-form on $S(K)$ which is the Chern form of the canonical bundle of $S(K)$. 
The cup-product with $\Omega^q$ --- or equivalently with the Chern form $c_q$ introduced above --- induces the $q$-th power of the Lefschetz operator:
$$L^q : H^{\bullet} (\mathrm{Sh} (G, X) , \Q) \to H^{\bullet + 2q} (\mathrm{Sh} (G, X) , \Q)$$
on cohomology which commutes with the action of $G(\A_{\Q}^f)$. We then set 
\begin{equation}
[\beta , \varphi ] := L^{q(n-t)}([\beta , \varphi ]^0)  \in H^{2qn} (\mathrm{Sh} (G, X) , \Q).
\end{equation}

For each $n$ with $0 \leq 0 \leq p$, let 
$$SC^{2nq} (\mathrm{Sh} (G, X))\subset H^{2qn} (\mathrm{Sh} (G, X) , \Q)$$ 
be the subspace spanned by the classes 
$[\beta , \varphi]$, where $\beta$ is any Hermitian (totally) positive semidefinite $n\times n$-matrix ($\beta \geq 0$). The subspace 
$SC^{2nq} (\mathrm{Sh} (G, X))$ is defined over $\Q$ and Hecke stable. We therefore have a direct sum decomposition into $\pi_f$-isotypical components:
\begin{equation}
SC^{2nq} (\mathrm{Sh} (G, X)) = \bigoplus_{\pi_f \in \mathrm{Coh}_f} SC^{2nq} (\mathrm{Sh} (G, X), \pi_f).
\end{equation}
We can now state our main result on special cycles.

\begin{thm} \label{thm:MainSC}
1. The space 
$$SC^{\bullet} (\mathrm{Sh} (G, X)) = \oplus_{n} SC^{2nq} (\mathrm{Sh} (G, X))$$ 
is a subring of $H^{\bullet} (\mathrm{Sh} (G, X) , \Q)$. 

2. For each $n$ with $0 \leq n \leq p$ we have:
$$SC^{2nq} (\mathrm{Sh} (G, X)) \subset SH^{nq, nq} (\mathrm{Sh} (G, X) , \C)
\cap H^{2qn} (\mathrm{Sh} (G, X) , \Q).$$

3. If we furthermore assume that $3n < p+q$, then the subspace $SC^{2nq}_{\rm prim} (\mathrm{Sh} (G, X))$ spanned by the projection of $SC^{2nq} (\mathrm{Sh} (G, X))$ into the {\rm primitive} part 
$$SH^{n \times q , n \times q} (\mathrm{Sh} (G, X) , \C)$$ of 
$SH^{nq, nq} (\mathrm{Sh} (G, X) , \C)$ is defined over $\Q$ and we have a direct sum decomposition:
\begin{equation}
SC^{2nq}_{\rm prim} (\mathrm{Sh} (G, X)) = \bigoplus_{\substack{\pi_f \in \mathrm{Coh}_f^{n,n} \\ \forall v < \infty, \ \varepsilon (\pi_v) =1}} H^{2nq} (\pi_f , \C)\otimes \pi_f.
\end{equation}
\end{thm}

\medskip
\noindent
{\it Remark.} The proof of Theorem \ref{thm:MainSC} is based on Kudla-Millson's theory \cite{KudlaMillson1,KudlaMillson2,KudlaMillson3} that give an explicit construction of Poincar\'e dual forms to the special cycles. 
The first part of Theorem \ref{thm:MainSC} immediately follows from their theory as was already pointed out by Kudla in \cite{Kudla}. The last part is the real new part; it will follow from Theorem \ref{thm:10.10} and the results of Section \ref{sec:KMlocal}.

\medskip

Before proving Theorem \ref{thm:MainSC} --- in the next section --- we review the relevant results of the Kudla-Millson theory.

\subsection{The forms of Kudla-Millson}
Recall that in  Equation \eqref{relationbetweencocycles} of Section \ref{sec:KMlocal}   we have defined an element
$$\varphi_{nq,nq} \in \mathrm{Hom}_{K} (\wedge^{nq, nq} \mathfrak{p} , \mathbf{S} (V^n))$$
for each $n$ with $0 \leq n \leq p$.
Here $V=V_{\tau_1}$ is the completion of $V$ w.r.t. the $\tau_1$-embedding, $G=\mathrm{GU} (p,q)$ and $K$ is the stabilizer 
of a fixed base point $x_0 \in X$. Then 
$$X \cong G / K  \cong \U (p,q) / (\U ( p) \times \U (q))$$
and the space of differential forms on $X$ of type $(a,b)$ is 
$$\Omega^{a,b} (X ) \cong  \mathrm{Hom}_{K} (\wedge^{ a, b} \mathfrak{p} ,C^{\infty} (G) ).$$
Evaluation at $x_0$ therefore yields an isomorphism 
$$[ \mathbf{S} (V^n ) \otimes \Omega^{nq,nq} (X)]^G \cong \mathrm{Hom}_{K} (\wedge^{nq, nq} \mathfrak{p} , \mathbf{S} (V^n)).$$
We will abusively denote by $\varphi_{nq, nq}$ the corresponding element in $[ \mathbf{S} (V^n ) \otimes \Omega^{nq,nq} (X)]^G$.

\subsection{} For each $i>1$, the Hermitian space $V = V_{\tau_i}$ is positive definite of dimension $m=p+q$ and we set 
$$\varphi_0 (\mathbf{x}) = \exp (- \pi \mathrm{trace} \ (\mathbf{x} , \mathbf{x})).$$
Then, under the Weil representation $\omega$ of $\U (n , n )$ associated to $V$, we have
$$\omega (k' , k'') \varphi_0 = \det (k')^{m} \det (k'')^{-m} \varphi_0 \quad ((k',k'') \in \U (n) \times \U(n)).$$ 
If $\mathbf{x} \in V^n$ with $\frac12 (\mathbf{x},\mathbf{x})=\beta$, then for $g' \in \U (n,n)$ 
we set
\begin{equation} \label{eq:whittaker}
W_{\beta} (g' ) = \omega (g') \varphi_0 (\mathbf{x}).
\end{equation}

\subsection{} Now we return to the global situation. Let $n$ be an integer with $1 \leq n \leq p$. For $\varphi \in \mathcal{S} (V(\A_{\Q}^f)^n)$, we
define
\begin{equation} \label{eq:phitilde}
\phi = \varphi_{nq,nq} \otimes \big( \bigotimes_{i=2}^d \varphi_0 \big)
\otimes \varphi  \in \left[ \mathbf{S} (V(\A_{\Q})^n) \otimes \Omega^{nq,nq} (X) \right]^{G (\R)}.
\end{equation}

Let $W$ be a $2n$-dimensional vector space over $E$ equipped with a {\it split} $\iota$-skew-Hermitian form and let $G' = \mathrm{Res}_{F/\Q} \U (W)$.
It corresponds to the splitting of $W$ a polarization $\mathbb{X} + \mathbb{Y}$ of $W \otimes_E V$ with $\mathbb{X} \cong V^n$. The global group $G' (\A_{\Q})$ acts in $\mathbf{S} (V(\A_{\Q})^n)$ via the global Weil representation associated to this polarization, our fixed additive character $\psi$ of $\A / F$ and some choice character $\chi$ of $\A_{E}^{\times} / E^{\times}$ whose restrictions to $\A^{\times}$ satisfy $\chi |_{\A^{\times}} = \epsilon_{E/F}^m$. If $\varphi$ is 
$K$-invariant, then for $g' \in G' (\A_{\Q})$ and $g \in G (\A_{\Q})$ we may then form the theta function $\theta_{\psi,\chi,\phi}(g,g')$ as in \S \ref{par:1.4}. As a function of $g$ it defines a closed $(nq,nq)$-form on $S(K)$ which we abusively denote by 
$\theta_{n} (g', \varphi)$. Let $[\theta_{n} (g' , \varphi)]$ be the corresponding class 
in 
$$H^{nq,nq} (S(K) , \R) \subset H^{nq,nq} (\mathrm{Sh} (G,X) , \R).$$ 

For $g' = (g'_1 , \ldots , g_d ') \in G' (\R)=\U(n,n)^d \subset G' (\A_{\Q})$ and for $\beta\geq 0$ Hermitian in $\mathrm{Her}_n (E)$  
we set 
$$W_{\beta} (g') = \prod_{i=1}^d W_{\beta^{\tau_i}} (g_i ') .$$

The following result  is the main theorem of \cite{KudlaMillson3} --- rephrased here in the adelic language following Kudla \cite{Kudla}. It relates the 
cohomology class $[\theta_{n} (g' , \varphi)]$ to those of the algebraic cycles $Z(\beta , \varphi)$ via Fourier decomposition as in the
classical work or Hirzebruch-Zagier \cite{HZ}.

\begin{prop}\label{prop:8.8}
For $g' \in G' (\R) \subset G'(\A_{\Q})$ and $\varphi \in \mathcal{S} (V(\A_{\Q}^f)^n)$, the Fourier expansion of $g' \mapsto [\theta_n (g' , \varphi) ]$ 
is given by
$$[\theta_n (g' , \varphi) ] = \sum_{\beta \geq 0} [\beta , \varphi] W_{\beta} (g').$$
\end{prop}

\subsection{Proof of Theorem \ref{thm:MainSC} I}
We first prove Theorem \ref{thm:MainSC}(1): Let $n_1$ and $n_2$ with $0 \leq n_1 , n_2 , \leq p$ and choose $W_1$ and $W_2$ split skew-Hermitian vector
space over $E$ of dimensions $2n_1$ and $2n_2$. Write $G'_{n_i} = \mathrm{Res}_{F/ \Q} \U (W_i)$ ($i=1,2$). Given two Schwartz functions 
$\varphi_i \in \mathcal{S} (V (\A_{\Q}^f )^{n_i})$ and two Hermitian matrices $\beta_i \geq 0$ in $\mathrm{Her}_{n_i} (E)$ ($i=1,2$), we want to prove that the cup product of $[\beta_1 , \varphi_1]$ and $[\beta_2 , \varphi_2]$ belongs to $SC^{2nq} (\mathrm{Sh} (G,X))$ where $n=n_1+n_2$. It is natural to introduce 
the skew-Hermitian vector space $W=W_1 \oplus W_2$, the groupe $G_n' = \mathrm{Res}_{F / \Q} \U (W)$ and the Schwartz function $\varphi = \varphi_1 
\otimes \varphi_2 \in  \mathcal{S} (V (\A_{\Q}^f )^{n})$. We have a natural homomorphism 
$$G_{n_1} ' (\A_{\Q}) \times G_{n_2} ' ( \A_{\Q}) \stackrel{\iota}{\to} G_n '(\A_{\Q} )$$
and the local product formula (Propositions \ref{completefactorization} and \ref{valuesofcocycles}) implies that for $g_i ' \in G_{n_i} ' (\R)$ ($i=1,2$) we have:
\begin{equation} \label{eq:11.12.1}
\theta_n (\iota (g_1 ' , g_2 ') , \varphi) = \theta_{n_1} (g_1 ' , \varphi_1) \wedge \theta_{n_2} (g_2 ' , \varphi_2).
\end{equation}
Taking cohomology classes, applying Proposition \ref{prop:8.8} and comparing Fourier coefficients yields that $[\beta_1 , \varphi_1] \cup [\beta_2 , \varphi_2]$
decomposes as the sum $\sum_{\beta \geq 0} [\beta , \varphi]$ over the $\beta$'s s.t. 
$$W_{\beta} (\iota (g_1 ' , g_2 ' )) = W_{\beta_1} (g_1 ') W_{\beta_2} (g_2 ').$$
This proves the first part of Theorem \ref{thm:MainSC}. 
 
To prove the second part it is enough to prove that the cohomology classes $[\theta_n (g' , \varphi)]$ belong to $SH^{nq,nq} (\mathrm{Sh} (G,X) , \R)$. 
This in turn follows from the fact that $\varphi_{nq,nq}$ --- seen as a $(nq,nq)$-form on $X$ with values in $\mathbf{S} (V^n)$ is $\SL (q)$-invariant. But 
this can be read out from the explicit formula for $\varphi_{nq,nq}$ --- see Proposition~\ref{formulaforcocycles} and Appendix \ref{AppendixC}. 

The last part of Theorem \ref{thm:MainSC} will be deduced from our main theorem which we state and prove in the next section.

\section{Main Theorem}

\subsection{Notations} We keep notations as in the previous section except that we will now always assume that 
$a$ and $b$ are integers such that $3(a+b)+|a-b| <2m$.

\subsection{The special lift} It follows from Theorem \ref{thm:10.10} that if $\pi_f \in \mathrm{Coh}_f^{b,a}$ then 
$\pi = A(b \times q , a \times q) \otimes \pi_f$ is in the image of the $\psi$-theta correspondence from a smaller group $\U(W)$ of signature 
$(a,b)$ at infinity. In particular the whole cohomology group $H^{b\times q , a \times q} (\mathrm{Sh} (G) , \C)$ is generated by
the automorphic functions $\theta^f_{\psi , \chi , \phi}$ as in \eqref{theta} where $\phi$ and $f$ vary. Here $f$ is an automorphic function 
of $\mathrm{GU} (W)$ and $\phi$ is a Schwartz function in the space $\mathbf{S} (\mathbb{X} (\A))$ associated to a choice of a complete (global)
polarization of the symplectic space $\mathbb{W}$. We may furthermore restrict to functions $\phi$ that are decomposable as $\phi_{\infty} \otimes \phi_f$. 

Now at infinity the Schwartz space $\mathbf{S} (\mathbb{X} (F_{\infty}))$ is a model for the Weil representation. We will abuse notation and  denote by 
$\varphi_{bq,aq}$ the Schwartz function in $\mathbf{S} (\mathbb{X} (F_{\infty}))$ which  is the tensor product of  $\varphi_{bq,aq}$ of Section \ref{sec:KMlocal}, Equation \eqref{relationbetweencocycles},  at the  infinite place where the real group is noncompact  and Gaussians at the other infinite places. 
We finally denote by $H^{b\times q , a \times q} (\mathrm{Sh} (G) , \C)_{\rm special}$ the subspace of {\it special lifts} that are generated by
the projections in $H^{b\times q , a \times q} (\mathrm{Sh} (G) , \C)$ of 
the automorphic functions $\theta^f_{\psi , \chi , \varphi_{bq,aq} \otimes \phi_f}$ as $\phi_f$, $\chi$ and $f$ vary.

We now prove that special lifts span the whole refined Hodge type $a \times q, b \times q$ in the cohomology of $\mathrm{Sh} (G)$ .

\begin{thm} \label{StepTwo}
We have:
$$H^{b\times q , a \times q} (\mathrm{Sh} (G) , \C)_{\rm special} = H^{b\times q , a \times q} (\mathrm{Sh} (G) , \C).$$
\end{thm}
\begin{proof} The proof follows the same lines as that of \cite[Theorem 10.5]{BMM}. First recall the following simple general observation: Suppose $K$ is a group and we have $K$-modules $A,B,U,V$. Suppose further that we have $H$-module homomorphisms $\Phi:U \to V$ and $\Psi: B \to A$.  Then we have a commutative diagram
\begin{equation}\label{genprin}
\begin{CD}
\mathrm{Hom}_K(A,U) @>\Phi_{*}>> \mathrm{Hom}_K(A,V) \\
@V\Psi^*VV                          @VV\Psi^*V \\
\mathrm{Hom}_K (B,U)  @>\Phi_{*}>> \mathrm{Hom}_K (B,V)
\end{CD}
\end{equation} 
Here $\Phi_*$ is postcomposition with $\Phi$ and $\Psi^*$ is precomposition with $\Psi$.

In what follows  $K$ will be the group $K_{\infty}$. We now define $K_{\infty}$-module homomorphisms $\Phi$ and $\Psi$ that will concern us here.
We begin with the $(\g,K_{\infty})$-module homomorphism $\Phi$.  Let $\pi_f \in \mathrm{Coh}_f^{b,a}$. Denote by $H_{\pi_f}$ the $\pi_f$-isotypical 
subspace in $L^2 (G , \chi (\pi_f ))$ and let 
$$H = \bigoplus_{\pi_f \in \mathrm{Coh}^{a,b}_f} \mathrm{m} (A(b \times q , a \times q) \otimes \pi_f) H_{\pi_f} .$$ 
Recall that we have:
\begin{equation}
\begin{split}
H^{b\times q , a \times q} (\mathrm{Sh} (G) , \C) & \cong H^{(a+b)q} (\g , K_{\infty} ; H) \\
& \cong \mathrm{Hom}_{K_{\infty}} (V(b,a) , H) .
\end{split}
\end{equation}

Theorem \ref{thm:10.10} implies that each automorphic representation $\pi = A(b \times q , a \times q) \otimes \pi_f$, with $\pi_f \in \mathrm{Coh}^{b,a}_f$ 
is in the image of the $\psi$-theta correspondence from a smaller
group $\U (W)$ of signature $(a,b)$. We realize the oscillator representation as a $(\g, K_{\infty}) \times G(\A_f)$-module 
in the subspace 
$$\mathbf{S} (\mathbb{X}(F_{v_0})) \times \mathcal{S} (\mathbb{X} (\A_f )) \subset \mathbf{S} (\mathbb{X}(\A)).$$
Here the inclusion maps an element $(\varphi_{\infty} , \varphi)$ of the right-hand side to 
$$\phi = \varphi_{\infty} \otimes \big( \bigotimes_{\substack{v | \infty \\ v \neq v_0}} \varphi_0 \big)
\otimes \varphi$$
where the factors $\varphi_0$ at the infinite places $v$ not equal to $v_0$
denotes the unique element (up to scalar multiples) that is fixed by the compact group $\U (m)$. We abusively write elements of $\mathbf{S} (\mathbb{X}(F_{v_0})) \times \mathcal{S} (\mathbb{X} (\A_f ))$ as $\varphi_{\infty} \otimes \varphi$.

Fix some pair of characters $\chi$ as in \S \ref{par:1.4} and let $H'$ be the direct sum of the spaces of cuspidal automorphic representations 
$\pi'$ of $\mathrm{GU} (W)$ s.t.
$\Theta_{\psi, \chi , W}^V (\pi' ) = A(b \times q , a \times q) \otimes \pi_f$ for some $\pi_f \in \mathrm{Coh}^{b,a}_f$. From now on we abbreviate
$$\mathbf{S} =\mathbf{S} (\mathbb{X}(F_{v_0})) \times \mathcal{S} (\mathbb{X} (\A_f )).$$
It follows from the definition of the global theta lift (see \S \ref{par:1.4}) that for any 
$f \in H'$  and   $\phi \in \mathbf{S}$  
the map $f \otimes \phi \mapsto \theta_{\psi , \chi,  \phi}^f$ is a 
$(\g , K_{\infty})$-module 
homomorphism from $H'\otimes \mathbf{S}$ to the space $H$. And Theorem \ref{thm:10.10} implies that the images of these as $\chi$ and $\psi$ vary
span the whole space $H$. We will drop the dependence of $\psi$ and $\chi$ henceforth and abbreviate this map to $\theta$ whence 
$f\otimes \phi \mapsto \theta (f \otimes \phi)$.  Then in the diagram \eqref{genprin}
we take $U= H'\otimes \mathbf{S}$ and $V= H$ and $\Phi =\theta$.

We now define the map $\Psi$.  We will take $A$ as above to be the vector space  $\wedge^{nq} \mathfrak{p}$, the vector space 
$B$ to be the submodule $V(b,a)$ and $\Psi$ to be the inclusion $i_{b,a} : V(b,a) \to \wedge^{nq}(\p)$ (note that there is a unique embedding up to scalars and the scalars are not important here). 

From the general diagram \eqref{genprin} we obtain the desired  commutative diagram
\begin{equation} \label{diagramofthetalift}
\begin{CD}
H' \otimes \mathrm{Hom}_{K_{\infty}}(\wedge^{nq}(\p), \mathbf{S}) @>\theta_*>> \mathrm{Hom}_{K_{\infty}}(\wedge^{nq}(\p), H )  \\
@Vi_{b,a}^*VV       @VVi_{b,a}^*V\\
H'\otimes \mathrm{Hom}_{K_{\infty}}(V(b,a), \mathbf{S}) @>\theta_* >> \mathrm{Hom}_{K_{\infty}}(V(b,a), H )
=H^{b \times q , a \times q}(\mathrm{Sh}(G), \C)
\end{CD}
\end{equation}
Now we examine the diagram.  Since $V(b,a)$ is a summand, the map on the left is {\it onto}.  Also by \eqref{VZKtype} the map on the right
is an isomorphism. 

We now define $U_{\varphi_{bq,aq}}$ to be 
the affine subspace of $\mathrm{Hom}_{K_{\infty}}(\wedge^{nq}(\p), \mathbf{S} )$ defined by $\varphi_{\infty} = \varphi_{bq,aq}$. 

The theorem will then follow from the equation
\begin{equation}\label{firstequivversionforthm95}
i_{b,a}^* \circ \theta_* (H' \otimes  U_{\varphi_{bq,aq}}) = 
\mathrm{Image} (i_{b,a}^* \circ \theta_*).
\end{equation}
Since the above diagram is commutative,  equation \eqref{firstequivversionforthm95} holds if and only if we have
\begin{equation}\label{secondequivversionforthm95}
\theta_* \circ   i_{b,a}^*(H' \otimes  U_{\varphi_{bq,aq}}) = \mathrm{Image} (i_{b,a}^* \circ \theta_*).
\end{equation}
Put $\overline{U}_{\varphi_{bq,aq}}    =i_{b,a}^*(U_{\varphi_{bq,aq}})$.  
Since the left-hand vertical arrow $i_{b,a}^*$ is onto, equation \eqref{secondequivversionforthm95} holds if and only if 
\begin{equation}\label{thirdequivversionforthm95}
\theta_*(H' \otimes \overline{U}_{\varphi_{bq,aq}}) = \theta_*(H' \otimes \mathrm{Hom}_{K_{\infty}}(V(b,a), \mathbf{S})).
\end{equation} 

We now prove equation \eqref{thirdequivversionforthm95}. 
 
To this end let $\xi \in  \theta_*(H' \otimes \mathrm{Hom}_{K_{\infty}}(V(b,a), \mathbf{S}))$.
Hence, by definition, there exists 
 $\phi = \varphi_{\infty} \otimes \varphi \in \mathrm{Hom}_{K_{\infty}} (V(b,a) , \mathbf{S})$
and $f \in H'$ such that 
\begin{equation}\label{hitxi}
 \theta_*( f \otimes \phi) = \xi.
\end{equation}

We claim that in equation \eqref{hitxi} (up to replacing the component  $f_{v_0}$) we may replace the  factor  $\varphi_{\infty}$ of $\phi$ by $\varphi_{bq,aq}$
 without changing the right-hand side $\xi$ of equation \eqref{hitxi}.  
Indeed by Theorem \ref{KMgenerates} 
there exists $Z \in \mathcal{U}(\mathfrak{u}(b,a)_{\C})$ such that
\begin{equation}\label{everythingcomesfromKM}
 \varphi_{\infty} = Z \varphi_{bq,aq}.
\end{equation}
Now by \cite[Lemma 6.9]{HoffmanHe} (with slightly changed notations)  we have
\begin{equation}
\theta_*( f \otimes Z\phi) = \theta_*( Z^*f \otimes  \phi).
\end{equation}

Hence setting $f' = Z^*f$ we obtain, for all $f \in H'$,
\begin{equation}\label{movedZover}
\begin{split}
\xi  & = \theta_* (f \otimes (\varphi_{\infty} \otimes  \varphi)) \\ 
& =  \theta_* (f \otimes (Z\varphi_{bq,aq} \otimes  \varphi)) \\
& = \theta_* (Z^*f \otimes  (\varphi_{bq,aq} \otimes  \varphi)) \\
& = \theta_*( f' \otimes (\varphi_{bq,aq} \otimes  \varphi)) .
\end{split}
\end{equation}
We conclude that 
the image of the space  $H' \otimes U_{\varphi_{bq,aq}}$ under $\theta_*$ coincides with the image 
of $H' \otimes \mathrm{Hom}_{K_{\infty}}  (V_{bq,aq}, \mathbf{S})$ as required. 
\end{proof}

We now prove the last part of Theorem \ref{thm:MainSC}.

\begin{prop} \label{SC=T}
Let $n$ be an integer s.t. $3n <m$. We then have a direct sum decomposition: 
$$SC^{2nq}_{\rm prim} (\mathrm{Sh} (G, X)) = \bigoplus_{\substack{\pi_f \in \mathrm{Coh}_f^{n,n} \\ \forall v < \infty, \ \varepsilon (\pi_v) =1}} H^{2nq} (\pi_f , \C)\otimes \pi_f.$$
\end{prop}
\begin{proof} Since $3n<m$, it follows from Theorem \ref{thm:10.10} that if $\pi_f  \in \mathrm{Coh}_f^{n,n}$ the automorphic representation $\pi = A(q^n , q^n) \otimes \pi_f$ is in the image of the $\psi$-theta correspondence from a smaller group $\U (W)$ of signature $(n,n)$ at infinity. By local theta dichotomy the global Hermitian space is completely determined by $\pi$; it is split if and only if $\varepsilon (\pi_v) =1$ for every finite place $v$. By Theorem \ref{StepTwo}, in the statement of Proposition \ref{SC=T} we may therefore replace 
the RHS by the subspace of $H^{n \times q , n \times q} (\mathrm{Sh} (G) , \C )$ generated by the projections of the cohomology classes $[\theta_n (g' , \varphi)]$
where $\varphi \in \mathcal{S} (V(\A_{\Q}^f)^n)$ and $g' \in G' (\A_{\Q})$. Let's denote by $\mathcal{H}$ this subspace.

Denote by $\langle , \rangle$ the Petersson scalar product restricted to the primitive part of $H^{2nq} (\mathrm{Sh} (G) , \C)$. Letting $SC_{\rm prim}^{2nq}(\mathrm{Sh}(G), \C )^{\perp}$ and $\mathcal{H}^{\perp}$ denote the respective
annihilators in $H^{2nq}_{\rm prim} (\mathrm{Sh} (G) , \C )$ it suffices to prove
\begin{equation} \label{inclusofperps}
SC_{\rm prim}^{2nq}(\mathrm{Sh}(G), \C )^{\perp} = \mathcal{H}^{\perp}.
\end{equation}
Now consider $\eta \in H^{2nq}_{\rm prim} (\mathrm{Sh} (G) , \C )$. Assume $\eta$ is $K$-invariant for some level $K$. 
It follows from proposition \ref{prop:8.8} that for $g' \in G' ( \R) \subset G' (\A_{\Q})$ the Fourier
expansion of the modular form 
$$\theta_{\varphi} (\eta) := \langle [\theta_{n} (g' , \varphi )] , \eta \rangle  \ \left( =  \int_{X_K} \theta_{n} (g' , \varphi ) \wedge * \eta \right)$$ 
is given by
\begin{equation*}
\begin{split}
\theta_{\varphi} ( \eta) & = \sum_{\beta \geq 0} \langle [\beta , \varphi] , \eta \rangle W_{\beta} (g') \\
& = \sum_{\beta \gg 0} \langle [\beta , \varphi] , \eta \rangle W_{\beta} (g'),
\end{split}
\end{equation*}
since $\eta$ is primitive. 
In particular: the form $\eta$ is orthogonal to $SC^{nq}_{\rm prim} (\mathrm{Sh} (G) , \C)$  if and only if all the Fourier 
coefficients of all the modular forms $\theta_{\varphi} ( \eta)$ vanish and therefore $\theta_{\varphi}
(\eta )=0$. 

On the other hand, it follows from a general principle --- see e.g. Lemma \ref{L:cusp} below --- that for each $\pi_f  \in \mathrm{Coh}_f^{n,n}$ the automorphic representation $\Theta_{\psi, \chi , V}^W ( A(n \times q , n \times q) \otimes \pi_f)$ is a {\rm cuspidal} automorphic representation of $\U (W)$. In particular: each modular form $\theta_{\varphi} (\eta)$ is a cuspidal 
automorphic form (which may be trivial). Now $\eta$ belongs to $\mathcal{H}^{\perp}$ if and only if for any $f\in H_{\sigma'}$ ($\sigma ' \in \mathcal{A}^c (\U (W))$), we have:
\begin{equation*}
\int_{X_K}  \theta (f , \varphi)[\varphi_{nq, nq}] \wedge * \eta  = 0.
\end{equation*}
Since 
$$\int_{X_K}  \theta (f , \varphi)[\varphi_{nq, nq}] \wedge * \eta  = \int_{[\U (W)]} \theta_{\varphi } (\eta) f(g') dg',$$
we get that $\eta$ belongs to $\mathcal{H}^{\perp}$ if and only if $\theta_{\varphi} (\eta )=0$. This concludes the proof.
\end{proof}

We can now prove our main theorem.

\begin{thm} \label{Thm:main9}
Assume that $b=a+c$ with $c>0$. Then the natural cup-product map
$$SC^{2aq} (\mathrm{Sh}(G) , \C) \times H^{cq,0} (\mathrm{Sh}(G) , \C) \rightarrow H^{b \times q , a \times q} (\mathrm{Sh}(G) , \C )$$
is surjective. If $c=0$ this is no longer true but the natural cup-product map
$$SC^{2(b-1)q} (\mathrm{Sh}(G) , \C) \times H^{q,q} (\mathrm{Sh}(G) , \C) \rightarrow H^{b \times q , b \times q} (\mathrm{Sh}(G) , \C )$$
is surjective.
\end{thm}
\begin{proof} Write $b=a+c$ with $c\geq 0$. It follows from Theorem \ref{StepTwo} that $H^{b \times q , a \times q} (\mathrm{Sh}(G) , \C )$
is generated by the automorphic functions $\theta_{\psi , \chi , \varphi_{bq,aq} \otimes \phi_f}^f$ as $\phi_f$, $\chi$ and $f$ vary. 
These data depend on a global Hermitian space $W$ of signature $(a,b)$ at infinity. Recall from \S \ref{par:5.10} that if $v$ is a finite place of 
$F$ the local Hermitian space $W=W_v$ decomposes as a sum 
\begin{equation} \label{Wdec}
W = \left\{ 
\begin{split} 
& W_{a,a} \oplus W_c \mbox{ if } c>0 \\
& W_{a-1 , a-1} \oplus W_2 \mbox{ otherwise},
\end{split}
\right.
\end{equation}
where $W_{r,r}$ denotes the Hermitian of dimension $2r$ with maximal isotropic subspaces of dimension $r$ and $W_k$ denotes the unique 
Hermitian space of dimension $k$ with the same local sign as $W$.\footnote{At infinite places one should ask that $W_c$ is positive definite in case $c>0$ and
that $W_2$ is of signature $(1,1)$ in case $c=0$.}
Since the $W_v$ are localizations of a global space, there exists a corresponding global $W_c$, or $W_2$ in case $c=0$, and the 
corresponding decomposition \eqref{Wdec} holds globally. Write $G_1' = \mathrm{Res}_{F/\Q} \mathrm{GU} (W_{b,b})$, resp. $G_1 ' =  \mathrm{Res}_{F/\Q} \mathrm{GU} (W_{b-1,b-1})$, and 
$G_2 ' = \mathrm{Res}_{F/\Q} \mathrm{GU} (W_c)$, resp. $G_2 ' = \mathrm{Res}_{F/\Q} \mathrm{GU} (W_2)$. We have a natural homomorphism
$$G_1' (\A_{\Q}) \times G_2 ' (\A_{\Q}) \stackrel{\iota}{\to} G' (\A_{\Q}).$$
We may now decompose the space $\mathbb{W} = W \otimes_E V$ accordingly:
\begin{equation} \label{Wdec2}
\mathbb{W} = \left\{ 
\begin{split} 
& (W_{a,a} \otimes_E V) \oplus \mathbb{W}_c  \mbox{ if } c>0 \\
& (W_{a-1 , a-1} \otimes_E V) \oplus \mathbb{W}_2  \mbox{ otherwise}.
\end{split}
\right.
\end{equation}
Here $\mathbb{W}_k = W_k \otimes_E V$. We finally choose a complete polarization 
$$\mathbb{W}_k = \mathbb{X}_k + \mathbb{Y}_k$$
and consider the associated polarization $\mathbb{X} + \mathbb{Y}$ of $\mathbb{W}$ where 
\begin{equation} \label{Xdec}
\mathbb{X} \cong \left\{ 
\begin{split} 
& V^a \oplus \mathbb{X}_c  \mbox{ if } c>0 \\
& V^{a-1} \oplus \mathbb{X}_2  \mbox{ otherwise}.
\end{split}
\right.
\end{equation}
The polynomial Fock space $\mathbf{S} (\mathbb{X} (\A))$ then contains $\mathbf{S} (V(\A)^a) \otimes \mathbf{S} (\mathbb{X}_c (\A))$, resp.
$\mathbf{S} (V(\A)^{a-1}) \otimes \mathbf{S} (\mathbb{X}_2 (\A))$, as a dense subspace. Now the local product formula at infinity (Propositions 
\ref{completefactorization} and \ref{valuesofcocycles}) decomposes
$\varphi_{bq,aq}$ as a cup-product $\varphi_{aq,aq} \wedge \varphi_{cq,0}$, resp. $\varphi_{(a-1)q, (a-1)q} \wedge \varphi_{q,q}$. As in the proof of Theorem \ref{thm:MainSC}(1) (see in particular \eqref{eq:11.12.1}) we conclude that if $\phi_f = \phi_{1f} \otimes \phi_{2f}$ belongs to $\mathbf{S} (V(\A)^a) \otimes \mathbf{S} (\mathbb{X}_c (\A))$, resp.
$\mathbf{S} (V(\A)^{a-1}) \otimes \mathbf{S} (\mathbb{X}_2 (\A))$, we have the following cup-product of differential forms:
\begin{equation} \label{diffcup}
\theta_{\psi , \chi , \varphi_{bq,aq} \otimes \phi_f } (\iota (g_1 ' , g_2 ' ) , \cdot) = \left\{ 
\begin{split} 
& \theta_a (g_1' , \phi_{1f} ) \wedge \theta_{\psi , \chi , \varphi_{cq,0}\otimes \phi_{2f}} (g_2 ' , \cdot)   \mbox{ if } c>0 \\
& \theta_{a-1} (g_1' , \phi_{1f} ) \wedge \theta_{\psi , \chi , \varphi_{q,q}\otimes \phi_{2f}} (g_2 ' , \cdot)  \mbox{ otherwise}.
\end{split}
\right.
\end{equation}

Recall that $H^{b \times q , a \times q} (\mathrm{Sh}(G) , \C )$
is generated by the projections of the cohomology classes of the differential forms $\theta_{\psi , \chi , \phi} (g' , \cdot)$.
Since $\theta_{\psi , \chi , \omega (g_1') \phi} (g' , \cdot) = \theta_{\psi , \chi , \phi} (g_1' g' , \cdot)$ and the space 
$\mathbf{S} (V(\A)^a) \otimes \mathbf{S} (\mathbb{X}_c (\A))$, resp.
$\mathbf{S} (V(\A)^{a-1}) \otimes \mathbf{S} (\mathbb{X}_2 (\A))$, is a dense subspace of $\mathbf{S} (\mathbb{X} (\A))$, we get that
$H^{b \times q , a \times q} (\mathrm{Sh}(G) , \C )$ is generated by the projections of the cohomology classes of the differential forms \eqref{diffcup}.
Theorem \ref{Thm:main9} now follows from Proposition \ref{SC=T}.
\end{proof}

\newpage

\part{Automorphic forms}

\section{On the global theta correspondence} \label{Sec:1}

\subsection{} In this Part of the paper we prove Theorem \ref{thm:10.10}. We keep notations as in Section 3. The main technical point is to prove that if $\pi  \in \mathcal{A}^c (\U (V))$
is such that its {\it local} component at infinity is ``sufficiently non-tempered'' (this has to be made precise) then the {\it global} representation $\pi$ is in the image of the cuspidal $\psi$-theta correspondence from a smaller group. 

As usual we encode local components of $\pi$ into an $L$-function. In fact we only consider its 
{\it partial $L$-function} $L^S (s, \pi)= \prod_{v \not\in S} L(s , \pi_v)$ where $S$ is a sufficiently big finite set of places such that 
$\pi_v$ is unramified for each $v \not\in S$. For such a $v$ we define the local factor $L (s , \pi_v)$
by considering the Langlands parameter of $\pi_v$. 

We may generalize these definitions to form the partial (Rankin-Selberg) $L$-functions $L^S (s , \pi \times \eta)$ for
any automorphic character $\eta$.

The goal of this section and the next is to prove the following theorem; it is a first important step toward the proof that a ``sufficiently non-tempered'' 
automorphic representation is in the image of the cuspidal $\psi$-theta correspondence from a smaller group. 
A large part of it is not new. The proof follows the pioneering work of Kudla and Rallis for the usual orthogonal-symplectic dual pair
(see also \cite{Moeglin97a}, \cite[Theorem 1.1 (1)]{GJS} and generalized in \cite{GanTakeda}). And most of the steps that appear to be different 
in the unitary case can now be found in the literature; in that respect one key ingredient is due to Ichino \cite{Ichino}. 

\begin{thm} \label{Thm:GJS}
Let $\pi \in  \mathcal{A}^c (\U (V))$ and let $\eta$ be a character of $ \A_E^{\times} / E^{\times} $. 
We assume that there exists some integer $a>1$ such that the partial $L$-function $L^S (s , \pi \times \eta)$ 
is holomorphic in the half-plane $\mathrm{Re} (s) > \frac12 (a-1)$
and has a pole in $s= \frac12 (a-1)$. 

Then, we have $\eta |_{\A^{\times}} = \epsilon_{E/F}^{m-a}$ and there exists some $n$-dimensional skew-Hermitian space over $E$, with $n=m-a$, such that 
$\pi$ is in the image of the cuspidal $\psi$-theta correspondence from the group $\U (W)$.
\end{thm}

\noindent
{\it Remark.} It will follow from the proof that letting $\chi_2 = \eta$ and fixing some arbitrary choice of character $\chi_1$ such that $\chi_1 |_{{\Bbb A}^{\times}} = \epsilon_{E/F}^m$, there exists a representation $\pi ' \in \mathcal{A}^c (\U (W))$ such that 
$$\pi = \Theta_{\psi , \chi , W}^V (\pi ').$$
Note that remark \ref{rem:2.4} implies that changing our choice of $\chi_1$ amounts to twist $\pi'$ by a character of $\A_E^1$. Similarly, replacing
$\psi$ by any other additive character $\psi_t (x) = \psi (tx)$ for some $t \in F$, amounts to rescale the Hermitian space $W$.

\subsection{Strategy of proof of Theorem \ref{Thm:GJS}}  It is based on Rallis' inner product formula.  
Let $\pi$ be as in Theorem \ref{Thm:GJS}. We want to construct a $n$-dimensional skew-Hermitian space $W$ over $E$ and some
pair of characters $\chi = (\chi_1 , \chi_2)$ of $\A_E^{\times} / E^{\times}$ (as in \S \ref{par:2.2}) such that 
\begin{equation} \label{eq:nonvan1}
\Theta_{\psi , \chi , V}^W (\pi ) \neq 0.
\end{equation}
Theorem \ref{Thm:GJS} will then follow by duality. Such a $W$ will be of the form $W=W_r= W_0 \oplus \H^r$ where
$\H$ denotes the hyperbolic plane, i.e. the split skew-Hermitian space of dimension $2$ and $W_0$ is {\it anisotropic}.
According to this decomposition we write $n=n_0+2r$. The family of spaces $\{ W_r \; : \; r \geq 0 \}$ forms a Witt 
tower of non-degenerate skew-Hermitian spaces. Replacing $W=W_r$ by $W_{r-1}$ we may assume that 
$\Theta_{\psi , \chi , V}^{W_{r-1}} (\pi ) =0$ so that $\theta_{\psi, \chi ,\phi}^f$ is a cusp form (possibly zero).
To prove \eqref{eq:nonvan1} we shall compute the square of its Petersson norm
\begin{eqnarray} \label{inner}
||\theta_{\psi, \chi ,\phi}^f ||^2 = \int_{[\U_n ]} \theta_{\psi, \chi ,\phi}^f (g') \overline{\theta_{\psi, \chi ,\phi}^f (g')} dg'
\end{eqnarray}
using Rallis' inner product formula. In our range $n \leq m$ the latter takes the rough form:
$$||\theta_{\psi, \chi ,\phi}^f ||^2 = c  \mathrm{Res}_{s = \frac{m-n}{2}} L^S (s + \frac12 , \pi \times \eta ),$$
where $c$ is a non-zero constant involving local coefficients of the oscillator representation. We will now provide the details of the proof.

We first recall the {\it doubling method} introduced by Piatetskii-Shapiro and Rallis \cite{PSRallis} in order to relate $L$-functions as the function $L^S (s , \pi \times \eta)$ in Theorem \ref{Thm:GJS} to the Weil representation. 

\subsection{Doubling the group}
Equip $V \oplus V$ with the split form $( , ) \oplus - ( , )$. Let $\U(V \oplus V)$ 
be the corresponding isometry group. We denote
by $V$ the subspace $V \oplus \{0 \}$ and by $-V$ the subspace $\{0 \} \oplus V$. There is a canonical embedding:
$$\j : \U(V) \times \U(-V) \rightarrow \U(V \oplus V) .$$

Let $P$ be the {\it Siegel parabolic} of $\U(V \oplus V)$ preserving the maximal isotropic subspace 
$$V^d = \{ (v,v) \; : \; v \in V \} \subset V \oplus V.$$
Let $P=MU$ be the Levi 
decomposition of $P$. Here $U$ is the unipotent radical of $P$, and $M$ is the subgroup which
preserves both $V^d$ and
$$V_d  = \{ (v,-v) \; : \; v \in V \} \subset V \oplus V .$$
Thus $M(F)$ is isomorphic to ${\rm GL}_n (E)$ via restriction to $V_d \cong E^n$; 
we write $\mu (a) \in M$ for the element corresponding to 
$a \in \GL_n (E)$. 
Note that 
$$P \cap (\U(V) \times \U(-V)) = \U(V)^d$$
where $\U(V)^d$ is the image of $\U(V)$ under the diagonal embedding $\U(V) \rightarrow \U(V) \times \U(-V)$.

\subsection{Siegel Eisenstein series}
Define the homomorphism $\det : M(F) \rightarrow E^{\times}$ by 
$\mu (a) \mapsto \det a$ and denote by $|\cdot |_E$ the norm map $E^{\times} 
\rightarrow F^{\times}$. Then $\det$ and $|\det |_E$ uniquely extend to
homomorphisms $\det : M({\Bbb A}) \rightarrow {\Bbb A}_E^{\times}$ and 
$|\det |_E : M({\Bbb A}) \rightarrow {\Bbb A}^{\times}$. Given a character $\eta$ of $E^{\times} \backslash \A_E^{\times}$ 
we define the quasi-character $\eta |\cdot|^s$ ($s \in {\Bbb C}$) of $M({\Bbb A})$ by
$$\mu \mapsto \eta (\det (\mu)) | \det (\mu) |_E^{s}.$$
 
Then define the induced representation
\begin{eqnarray} \label{repind}
I (s, \eta ) = {\rm Ind}_{P({\Bbb A})}^{\U_{m,m} (\A)}  \eta |\cdot|^s
\end{eqnarray}
where we denote by $\U_{m,m} (\A)$ the adelic points of $\U (V \oplus V)$ and the induction is normalized. 
(Note that the modular character $\delta : P \rightarrow F^{\times}$
is equal to $\mu \mapsto |\det (\mu) |_E^{m}$.)
Starting with a section $\Phi (\cdot , s) \in I(s, \eta )$ we may form the
Eisenstein series
\begin{eqnarray} \label{ES}
E(h,s, \Phi ) = \sum_{\gamma \in P(F) \backslash \U (V \oplus V)} \Phi (\gamma h , s).
\end{eqnarray}
The group $\U (V \oplus V)$ being quasi-split, we may fix a standard maximal compact subgroup $K= \U (V \oplus V) (\mathcal{O}_F)$ of $\U_{m,m} (\A)$. We then 
say that $\Phi (\cdot , s) \in I(s, \eta )$ is {\it standard} if it is holomorphic in $s$ and its restriction to $K$ is independent
of $s$.  It follows from \cite[p. 56]{Weil} that if $\Phi (\cdot , s)$ is holomorphic in $s$ the Siegel Eisenstein series (\ref{ES}) converges absolutely for ${\rm Re} (s) > \frac{m}{2}$ and has a meromorphic continuation to the whole $s$-plane. If $\Phi (\cdot, s)$ is standard, Tan \cite{Tan} proves that $E(h,s, \Phi)$ is entire in the 
half-plane $\mathrm{Re} (s) \geq 0$ if $\eta \neq \bar \eta^{-1}$. If $\eta = \bar \eta^{-1}$, then $\eta_{| \A^*} = \epsilon_{E/F}^i$ with $i=0$ or $1$ and the poles of $E(h,s, \Phi)$ in the half-plane
$\mathrm{Re} (s) \geq 0$ are at most simple and occur at the points 
$$s \in \left\{ \frac{m-i}{2}  , \frac{m-i}{2} -1 , \ldots  \right\}.$$

If $s_0$ is a pole of $E(h,s, \Phi)$ we may consider the Laurent expansion 
$$E(h , s, \Phi ) = \frac{A (h , \Phi )}{s-s_0} + A_0 (h , \Phi ) + O (s-s_0),$$
where $A(h, \Phi) = \mathrm{Res}_{s=s_0} E(h , s, \Phi )$. 
The function of $h \mapsto  A(h , \Phi)$ defines an automorphic form on $\U (V \oplus V)$. 

\subsection{The global doubling integral} Consider now an irreducible automorphic cuspidal representation $\pi \in \mathcal{A}^c (\U (V))$.
Denote by $\pi^{\vee}$ the contragredient representation of $\pi$ and write 
$\eta$ for the character $\eta \circ \det$ of $\U_m (\A)$, identified with the adelic points of the subgroup $\U(V)^d$ of $P$.
Let $H_{\pi}$ (resp. $H_{\pi^{\vee}}$) be the space of $\pi$ (resp. $\pi^{\vee}$). For
$f \in H_{\pi}$ and $f^{\vee} \in H_{\pi^{\vee}}$, let 
$$\phi (g) = \langle \pi (g) f , f^{\vee} \rangle$$
where $\langle, \rangle$ is the standard pairing. 
Given a section $\Phi(\cdot ,s) \in I(s,\eta )$, define
\begin{eqnarray} \label{Z}
Z(s, f , f^{\vee},  \Phi ) = \int_{\U_m (\A)}  \phi (g) \Phi (\j (g,1) , s) dg.
\end{eqnarray}
This integral converges absolutely for ${\rm Re} (s)
> m$.

The Basic Identity of \cite[p. 3]{PSRallis} relates (\ref{Z}) to some Rankin-Selberg integral:
\begin{multline} \label{BI}
Z(s,f,f^{\vee}, \Phi) 
 = \int_{[\U_m  \times \U_m]} E(\j (g_1, g_2 ) , s, \Phi ) f (g_1 ) f^{\vee} (g_2 ) \eta^{-1} (g_2 ) dg_1 dg_2 .
\end{multline}

The point here is that:
\begin{itemize}
\item If $\Phi (\cdot , s) = \otimes_v \Phi_v (\cdot , s)$, $f = \otimes f_{v}$ and $f^{\vee} = \otimes f_{v}^{\vee}$ are factorizable, then the zeta integral (\ref{Z}) equals the product 
$$\prod_v Z ( s, f_{v} , f_{v}^{\vee} ,  \Phi_v )  $$
where, letting $\phi_v (g) = \langle \pi_v (g) f_{v} , f_{v}^{\vee} \rangle$, 
\begin{eqnarray} \label{Zloc}
Z ( s, f_{v} , f_{v}^{\vee}, \eta_v  , \Phi_v )   = \int_{\U (V\otimes_F F_v )} \phi_v (g) \Phi_v (\j (g,1) ,s ) dg
\end{eqnarray}
is the local zeta integral.
\item The Eisentein series (\ref{ES}) admit meromorphic continuation to the entire complex plane.
\end{itemize}
All this thus leads to the meromorphic continuation of some Euler products. We now study
the local zeta integrals (\ref{Zloc}). 

\subsection{Local zeta integrals} \label{par:local}
Let $S$ be the finite set of ramified places (those $v$ which are either
infinite or finite and either $\eta_v$ is ramified or $\pi_v$ not spherical). Assume that $\Phi$ is standard at each finite place $v \notin S$.
Let $v$ be a finite place of $F$ outside $S$ and $q$ be the residue characteristic of $F$. 
In this paragraph we deal with the local field $F_v$ but we 
drop the subscript $v$ everywhere. Thus $F$ is a 
non-Archimedean local field of characteristic $0$, $E$ is a quadratic extension of $F$,
$V$ is a $m$-dimensional vector space over $E$ endowed with a non-degenerate Hermitian
form $(,)$ and $\pi$ is an irreducible unitary representation
of the group $G=\U (V )$. It may also be that $E=F\oplus F$ is a split extension; then $G\cong \GL(m,F)$ and
(see \cite{PSRallis}) 
the computations below still work and are equivalent to the computations made by Godement and
Jacquet \cite{GJ}. 

Denote by $H$ the double of $G$ so that $\j: G \times G \hookrightarrow H$. For $s \in {\Bbb C}$ and for any character $\eta$ of 
$E^{\times}$, let $I(s,\eta )$ be the so-called {\it degenerate principal series} consisting of smooth functions on $H$ which satisfy 
\begin{eqnarray} \label{section}
\Phi (u\mu h,s ) = \eta (\det (\mu) ) |\det (\mu)|_E^{s+ \frac{m}{2}} \Phi (h,s) , 
\end{eqnarray}
where $u \in U$ and $\mu \in M$. 

The local zeta integral is defined as follows. For $f \in \pi$ and $f^{\vee} \in \pi^{\vee}$ let $\phi (g) = \langle \pi (g) f , 
f^{\vee} \rangle$ be the corresponding matrix coefficient. For a section $\Phi (\cdot , s) \in I(s, \eta)$, define
\begin{eqnarray} \label{Zloc2}
Z(s, f , f^{\vee}, \Phi ) = \int_G \phi (g) \Phi (\j (g,1) , s) dg.
\end{eqnarray}
The integral converges for large ${\rm Re} (s)$ and defines a non-zero element 
\begin{eqnarray} \label{Zs}
Z(s) \in {\rm Hom}_{G\times G} ( I(s,\eta ) , \pi^{\vee} \otimes (\eta \cdot \pi )).
\end{eqnarray}

We may identify the group $H$ as the subgroup of
$\GL(2m,E)$ which preserves the Hermitian form with matrix $\left(
\begin{array}{cc}
0 & 1_m \\
 1_m & 0 
\end{array} \right)$. Let then $K$ be the maximal compact subgroup of $H$ obtained by intersecting 
this group with $\GL(2m,{\cal O}_E)$, where ${\cal O}_E$ is the ring of integers of $E$. We 
let $\Phi^0 (\cdot , s)$ be the standard section whose restriction to $K$ is $1$.
Recall (see \cite{Li}) that if $G$ is the split group and $\pi$ is an unramified principal series, and if $\eta$ is unramified, 
and $f$ and $f^{\vee}$ are non-trivial and $K$-fixed, then, up to a non-zero constant, $\phi (g)$
is the zonal spherical constant associated to $\pi$ and the local zeta integral
\begin{eqnarray} \label{Znr}
Z(s, f , f^{\vee}, \Phi^0 ) = \frac{L(s+ \frac12 , \pi \times \eta )}{b(s, \eta)},
\end{eqnarray}
where $L(s, \pi \times \eta)$ is the usual Langlands Euler factor associated to $\pi$ , $\eta$ and 
the standard representation of the $L$-group ${}^L G$ and 
$$b(s, \eta ) = \prod_{i=0}^{m-1} L(2s+m-i , \eta^0 \epsilon_{E/F}^i ),$$
where $\eta^0$ is the restriction of $\eta$ to $F^{\times}$.
Note that in $b(s , \eta )$ each $L$-factor is a {\it local} (Tate) factor. 

\subsection{} From now on we fix $f_v$, $f_v^{\vee}$ and $\Phi_v = \Phi_v^0$ as above for each $v \notin S$. We conclude that:
\begin{equation} \label{Z1}
Z(s, f , f^{\vee}, \Phi ) = \frac{1}{b^S (s, \eta )} L^S (s + \frac12 , \pi \times \eta ) \prod_{v \in S} Z (s , f_v , f_v^{\vee} , \Phi_v),
\end{equation}
where $b^S (s, \eta)$ is the product of the local factors $b (s , \eta_v)$ over the set of finite places $v \notin S$.

We remark that the proof of  \cite[Proposition 7.2.1]{KR} --- which 
generalizes immediately to the unitary case --- implies that
for any point $s \in \C$ there exist choices of $f$, $f^{\vee}$ and $\Phi$ such that the local
zeta integral $Z_v(s, f , f^{\vee} , \Phi)$ is non zero. 

Now recall that we have
\begin{multline} \label{Z2}
b^S (s , \eta)  Z(s, f , f^{\vee}, \Phi ) \\ = \int_{[\U_m \times \U_m]} E^*(\j (g_1, g_2 ) , s, \Phi ) f (g_1 ) f^{\vee} (g_2 ) \eta^{-1} (g_2 ) dg_1 dg_2 ,
\end{multline}
where $E^*(h , s, \Phi ) = b^S (s , \eta) E(h , s, \Phi )$ is the {\it normalized Eisenstein series}. 

We conclude from \eqref{Z1} and \eqref{Z2} that any pole of $L^S (s + \frac12 , \pi \times \eta )$ must be a pole of $b^S (s , \eta)  Z(s, f , f^{\vee}, \Phi )$ for a suitable choice of $\Phi$ and hence also a pole of the normalized Eisenstein series $E^*(h , s, \Phi )$. Since $b^S (s , \eta)$ doesn't vanish in the half-plane $\mathrm{Re} (s) \geq 0$, we conclude:

\begin{prop} \label{P:Tan}
\begin{enumerate}
\item If $\eta \neq \bar \eta^{-1}$, the partial $L$-function $L^S (s + \frac12 , \pi \times \eta )$ is entire in the half-plane $\mathrm{Re} (s) \geq 0$.
\item If $\eta = \bar \eta^{-1}$ then $\eta = \epsilon_{E/F}^i$ with $i=0$ or $1$. Then the partial $L$-function $L^S (s + \frac12 , \pi \times \eta )$ 
has at most simple poles in the half-plane $\mathrm{Re} (s) \geq 0$ and these can only occur for 
$$s \in \left\{ \frac{m-i}{2}  , \frac{m-i}{2} -1 , \ldots \right\}.$$
\end{enumerate}
\end{prop} 

We now explain the relation of the doubling method with the Weil representation.

\subsection{Doubling the Weil representation} Consider a $n$-dimensional vector space $W$ over $E$ equipped with a $\tau$-skew-Hermitian $\langle , \rangle$, 
as in \S \ref{par:2.2}. The space
${\Bbb W}+{\Bbb W}= W \otimes_E (V \oplus V) = {\Bbb W} \oplus {\Bbb W}$ is then endowed with the symplectic
form $[, ] \oplus -[ , ]$. We denote by $\Omega$ the Weil
representation of $\Mp_{4nm} (\A)$ --- the group of adelic points of ${\rm Mp} ({\Bbb W}+{\Bbb W} )$ --- corresponding to the same character 
$\psi$ of ${\Bbb A} /F$. The obvious embedding 
$$i: {\rm Sp}_{2nm} (\A) \times {\rm Sp}_{2nm} (\A ) \rightarrow \Sp_{4nm} (\A)$$
leads to a homomorphism 
$$\tilde{i}: \Mp_{2nm} (\A) \times \Mp_{2nm} (\A ) \rightarrow \Mp_{4nm} (\A)$$
such that
$$\Omega \circ \tilde{i} = \omega \otimes \omega^{\vee}$$
where $\omega^{\vee}$ is the contragredient representation of $\omega$, and is the same as $\omega_{\overline{\psi}}$ (the Weil representation associated to $\overline{\psi}$). Clearly ${\Bbb W} +{\Bbb W}
= ({\Bbb X} \oplus {\Bbb X}) + ({\Bbb Y} \oplus {\Bbb Y})$ is a complete polarization of 
${\Bbb W} +{\Bbb W}$ and $\Omega$ can be realized on $\mathcal{S}(({\Bbb X} \oplus {\Bbb X})({\Bbb A}))$.

The choice of $\chi$ also defines a homomorphism
\begin{eqnarray} \label{j}
\U_n ({\Bbb A}) \times \U_{2m} ({\Bbb A}) \rightarrow \Mp_{4nm} (\A)
\end{eqnarray}
lifting the natural map 
$$\U_n ({\Bbb A}) \times \U_{2m} ({\Bbb A}) \rightarrow \Sp_{4nm} (\A),$$
and so, we obtain a representation $\Omega_{\chi}$ of $\U_n ({\Bbb A}) \times \U_{2m} ({\Bbb A})$
on $\mathcal{S}(({\Bbb X} \oplus {\Bbb X})({\Bbb A}))$. We will rather work 
with another model of the Weil representation which we now describe.

\subsection{} Set 
$${\Bbb W}^d = W \otimes_E V^d  \mbox{ and } {\Bbb W}_d = W \otimes_E V_d.$$
Then ${\Bbb W}+{\Bbb W} = {\Bbb W}_d + {\Bbb W}^d$ is another complete polarization of 
${\Bbb W}+{\Bbb W}$. Thus $\Omega$ can also be realized on $L^2 ({\Bbb W}_d ({\Bbb A}))$. We denote by $\Omega^-$ this realization.
There is an isometry 
$$\delta : L^2 (({\Bbb X} + {\Bbb X})({\Bbb A})) \rightarrow L^2 ({\Bbb W}_d ({\Bbb A}))$$
intertwinning the action of $\Mp_{4nm} (\A)$ on the two spaces. 
Explicitly $\delta$ is given as follows. We identify ${\Bbb W}_d$ and ${\Bbb W}$ via the map
$$(w,-w) \mapsto w$$
and write $w\in {\Bbb W}$ as $w=(x,y)$ according to the decomposition ${\Bbb W}={\Bbb X}+{\Bbb Y}$.
Then for $\Psi \in L^2  (({\Bbb X} + {\Bbb X})({\Bbb A}))$ we have 
\begin{eqnarray} \label{delta}
\delta (\Psi ) (w) = \int_{{\Bbb X} ({\Bbb A})} \psi (2 [ u,y ] ) \Psi (u+x, u-x) du .
\end{eqnarray}

\subsection{Ichino sections}
It follows from \cite{Kudla2} that the action of $M({\Bbb A})$ on $\varphi \in \mathcal{S}({\Bbb W}_d ({\Bbb A}))$
is given by 
\begin{eqnarray} \label{actionLevi}
(\Omega_{\chi}^- (\mu) \varphi ) (\omega ) = \chi_2 (\det (\mu)) |\det (\mu) |^{\frac{n}{2}} \varphi (\mu^{-1} \omega).
\end{eqnarray}
In particular, setting
\begin{eqnarray} \label{Phi}
\Phi (h ) = (\Omega_{\chi}^- (h) \varphi ) (0), 
\end{eqnarray}
we have 
\begin{eqnarray} \label{Phisection}
\Phi \in I(s_0 , \chi_2 ),
\end{eqnarray}
with $s_0=\frac{n - m}{2}$. In general a holomorphic section $\Phi (\cdot , s)$ is said to be associated to $\varphi$ if it is
holomorphic in $s$ and $\Phi (h, s_0) = (\Omega_{\chi}^- (h) \varphi ) (0)$. We call {\it Ichino sections} the sections which are associated 
to some $\varphi \in \mathcal{S}({\Bbb W}_d ({\Bbb A}))$.

\section{Rallis' inner product formula and the proof of Theorem \ref{Thm:GJS}}

We first want to construct a $n$-dimensional skew-Hermitian space $W$ over $E$ and some
pair of characters $\chi = (\chi_1 , \chi_2)$ of $\A_E^{\times} / E^{\times}$ (as in \S \ref{par:2.2}) such that 
\begin{equation} \label{eq:nonvan}
\Theta_{\psi , \chi , V}^W (\pi ) \neq 0.
\end{equation}

To prove \eqref{eq:nonvan} we shall compute the Petersson product
\begin{eqnarray} \label{inner}
\langle \theta_{\psi, \chi ,\phi_1}^{f_1} , \theta_{\psi , \chi , \phi_2}^{f_2} \rangle = \int_{[\U_n ]} \theta_{\psi, \chi ,\phi_1}^{f_1} (g') \overline{\theta_{\psi, \chi ,\phi_2}^{f_2} (g')} dg',
\end{eqnarray}
where $f_1 , f_2 \in \mathcal{H}_{\pi}$, using Rallis' inner product formula.

\subsection{Rallis' inner product formula}
We work with the dual pair $(\U(V\oplus V) ,\U (W))$ and the associated Weil representation $\Omega_{\chi}$. 
The corresponding theta distribution is 
$$\Theta (\phi ) = \sum_{\xi , \eta \in {\Bbb X} (F)} \phi (\xi , \eta )$$
and it follows from \cite[\S 1]{HKS} that, as a representation of $\U_m (\A ) \times \U_m (\A) \subset \U_{2m} (\A)$,
$$\Omega_{\chi} \circ \j = \omega_{\chi} \otimes \left( \chi_2 \cdot \omega_{\chi}^{\vee} \right),$$
where $\j$ is the obvious embedding $\U_m (\A ) \times \U_m (\A) \subset \U_{2m} (\A)$. 

With notation as in (\ref{inner}), we have $\phi_1 \otimes \overline{\phi}_2 \in \mathcal{S}(({\Bbb X} \oplus {\Bbb X})({\Bbb A}))$ and a simple {\it formal} 
calculation gives the right hand side of (\ref{inner}) as 
\begin{multline} \label{inner2}
\int_{[\U_m \times \U_m]} f_1(g_1) \chi_2^{-1} (g_2 ) \overline{f_2(g_2)}  
\left( \int_{[\U_n ]} \Theta_{\psi , \chi , \phi_1 \otimes \overline{\phi}_2} (\j(g_1 , g_2) , g' )  dg' \right)
dg_1 dg_2,
\end{multline}
where $\Theta_{\psi , \chi , \phi_1 \otimes \overline{\phi}_2}$ is defined by the obvious analogue of (\ref{theta}). Unfortunately the inner theta integral
$$\int_{[\U_n ]} \Theta_{\psi , \chi , \phi_1 \otimes \overline{\phi}_2} (\j(g_1 , g_2) , g' )  dg'$$
diverges in general. Following Kudla and Rallis \cite{KR} we will regularize this integral. But here again we shall rather work 
with the model $(\Omega_{\chi}^- , \mathcal{S}({\Bbb W}_d ({\Bbb A})))$.

It follows from \eqref{delta} that we have:
\begin{eqnarray} \label{delta0}
\delta (\phi_1 \otimes \overline{\phi}_2) (0) = \langle \phi_1 , \phi_2 \rangle,
\end{eqnarray}
where $\langle , \rangle$ denotes the scalar product in the Hilbert space $L^2 ({\Bbb X} ({\Bbb A}))$. 

Now on the space $\mathcal{S}({\Bbb W}_d ({\Bbb A}))$ the theta distribution takes the form
$$\Theta^- (\varphi ) = \sum_{\xi \in {\Bbb W}_d (F)} \varphi (\xi ).$$
Setting 
\begin{eqnarray} \label{deltaphi}
\varphi = \delta (\phi_1 \otimes \overline{\phi}_2)
\end{eqnarray}
we therefore see that (\ref{inner2}) is equal to
\begin{multline} \label{inner3}
\int_{[\U_m \times \U_m]} f_1(g_1) \chi_2^{-1} (g_2 )\overline{f_2(g_2)}
\left( \int_{[\U_n ]} \Theta^-_{\psi , \chi , \varphi} (\j(g_1 , g_2) , g' )  dg' \right)
dg_1 dg_2.
\end{multline}
Here again the inner theta integral
$$\int_{[\U_n ]} \Theta^-_{\psi , \chi , \varphi} (\j(g_1 , g_2) , g' )  dg' $$
diverges in general but we now describe how Kudla and Rallis \cite{KR} and Ichino \cite{Ichino} introduce a regularization of it.

\subsection{The regularized theta integral} 
Suppose $m \geq n$. It then follows from \cite[\S 2]{Ichino} that for a given finite place $v$ of $F$, there is an element $\alpha$ of the spherical Hecke algebras of 
$\U_{2m} (F_v)$ such that, for every $\varphi \in \mathcal{S}({\Bbb W}_d ({\Bbb A}))$ and $h \in 
\U_{2m} (\A)$, the theta function $\Theta^-_{\psi , \chi , \Omega^- (\alpha) \varphi} (h , \cdot )$ is rapidly decreasing on $\U (W) \backslash \U_n (\A)$.\footnote{Note that $\alpha$ acts on $\varphi$ through the Weil representation.}
It follows that the theta integral
$$\int_{[\U_n]} \Theta^-_{\psi , \chi , \Omega^- (\alpha ) \varphi} (h , g' ) dg'$$
is absolutely convergent. Now if $\varphi$ is such that the original theta-integral 
$$\int_{[\U_n]} \Theta^-_{\psi , \chi , \varphi} (h , g' ) dg'$$
is absolutely convergent, Ichino \cite[Lemma 2.3 and 2.2]{Ichino} proves that there exists some non-zero constant $c_{\alpha}$ --- independent of $\varphi$ --- such that
$$\int_{[\U_n]} \Theta^-_{\psi , \chi , \Omega^- (\alpha ) \varphi} (h , g' ) dg' = c_{\alpha} \int_{[\U_n]} \Theta^-_{\psi , \chi , \varphi} (h , g' ) dg'.$$
The {\it regularized theta integral} is then defined to be the function
$$B(g, \varphi) = \frac{1}{c_{\alpha}} \int_{[\U_n ]} \Theta^-_{\psi , \chi , \Omega^- (\alpha ) \varphi} (h , g') dg'.$$
It is well-defined and independent of the choices of $v$ and $\alpha$, see \cite[\S 2]{Ichino}. 
The function $h \mapsto B (h , \varphi)$ defines an automorphic form on $\U (V\oplus V)$ and the linear map 
$$\mathcal{S}({\Bbb W}_d ({\Bbb A})) \to \mathcal{A} (\U (V \oplus V)); \quad \varphi \mapsto B (\cdot  , \varphi)$$
is $\U_{2m} (\A)$-equivariant. 

\medskip
\noindent
{\it Remark.} By the Howe duality \cite{Howe,MVW} the spherical Hecke algebras of $\U_{2m} (F_v)$ and $\U_n (F_v)$ generate the same algebra of operators
on $\mathcal{S} ({\Bbb W}_d (F_v))$ through the Weil representation $\Omega^-$. It follows from \cite[\S 2]{Ichino} that one may associate to $\alpha$ an 
element $\alpha '$ of the spherical Hecke algebra of $\U_n (F_v)$ such that $\Omega^- (\alpha) = \Omega^- (\alpha ')$ and
\begin{equation}
\alpha ' \cdot \mathbf{1} = c_{\alpha} \cdot \mathbf{1},
\end{equation}
where $\mathbf{1}$ is the trivial representation of $\U_n (F_v)$.

\begin{prop} \label{P:Ripf}
Let $\psi$, $\chi$, $\phi$ and $f$ as above. Then we have:
$$\langle \theta_{\psi, \chi ,\phi_1}^{f_1} , \theta_{\psi , \chi , \phi_2}^{f_2} \rangle = \int_{[\U_m \times \U_m]} f_1 (g_1) \chi_2^{-1} (g_2 )\overline{f_2 (g_2)} 
B (\j (g_1 , g_2) , \varphi) dg_1 dg_2,$$
where $\varphi$ is given by \eqref{deltaphi}.
\end{prop}
\begin{proof} We have:
\begin{multline*}
\int_{[\U_m \times \U_m]} f_1 (g_1) \chi_2^{-1} (g_2 )\overline{f_2 (g_2)} B(\j (g_1 , g_2) , \varphi) dg_1 dg_2 \\
\begin{split}
& = \frac{1}{c_{\alpha}} \int_{[\U_m \times \U_m]} f_1 (g_1) \chi_2^{-1} (g_2 )\overline{f_2 (g_2)}  \int_{[ \U_n]} \Theta^-_{\psi , \chi , \Omega^- (\alpha) \varphi} (\j(g_1 , g_2) , g') dg_1 dg_2 dg' \\
& =  \frac{1}{c_{\alpha}}  \int_{[\U_n]}  \left( \int_{[\U_m \times \U_m]} f_1 (g_1) \chi_2^{-1} (g_2 )\overline{f_2 (g_2)} 
\Theta^-_{\psi , \chi , \Omega^- (\alpha ) \varphi} (\j(g_1 , g_2) , g') dg_1 dg_2 \right) \mathrm{1} (g') dg' \\
& =  \frac{1}{c_{\alpha}}  \int_{[\U_n]}  \left( \int_{[\U_m \times \U_m]} f_1 (g_1) \chi_2^{-1} (g_2 )\overline{f_2 (g_2)} 
\Theta^-_{\psi , \chi , \varphi} (\j(g_1 , g_2) , g') dg_1 dg_2 \right) (\alpha ' \cdot \mathrm{1})  (g') dg' \\
& =  \int_{[ \U_n ]}  \left( \int_{[\U_m \times \U_m]} f_1 (g_1) \chi_2^{-1} (g_2 )\overline{f_2 (g_2)} 
\Theta^-_{\psi , \chi , \varphi} (\j(g_1 , g_2) , g') dg_1 dg_2 \right) dg' \\
&=  \int_{[\U_n ]} \theta_{\psi, \chi ,\phi_1}^{f_1} (g') \overline{\theta_{\psi, \chi ,\phi_2}^{f_2} (g')} dg',
\end{split}
\end{multline*}
where we have used the remark above.
\end{proof}

\subsection{Ichino's regularized Siegel-Weil formula} The proof of Theorem \ref{Thm:GJS} will follow from the following reformulation of the main result
of \cite{Ichino}. We provide the details of the proof in Appendix B. 

\begin{thm} \label{Thm:Ichino}
Let $s_0$ a positive real number $< \frac{m}{2}$ and let $\Phi (\cdot , s)$ be an holomorphic section of $I(s , \chi_2)$ such that the Siegel Eisenstein series
$E (h , s, \Phi)$ has a simple pole in $s=s_0$ whose residue $A (\cdot , \Phi)$ generate an irreducible automorphic representation. Then there exists a global 
skew-Hermitian space $W$ over $E$ of dimension $n=m-2s_0$ and a function $\varphi \in \mathcal{S}({\Bbb W}_d ({\Bbb A}))$ such that
$$A (h , \Phi ) (=\mathrm{Res}_{s=s_0} E (h,s, \Phi)) = c B(h , \varphi),$$
where $c$ is a non-zero explicit constant. 
\end{thm}

Assuming Theorem \ref{Thm:Ichino} we can now prove Theorem \ref{Thm:GJS}.

\subsection{Proof of Theorem \ref{Thm:GJS}}
Let $s_0$ be a pole of the partial $L$-function $L^S (s + \frac12 , \pi \times \eta )$ where $\eta$ is some character of $E^{\times} \backslash \A_E^{\times}$. 
It follows from Proposition \ref{P:Tan} that $\eta = \bar \eta^{-1}$ and
$s_0$ is a half-integer $\leq m/2$. We set $n=m-2s_0$ and suppose that $n$ is positive. Note that we necessarily have $\eta |_{\A^{\times}} = \epsilon_{E/F}^n$. We set $\chi_2 = \eta$ and fix some arbitrary choice of character $\chi_1$ such that $\chi_1 |_{\A^{\times}} = \epsilon_{E/F}^m$.

Since $s_0$ is a pole of the partial $L$-function $L^S (s + \frac12 , \pi \times \chi_2 )$ there are choices of $f \in \pi$, 
$f^{\vee} \in \pi^{\vee}$ and of a holomorphic section $\Phi (\cdot , s)$ of $I (s, \chi_2)$ such that 
\begin{equation} \label{eq:nv}
\int_{[\U_m \times \U_m ]} A(\j (g_1 , g_2) , \Phi) f (g_1) f^{\vee} (g_2) \chi_2^{-1} (g_2) dg_1 dg_2 \neq 0.
\end{equation}
In particular the Siegel Eisenstein series
$E (h , s, \Phi)$ has a simple pole in $s=s_0$. We may moreover modify $\Phi$ so that $A (\cdot , \Phi)$ generate an irreducible automorphic representation
and \eqref{eq:nv} is still non-zero. Theorem \ref{Thm:Ichino} then implies that there exists a global 
skew-Hermitian space $W$ over $E$ of dimension $n=m-2s_0$ and a function $\varphi \in \mathcal{S}({\Bbb W}_d ({\Bbb A}))$ such that
\begin{equation} \label{eq:SWI}
A (h , \Phi ) = c B(h , \varphi),
\end{equation}
where $c$ is a non-zero explicit constant. Set $f_1 = f$ and let $f_2 \in \mathcal{H}_{\pi}$ be the element corresponding to the conjugate of $f^{\vee} \in \mathcal{H}_{\pi^{\vee}}$. Writing $\varphi$ as a linear combination of $\delta (\phi_1 \otimes \bar \phi_2)$, it follows from Proposition \ref{P:Ripf} and equations \eqref{eq:nv} and \eqref{eq:SWI} 
that there exist elements $\phi_1, \phi_2 \in \mathcal{S} ({\Bbb X} (\A))$ such that 
$$\langle \theta_{\psi, \chi ,\phi_1}^{f_1} , \theta_{\psi , \chi , \phi_2}^{f_2} \rangle  = \int_{[\U_m \times \U_m]} f_1(g_1) \chi_2^{-1} (g_2 )\overline{f_2(g_2)} 
B (i (g_1 , g_2) , \varphi) dg_1 dg_2$$
is non-zero. This proves that $\Theta_{\psi , \chi , V}^{W} (\pi ) \neq 0$.

\begin{lem} \label{L:cusp}
The representation $\Theta_{\psi , \chi , V}^{W} (\pi )$ is a {\rm cuspidal} automorphic representation of $\U (W)$.
\end{lem}
\begin{proof} Let $W'$ be the smallest element of the Witt tower $\{ W_r \}$ such that 
$\Theta_{\psi , \chi , V}^{W'} (\pi ) \neq 0$. By the Rallis theta tower property \cite{Rallis}
$\Theta_{\psi , \chi , V}^{W'} (\pi )$ is a {\rm cuspidal} automorphic representation of $\U (W)$. Let $n'$ be the dimension of $W'$. We have
$n' \leq n$ and, by inverting the above arguments, we get that $L^S (s + \frac12 , \pi \times \eta )$ has a pole in $s = \frac12 (m-n')$. Since by hypothesis
$L^S (s , \pi \times \eta )$ is holomorphic in the half-plane $\mathrm{Re} (s) > \frac12 (a-1)$ where $a=m-n$ we conclude that $n' \geq n$. The lemma follows.
\end{proof}

\subsection{} We can now conclude the proof of Theorem \ref{Thm:GJS}: the non-vanishing of 
$$\langle \theta_{\psi, \chi ,\phi_1}^{f_1} , \theta_{\psi , \chi , \phi_2}^{f_2} \rangle$$
implies the non-vanishing of
\begin{multline} \label{eq:revers}
\int_{[\U_n]}  \overline{\theta_{\psi , \chi , \phi_2}^{f_2}(g')}  \left( \int_{[\U_m]}  \theta_{\psi , \chi , \phi_1} (g,g') f_1 (g) dg \right) dg' \\ = \int_{[\U_n]} \theta_{\psi , \chi , \phi_1}^{f_1} (g') \overline{\theta_{\psi , \chi , \phi_2}^{f_2} (g')} dg'.
\end{multline}
But it follows from Lemma \ref{L:cusp} that $g' \mapsto \overline{\theta_{\psi , \chi , \phi_2}^{f_2}(g')}$ is rapidly decreasing and that we may therefore invert the order of integration
in the LHS of \eqref{eq:revers}. Setting $F(g')=  \overline{\theta_{\psi , \chi , \phi_2}^{f_2}(g')}$ we get that
$$\theta^F_{\psi , \chi , \phi_1} (g) = \int_{[\U_n]} F(g') \theta_{\psi , \chi , \phi_1} (g,g') dg'$$
is non-zero and is not orthogonal to the space of $\pi$. Since the image of $\Theta_{\psi , \chi , W}^V$ is an automorphic subrepresentation of 
$L^2 ([\U_m])$ this proves that $\pi$ is in the image of the cuspidal $\psi$-theta correspondence from the group $\U (W)$.

\section{Weak Arthur theory}

In this section we recall a small part of Arthur's work on the endoscopic classification of automorphic representations of classical groups. This
will be used in the following section to verify the hypotheses of Theorem \ref{Thm:GJS} in our (geometric) cases.

\subsection{Notations} \label{notations2}
Let $E$ be a CM-field with totally real maximal subfield $F$ be a number field, $\A$ the ring of adeles of $F$. We will always assume that
a specified embedding of $E$ into the algebraic closure of $F$ has been fixed. We set $\Gamma_F  = {\rm Gal} (\overline{{\mathbb Q}} / F)$. Let $V$ be a nondegenerate Hermitian vector space over $E$ with $\dim_E V =m$. 

We let $G$ be an inner form of $\mathrm{U}_{E/F} (m)$ the quasi-split unitary group over $F$, whose group of $F$-points is given by 
$$\mathrm{U}_{E/F} (m) (F) = \{ g \in \GL_m (E) \; : \; {}^t \- \bar{g} J g = J \}.$$ 
Here $J$ is the anti-diagonal matrix 
$$J= \left(
\begin{array}{ccc}
0 & & 1 \\
 & \adots & \\
1 &  & 0 
\end{array} \right)$$ 
and $z \mapsto \bar{z}$ is the Galois conjugation of $E/F$. In this section we simply denote by $\mathrm{U} (m)$ the unitary group $\mathrm{U}_{E/F} (m)$.
We finally denote by $\GL(m)$ the $F$-algebraic group $\mathrm{Res}_{E/F} (\GL_m {}_{| E})$. We will identify automorphic representations of $\GL (m)$ with automorphic representations of $\GL_m (\A_E )$.

\subsection{$L$-groups} The (complex) dual group of $\mathrm{U}(m)$ is 
$$\mathrm{U}(m)^{\vee} = \GL_m ( \C)$$ 
and the $L$-group of $\mathrm{U}(m)$ is the semi-direct product 
$${}^L \- \mathrm{U} (m) = \GL_m (\C) \rtimes \Gamma_F,$$
where the action of $\Gamma_F$ factors through $\mathrm{Gal} (E/F)$ and the action of the non-trivial element $\sigma \in \mathrm{Gal} (E/F)$ is given by
$$\sigma (g) = \Phi_m {}^t \- g^{-1} \Phi_m^{-1} \quad (g \in \GL_m (\C))$$
where $\Phi_m$ is the anti-diagonal matrix with alternating $\pm 1$ entries:
$$\Phi_m = \left(
\begin{array}{ccc}
0 & & 1 \\
 & \adots & \\
(-1)^{m-1} &  & 0 
\end{array} \right).$$ 
Note that $\Phi_m^2 = (-1)^{m-1}$ so that $\sigma$ is of order $2$; it furthermore fixes the standard splitting of $\GL_m (\C)$.

Now the (complex) dual group of $\GL(m,E)$, seen as a group over $F$, is 
$$\GL (m)^{\vee} = \GL_m (\C) \times \GL_m (\C)$$
and the $L$-group of $\GL (m)$ is the semi-direct product 
$${}^L \- \GL (m) = (\GL_m (\C) \times \GL_m (\C)) \rtimes \Gamma_F$$
where now $\sigma$ acts by
$$\sigma ( g , g') = (g' , g) \quad ( g , g' \in \GL_m (\C )).$$

\subsection{Representations induced from square integrable automorphic representations} \label{par:6.3}
By \cite{MW} one may parametrize the
discrete automorphic spectrum of $\GL (m)$ by a set of formal tensor products
$$\Psi = \mu \boxtimes R,$$
where $\mu$ is an irreducible, unitary, cuspidal automorphic representation of $\GL (d)$ and $R$
is an irreducible representation of $\SL_2 (\C)$ of dimension $n$, for positive integers $d$ and $n$ such that $m=dn$: for any such $\Psi$, 
we form the induced representation
$$\mathrm{ind} (\mu | \cdot |_{\A_E}^{\frac12 (n-1)} , \mu | \cdot |_{\A_E}^{\frac12 (n-3)} , \ldots , \mu | \cdot |_{\A_E}^{\frac12 (1-n)} )$$
(normalized induction from the standard parabolic subgroup of type $(d, \ldots , d)$). We then write 
$\Pi_{\Psi}$ for the unique irreducible quotient of this representation. 

We may more generally associate a representation $\Pi_{\Psi}$ of $\GL(m)$, induced from square integrable automorphic representations, to a 
formal sum of formal tensor products
\begin{equation} \label{psi}
\Psi = \mu_1 \boxtimes R_{n_1} \boxplus \ldots \boxplus \mu_r \boxtimes R_{n_r}
\end{equation}
where each $\mu_j$ is an irreducible, unitary, cuspidal automorphic representation of $\GL (d_i) / F$, 
$R_{n_j}$ is an irreducible representation of $\SL_2 (\C)$ of dimension $n_j$ and $m=n_1d_1 + 
\ldots + n_r d_r$: to each $\mu_j \boxtimes R_{n_j}$ we associate a square integrable automorphic form $\Pi_i$ of $\GL (n_j d_j)$. We then define $\Pi_{\Psi}$ as the induced representation 
$$\mathrm{ind} (\Pi_1 \otimes \ldots \otimes \Pi_r )$$
(normalized induction from the standard parabolic subgroup of type $(n_1 d_1, \ldots , n_r d_r )$). It is an irreducible representation 
of $\GL_m (\A_E)$ because it is proved by Tadic \cite{Tadic} and Vogan \cite{Vogan} that an induced representation from a uitary irreducible one, is irreducible.

\subsection{Unramified base change}
Let $\pi=\otimes_v ' \pi_v$ be an automorphic representation of $G (\A)$ occuring in the discrete spectrum.\footnote{Note that it is always the case if $G$ is anisotropic.} For almost every finite place $v$ of $F$ both $G(F_v)$ and $\pi_v$ are unramified. Let $F_v$ be the local field associated to such a place and let $W_{F_v} '$ be its Weil-Deligne group. Then by the Satake isomorphism $\pi_v$ is associated to an $L$-parameter $\varphi_v : W_{F_v} ' \to {}^L \- \mathrm{U} (m)$, see e.g. \cite{Minguez}. 

The (fixed) embedding of $E$ into the algebraic closure of $F$ specifies $E_v = E\otimes_F F_v$. This realizes $W_{E_v}'$ as a subgroup of $W_{F_v}'$.
Restricting $\varphi_v$ to $W_{E_v} '$ one gets an $L$-parameter for an unramified irreducible representation $\Pi_v$ of $\GL_m (E_v )$ --- the {\it principal
base change} of $\pi_v$. 

\medskip
\noindent
{\it Remark.} There is another base change\footnote{It is sometimes called {\it unstable base change} but we will avoid this confusing terminology.}: Back to the global extension $E/F$, we fix a unitary character $\chi : \A_E^{\times} / E^{\times} \to \C^{\times}$ whose restriction to $\A^{\times}$ is the character $\epsilon_{E/F}$ associated, by classfield theory, to the quadratic extension $E/F$.  Outside some finite $S$ of places of $F$ the character $\chi$ is unramified. 
Let $v \notin S$, the {\it non-principal base change} for unramified representations of $G(F_v)$ is the one described above but twisted by the character $\chi_v$. The definition depends on the choice of the character but so does the transfer between functions, see e.g. \cite{Mok}. In the statement below we will solely consider the principal base change, note however that --- also we won't make this very explicit --- we have to consider both base-change in the proofs. The definition depends on the choice of the character but only up to a twist: two different choices differ by a character of $E_v^{\times}$ trivial on $F_v^{\times}$. The group of such characters of $E_v^{\times}$ is exactly the group of characters of $\U_{E_v / F_v} (m)$ by the following correspondance: let $\omega$ be a character of $E_v^{\times}$ trivial on $F_v^{\times}$ and denote by $\omega^1$ the character of the subgroup of $\U_{E_v / F_v} (m)$ defined by $\omega^1 (g) = \omega (z)$ where $z$ is any element of $E_v^{\times}$ such that $\det (g) = z/ \overline{z}$. The principal and the non principal base change commute with the twist by $\omega \circ \det$ on the $\GL (m, E_v)$ side and the twist by $\omega^1$ on the $\U_{E_v / F_v} (m)$ side.

\medskip

\subsection{Weak base change} The following proposition is essentially due to Arthur \cite[Corollary 3.4.3]{Arthur} though it is not stated for unitary groups --- see \cite[Corollary 4.3.8]{Mok} for a statement in the latter case when $G$ is quasi-split. The reduction from the general case to the quasi-split case follows the same lines\footnote{The use of the stable {\it twisted} trace formula being replaced by the (untwisted) stable trace formula.}; we provide some details in Appendix A. 

\begin{prop} \label{Prop:Arthur}
Let $\pi$ be an irreducible automorphic representation of $G (\A)$ which occurs (discretely) 
as an irreducible subspace of $L^2 (G (F) \backslash G (\A))$. Then, there exists a (unique) global 
representation $\Pi = \Pi_{\Psi}$ of $\GL(m , \A_E)$, induced from square integrable automorphic representations, associated to 
a parameter $\Psi$ as in \eqref{psi} and a finite set $S$ of places of $F$ containing all Archimedean ones 
such that for all $v\not\in S$, the representations $\pi_v$ and $\Pi_v$ are both unramified and $\Pi_v$ is the principal base change of
$\pi_v$.
\end{prop}
We will refer to $\Pi$ as the {\it weak base change} of $\pi$.

\medskip
\noindent
{\it Remark.} Proposition \ref{Prop:Arthur} puts serious limitations on the kind of non-tempered
representations that can occur discretely: e.g. an automorphic representation $\pi$ of $G (\A)$ which occurs discretely in $L^2 (G (F) \backslash G (\A))$ and which is non-tempered at one place $v \notin S$ is non-tempered at all places outside $S$. 

\medskip

The above remark explains how Arthur's theory will be used in our proof. We now want to get a global control on the automorphic representations with a prescribed type at infinity. 

\subsection{Standard representations and characters}
%\subsection{On the local Arthur packets}
We first recall the results of local harmonic analysis we will need. We therefore fix a place $v$ of $F$ and, until further notice, let $G$ denote the group of $F_v$-points of the unitary group. Denote by $\mathcal{H} (G)$ the Hecke algebra of locally constant functions of compact
support on $G$.\footnote{We will mainly deal with the situation where $v$ is Archimedean --- so that $\mathcal{H}(G) = C_c^{\infty} (G)$ ---
except in Appendice A where we have to deal with all places.}

%\subsection{Standard representations and characters}
By Langlands' classification, any admissible representation $\pi$ of $G$ can be realized as the unique irreducible sub-quotient of some {\it standard representation}. Recall --- see \cite[(3.5.5)]{Arthur} --- that the latter can be identified with $G$-orbit of triples 
$\rho = (M, \sigma , \xi )$ where $M \subset G$ is a Levi subgroup, $\sigma$ an irreducible representation of $M$ that is square 
integrable modulo the center, and $\xi$ an irreducible constituent of the induced representation $\mathrm{ind} (\sigma)$.

Both $\pi$ and $\rho$ determines real linear forms $\Lambda_{\pi}$ and $\Lambda_{\rho}$ --- the {\it exponents} --- on $\mathfrak{a}_M$; it measures the failure of the representation to be tempered. 

Following Arthur we denote by $\rho_{\pi} = (M_{\pi}, \sigma_{\pi} , \xi_{\pi})$ the standard representation corresponding to $\pi$.
We furthermore recall that the distribution character of $\pi$ has a decomposition 
$$\mathrm{trace} \ \pi (f) = \sum_{\rho} n(\pi , \rho ) \mathrm{trace} \ \rho (f) \quad (f \in \mathcal{H} (G))$$
into standard characters $\rho$, where the coefficients $n(\pi , \rho)$ are uniquely determined integers s.t. all but finitely many of them are equal to $0$ and $n(\pi , \rho_{\pi}) = 1$. If $n(\pi , \rho ) \neq 0$, then $\Lambda_{\rho} \leq \Lambda_{\pi}$ in the usual sense that $\Lambda_{\pi} - \Lambda_{\rho}$ is a nonnegative integral combinaison of simple roots of the root system associated to the inducing parabolic,\footnote{In loose terms: the Langlands quotient of 
$\rho$ is {\it more tempered} than $\pi$.} with equality $\Lambda_{\pi} = \Lambda_{\rho}$ 
if and only if $\rho = \rho_{\pi}$.

Standard characters can in turn be decomposed into virtual characters parametrized by the set $T(G)$ of $G$-orbits of triples $\tau = (M, \sigma , r)$ where $M \subset G$ is a Levi subgroup, $\sigma$ an irreducible representation of $M$ that is square 
integrable modulo the center, and $r$ is an element in the $R$-group of $\sigma$ in $G$ (see \cite[\S 3.5]{Arthur}). We refer to \cite[(3.5.3)]{Arthur} for 
the definition of the virtual character $\mathrm{trace} \ \tau$ for $\tau \in T(G)$ and simply note that if $\tau = (M , \sigma , 1)$ then $\mathrm{trace} \ \tau$ is just 
the distribution of the virtual representation $\mathrm{ind} (\sigma)$. In general it is still true that it corresponds to $\tau$ a representation whose distribution character is precisely $\mathrm{trace} \ \tau$; we will abusively denote by $\tau$ as well this representation. It follows easily from the definitions that 
\begin{equation}
\mathrm{trace} \ \pi (f) = \sum_{\tau \in T (G)} n (\pi , \tau ) \mathrm{trace} \ \tau(f) \quad (f \in \mathcal{H} (G)).
\end{equation}
Moreover: if we define the element $\tau_{\pi} \in T(G)$ associated to $\pi$ by the triplet 
$(M_{\pi} , \sigma_{\pi} , 1)$, we have $n(\pi , \tau_{\pi}) >0$. For each $\tau \in T(G)$ there is a natural notion of exponent $\Lambda_{\tau}$, see \cite[p.156]{Arthur}, we simply note that $\Lambda_{\tau_{\pi}} = \Lambda_{\rho_{\pi}} = \Lambda_{\pi}$ and that $n(\pi , \tau) \neq 0$ implies $\Lambda_{\tau} \leq \Lambda_{\pi}$. 

{\it Elliptic} elements in $T(G)$ play a special role. We refer to \cite[\S 3.5]{Arthur} for their definition; we will denote by $T_{\rm ell} (G)$ the subset of elliptic
elements in $T(G)$ and call {\it elliptic representations} the corresponding representations of $G$;  these can be regarded as basic building blocks of the set of admissible representation of $G$, see \cite{ArthurActa}.

\subsection{Archimedean packets} \label{par:6.2}

Suppose now that $v$ is a Archimedean place. Given any $\tau \in T(G)$, we write $M_{\tau}$ for a Levi subgroup such that $\tau$ lies in the image of 
$T_{\rm ell} (M_{\tau})$. Writing $\tau = (M_{\tau} , \sigma , r)$, we define 
the {\it stable standard module} associated to $\tau$ to be the virtual representation obtained as the induced module 
$\mathrm{ind}_{P_{\tau}} \mathrm{St} (\sigma)$ where $\mathrm{St} (\sigma)$ is the 
sum of the representations in the $L$-packet of $\sigma$. 

If $\pi$ is any irreducible admissible representation of $G$ we denote by $\mathrm{St} (\pi)$ the stable standard module associated to $\tau_{\pi}$. 
We will always assume that $\pi$ has regular infinitesimal character. Then, in the Grothendieck group of representations of $G$, we may write $\pi$ 
as a linear combination of the basis represented by representations induced from elliptic ones $\mathrm{ind}_Q \sigma$ and each $\sigma$ is a discrete series. 
By definition the virtual representation $\mathrm{St} (\pi)$ is therefore the sum of the standard representations $\mathrm{ind}_Q \sigma '$ where $\sigma '$ 
belongs to the same $L$-packet as $\sigma$. 

Now local principal base change associates to all the standard modules occurring in $\mathrm{St} (\pi)$ a unique standard module of $\GL_m (\C)$ that we denote by 
$\mathrm{St} (\pi )^{\rm BC}$.

We now describe this module in terms of $L$-parameters: recall that $L$-packets are in 1-1 correspondence with
admissible $L$-parameters $\varphi : W_{F_v} ' \to {}^L\mathrm{U} (m)$.

If $\pi$ is a discrete series representation of $G$ then $\mathrm{St} (\pi)$ is a linear combination of representations in the 
$L$-packet of $\pi$. Restricting $\varphi$ to $W_{E_v} '$ one gets an $L$-parameter which defines 
a standard module of $\GL_m (\C)$; this representation is precisely $\mathrm{St} (\pi)^{\rm BC}$.

In general there is a local analogue to \S \ref{par:6.3}: a standard representation of $\GL_m (\C)$ can be parametrized by a formal sum of formal tensor product \eqref{psi} where each $\mu_j$ is now a tempered
irreducible representation of $\GL(d_j, \C)$.
The other components $R_{n_j}$ remain irreducible representations of $\SL_2 (\C)$. 
To each $\mu_j \boxtimes R_{n_j}$ we associate the unique irreducible quotient $\Pi_j$ of 
$$\mathrm{ind} (\mu_j | \cdot |^{\frac12 (n_j-1)} , \mu_j | \cdot |^{\frac12 (n_j-3)} , \ldots , \mu_j | \cdot |^{\frac12 (1-n_j)} )$$
(normalized induction from the standard parabolic subgroup of type $(d_j, \ldots , d_j)$). We then define
$\Pi_{\Psi}$ as the induced representation
$$\mathrm{ind} (\Pi_1 \otimes \ldots \otimes \Pi_r )$$
(normalized induction from the standard parabolic subgroup of type $(n_1d_1 , \ldots , n_r d_r)$). It is 
irreducible and unitary. We will abusively denote by $\Psi$ the standard representation associated to $\Pi_{\Psi}$.
Now by the local Langlands correspondence, the standard module $\Psi$ can be represented as a
homomorphism 
\begin{equation} \label{Aparam}
\Psi : W_{E_v}' \times \SL_2 (\C) \rightarrow \GL_m (\C).
\end{equation}
Arthur associates to such a parameter the $L$-parameter $\varphi_{\Psi} : W_{E_v} ' \rightarrow \GL_m (\C)$
given by 
$$\varphi_{\Psi} (w) = \Psi \left( w, 
\left(
\begin{array}{cc}
|w|^{1/2} & \\
& |w|^{-1/2} 
\end{array} \right) \right).$$
And $\mathrm{St} (\pi)^{\rm BC} = \Psi$ if and only if $\varphi_{| W_{E_v} '} = \varphi_{\Psi}$.

\subsection{Weak classification}
We now come back to the global situation. Let $\pi$ be an irreducible automorphic representation of $G (\A)$ which occurs (discretely) 
as an irreducible subspace of $L^2 (G (F) \backslash G (\A))$. It follows from Proposition \ref{Prop:Arthur} that $\pi$ determines an irreducible automorphic representation $\Pi=\Pi_{\Psi}$ of $\GL_m (\A_E)$. Given an Archimedean place $v$, it is not true in general that $\mathrm{St} (\pi_v)^{\rm BC} = \Psi_v$. Our main technical result in this second part of the paper is the following theorem whose proof --- a slight refinement of the proof of Proposition \ref{Prop:Arthur} --- is postponed to Appendix~A. 

\begin{thm} \label{Thm:Arthur}
Let $\pi$ be an irreducible automorphic representation of $G (\A)$ which occurs (discretely) 
as an irreducible subspace of $L^2 (G (F) \backslash G (\A))$. Assume that for every Archimedean place $v$ the representation $\pi_v$ has regular infinitesimal character. Let  $\Pi_{\Psi}$ be the automorphic representation of $\GL (m , \A_E)$ associated to $\pi$ by weak base change. Then, for every Archimedean place $v$, the (unique) irreducible quotient of the standard $\GL_m (\C)$-module $\mathrm{St} (\pi_v)^{\rm BC}$ 
occurs as an (irreducible) sub-quotient of the local standard representation $\Psi_v$. 
\end{thm}

\medskip
\noindent
{\it Remark.} When $G$ is quasi-split Theorem \ref{Thm:Arthur} follows from Arthur's endoscopic classification of automorphic representations. This has been 
worked out by Mok \cite{Mok} in the case of (quasi-split) unitary groups. Using the stable trace formula the general case easily reduces to the quasi-split case
as in \cite[Appendix B, \S 17.15 and after]{BMM}. For the convenience of the reader we provide a more direct proof in Appendix~A.

\section{Applications} \label{sec:appl}

In this section we derive the corollaries to Theorem \ref{Thm:Arthur} that are used in the paper.

\subsection{Application to $L$-functions}  \label{par:8.3}
Let $\pi$ be an irreducible automorphic representation of $G (\A)$ which occurs (discretely) 
as an irreducible subspace of $L^2 (G (F) \backslash G (\A))$ and let $\Pi = \Pi_{\Psi}$ be the automorphic representation of $\GL (m)$  associated to $\pi$ by weak base change. Write
$$\Psi = \mu_1 \boxtimes R_{n_1} \boxplus \ldots \boxplus \mu_r \boxtimes R_{n_r}.$$
We factor each $\mu_j = \otimes_v \mu_{j,v}$ where $v$ runs over 
all places of $F$. Let $S$ be a finite set of places of $F$ containing the set $S$ of Proposition \ref{Prop:Arthur}, and all $v$ for which either one $\mu_{j,v}$ or $\pi_v$ is ramified.
We can then define the formal Euler product
$$L^S (s, \Pi_{\Psi})= \prod_{j=1}^r \prod_{v \not\in S} L_{v}(s- \frac{n_j -1}{2} , \mu_{j,v}) L_{v}(s- \frac{n_j -3}{2} , \mu_{j,v}) \ldots L _{v}(s- \frac{1-n_j}{2} , \mu_{j,v}).$$
Note that 
$L^S (s , \Pi_{\Psi})$ is the 
partial $L$-function of a very special automorphic representation of $\GL (m)$; it is the product of partial $L$-functions of the square integrable automorphic representations associated to the parameters $\mu_j \boxtimes R_{n_j}$. According to Jacquet and Shalika \cite{JacquetShalika} $L^S (s , \Pi_{\Psi})$, which is a product absolutely convergent for $ \mathrm{Re}( s)\gg 0$,  extends to a meromorphic function of $s$. It moreover
follows from Proposition \ref{Prop:Arthur} and the definition of $L^S (s , \pi)$ that:
$$L^S (s , \pi ) = L^S (s , \Pi_{\Psi}).$$

\subsection{Infinitesimal character}
It follows from Theorem \ref{Thm:Arthur} that for each infinite place $v$ the representations $\pi_v$ and $\Pi_v$ both have the same infinitesimal character. 
It is computed in the following way: Let $v_0$ be an Archimedean place of $F$. 
We may associate to $\Psi$ the parameter $\varphi_{\Psi_{v_0}} : 
\C^* \rightarrow G^{\vee} \subset \GL_m ( \C)$ given by 
$$\varphi_{\Psi_{v_0}} (z) = \Psi_{v_0} \left( z , \left(
\begin{array}{cc}
(z \overline{z})^{1/2} & \\
& (z\overline{z})^{-1/2} 
\end{array} \right) \right).$$
Being semisimple, it is conjugate into the diagonal torus 
$\{ \mathrm{diag} (x_1 , \ldots , x_m) \}$. We may therefore write $\varphi_{\Psi_{v_0}} = 
(\eta_1 , \ldots ,\eta_m)$ where each $\eta_j$ is a character $z \mapsto z^{P_j} \overline{z}^{Q_j}$. One easily checks that the vector 
$$\nu_{\Psi} = (P_1 , \ldots , P_{m}) \in \C^{m} \cong \mathrm{Lie}(T) \otimes \C$$
is uniquely defined modulo the action of the Weyl group $W= \mathfrak{S}_m$ of $G(F_{v_0})$. The following proposition
is detailed in \cite{BC1}.

\begin{prop}
The infinitesimal character of $\pi_{v_0}$ is the image of $\nu_{\Psi}$ in $\C^{m} / \mathfrak{S}_m$.
\end{prop}

From now on we assume that $\pi$ has a {\it regular}\footnote{Equivalently: the $P_j$ are all distinct.} and {\it integral} infinitesimal character at every infinite place. This forces the Archimedean localizations of
the cuspidal automorphic representations $\mu_j$ to be induced of unitary characters of type $(z/\bar{z})^{p/2}$ where $p \in \Z$. Moreover: we have
$$\frac{p}{2}+ \frac{n_j -1}{2} - \frac{m-1}{2} \in \Z.$$

We can now relate Arthur's theory to Theorem \ref{Thm:GJS}: 

\begin{prop} \label{Lem:L}
Assume that for some Archimedean place $v_0$ of $F$ the local representation $\pi_{v_0}$ is a Langlands' quotient of a standard representation with an exponent
$(z/\bar{z} )^{p/2} (z \bar{z})^{(a-1) /2}$.  
\begin{enumerate}
\item In the parameter $\Psi$ some of the factor $\mu_j \boxtimes R_{n_j}$ is such that $n_j \geq a$; in particular if $a>m/2$, the representation $\mu_j$ is a character.
\item If we assume that $\pi$ has trivial infinitesimal character and that $3a > m+ |p|$ then 
in the parameter $\Psi$ some of the factor $\mu_j \boxtimes R_{n_j}$ is such that $n_j \geq a$ and the representation $\mu_j$ is a character.
\end{enumerate}
\end{prop}
\begin{proof} It follows from Theorem \ref{Thm:Arthur} that if $\pi_{v_0}$ is a Langlands' quotient of a standard representation with an exponent 
$(z/\bar{z} )^{p/2} (z \bar{z})^{(a-1) /2}$ then $\Psi_{v_0}$ --- the associated standard representation (of $\GL_m (\C)$)  --- contains a 
character of absolute value $\geq (a-1)/2$. This forces one of the factor $\mu_j \boxtimes R_{n_j}$ in $\Psi$ to be such that $n_j \geq a$.
Since $\sum_j n_j = m$, if $a> m/2$, there can be only factor $\mu_j \boxtimes R_{n_j}$ in $\Psi$ with $n_j \geq a$ and $\mu_j$ is a character.
This proves the first part of the proposition.

We now assume that $\pi$ has trivial infinitesimal character. It follows in particular that $p$ and $m-a$ have the same parity. 
Suppose that $3a > m+ |p|$. Note that if $|p| \geq a$ then $a >m/2$ and the result follows from the first case; we will therefore assume that $|p| \leq a$. 

Now we have $a >m/3$ and, as above, this forces one of the factor $\mu_j \boxtimes R_{n_j}$ in $\Psi$ to be such that $n_j \geq a$ and either $\mu_j$ is a character or $\mu_j$ is $2$-dimensional. We only have to deal with the latter case.
Set $n=n_j$. The localization $\mu_j$ in $v_0$ is induced 
from two characters $(z/\bar{z})^{p_i/2}$, $i=1,2$, and Theorem \ref{Thm:Arthur} implies that the set
\begin{equation} \label{seg}
\left\{ \frac{a+p-1}{2}, \frac{a+p-3}{2} , \ldots ,   \frac{p+1-a}{2} \right\}
\end{equation}
is contained in one of the two sets 
\begin{equation*} \label{interv}
I_i= \left\{ \frac{n+p_i-1}{2}, \frac{n+p_i-3}{2} , \ldots ,   \frac{p_i+1-n}{2} \right\}  \quad (i=1,2),
\end{equation*}
say $i=1$. The infinitesimal character of $\pi_{v_0}$ is regular and integral, it is therefore a collection of $m$ distinct half-integers. Let $m_+$
be the number of positive entries and $m_-$ be the number of negative entries. The entries must include the (necessary disjoint) sets $I_i$ ($i=1,2$).
Now since $|p| < a$ (by the reduction already made), the set \eqref{seg} contains either $0$ or $\pm 1/2$ and consequently so does $I_1$. 
This forces $I_2$ to be totally positive or negative. 
And since \eqref{seg} contains $[(a+p)/2]$ positive elements and $[(a-p)/2]$ negative elements, we conclude that 
$$\max (m_- , m_+ ) \geq n+ \left[ \frac{a-|p|}{2} \right] \geq a+ \left[ \frac{a-|p|}{2} \right] .$$
On the other hand $\max (m_- , m_+) = [m/2] + |m_+ - m_-|$ and since $a-|p|$ and $m$ have the same parity, we end up with:
$$3a - |p| \leq m+2|m_+ - m_- |.$$
Since $m_+ = m_-$ ($\pi_{v_0}$ has trivial infinitesimal character) we get a contradiction. 
\end{proof}

\medskip
\noindent
{\it Remark.} It follows from the proof that we can replace the assumption that the infinitesimal character is that of the trivial representation by the more 
general assumption that its entries are all integral, resp. half-integral but non integral, and that it is {\it balanced}, i.e. that it contains as many positive and negative entries.

\medskip

\begin{prop} \label{Prop:8.7}
Let $\pi$ as in Proposition \ref{Lem:L} with either $a>m/2$ or with trivial infinitesimal character and $3a > 1+ |p|$. 
Then there exists a character $\eta$ of $\A_E^{\times} / E^{\times}$ and an integer $b\geq a$ such that 
the partial $L$-function $L^S (s,\eta\times \pi)$ -- here $S$ is a finite set of places which contains all the Archimedean places and
all the places where $\pi$ ramifies -- is holomorphic in the half-plane $\mathrm{Re} (s) > (b+1)/2$ and
has a simple pole in $s=(b+1)/2$.
\end{prop}
\begin{proof}
Proposition \ref{Lem:L} provides a character $\eta = \mu_j$ of $\A_E$ and an integer $b=n_j \geq a$.
Writing $L^S(s,\eta\times \pi)$ explicitly on a right half-plane of absolute convergence using the description of \S \ref{par:8.3}; 
we get a product of $L^S (s-(b-1)/2, \eta \times \eta)$ by factors $L^S (s-(b'-1)/2,\eta\times \rho)$. Our hypothesis on $a$ forces $b' < a \leq b$; so such a 
factor is non zero at $s=(b+1)/2$.
The conclusion of the proposition follows.
\end{proof}

\subsection{Langlands' parameters of cohomological representations}
The restriction to $\C^{\times}$ of the Langlands' parameter of a cohomological representation $A_{\mathfrak{q}}$ is explicitly described in 
\cite[\S 5.3]{Asterisque}. We use here the following alternate description given by Cossutta \cite{Cossutta}:

Recall that we have $\mathfrak{q} = \mathfrak{q} (X)$ with $X=(t_1 , \ldots , t_{p+q}) \in \R^{p+q}$ s.t. $t_1 \geq \ldots  \geq t_p$ 
and $t_{p+q} \geq \ldots \geq t_{p+1}$. Since $A_{\mathfrak{q}}$ only depends on the intersection $\mathfrak{u} \cap \mathfrak{p}$, we may furthermore choose $X$ such that the Levi subgroup $L$ associated to $\mathfrak{l}$ has no compact (non abelian) simple factor.

We associate to these data a parameter
$$\Psi : W_{\R} \times \SL_2 (\C) \rightarrow {}^L \- \U(m)$$
such that 
\begin{enumerate}
\item $\Psi$ factors through ${}^L \- L$, that is $\Psi : W_{\R} \times \SL_2 (\C) \stackrel{\Psi_L}{\rightarrow} {}^L \- L \rightarrow {}^L \- \U(m)$ where the last map is the canonical extension \cite[Proposition 1.3.5]{Shelstad} of the injection $L^{\vee} \subset \U (m)^{\vee}$, and
\item $\varphi_{\Psi_L}$ is the $L$-parameter of the trivial representation of $L$.
\end{enumerate}
The restriction of the parameter $\Psi$ to $\SL_2 (\C)$ therefore maps $\left( \begin{smallmatrix} 1 & 1 \\ 0 & 1 \end{smallmatrix} \right)$ to a principal unipotent element in $L^{\vee} \subset \U (m)^{\vee}$. The restriction $\varphi : \C^{\times} \to \GL_m (\C)$  of the Langlands' parameter of $A_{\mathfrak{q}}$ to $\C^{\times}$ is just given
by
$$\varphi (z ) = \Psi \left( z , \left(
\begin{array}{cc}
(z \overline{z})^{1/2} & \\
& (z\overline{z})^{-1/2} 
\end{array} \right) \right).$$

More explicitly: Let $z_1 , \ldots , z_r$ be the different values of the $(t_j)_{1 \leq j \leq m}$ and let $(p_i )_{1 \leq i \leq r}$ and $(q_i )_{1 \leq i \leq r}$ be the integers s.t.
$$(t_1 , \ldots , t_p ) = ( \underbrace{z_1 , \ldots , z_1}_{p_1 \ \mathrm{times}} , \ldots , \underbrace{z_r , \ldots , z_r}_{p_r \ \mathrm{times}})$$
and
$$(t_{p+q} , \ldots , t_{p+1} ) = ( \underbrace{z_1 , \ldots , z_1}_{q_1 \ \mathrm{times}} , \ldots , \underbrace{z_r , \ldots , z_r}_{q_r \ \mathrm{times}}).$$
We then have: 
$$L = \prod_{j=1}^r \U (p_j , q_j )$$
with  $\sum_j p_j =p$ and $\sum_j q_j =q$. Moreover: if $p_jq_j=0$ then either $p_j$ or $q_j$ is equal to $1$. We let $m_j = p_j + q_j$ ($j=0 , \ldots , r$) and 
set
$$k_j = -m_1 - \ldots - m_{j-1} +m_{j+1} + \ldots + m_{r}.$$
The canonical extension ${}^L \- L \to {}^L \- \U (m)$ of the block diagonal map $\GL_{m_1} (\C) \times \ldots \times \GL_{m_r} (\C) \to \GL_m (\C)$ then maps
$z \in \C^{\times} \subset W_{\R}$ to 
$$\left( \begin{array}{ccc}
(z / \bar{z})^{k_1 /2} I_{m_1} & & \\
& \ddots & \\
& & (z / \bar{z})^{k_r /2} I_{m_r}
\end{array} \right).$$
Now the parameter $\Psi_L$ maps $\left( \begin{smallmatrix} 1 & 1 \\ 0 & 1 \end{smallmatrix} \right)$ to a principal unipotent element in each factor of $L^{\vee} \subset \U(m)^{\vee}$. The parameter $\Psi_L$ therefore contains a $\SL_2 (\C)$ factor of the maximal dimension
in each factor of $L^{\vee}$. These factors consist of $\GL_{m_j} (\C)$, $j=1 , \ldots , r$. The biggest possible $\SL_2 (\C)$ representation in the $j$-th factor is $R_{m_j}$. We conclude that $\Psi$ decomposes as:
\begin{equation} \label{AJparam1}
(\mu_1 \boxtimes R_{m_1} \boxplus \mu_1^{-1} \boxtimes R_{m_1} ) \boxplus 
\ldots \boxplus (\mu_r \boxtimes R_{m_r} \boxplus \mu_r^{-1} \boxtimes R_{m_r} ) 
\end{equation}
where $\mu_j$ is the unitary characters of $\C^{\times}$ given by $\mu_j (z) = (z / \bar{z})^{k_j / 2}$. Denoting the character $z \mapsto z/\bar{z}$ by $\mu$, we conclude that we have:
\begin{equation} \label{Aparam}
\mathrm{St} (A_{\mathfrak{q}})^{\rm BC} = \mu^{k_1/2} \boxtimes R_{m_1}  \boxplus 
\ldots \boxplus \mu^{k_r/2} \boxtimes R_{m_r} .
\end{equation}

\subsection{} It follows in particular from \eqref{Aparam} that 
\begin{multline*}
\mathrm{St} (A (b \times q , a \times q))^{\rm BC} \\ 
= \mu^{\frac{m-1}{2}}  \boxplus \mu^{\frac{m-3}{2}} \boxplus \cdots \boxplus \mu^{\frac{m-2b+1}{2}} \boxplus (
\mu^{\frac{a-b}{2}} \boxtimes R_{p+q-a-b} ) \boxplus \mu^{\frac{2a-1-m}{2}} \boxplus \cdots \boxplus \mu^{\frac{1-m}{2}}.
\end{multline*}

We therefore deduce from Theorem \ref{Thm:GJS}, Proposition \ref{Lem:L} and the paragraph following it:

\begin{prop} \label{prop:main}
Let $\pi \in \mathcal{A}^c (\U (V))$ and let $v$ be an infinite place of $F$ such that $\U (V) (F_v) \cong \U (p,q)$. Assume that $\pi_v$ is (isomorphic to)
the cohomological representation $A (b \times q , a \times q)$ of $\U (p,q)$ with $3(a+b)+ |a-b| < 2m$.
Then, there exists some $(a+b)$-dimensional skew-Hermitian space $W$ over $E$ such that $\pi$ is in the image of the cuspidal $\psi$-theta correspondence
from the group $\U (W)$. 
\end{prop}
\begin{proof}
One begins to make a translation between the notations of  \ref{par:8.3} and the notations of this section: the $|p|$ of loc. cite is here $|a-b|$ and the $b$ of loc.cite is $p+q-a-b=m-(a+b)$. The hypothesis of this proposition is the same as the hypothesis of   loc.cite.  One deduces the fact that the partial $L$-function as in \ref{par:8.3} has a pole at a point $s_{0}$ with $s_{0}\geq (m-(a+b)+1)/2$. Using now \ref{Thm:GJS}, we obtain the fact that $\pi$ is in the image of the $\psi$-theta correspondance  with a skew-hemitian space $W$ of dimension $m-2s_{0}-1 \geq m-(m-(a+b))=(a+b)$. But it is easy to see that a strict inequality is impossible at the infinite place $v$ and we obtain the equality as in the statement of the proposition. Moreover the representation of $\U(W)$ in this correspondance is a discrete series at the place $v$ and is, therefore, necessarily a cuspidal representation.
\end{proof}

\subsection{Proof of Theorem \ref{thm:rational}(2)}
It suffices to prove that if $\pi_f^{\sigma} \in \mathrm{Coh}_f^{b' , a'}$ for some $\sigma \in \mathrm{Gal} (\overline{\Q} / \Q)$
and some integers $b'$ and $a'$ s.t. $a+b=a'+b'$, then either $(a',b')=(a,b)$ or $(a', b') = (b,a)$. To do so recall that it corresponds to $\pi_f$ 
a parameter $\Psi$ given by Proposition \ref{Prop:Arthur}. 
Now fix an unramified finite prime $v$ and let $\omega_{\pi_v} : \mathcal{H}_v \to \C$ be the unramified character of the Hecke algebra associated to 
the local (unramified) representation $\pi_v$ by Satake transform \cite{Minguez}. Recall that $\omega_{\pi_v}$ is associated to some unramified character 
$\chi_{\pi_v}$ of a maximal torus of $G(\Q_v)$ (considered up to the Weyl group action). Since the Hecke algebra and its action on $H^{\bullet} (S(K) , \C)$ 
admit a definition over $\Q$, the Galois group $\mathrm{Gal} (\overline{\Q} / \Q)$ acts on the characters of the Hecke algebra associated to representations
$\pi_f \in \mathrm{Coh}_f$, and therefore on $\chi_{\pi_v}$, so that $\chi_{\pi_v^{\sigma}} = \chi_{\pi_v}^{\sigma}$. 
Note that the Galois action preserves the norm. Now if $\pi_f \in \mathrm{Coh}_f^{a,b}$ with $3(a+b)+ |a-b| <2m$, the global parameter $\Psi$ contains a unique factor 
$\mu \boxtimes R_{m-a-b}$ where $\mu$ is a unitary Hecke character, and all the other factor are associated with smaller $\SL_2 (\C)$-representations --- as follows from Proposition \ref{prop:main} and (the proof of) Proposition \ref{Lem:L}. 
Localizing this parameter $\Psi$ at the finite unramified place $v$, we conclude that the character $\chi_{\pi_v}$ 
has a constituent which is a Hecke character of `big' norm --- corresponding to the $\SL_2 (\C)$-representation of dimension $m-a-b$. This singularizes the character 
$\mu_v$ so that the Galois action on $\chi_{\pi_v}$ yields the usual action of $\mathrm{Gal} (\overline{\Q} / \Q)$ on $\mu$. Now if $\mu$ is  
$(z/\bar{z})^{(b-a)/2}$ at infinity then for every $\sigma \in \mathrm{Gal} (\overline{\Q} / \Q)$ the character $\mu^{\sigma}$ is either $(z/\bar{z})^{(a-b)/2}$
or $(z/\bar{z})^{(b-a)/2}$ at infinity. And we conclude from Proposition \ref{prop:main} that the contribution of $\pi_f^{\sigma}$ to $SH^{(a+b)q} (S(K) , \C)$
can only occur in 
$$H^{a\times q , b \times q} (S(K) , \C) \oplus H^{b\times q , a \times q} (S(K) , \C).$$
\qed

\newpage

\setcounter{section}{1}

\setcounter{subsection}{1}

\numberwithin{subsection}{section}

\setcounter{equation}{0}

\part*{Appendices}

\section*{Appendix A: proof of Theorem \ref{Thm:Arthur}}

\numberwithin{equation}{section}

\subsection*{The quasi-split case}
Suppose that $G$ is quasi-split. We want to relate the discrete automorphic spectra of $G$ and that of the (disconnected) group 
$\widetilde{G}$ equal to $\GL (m)$ twisted by the exterior automorphism $\theta : g \mapsto {}^t \bar{g}^{-1}$. Following Arthur this goes through the use of the 
stable trace formula. Notations are as in the preceding paragraph except that we will use $\widetilde{G}$ to denote the twisted linear group as in \cite{Mok}.

\subsubsection*{Discrete parts of trace formulae}
Let $f= \otimes_v f_v$ be a decomposable smooth function with compact support in $G(\A)$. We are essentially interested in 
\begin{equation} \label{tracetotale}
\mathrm{trace} \ R_{\rm dis}^G (f) = \mathrm{trace} \left( f | L_{\rm dis}^2 (G(F) \backslash G (\A) \right),
\end{equation}
where $L^2_{\rm dis}$ is the discrete part of the space of automorphic forms on $G(F) \backslash G(\A)$.
We fix $t$ a positive real number and -- following Arthur -- we will only compare sums -- relative 
to $G$ and $\GL (m)$ -- over representations whose infinitesimal character has norm $\leq t$.

For fix $t$, Arthur defines a distribution $f \mapsto I_{{\rm disc} , t}^G (f)$ as a sum of the part of \eqref{tracetotale} relative to $t$ and of terms associated to some representations induced from Levi subgroups, see \cite[(4.1.1)]{Mok}. There is an analogous distribution $f \mapsto \tilde{I}^m_{{\rm disc} , t} (f)$ in the twisted case, see \cite[(4.1.3)]{Mok}. Proposition \ref{Prop:Arthur} is based on the comparison of these two distributions via their stabilized versions.

\subsubsection*{Stable distribution, endoscopic groups and transfer}
Langlands and Shelstad \cite{LanglandsShelstad} have conjectured the existence of a remarkable family of maps --- or rather of correspondences --
$f \leadsto f^H$ which transfer functions on $G$ to functions on so-called endoscopic groups $H$, certain quasi-split groups of dimension
smaller than $\dim G$. These maps are an analytic counterpart to the fact that nonconjugate elements in $G$ can be conjugate over the 
algebraic closure $G(\overline{F}_v )$. This goes through the study of orbital integrals: 

Two functions in $\mathcal{H} (G)$ are {\it equivalent}, resp. {\it stably equivalent}, 
if they have the same orbital integrals, resp. stable orbital integrals (see \cite{ArthurSelecta}); we denote
by $\mathcal{I} (G)$, resp. $\mathcal{SI} (G)$, the corresponding orbit space. It is known that invariant distribution on $G$ annihilates any
function $f \in \mathcal{H} (G)$ such that all the orbital integrals of $f$ vanish. An invariant distribution is a {\it stable distribution} on $G$ if 
it annihilates any function $f \in \mathcal{H} (G)$ such that all the stable
orbital integrals of $f$ vanish. Equivalently it is an invariant distribution which lies in the closed linear span of the stable orbital integral, 
see e.g. \cite{LanglandsShelstad}. The theory of endoscopy describes invariant distributions on $G$ in terms
of stable distributions on certain groups of dimension less than or equal to $G$ --- the endoscopic groups. 

Thanks to the recent proofs by Ngo \cite{Ngo} of the fundamental lemma and Waldspurger's and Shelstad's work, the Langlands-Shelstad 
conjecture is now a theorem; we need the more general situation which include  the disconnected group $\widetilde{G}$, and which is also known thanks to the same authors 
\cite{WaldspurgerStab}, \cite{transfarchi}, \cite{transfKfini}. In both cases endoscopic groups are described in \cite{WaldspurgerFormulaire}, see also \cite{Mok}. The group $G$ appears
as a (principal) endoscopic subgroup for $\widetilde{G}$; this is the key point to relate representations of $G$ and $\GL (m)$. We denote by 
$\widetilde{\mathcal{I}} (m)$ and $\widetilde{\mathcal{SI}} (m)$ the quotients of $\widetilde{\mathcal{H}} (m)$ defined as above.

\subsubsection*{Local stabilization}
Let $\mathcal{I}_{\rm cusp} (G)$ be 
the image in $\mathcal{I} (G)$ 
of the {\it cuspidal} functions on $\mathcal{H}(G)$, i.e. 
the functions whose orbital integrals associated to semi-simple elements contained in a proper parabolic subgroup all vanish. 
We similarly define $\widetilde{\mathcal{I}}_{\rm cusp} (m)$ (see \cite{WaldspurgerFTLT}). 

Now write $T(G)$ --- the set of $G$-orbit of triples parametrizing virtual characters ---  as a disjoint union
$$T(G) = \bigsqcup_{\{ M \}} (T_{\rm ell} (M) / W(M))$$
over conjugacy classes of Levi subgroups of $G$. Arthur \cite{ArthurSelecta}, and Waldspurger \cite{WaldspurgerStab} in the twisted case, prove 
that there is a linear isomorphism from $\mathcal{I}_{\rm cusp} (G)$ to the space of functions
of finite support on $T_{\rm ell} (G)$; it is given by sending $f \in \mathcal{I}_{\rm cusp} (G)$ to the function:
$$\tau \mapsto \mathrm{trace} \ \tau (f).$$
Here if $\tau = (M, \sigma , 1)$ with $M=G$, then $\sigma$ is a discrete series of $G$ and $f \mapsto \mathrm{trace} \ \tau (f) = \mathrm{trace} \ \sigma (f)$
is the usual character of $\sigma$. 

Arthur \cite{ArthurSelecta} then stabilizes $\mathcal{I}_{\rm cusp} (G)$. 
In particular he defines the stable part $\mathcal{SI}_{\rm cusp} (G)$ of $\mathcal{I}_{\rm cusp} (G)$,
and similarly for all the endoscopic groups (or rather endoscopic data):  The transfer maps $f \mapsto \oplus_H f^H$ induce a linear isomorphism
\begin{equation} \label{localStab}
\mathcal{I}_{\rm cusp} (G) \cong \bigoplus_{H} \mathcal{SI}_{\rm cusp} (H)^{\mathrm{Out}_G (H)},
\end{equation}
where we sum over the endoscopic groups (or rather endoscopic data) not forgetting $G$ itself. The twisted analogue of \eqref{localStab} is:
\begin{equation} \label{localStab2}
\widetilde{\mathcal{I}}_{\rm cusp} (m) \cong \bigoplus_{H} \mathcal{SI}_{\rm cusp} (H)^{\mathrm{Out}_{\widetilde{G}} (H)};
\end{equation}
see \cite{WaldspurgerStab} where Waldspurger deals with a much more general situation.

We say that a virtual representation $\pi$ --- in the Grothendieck ring of the representations of $G$ ---  is {\it stable}
if $f \mapsto \mathrm{trace} \ \pi (f)$  is a stable distribution. Assume that $\pi$ is a finite linear combination of elliptic representations of $G$ then after \cite{ArthurSelecta}, $\pi$ is stable if and only if $\mathrm{trace} \ \pi(f)=0$ for any $f\in \mathcal{I}_{\rm cusp} (G)$ in the kernel of the projection of $\mathcal{I}_{\rm cusp} (G)$ onto $ \mathcal{SI}_{\rm cusp} (G)$ in the above decomposition. Moreover: the map $f \mapsto  ( \pi \mapsto \mathrm{trace} \ \pi(f))$ induces an isomorphism from 
$\mathcal{SI}_{\rm cusp} (G)$ to the space of linear forms with finite support on the elliptic and stable virtual representations. 

Now let $\pi$ be any virtual representation of $G$ which corresponds to an elliptic element in $T(G)$. Its distribution character defines a distribution on $\mathcal{I}_{\rm cusp} (G)$ and 
it follows from \eqref{localStab} that, for every endoscopic data $H$, we can associate to $\pi$ a virtual representation $\mathrm{St}_H (\pi)$ such that 
for every $f \in \mathcal{I}_{\rm cusp} (G)$ we have:
\begin{equation}
\mathrm{trace} \ \pi (f) = \sum_H \mathrm{trace} \ \mathrm{St}_H (\pi) (f^H).
\end{equation}
We simply denote by $\mathrm{St} (\pi)$ the virtual representation $\mathrm{St}_G (\pi)$ and as we have recall above, after Arthur it is a stable, elliptic representation. 

In the Archimedean case, it follows from \cite{Shelstad} that if $\pi$ is a discrete series representation of $G$, then the virtual representation $\mathrm{St} (\pi)$ is the sum of the representations in the $L$-packet of $\pi$. Moreover if $\pi$ is only elliptic but not discrete then $\mathrm{St} (\pi)=0$. This is coherent with the notation
in \S \ref{par:6.2}.

\subsubsection*{Stable standard modules}
Recall that given any $\tau \in T(G)$, we write $M_{\tau}$ for a Levi subgroup such that $\tau$ lies in the image of $T_{\rm ell} (M_{\tau})$. Writing 
$\tau = (M_{\tau} , \sigma , r)$, we define 
the {\it stable standard module} associated to $T(G)$ to be the virtual representation obtained as the induced module $\mathrm{ind}_{P_{\tau}} \mathrm{St} (\sigma)$. If $\pi$ is any irreducible admissible representation of $G$ we denote by $\mathrm{St} (\pi)$ the stable standard module associated to $\tau_{\pi}$. 
This is a virtual representation induced from a Langlands' packet of discrete series. The next proposition again follows from \cite{ArthurSelecta}.

\begin{prop*}
The set of stable standard modules associated to $T(G)$ is a basis of the vector space of stable distributions that are supported on a finite set of characters of $G$.
\end{prop*}

If $H$ is an endoscopic data in $\widetilde{G}$, since $H$ is a product of unitary groups, given an 
admissible irreducible representation $\pi$ of $H$ we have associate to it a stable standard module $\mathrm{St} (\pi)$ and 
we may transfer it to a standard module for $\widetilde{G}$ and therefore for $\GL_m (E_v )$; we denote it by $\mathrm{St} (\pi)^{\rm BC}$. 

We now come back to the global situation.

\subsubsection*{Stabilization of trace formulae} The stabilization of the distributions $I_{{\rm disc} , t}^G$ refers to a decomposition:
\begin{equation} \label{STFF}
I_{{\rm disc} , t}^G (f) = \sum_{ H \in \mathcal{E}_{\rm ell} (G)} \iota (G,H) S_{{\rm disc} , t}^H (f^H), \quad f \in \mathcal{H} (G)
\end{equation}
see e.g. \cite[(4.2.1)]{Mok}. Here we sum over a set of representatives of endoscopic data in $G$, we denote by $f^H$ the Langlands-Kottwitz-Shelstad transfer of $f$ to $H$ and for every $H$, $S_{{\rm disc} , t}^H$ is a stable distribution. The coefficients $\iota (G,H)$ are positive rational numbers.

Similarly, the stabilization of the distributions $\widetilde{I}_{{\rm disc} , t}^m$ refers to a decomposition:
\begin{equation} \label{STFFt}
\widetilde{I}_{{\rm disc} , t}^m (\widetilde{f}) = \sum_{ H \in \widetilde{\mathcal{E}}_{\rm ell} (m)} \widetilde{\iota} (m,H) S_{{\rm disc} , t}^H (\widetilde{f}^H), \quad \widetilde{f} \in \widetilde{\mathcal{H}} (m).
\end{equation}

Now we fix a finite set $S$ of places of $F$ which
contains all the Archimidean places of $F$. We moreover assume that $S$ contains all the ramification
places of $G$. If $v \notin S$, $G \times F_v$ is isomorphic to the (quasi-split) group $G (F_v)$ and splits over a finite unramified extension of $F_v$; in particular $G \times F_v$ contains a hyperspecial 
compact subgroup $K_v$, see \cite[1.10.2]{Tits}. Let $\mathcal{H}_v$ be the corresponding (spherical)
Hecke algebra and let $\mathcal{H}^S = \prod_{v \notin S} \mathcal{H}_v$.

We may decompose \eqref{STFF} according to characters of $\mathcal{H}^S$: Let $f \in \mathcal{H}^S$, we have a decomposition (see \cite[(4.3.1)]{Mok}):
$$I^G_{{\rm disc} , t} (f) = \sum_{c^S} I^G_{{\rm disc} , c^S , t} (f)$$
where $c^S = (c_v)_{v \in S}$ runs over a family of compatible Satake parameters --- called Hecke eigenfamilies in \cite{Arthur,Mok} --- consisting
of those families that arise from automorphic representation of $G(\A)$, and $ I^G_{{\rm disc} , c^S , t}$ is the $c^S$ eigencomponent of $I^G_{{\rm disc} , t}$.
It then follows from \cite[Lemma 3.3.1]{Arthur} or \cite[Lemma 4.3.2]{Mok} that 
\begin{equation} \label{STFFchi}
I_{{\rm disc} , c^S, t}^G (f) = \sum_{ H \in \mathcal{E}_{\rm ell} (G)} \iota (G,H) S_{{\rm disc}, c^S , t}^H (f^H),
\end{equation}
where on the RHS $c^S$ is rather the Hecke eigenfamily for $H$ which corresponds to $c^S$ under the $L$-embedding ${}^L H \to {}^L G$ that is part of the endoscopic data and $S_{{\rm disc}, c^S , t}^H$ is the stable part of the trace formula for $H$ restricted to the representations which are unramified outside $S$ and belong to the $c^S$ eigen-component. 

Similarly, in the twisted case we have:
\begin{equation} \label{STFFchit}
\widetilde{I}^m_{{\rm disc} , c^S , t} (\widetilde{f})= \sum_{ H \in \widetilde{\mathcal{E}}_{\rm ell} (m)} \widetilde{\iota} (m,H) S_{{\rm disc}, c^S , t}^H (\widetilde{f}^H).
\end{equation}

A Hecke eigenfamily $c^S$ determines at most one irreducible automorphic representation $\Pi=\Pi_{\Psi}$ of $\GL_m (\A_E)$ s.t. for every $v \notin S$ the unramified representation $\Pi_v$ corresponds to the Satake parameter $c_v$. We write $\Pi=0$ or $\Psi = 0$ if $\Pi$ does not exist.
Note that Proposition \ref{Prop:Arthur} states that if $\Pi = 0$ then the $c^S$ eigencomponent of the discrete part of $L^2 (G(F) \backslash G(\A))$ is trivial. In fact Arthur (and Mok) prove that if $\Pi = 0$, then for all $f \in \mathcal{H}^S$ we have:
$$I^G_{{\rm disc} , c^S , t} (f) = 0 = S^G_{{\rm disc} , c^S , t} (f).$$
The proof goes by induction on $m$, see \cite[Proposition 3.4.1]{Arthur} and \cite[Proposition 4.3.4]{Mok}. Following their proof, we now prove the following refined version of Proposition \ref{Prop:Arthur}.

\begin{prop*} 
Let $\pi$ be an irreducible automorphic representation of $G (\A)$ which occurs (discretely) 
as an irreducible subspace of $L^2 (G (F) \backslash G (\A))$. Assume that for every Archimedean place $v$ the representation $\pi_v$ has regular infinitesimal 
character. Let $S$ be a finite set of places --- including all the Archimedean ones --- such that $\pi$ is unramified outside $S$ and belongs to the $c^S$ eigencomponent of the discrete part of $L^2 (G(F) \backslash G(\A))$ and let $\Pi = \Pi_{\Psi}$ the associated automorphic representation. Then, for every Archimedean place $v$, the (unique) irreducible quotient of the standard $\GL_m (\C)$-module $\mathrm{St} (\pi_v)^{\rm BC}$ occurs as an (irreducible) sub-quotient of $\Psi_v$.\footnote{Here $\Psi_v$ is considered as a standard module, see \S \ref{par:6.2}.} 
\end{prop*}

\medskip
\noindent
{\it Remark.} If $\Psi=0$ it follows in particular from the proposition that $\pi$ cannot exist as proved in \cite[Proposition 4.3.4]{Mok}.
\medskip

\begin{proof} Mok \cite{Mok} has proved that $I^H_{{\rm disc} , c^S , t}$ is of finite length; we will freely use that fact to simplify the proof.
Let $v\in S$ be an Archimedean place and let $\pi_{S}\in I^H_{{\rm disc} , c^S , t}$ an irreducible representation. Denote by $\gamma_{v}$ the collection of $m$ characters of ${\mathbb C}^*$ obtained as the restriction of the Langlands parameter of $\pi_{v}$ to ${\mathbb C}^*$; it corresponds to $\gamma_v$ a standard representation of $\GL_m (\C)$ that is precisely $\mathrm{St}(\pi_{v})^{\rm BC}$. We recall that $\Psi_{v}$ is the local component of the  representation of $\GL_m (\A )$ defined by $c^S$. The proposition is a corollary of Proposition \ref{Prop:Arthur} and the following lemma.
\end{proof}

\begin{lem*}
The Langlands quotient of $\mathrm{St}(\pi_{v})^{\rm BC}$ is a subquotient of the standard module associated to $\Psi_{v}$.
\end{lem*}
\begin{proof}
We call an irreducible repr\'esentation included in $I^H_{{\rm disc} , c^S , t}$ {\it maximal} if this representation does not appear as a subquotient of the standard module (product on all places of the local standard modules) of another representation entering $S^H_{{\rm disc} , c^S , t}$. We will prove the lemma for maximal representation. Denote by $\pi'_{S}$ such a representation.
 
We first prove that if the lemma holds for any maximal element of $I^H_{{\rm disc} , c^S , t}$ then it holds for any element of $I^H_{{\rm disc} , c^S , t}$. Take $\pi_{S}$ any element of $I^H_{{\rm disc} , c^S , t}$ and assume that its standard module occurs in the decomposition of the standard module of the maximal element $\pi'_{S}$ of $I^H_{{\rm disc} , c^S , t}$ and that we know the lemma for $\pi'_{S}$.
 Here we use the deep result of Salamanca-Riba (\cite{salamanca}): if $\pi'_{v}$ is unitary and has an infinitesimal character which is integral and regular then $\pi'_{v}$ has cohomology. So $\pi'_{v}$ is an $A_{{\mathfrak{q}}'}(\lambda')$ and Johnson has written in his thesis {\cite{johnson}} the decomposition of such a module in terms of standard modules. We provide details below (with explicit parameters).
 
Denote by $\gamma_v'$ the analogue of $\gamma_v$ (as defined before the lemma) for $\pi'$. We recall that the normalizer of ${\mathfrak{q}}'$ in $H$ is a product of unitary groups: $\prod_{i\in [1,\ell]} \U(p_{i},q_{i})$ and that $\lambda'$ gives a character of this group, this means a set of half integer $r_{i}$ for $i\in [1,\ell]$. It is not necessary to know exactly what is $\gamma_v'$ it is enough to know that it is a collection of characters $z^{x_{i,j_{i}}}{\overline{z}}^{x'_{i,j_{i}}}$ where $i\in [1,\ell]$ and for $i$ fixed $i_{j}\in [1,m_{i}:=(p_{i}+q_{i})]$ and 
\begin{equation} \label{eqno1}
\{x_{i,j_{i}} : j_{i}\in [1,m_{i}]\}= \left\{r_{i}+k : k\in \left[\frac{m_{i}-1}{2},- \frac{m_{i}-1}{2} \right] \right\} 
\end{equation}
\begin{equation} \label{eqno2}
\{x_{i,j_{i}} : j_{i}\in [1,m_{i}]\}= \left\{-r_{i}+k :  k\in \left[ \frac{m_{i}-1}{2},- \frac{m_{i}-1}{2} \right] \right\}. 
\end{equation}
We also know that $\Psi_{v}$ in an induced representation of unitary characters; so we can decompose $m=\sum_{t\in [1,\ell']}m_{t}$ and for any $t$ we have a unitary character of ${\mathbb C}^*$, $(z/{\overline{z}})^{n_{t}}$. The subquotients of the standard module $\Psi_{v}$ are exactly the representations whose Langlands' parameters are collections of $m$ characters of ${\mathbb C}^*$ that can be partitioned in $\ell'$ subsets s.t. in each subset, indexed by $t$, the characters are of the form 
$z^{x}{\overline{z}}^{x'}$ with $x\in n_{t}+[(m_{t}-1)/2,-(m_{t}-1)/2]$ and $x'\in -n_{t}+[(m_{t}-1)/2,-(m_{t}-1)/2]$.

Recall that by hypothesis the lemma holds for $\pi'_{v}$. Now for each $i\in [1,\ell]$ there exist $t\in [1,\ell']$ such that 
$$
\left\{r_{i}+k : k\in \left[\frac{m_{i}-1}{2},- \frac{m_{i}-1}{2} \right] \right\}  \subset n_{t}+\left[ \frac{m_{t}-1}{2},- \frac{m_{t}-1}{2} \right]$$ 
and, by symmetry, 
$$
\left\{-r_{i}+k :  k\in \left[ \frac{m_{i}-1}{2},- \frac{m_{i}-1}{2} \right] \right\} \subset -n_{t}+\left[ \frac{m_{t}-1}{2},-\frac{m_{t}-1}{2} \right].$$
To prove the lemma for $\pi_v$ we therefore only have to prove that the `Langlands' parameter' $\gamma_v$ of $\pi_{v}$ is a collection of $m$ characters of ${\mathbb C}^*$ which can be decomposed in $\ell$ subsets satisfying exactly the property \eqref{eqno1} and \eqref{eqno2} above {\it with the same numbers}. But this follows from Johnson's thesis: in the exact sequence \cite[(3) page 378]{johnson}, the standard modules which appear all satisfy  $\Delta^+ \supset \Delta ({\overline{\mathfrak {u}}})$ where ${\overline{\mathfrak {u}}}$ is the nilradical of the opposite parabolic subalgebra of ${\mathfrak q}$. Our assertion follows: we can permute inside the blocks but not between to different blocks. 

\medskip
\noindent
{\it Remark.} This is not mysterious --- at least in our case. It is just the fact that the induced representation of $\GL_2 ({\mathbb C})$ of two characters $z^{x}{\overline{z}}^{x'}$ and $z^y {\overline{z}}^{y'}$ is irreducible if $x>y$ but $x'<y'$ or the symmetric relations and this is precisely what occurs if the two characters are in subsets indexed by $i,i'$ with $i\neq i'$.

\medskip
 
We now prove the lemma for maximal representations.
When we decompose $I^H_{{\rm disc} , c^S , t}$ in terms of standard modules, we are sure that the standard module associated to any maximal representation, $\pi_{S}'$, occur with the same coefficient that the representation itself. We now look at the coefficient with which  the stable standard module of $\pi'_{S}$ occurs in  $S^H_{{\rm disc} , c^S , t}$; up to a positive global, constant, $i(H)$, either it is the same coefficient as in $I^H_{{\rm disc} , c^S , t}$ or $\pi'_{S}$ comes from endoscopy.  In this latter case we argue by induction because endoscopic groups are product of smaller unitary groups. So we assume that the standard module of $\pi'_{S}$ occurs in $S^H_{{\rm disc} , c^S , t}$ with the product of $i(H)$ by the multiplicity of $\pi$ in $I^H_{{\rm disc} , c^S , t}$. We now look at the sum of all $H$ and decompose in the Grothendieck group; the standard module of $\pi'_{S}$ can be cancelled by a representation occuring in the decomposition of the standard module of $\pi^{H'}_{S}$ where $\pi^{H'}_{S}$ occurs in $S^{H'}_{{\rm disc} , c^S , t}$. If this does not occur then by transfert  the standard module of $\pi'_{S}$ appears in the decomposition of $\Psi_{S}$ and we are done for $\pi'_{S}$ --- in fact in that case what we get is stronger more than the statement of the lemma.
 
In the case where we have a simplification, by positivity, $\pi^{H'}_{S}$ is not maximal for $H'$ but by an easy induction we know the lemma for $\pi^{H'}_{S}$. We argue explicitly as above that this also prove the lemma for $\pi_{S}$, we leave the details to the reader especially as the deep results of Mok ultimately yield that this cannot happen: the representation of $\GL_m (\A)$ determined by $c^S$ is the transfer of a {\it unique} endoscopic group.
\end{proof}

\subsection*{The general (non quasi-split) case}

Notations are as in the preceding two paragraphs. We don't assume $G$ to be quasi-split anymore.

Let $\pi$ be an irreducible automorphic representation of $G (\A)$ which occurs (discretely) 
as an irreducible subspace of $L^2 (G (F) \backslash G (\A))$. Let $S$ be a finite set of places --- including all the Archimedean ones --- such that both $G$ and 
$\pi$ are unramified outside $S$. It still make sense to consider a Hecke eigenfamily (outside $S$) $c^S$. Again it determines at most one irreducible automorphic representation $\Pi=\Pi_{\Psi}$ of $\GL_m (\A_E)$ s.t. for every $v \notin S$ the unramified representation $\Pi_v$ corresponds to the Satake parameter $c_v$, and we write $\Pi = 0$, or $\Psi = 0$, if $\Pi_{\Psi}$ does not exist.

Everything we have proved in the quasi-split case remains valid except that the stable standard representation 
$\mathrm{St} (\pi_v )_{\rm st}$ does not make sense. However in the Grothendieck group of representations of $G(F_v)$ we may write $\pi$ as 
a linear combination of the basis represented by representations induced from elliptic ones $\mathrm{ind}_Q \tau$ and the virtual representation $\mathrm{St} (\pi_v)$ we associate to $\pi_v$ is the sum of the standard representations $\mathrm{ind}_Q \tau '$ where $\tau '$ belongs to the same $L$-packet as $\tau$. 
Here again, the local principal base change associates to all these standard a unique standard module of $\GL_m (\C)$ that we denote by 
$\mathrm{St} (\pi_v )^{\rm BC}$.  

We may restate Theorem \ref{Thm:Arthur} in the following way:

\begin{thm*} 
Let $\pi$ be an irreducible automorphic representation of $G (\A)$ which occurs (discretely) 
as an irreducible subspace of $L^2 (G (F) \backslash G (\A))$. Assume that for every Archimedean place $v$ the representation $\pi_v$ has regular infinitesimal character. Let $S$ be a finite set of places --- including all the Archimedean ones --- such that 
both $G$ and $\pi$ are unramified outside $S$. Let $c^S$ be the Hecke eigenfamily associated to $\pi$ and let $\Pi_{\Psi}$ be the associated automorphic representation of $\GL (m)$. Then, for every Archimedean place $v$, the (unique) irreducible quotient of the standard $\GL_m (\C)$-module $\mathrm{St} (\pi_v)^{\rm BC}$ 
occurs as an (irreducible) sub-quotient of 
the local standard representation $\Psi_v$. 
\end{thm*}
\begin{proof} \footnote{We refer to \cite[Appendix B, \S 17.15 and after]{BMM} for more details.}
Here we use the stable trace formula \eqref{STFFchi} for the group $G$. We write the LHS of \eqref{STFFchi} as a linear combination of standard modules.
The contribution of the standard module associated to $\pi_v$ might be zero but then there must exist $\pi'$, an irreducible automorphic representation of $G (\A)$ which occurs (discretely) as an irreducible subspace of $L^2 (G (F) \backslash G (\A))$, such that $\pi_v'$ and $\pi_v$ share the same infinitesimal character and $\pi_v$ is a subquotient of the standard module associated to $\pi_v'$. It is then enough to prove the theorem for $\pi'$. From now
on we will therefore assume that the standard module associated to $\pi_v$ contribute to the LHS of \eqref{STFFchi}.

We now write the RHS as a sum of {\it stable} standard modules. At least one of these has an $L$-parameter whose restriction to $\C^{\times}$ is the parameter of the standard module of $\pi_v$. At this point it is not clear that this stable standard module is associated to a {\it square integrable} automorphic
representation of an endoscopic group. Using the same induction as above one may however assume this is the case. Now since endoscopic groups are (product of) quasi-split unitary groups the theorem follows from the quasi-split case (see the proposition above). 
\end{proof}

\setcounter{section}{2}

\setcounter{subsection}{1}

\setcounter{equation}{0}

\section*{Appendix B: proof of Theorem \ref{Thm:Ichino}}

We fix $s_0$ and $\chi=\chi_2$ a character of $\A_E^{\times}/E^{\times}$ as in Theorem \ref{Thm:Ichino}. Let $n=m-2s_0$ and $n'=m+2s_0$ so that 
$$-s_0 = \frac12 (m-n').$$
Note that we have 
$$m<n' < 2m$$
and
$$\chi |_{\A_F^{\times}} = \epsilon_{E/F}^n = \epsilon_{E/F}^{n'}.$$

\subsection*{Representations associated to skew-Hermitian spaces} 
In this paragraph we fix a prime and omit it from notation as in \S \ref{par:local}. 
Recall that the isometry class of a $\tau$-skew-Hermitian space $W$ of dimension $n$ over $E$ is determined by the 
Hasse invariant $\epsilon = \epsilon (W) = \pm 1$. We will write $W_1$ and $W_2$ for the two distinct $\tau$-skew-Hermitian spaces
of dimension $n$ over $E$. 

We fix an arbitrary choice of character $\chi_1$ of $E^{\times}$ such that $\chi_1 |_{F^{\times}} = \epsilon_{E/F}^m$. We may consider the local analogue of the Weil representation $\Omega_{\chi}^-$ of the preceding
paragraph, see (\ref{actionLevi}). This yields a representation of $\U_{2m} (F) \times \U_n (F)$ on $\mathcal{S}(W^m )$.  
The group $\U_n (F)$ acts via a twist $\chi_1$ of its linear action on $\mathcal{S}(W ^m)$. Let $R (W, \chi)$ be the maximal quotient of $\mathcal{S}(W ^m)$ on which
$\U_n (F)$ acts by $\chi_1$. Kudla and Sweet \cite{KudlaSweet} show that 
$R (W, \chi)$ is an admissible representation of $\U_{2m} (F)$ of finite length and with a unique irreducible
quotient. Moreover: since $1 \leq n \leq m$, it follows from \cite[Theorem 1.2]{KS} that $R(W_1 , \chi )$ and $R(W_2 , \chi)$ are irreducible inequivalent representations of $\U_{2m} (F)$.

\subsection*{Automorphic representations associated to skew-Hermitian spaces} 
We now return to the global situation. Assume that $1 \leq n < m$ and let $\mathcal{C}= \{ W_v \}$ be a collection of local skew-Hermitian spaces of dimension $n$ such that $W_v$ is unramified 
outside of a finite set of places of $F$, then --- following Kudla-Rallis \cite[\S 3]{KR} --- we may define  
a global irreducible representation
\begin{equation} \label{eq:PI_C}
\Pi (\mathcal{C}) = \bigotimes_v R(W_v , \chi_v)
\end{equation}
of $\U_{2m} (\A)$. Such representations are of two types: those for which the $W_v$'s are the localizations of some (unique) global skew-Hermitian space $W$ over $E$ ---  in this case 
we write $\Pi (W)$ for $\Pi (\mathcal{C})$ --- and those for which no such global space exists. Given a collection $\mathcal{C}= \{ W_v \}$ the obstruction to the existence of a global space is just the
requirement that 
$$\prod_v \epsilon (W_v ) = 1.$$
Let $\mathcal{A} (\U (V \oplus V))$ be the set of irreducible automorphic representations of $\U_{2m} (\A)$. The following proposition is the analogue of \cite[Theorem 3.1]{KR} in the unitary case.

\begin{prop*}
Let $\mathcal{C}= \{ W_v \}$ be a collection of local skew-Hermitian spaces of dimension $n$ such that $W_v$ is unramified 
outside of a finite set of places of $F$ and $\dim \mathrm{Hom}_{\U_{2m} (\A)} (\Pi (\mathcal{C}) , \mathcal{A} (\U(V \oplus V))) \neq 0$. Then, there exists a global skew-Hermitian space $W$ over $E$ 
such that the $W_v$'s are the localizations of $W$.
\end{prop*}
\begin{proof} 
It makes use of ideas of Howe: we consider Fourier coefficients of automorphic representations w.r.t. the unipotent radical of the Siegel parabolic subgroup
$P=MN$ of the quasi-split unitary group $\U(V \oplus V)$. Identifying the latter with 
the subgroup of
$\GL(2m,E)$ which preserves the Hermitian form with matrix $\left(
\begin{smallmatrix}
0 & 1_m \\
-1_m & 0 
\end{smallmatrix} \right)$ we have:
$$N = \left\{ \left( \begin{array}{cc}
1_m & b \\
0 & 1_m 
\end{array} \right) \; : \; b \in \mathrm{Her}_m (F) \right\},$$
where 
$$\mathrm{Her}_m (F) = \left\{ b \in \mathrm{M}_m (F) \; : \; b =  {}^t \- \overline{b}  \right\}.$$
For $\beta \in \mathrm{Her}_m (F) $ we define the character
$\psi_{\beta}$ of $N(\A)$ by
$$n(b)=\left( \begin{array}{cc}
1_m & b \\
0 & 1_m 
\end{array} \right) \mapsto \psi ( \mathrm{trace} (b \beta ))$$
and the $\beta$th Fourier coefficient of $f \in \mathcal{A} (\U (V \oplus V))$ by 
\begin{equation}
W_{\beta} (f) (g) = \int_{N(F) \backslash N (\A)} f(n(b) g) \psi (-\mathrm{trace} (\beta b)) db.
\end{equation}
The latter defines a linear functional 
\begin{equation}
W_{\beta} : \mathcal{A} (\U (V \oplus V)) \to \C ; \quad f \mapsto W_{\beta} (f) (e).
\end{equation}

Now if $A \in \mathrm{Hom}_{\U_{2m} (\A)} (\Pi (\mathcal{C}) , \mathcal{A} (\U(V \oplus V))$ then $A_{\beta} = W_{\beta} \circ A$ defines a linear functional on $\Pi (\mathcal{C})$ such that
$$A_{\beta} ( \pi (n) f) = \psi_{\beta} (n) A_{\beta} (f), \quad \forall n \in N( \A_f ),$$
$$A_{\beta} (\pi (X) f) = d\psi_{\beta} (X) A_{\beta} (f), \quad \forall X \in \mathrm{Lie} (N_{\infty})$$
and $A_{\beta}$ has continuous extension to the smooth vectors of $\Pi (\mathcal{C})$. 

Let
$$\Omega_{\beta} = \{ w \in W^m  \; : \; \langle w , w \rangle = \beta \}.$$
It follows from \cite[Lemma 5.1 \& 5.2]{Ichino} on the local functionals that if $\Omega_{\beta} = \emptyset$ then $A_{\beta}= 0$.
In particular if $\mathrm{rank} (\beta) > n$ then $A_{\beta}=0$. Now if $\Omega_{\beta} \neq \emptyset$ and $\mathrm{rank} (\beta) = n$ we must 
have equality of Hasse invariants $\epsilon_v (\beta) = \epsilon_v (W_v)$. So that either $\mathcal{C}$ correspond to a global skew-Hermitian space or 
$A_{\beta} =0$ for all $\beta$ with $\mathrm{rank} (\beta) \geq n$. The proposition therefore follows from the next lemma.
\end{proof}

\begin{lem*}
Let $A \in \mathrm{Hom}_{\U_{2m} (\A)} (\Pi (\mathcal{C}) , \mathcal{A} (\U(V \oplus V))$ such that $A_{\beta} =0$ for all $\beta$ with $\mathrm{rank} (\beta) \geq n$, then $A=0$.
\end{lem*}
\begin{proof} Fix a non-Archimedean place $v$ and let $\phi$ be a compactly supported function on $N_v$ whose Fourier transform vanishes on the set of $\beta_v \in \mathrm{Her}_m (F_v)$ of rank $<n$ and is
non-zero on the set of $\beta_v$ such that $\Omega_{\beta_v} \neq \emptyset$. Then \cite[Lemma 5.1]{Ichino} shows that $\phi$ does not act by zero in $R(W_v , \chi_v)$ or $\Pi (\mathcal{C})$. On the other hand
by hypothesis $\phi$ acts by zero in the image of $A$. Thus $A=0$ by irreducibility of $\Pi (\mathcal{C})$.
\end{proof}

\subsection*{Proof of Theorem \ref{Thm:Ichino}} 
If $\Phi$ is a section of $I(s_0 , \chi)$ we may extend it to an holomorphic section $\Phi (\cdot, s)$ of $I(s , \chi)$ and consider the residue $A(\cdot , \Phi)$ in $s=s_0$ of the Siegel Eisenstein series $E(h,s, \Phi)$.
This residue does not depend on the holomorphic extension. We therefore get an $\U_{2m} (\A)$-intertwining map $A: I(s_0 , \chi) \to \mathcal{A} (\U (V \oplus V))$. Now it follows from \cite[Lemma 6.1]{Ichino} that this map
factors through the quotient 
$$I(s_0 , \chi)_{\infty} \otimes \left( \bigoplus_{\mathcal{C}} \Pi_f (\mathcal{C})\right) ,$$
where $\mathcal{C}$ runs over all collections of local skew-Hermitian spaces, as above, of dimension $n$. The proposition above therefore associates to any irreducible residue of a Siegel Eisenstein series 
a global space $W$ of dimension $n$ over $E$. Theorem \ref{Thm:Ichino} then follows from \cite[Theorem 4.1]{Ichino}.\footnote{Beware that our $W$ is Ichino's $V'$.}

\newpage

\setcounter{section}{3}

\setcounter{subsection}{1}

\setcounter{equation}{0}

\section*{Appendix C: the local product formula}\label{AppendixC}

In this appendix we show that  the cocycles $\psi_{nq,nq}$ of this paper (when transformed into cocycles with values in the appropriate Schr\"odinger model) are equal to  the cocycles $\varphi_{nq,nq}$ introduced by Kudla-Millson in \cite{KudlaMillson1}, see also \cite{KudlaMillson3}. We stated this
relation without proof in Proposition \ref{formulaforcocycles}(5). It is almost proved in \cite{KudlaMillson3}:  on the fifth line of page 158 the cocycle $\varphi_{q,q}$ is defined by  (see \eqref{KM} below)
$$\varphi_{q,q} = \frac{1}{2^{2q}}D^+  \overline{D}^+  \varphi_0.$$
Then on the next page, Theorem 5.2 (ii), it is stated that
\begin{equation} \label{factorizationofKM}
\varphi_{nq,nq} = \varphi_{q,q} \wedge \ldots \wedge \varphi_{q,q}.
\end{equation}
The previous equation reduces the problem  to proving the equality of cocycles to the $n= 1$ case.  This is because the intertwiner 
$B_{V_+ \otimes W} \otimes 1$ below commutes with the { \it outer} exterior product and the cocycle $\psi_{nq,nq}$ factors as above by definition.     In what follows  we will be using the symplectic vector space obtained from the tensor product of Hermitian spaces $V \otimes W$ where $V,( \ ,\ )_V$ is a Hermitian space of signature $(p,q)$ and $W, (\, \ )_W$ where $W$ is a Hermitian space of signature $(1,1)$.  Since the analysis in what follows will be controlled by $W$, with $V$ essentially a parameter space, we will first describe the required structures on $W$ alone.     
\subsection{The Schr\"odinger and Fock  models of the Weil representation for $\U(W)$}
We therefore consider a Hermitian space $(W , (,)_W)$ of signature $(1,1)$.  We let $W_{\R}$ denote the real vector underlying  $W$.  Then $W_{\R}$ is has the  standard integrable almost complex structure $J_W$ induced by multiplication by $i$. It is also equipped with the symplectic form 
$$\langle , \rangle_{W} = - \mathrm{Im} (,)_{W}.$$
Finally recall that we denote by $\theta_W$ the Cartan involution of $W$ and let 
$$J_0= J_W \circ \theta_W = \theta_W \circ J_W $$
then $J_0$ is positive definite with respect to $\langle , \rangle_{W}$

We now describe two bases for  $W$.  Let 
$(\epsilon_1 , \epsilon_2)$ be the orthogonal complex basis of $W$ such that
$$(\epsilon_1 , \epsilon_1 ) = 1 \mbox{ and } (\epsilon_2 , \epsilon_2 ) = -1.$$
Set 
$$e_1 = \frac{1}{\sqrt{2}} (\epsilon_1 - i \epsilon_2) , \quad  e_2 =  \frac{1}{\sqrt{2}} (i \epsilon_1 + \epsilon_2) = J(e_1)$$
and
$$f_1 =  \frac{1}{\sqrt{2}} (i \epsilon_1 -  \epsilon_2)  = J_0 (e_1), \quad f_2 =  \frac{1}{\sqrt{2}} (- \epsilon_1 - i \epsilon_2) = J_0 (e_2).$$
Then $(e_1, e_2 , f_1 , f_2)$ is a (real) symplectic basis of the underlying real vector space $W_{\R}$.  
Let $E = \mathrm{span}_{\R} (e_1 , e_2)$ and $F = J_0(E) = \mathrm{span}_{\R} (e_1 , e_2)$.  Then
\begin{equation} \label{Wsplitting}
W_{\R} = E + F
\end{equation}  
is a Lagrangian splitting of the symplectic vector space and we obtain a Schr\"odinger model
$\mathcal{S} (E)$ of the Weil representation of $\U(W)$ realized in the Schwartz space $\mathcal{S} (E)$.  
We let $x$ and $y$ be the coordinates of $E$ associated to the basis $e_1,e_2$.  We put $z = x + i y$.  We let 
$\mathcal{P}(E)$ denote the subspace of $\mathcal{S} (E)$ given by products of complex-valued polynomials in $x$ and $y$ and the Gaussian
$\varphi_0 = \exp{(- \pi( x^2 + y^2) )}$ or equivalently given by products of complex-valued polynomials in $z$ and $\overline{z} $ and the Gaussian $\varphi_0 = \exp{(- \pi z \overline{z} )}$

We next give two sets of coordinates for  the space $W^{\prime_0}$ associated to the positive complex structure $J_0$.
We first note that 

\begin{align*}
e^{\prime_0}_1 = & \frac{1}{\sqrt{2}} (\epsilon^{\prime_0} _1 + i \epsilon^{\prime_0}_2)\\
e^{\prime_0}_2 = & \frac{i}{\sqrt{2}} (\epsilon^{\prime_0} _1 - i \epsilon^{\prime_0}_2)
\end{align*}
Hence,  the vectors $e^{\prime_0}_1$ and $e^{\prime_0}_2$ are independent (over $\C$)
and we have two bases $W^{\prime_0}$,  the basis $ \{ e^{\prime_0}_1, e^{\prime_0}_2\}$ and the basis
$\{\epsilon^{\prime_0} _1, \epsilon^{\prime_0}_2\}$.  We let $s,t$ be the (complex) coordinates for $W^{\prime_0}$ relative to
the first basis. We will call these coordinates { \it split} coordinates. 
We let $a',b''$ be the coordinates  for $W^{\prime_0}$  relative to the second basis $\epsilon^{\prime_0} _1$ 
and $\epsilon^{\prime_0}_2$. We will call $a',b''$ {\it product} coordinates. 
In order to understand the superscripts attached to these coordinates and the terminology note that 
$\epsilon^{\prime_0}_1 = \epsilon^{\prime}_1 \ \text{and} \ \epsilon^{\prime_0}_2 = \epsilon^{\prime \prime }_2$
and hence  
$$\ W^{\prime_0} = W^{\prime}_+ \oplus W^{\prime \prime} _- =  \C \epsilon^{\prime_0}_1 \oplus   \C \epsilon^{\prime_0}_2$$
and
$$\Pol(W^{\prime_0}) =  \Pol(W^{\prime}_+) \otimes \Pol(W^{\prime \prime}_-) \cong \C[a'] \otimes \C[b''].$$

We next note that the split coordinates and the product coordinates are related by 
\begin{lem} \label{changeofcoordinatesforW}
\begin{align*}
a' = & \frac{1}{\sqrt{2}} (s + i t)\\
b'' = & \frac{i}{\sqrt{2}} (s - it)
\end{align*}
\end{lem}
We now have 
\begin{lem}\label{Wintertwiner}
There is a $\mathfrak{u}(W)$-equivariant mapping  $B_W: \Pol(W^{\prime_0}) \to \mathcal{P} (E)$ satisfying
\begin{enumerate}
\item $B_W(1)  = \varphi_0$
\item  $B_W \circ s \circ B_W^{-1} = x - \frac{1}{\pi} \frac{\partial}{\partial x}$
\item $B_W \circ t \circ B_W^{-1} = y - \frac{1}{\pi} \frac{\partial}{\partial y}$
\item $B_W \circ \frac{1}{\pi}\frac{\partial}{\partial s} \circ B_W^{-1} = x + \frac{1}{\pi} \frac{\partial}{\partial x}$
\item $B_W \circ \frac{1}{\pi} \frac{\partial}{\partial t} \circ B_W^{-1} = y + \frac{1}{\pi} \frac{\partial}{\partial y}$
\end{enumerate}
\end{lem} 

Hence by Lemma \ref{changeofcoordinatesforW} we have 

\begin{lem}\label{complexintertwiner}
\hfill

\begin{enumerate}
\item  $B_W \circ a' \circ B_W^{-1} = \frac{1}{\sqrt{2}} (z - \frac{1}{\pi} \frac{\partial}{\partial \overline{z}})$
\item $B_W \circ b'' \circ B_W^{-1} = \frac{1}{\sqrt{2}} (\overline{z} - \frac{1}{\pi} \frac{\partial}{\partial z})$
\end{enumerate}
\end{lem}

\subsection{The Schr\"odinger and Fock models for $\U(V) \times \U(W)$ }

We now describe and compare two different realizations of the infinitesimal Weil representation associated to the pair $\U(V) \times \U(W)$, the `` split Schr\"odinger model'' (there is another Schr\"odinger model for Hermitian spaces, the real points in $V \otimes W$, that we will not use here, but it is the one used in Chapter VIII of \cite{BorelWallach})  and the Fock model (with two sets of coordinates). So we now need to bring the space $V$ into the picture. We will be brief since all the essential ideas are contained in the previous section.

\subsection*{The split Schr\"odinger model for $\U(V) \times \U(W)$} The split Schr\"odinger model  is realized in the Schwartz space $\mathcal{S} ((V \otimes E)_{\R})$ using  the polarization 
$$(V \otimes W)_{\R} = (V \otimes E)_{\R} + (V \otimes F)_{\R}$$
 inherited from that of $W$ in Equation \eqref{Wsplitting}. 
Here the tensor product is over $\C$.

Recall that throughout the body of the paper we have used a basis $\{v_j: 1 \leq j \leq p+q\}$ for $V$. Hence we have a basis 
$\{v_j \otimes e_1 :1 \leq j \leq p+q\}$ for the complex vector space $V \otimes E$ and noting that
$i(v_j \otimes e_1) = v_j \otimes e_2$ we obtain a basis 
$\{v_j \otimes e_1,v_j \otimes e_2 : 1 \leq j \leq p+q\}$ for the underlying real vector space $(V \otimes E)_{\R}$. We let 
$\{x_j,y_j : 1 \leq j \leq p+q\}$ be the corresponding coordinates.   We will  also use the {\it complex} coordinates $z_j =x_j + i y_j, 1 \leq j \leq p+q$. We will regard a function in  
$\mathcal{S} ((V\otimes E)_{\R})$ as a function of the $z_j$'s (and their complex conjugates).  Once again we have the space
$\mathcal{P}(V \otimes E)$ consisting of the product of complex-valued polynomials in $x_j,y_j, 1 \leq j \leq p+q,$ with the Gaussian
$\varphi_0 = \exp{(- \pi(\sum \limits_{j=1}^{p+q} x^2_j + y^2_j))}$ or equivalently the  product of polynomials in $z_j,
\overline{z}_j, 1 \leq j \leq p+q,$ with the Gaussian
$\varphi_0 = \exp{(- \pi(\sum\limits_{j=1}^{p+q} z_j  \overline{z}_j))}$.

This is the model where the cocycles of Kudla-Millson $\varphi_{q,q}$ were originally defined,  see  Proposition 5.2 of \cite{KudlaMillson1} or  pg. 148 of  \cite{KudlaMillson3}. To state their formula we need more notation.  In what follows $A(\xi'_{\alpha,\mu})$ resp. $A(\xi''_{\alpha,\mu})$ will denote the operation of left exterior multiplication by $\xi'_{\alpha,\mu}$, resp. $\xi''_{\alpha,\mu}$.

For $\mu$ with $p+1 \leq \mu \leq p+q $ define 
\begin{equation} \label{Doperator}
D^+_{\mu} = \sum_{\alpha =1}^p \bigg( \big(\overline{z}_{\alpha}  - 
\frac{1}{\pi} \frac{\partial}{\partial z_{\alpha}}\big)
 \otimes A(\xi'_{\alpha,\mu}) \bigg)
\end{equation}

and 

\begin{equation} \label{barDoperator}
\overline{D}^+_{\mu} = \sum_{\alpha =1}^p \bigg( \big(z_{\alpha}  - 
\frac{1}{\pi} \frac{\partial}{\partial \overline{z}_{\alpha}}\big)
 \otimes A(\xi''_{\alpha,\mu}) \bigg)
\end{equation}

The formula of Kudla and Millson is then 
\begin{equation} \label{KM}
\varphi_{q,q}  = \frac{1}{2^{2q}} \bigg( \big(\prod_{\mu= p+1}^{p+q} D^+_{\mu} \big)
\circ \big(\prod_{\nu= p+1}^{p+q} \overline{D}^+_{\nu} \varphi_0\big) \bigg) \bullet \varphi_0.
\end{equation}

\subsection*{The Fock model for $\U(V) \times \U(W)$} 
The second realization of the Weil representation is the polynomial Fock model $\mathrm{Pol} ((V \otimes W)^{\prime_0})$
considered in this paper. In what follows we will not need the entire space $(V \otimes W)^{\prime_0}$ but only the subspace $(V_+ \otimes W)^{\prime_0}$.  We will give two bases for $(V_+ \otimes W)^{\prime_0}$.  Our computations are then   simplified by   
\begin{lem}\label{simplificationofprojections}
$$ p_{V \otimes W}^{\prime_0}| \big((V_+ \otimes W)\otimes_{\R} \C \big) = I_{V_+} \otimes p_W^{\prime_0}$$
or in more concise form
$$ (v \otimes w)^{\prime_0} = v \otimes w^{\prime_0}, \ \text{for} \ v \in V_+.$$
\end{lem}
\begin{proof}
By definition
$$p_{V \otimes W}^{\prime_0} = \frac{1}{2}\big( I_V \otimes I_W \otimes 1 - \theta_V \otimes (J_W \circ \theta_W) \otimes i \big)$$
and hence
$$p_{V \otimes W}^{\prime_0}| \big((V_+ \otimes W)_{\R} \otimes_{\R} \C \big) =   \frac{1}{2}\big( I_{V_+} \otimes I_W \otimes 1 - I_{V_+} \otimes (J_W \circ \theta_W) \otimes i \big) = I_{V_+}  \otimes \big(I_W \otimes 1 - J_{W,0} \otimes i \big).$$
\end{proof}

\medskip
\noindent
\begin{rem}
Lemma \ref{simplificationofprojections} implies  that we can  carry over the computations of the previous section to the ones we need  by simply ``tensoring with the standard basis for $V_+$ ''. 
\end{rem}
\medskip
Now recall that 
\begin{equation} \label{factoringofFockone}
\Pol((V_+ \otimes W)^{\prime_0}) = \Pol(V'_+ \otimes W'_+) \otimes \Pol(V''_+ \otimes W''_-)
\end{equation}

The first basis for $(V_+ \otimes W)^{\prime_0}$ is adapted to the split Schr\"odinger model.
It is  $\{(v_{\alpha} \otimes e^{\prime_0}_1,v_{\alpha} \otimes e^{\prime_0}_2 : 1 \leq \alpha \leq p \}$. 
We  define the  coordinates  $\{s_{\alpha} ,t_{\alpha}  : 1 \leq {\alpha}  \leq p\}$ to be the coordinates associated to this 
basis.  We will again call these coordinates split coordinates .
The second basis for $(V_+ \otimes W)^{\prime_0}$ is $\{(v_{\alpha}  \otimes \epsilon^{\prime_0}_1,v_{\alpha}  \otimes \epsilon^{\prime_0}_2 : 1 \leq {\alpha}  \leq p \}$
We let $\{a'_{\alpha} ,b''_{\alpha}  : 1 \leq {\alpha} \leq p\}$ be the corresponding coordinates with the same explanation for the name and the superscripts as  before. Thus the tensor product in Equation \eqref{factoringofFockone} corresponds to the tensor product decomposition 
\begin{equation}\label{factoringofFocktwo}
\Pol({V_+ \otimes W}^{\prime_0})\cong \C[a'_1,\cdots,a'_p] \otimes \C[b''_1,\cdots,b''_p]
\end{equation} 

We next note that the split coordinates and the product coordinates are related by 
\begin{lem} \label{changeofcoordinatesforVtensorW}
\begin{align*}
a'_{\alpha}  = & \frac{1}{\sqrt{2}} (s_{\alpha}  + i t_{\alpha} )\\
b''_{\alpha}  = & \frac{i}{\sqrt{2}} (s_{\alpha}  - it_{\alpha} )
\end{align*}
\end{lem}

As before we have
\begin{lem}\label{intertwiner}
There is a $\mathfrak{u}(V)\times \mathfrak{u}(W)$-equivariant embedding (not onto) $B_{V_+ \otimes W}: \Pol({V_+ \otimes W}^{\prime_0})  \to \mathcal{P} (V \otimes E)$ satisfying
\begin{enumerate}
\item $B_{V_+ \otimes W}(1 \otimes 1)  = \varphi_0$
\item  $B_{V_+ \otimes W} \circ s_{\alpha}  \circ B_{V_+ \otimes W}^{-1} = x_{\alpha}  - \frac{1}{\pi} \frac{\partial}{\partial x_{\alpha} }$
\item $B_{V_+ \otimes W} \circ t_{\alpha}  \circ B_{V_+ \otimes W}^{-1} = y_{\alpha}  - \frac{1}{\pi} \frac{\partial}{\partial y_{\alpha} }$
\item $B_{V_+ \otimes W} \circ \frac{1}{\pi} \frac{\partial}{\partial s_{\alpha} } \circ B_{V_+ \otimes W}^{-1} = x_{\alpha}  + \frac{1}{\pi} \frac{\partial}{\partial x_{\alpha} }$
\item $B_{V_+ \otimes W} \circ\frac{1}{\pi}  \frac{\partial}{\partial t_{\alpha} } \circ B_{V_+ \otimes W}^{-1} = y_{\alpha}  + \frac{1}{\pi} \frac{\partial}{\partial y_{\alpha} }$
\end{enumerate}
\end{lem} 

Hence by Lemma \ref{changeofcoordinatesforVtensorW} we have 

\begin{lem}\label{complexintertwiner}
\hfill

\begin{enumerate}
\item  $B_{V_+ \otimes W} \circ a'_{\alpha}  \circ B_{V_+ \otimes W}^{-1} = \frac{1}{\sqrt{2}}(z_{\alpha} - \frac{1}{\pi} \frac{\partial}{\partial \overline{z_{\alpha} }})$
\item $B_{V_+ \otimes W} \circ b''_{\alpha}  \circ B_{V_+ \otimes W}^{-1} = \frac{1}{\sqrt{2}}(\overline{z}_{\alpha}  - \frac{1}{\pi} \frac{\partial}{\partial z_{\alpha} })$
\end{enumerate}
\end{lem} 

We now define operators $C^+_{\mu}, p+1 \leq \mu \leq p+q$ and $\overline{C}^+_{\mu}, p+1 \leq \mu \leq p+q$ by 
\begin{enumerate}
\item $C^+_{\mu} = \sum\nolimits_{\alpha= 1}^p a'_{\alpha} \otimes A(\xi^{\prime \prime}_{\alpha,\mu})$
\item $\overline{C}^+_{\mu} = \sum\nolimits_{\alpha= 1}^p b''_{\alpha} \otimes A(\xi^{\prime }_{\alpha,\mu})$
\end{enumerate}

We leave the  proof of the  next lemma to reader, see Lemma \ref{formulaforvarphi} for (1) and Lemma \ref{formulaforvarphitwo} for (2).

\begin{lem}
\hfill

\begin{enumerate}
\item $\big( \prod\limits_{\mu = p+1}^{p+q} C^+_{\mu}\big) ( 1) = \psi_{0,q}$
\item $\big( \prod\limits_{\mu = p+1}^{p+q}\overline{C}^+_{\mu}\big)(1) = \psi_{q,0}$
\item $ \big( \prod\limits_{\mu = p+1}^{p+q} C^+_{\mu}\big)\circ \big( \prod\limits_{\mu = p+1}^{p+q}\overline{C}^+_{\mu}\big)(1)
= \psi_{0,q} \wedge \psi_{q,0}$
\end{enumerate}
\end{lem}

Again we leave the proof of the next lemma to the reader. 
\begin{lem}
\hfill

\begin{enumerate}
\item $B_{V_+ \otimes W} \circ \big(\prod\limits_{\mu = p+1}^{p+q} C^+_{\mu}\big) \circ B^{-1} _{V_+ \otimes W} 
= \frac{1}{2^{q/2}}(\prod\limits _{\mu= p+1}^{p+q} D^+_{\mu} )$
\item $B_{V_+ \otimes W} \circ \big( \prod\limits_{\mu = p+1}^{p+q}\overline{C}^+_{\mu}\big) \circ B^{-1} _{V_+ \otimes W} = \frac{1}{2^{q/2}}(\prod\limits _{\mu= p+1}^{p+q} \overline{D}^+_{\mu}) $
\end{enumerate}
\end{lem}

We can now prove the local product formula.

\begin{prop}\label{productformulafor11}
\hfill

$( B_{V_+ \otimes W}  \otimes 1)  (\psi_{q,0} \wedge \psi_{0,q} )  =  2^q \ \varphi_{q,q}$
\end{prop}
\begin{proof}
\begin{align*}
&( B_{V_+ \otimes W} \otimes 1)  (\psi_{q,0} \otimes \psi_{0,q}) = \big((B_{V_+ \otimes W} \otimes 1) \bigg( \big(\prod\limits_{\mu = p+1}^{p+q} C^+_{\mu}\big) \circ   \big( \prod\limits_{\mu = p+1}^{p+q}\overline{C}^+_{\mu}\big)(1)\bigg) \\
& = \big((B_{V_+ \otimes W} \otimes 1) \circ \big(\prod\limits_{\mu = p+1}^{p+q} C^+_{\mu}\big) \circ   \big( \prod\limits_{\mu = p+1}^{p+q}\overline{C}^+_{\mu}\big)
\circ (B_{V_+ \otimes W}^{-1}\otimes 1)\big)(B_{V_+\otimes W} \otimes 1)(1))\\
&= \frac{1}{2^q} \big(\prod\limits _{\mu= p+1}^{p+q} D^+_{\mu} \big)\circ (\big(\prod\limits _{\mu= p+1}^{p+q} \overline{D}^+_{\mu} \big))  (\varphi_0 ) = 2^q  \ \varphi_{q,q}.
\end{align*} 
\end{proof}

\begin{rem} \label{KMdoesnotfactor}
We warn the reader that the KM-cocycle $\varphi_{q,q}$  does not factor in the split Schr\"odinger model.  The cocycles $\varphi_{q,0}: = 
(B_{V_+ \otimes W} \otimes 1)(\psi_{q,0})$ and $\varphi_{0,q}: = 
(B_{V_+ \otimes W} \otimes 1) (\psi_{0,q})$ both exist in the split Schr\"odinger model (of course) but it is not true that we have
$ \varphi_{q,0} \wedge  \varphi_{0,q} = \varphi_{q,q}$.  The problem  is that the product in $\mathcal{P}(V \otimes E)$ is induced by 
 internal i.e the usual multiplication of functions, whereas the product in $\Pol(V'_+\otimes  W'_+) \otimes \Pol(V''_+ \otimes W''_-)$ is external i.e. the tensor product multiplication.  For example, when $q=1$, we have 
\begin{align*}
(B_{V_+ \otimes W} \otimes 1)(\psi_{1,0}) = &(B_{V_+ \otimes W} \otimes 1)\big(\sum\limits_{\alpha = 1}^p \overline{z}_{\alpha} \otimes \xi'_{\alpha,p+1}\big ) = \sqrt{2} \varphi_0 (\sum\limits_{\alpha = 1}^p \overline{z}_{\alpha}  \otimes \xi'_{\alpha,p+1})\\
(B_{V_+ \otimes W} \otimes 1)(\psi_{0,1}) = &(B_{V_+ \otimes W} \otimes 1)\big(\sum\limits_{\alpha = 1}^p 
z_{\beta} \otimes \xi'_{\beta,p+1}\big ) = \sqrt{2} \varphi_0 (\sum\limits_{\beta = 1}^p z_{\beta}  \otimes \xi''_{\beta,p+1})
\end{align*}
and hence
$$((B_{V_+ \otimes W} \otimes 1)(\psi_{1,0})\wedge ((B_{V_+ \otimes W} \otimes 1)(\psi_{0,1}) = 2(\varphi_0)^2 (\sum\limits_{\alpha,\beta = 1}^p\overline{z}_{\alpha} z_{\beta}  \otimes \xi'_{\alpha,p+1} \wedge \xi''_{\beta,p+1})$$
whereas
$$ \varphi_{1,1} = 2\varphi_0 \ (\sum\limits_{\alpha,\beta = 1}^p\overline{z}_{\alpha} z_{\beta}  \otimes \xi'_{\alpha,p+1} \wedge \xi''_{\beta,p+1}).$$
\end{rem}

\bibliography{bibli}

\bibliographystyle{plain}

\end{document}